\documentclass[reqno,draft]{amsart}
\usepackage{amsaddr}

\RequirePackage{amsthm,amsmath,amsfonts,amssymb}
\RequirePackage[numbers]{natbib}
\RequirePackage[colorlinks,citecolor=blue,urlcolor=blue]{hyperref}
\RequirePackage[capitalise]{cleveref}
\RequirePackage{graphicx}

\numberwithin{equation}{section}

\theoremstyle{plain}

\usepackage{thmtools,thm-restate}
\newtheorem{theorem}{Theorem}[section]
\newtheorem{lemma}[theorem]{Lemma}
\newtheorem{proposition}[theorem]{Proposition}

\newtheorem{conjecture}[theorem]{Conjecture}

\theoremstyle{remark}
\newtheorem{remark}[theorem]{Remark}

\newtheorem{example}[theorem]{Example}
\newtheorem{definition}[theorem]{Definition}

\usepackage[margin=1.5in,includefoot,footskip=30pt]{geometry}

\usepackage[utf8]{inputenc} % So können Umlaute auch direkt eingegeben werden.
\usepackage{IEEEtrantools}
\usepackage{stmaryrd} % \mapsfrom
\usepackage{mathtools} % noch mehr praktisches Mathezeugs
\usepackage{dsfont}
\usepackage[english]{babel}
\usepackage{xcolor}
\usepackage{subcaption}
\usepackage{verbatim}
\usepackage{faktor}
\usepackage{caption}
\usepackage{subcaption}
\usepackage{tikz}
\usetikzlibrary{math,calc,intersections,decorations.pathreplacing,fit,positioning,backgrounds,arrows.meta}
\usepackage{adjustbox}
\usepackage{enumitem}
\setenumerate[0]{label=(\roman*)}
\setlist{topsep=8pt,itemsep=4pt,partopsep=4pt, parsep=4pt}
\usepackage{scalerel}
\usepackage{bm}

\crefname{equation}{}{}
\Crefname{equation}{Eq.}{Eqs.}
\crefname{assumption}{Assumption}{Assumptions}

\newcommand{\R}{\mathbb{R}} % reelle
\newcommand{\Q}{\mathbb{Q}} % rationale
\newcommand{\Z}{\mathbb{Z}} % ganze
\newcommand{\N}{\mathbb{N}} % natuerliche
\newcommand{\A}{\mathcal{A}} % sigma algebra
\renewcommand{\P}{\mathbb{P}} % W-Maß
\newcommand{\E}{\mathbb{E}} % expectancy
\newcommand*\diff{\mathop{}\!\mathrm{d}}

\newcommand{\tendsto}[1]{\stackrel{#1}{\longrightarrow}}
\newcommand{\id}{\mathrm{id}}
\newcommand{\ind}{\mathds{1}}
\newcommand{\e}{\mathrm{e}}

\DeclareMathOperator{\dom}{dom}

\DeclareMathOperator{\ch}{ch}
\DeclareMathOperator{\pr}{pr}

\newcommand{\fpart}{\Pi} % \Pi, P, \eta
\newcommand{\rates}{R^\circ} %R
\newcommand{\nurates}{R} %\widetilde{R}
\newcommand{\mergers}{M}
\DeclareMathOperator{\lf}{lf}
\DeclareMathOperator{\rt}{rt}
\DeclareMathOperator{\nd}{nd}
\DeclareMathOperator{\dcb}{dec}
\DeclareMathOperator{\dct}{tm}
\DeclareMathOperator{\dcs}{sp}
\newcommand{\kbr}{K} % stands for k_ernel br_idge
\newcommand{\semi}{S}

\newcommand{\rfy}[1]{\widetilde{{#1}}} % for an Rd-version of something defined for the torus.

\DeclareMathOperator{\fr}{Fr}
\DeclareMathOperator{\tm}{Tm}
\let\sp\relax % workaround because \RedeclareMathOperator doesnt exist
\DeclareMathOperator{\sp}{Sp}
\DeclareMathOperator{\dc}{Dec}
\newcommand{\pth}{\mathrm{Path}}

\newcommand{\smashop}[1]{\,\,\smashoperator{#1}\,\,}
\newcommand{\wo}{\mathbin{{\setminus}\mspace{-5mu}{\setminus}}}
\newcommand{\ftm}{f_\mathrm{tm}^{\boldsymbol \lambda}}
\newcommand{\fsp}{f_\mathrm{sp}}

\newcommand{\fnu}{f_{\boldsymbol \nu}}
\newcommand{\interp}{\overline{x}} % for lem:NFaux

\newcommand{\nn}{\widetilde{N}} % replacement for \beta
\newcommand{\la}{\beta} % replacement for \lambda
\newcommand{\nut}[1]{|\nu_{#1}|}

\newcommand{\nkmap}{\kappa} % \phi \chi \eta \theta
\newcommand{\labelinv}{\mathcal{I}} % set of bijections
\newcommand{\nunk}{\nu_{n,\smash{\vec{k}}}}
\newcommand{\gnk}{\nu_{n,\smash{\vec{k}}}}
\newcommand{\lank}{\lambda_{n,\smash{\vec{k}}}}

\newcommand{\cnk}{c}%{c_{n,\smash{\vec{k}}}}

\newcommand{\treethree}[1]{%
  \vcenter{\hbox{%
    \begin{tikzpicture}[scale=0.10, line width=.5pt]
      \draw[#1] (0,0) -- (0,1);
      \draw[#1] (1,0) -- (1,1);
      \draw[#1] (2,0) -- (2,2);
      \draw[#1] (0,1) -- (1,1);
      \draw[#1] (.5,1) -- (.5,2);
      \draw[#1] (.5,2) -- (2,2);
      \draw[#1] (1.25,2) -- (1.25,3);
    \end{tikzpicture}%
  }}%
}

\usepackage{accents}
\newcommand{\rs}[1]{\accentset{\circ}{#1}} % for restricting forests and decorations from n + 1 leaves to n leaves.
\newcommand{\rsto}[1]{\mathord\downarrow_{#1}}

\def\step{0.015}
\def\pt{28.45274}
\def\sig{0.1}
\newcommand\BrownianBridgeX{} % just for safety
\def\BrownianBridgeX[#1](#2)(#3){%
    \tikzmath{
        int \n;
        coordinate \pstart, \pend;
        \pstart = (#2);
        \pend = (#3);
        \n = {((\pendx-\pstartx)/(\step*\pt))};
        \n = \n + 1;
    }
    \path (#2)
    \foreach \x [count=\i] in {1,...,\n} {
        to ++(\step,rand*\sig)
        node (A\i) {}
    } node (END) {};
    \draw[#1] (#2)
    \foreach \i in {1,...,\n} {
        let \p1=(#2), \p2=(#3), \p3 = (A\i), \p4 = (END) in to (\x3,{\y3+(\i*\step*\pt/(\x2-\x1))*(\y2-\y4)})
    };
}
\usepackage{xparse}
\NewDocumentCommand{\BBY}{ommoo}{%
    \IfValueT{#4}{\pgfmathsetseed{#4}}
    \IfValueTF{#5}{\def\sig{#5}}{\def\sig{0.07}}
    \tikzmath{
        int \n;
        coordinate \pstart, \pend;
        \pstart = (#2);
        \pend = (#3);
        \n = {((\pendy-\pstarty)/(\step*\pt))};
        \n = \n + 1;
    }
    \path (#2)
    \foreach \x [count=\i] in {1,...,\n} {
        to ++(rand*\sig,\step)
        node (A\i) {}
    } node (END) {};
    \IfValueTF{#1}{\draw[#1]}{\draw} (#2)
    \foreach \i in {1,...,\n} {
        let \p1=(#2), \p2=(#3), \p3 = (A\i), \p4 = (END) in to ({\x3+(\i*\step*\pt/(\y2-\y1))*(\x2-\x4)},\y3)
    };
}

\title{The Brownian Spatial Coalescent}
\author{Peter Koepernik\textsuperscript{1,2}}
\thanks{\textsuperscript{2}The author gratefully acknowledges support from an EPSRC grant EP/W523781/1.}
\address{\textsuperscript{1}Department of Statistics, University of Oxford, 24–29 St Giles’, Oxford, OX1 3LB, United Kingdom}
\email{peter.koepernik@stats.ox.ac.uk}
\date{16 January 2024}
 %sampling consistency, stationary

\begin{document}

\begin{abstract}

    We introduce a class of Markov coalescent processes on the continuous $d$-dimensional torus, in the most general setting of simultaneous multiple mergers, called the Brownian spatial coalescent. It is axiomatically defined through a property that is satisfied by the genealogies of any population model in which individuals follow independent Brownian motions forwards in time, regardless of the branching mechanism.

    We prove that a Brownian spatial coalescent is characterised by a set of \emph{transition measures}, reminiscent of the transition rates that characterise a non-spatial coalescent.
    We prove that it is sampling consistent in a suitable sense if and only if all transition measures are uniform with intensity given by the transition rates of a $\Xi$-coalescent. %, or a $\Lambda$-coalescent if simultaneous mergers are not allowed.
    This defines the \emph{Brownian spatial $\Xi$-coalescent}, which we show describes the genealogies of neutral population models with Brownian movement in the limit of large population size, and in particular those of the $\Xi$-Fleming-Viot process---a generalisation of the well-known Fleming-Viot process---at stationarity.
    An important consequence of our results is that all spatial population models in which individuals follow independent Brownian motions and the branching mechanism is not neutral, that is, depends non-trivially on the spatial distribution, for example through local regulation, have non-Markovian genealogies.

    Byproducts of our results include explicit formulas for samples from the stationary distribution of a $\Xi$-Fleming-Viot process, and a representation of the backward dynamics of lineages in terms of Brownian motions with coupled drift. This includes calculations of the drift that leads to multiple or even simultaneous mergers in any dimension.

\end{abstract}

\maketitle

%\newpage
%\tableofcontents

\section{Introduction}
%todo cite wright. From nicks email: Wright’s (1943, 1946 Genetics) papers are quite impenetrable, but still worth citing, as the first to calculate the distribution of coalescence times in spatially continuous populations.  For nearby genes in 2D, this is ~1/t, and leads to the same approximations as Mal\'ecot (though that isn't at all obvious).
%todo another thing to cite. from nicks email: Barton & Wilson (1995) explore the model where total population size is kept constant; this is a sensible model for evolution of a neutral quantitative trait, but as a spatial model, leads to extreme clumping (as pointed out by Felsenstein). The paper also includes an attempt to explain Wright’s reasoning, and unsuccessful efforts to justify his approximation.  I'm mentioning this mainly because it might be helpful to introduce this simple neutral model early on, to contrast with models with (some) local regulation.
Coalescent processes arise when describing the genealogies, that is family relations, of a biological population by tracing the ancestries of a sample from the population backwards in time. If individuals in the population have exactly one parent, such as in haploid asexual populations, then every sampled individual has a single ancestral lineage going backwards in time, and the lineages associated with a set of individuals coalesce when their common ancestor is reached. We will not consider diploid populations, but remark that they can be treated as haploid populations of twice the size if one is only interested in a single gene, see e.g.~\cite{E2008}, Remark 2.1.
The most basic coalescent process is Kingman's coalescent, introduced in 1982 \cite{kingman1982}, in which every pair of lineages coalesces at constant rate, independently of all other pairs.

In this project, we are ultimately interested in the genealogies of \emph{structured} populations, where every individual has an evolving type or location, but we begin with a brief overview of some classical theory in the unstructured case.
%We begin with a brief overview of some classical theory.
The reader familiar with the field should feel free to skim \cref{sec:coalintro}.

\subsection{Coalescents}\label{sec:coalintro}
Mathematically, a coalescent is typically described by giving disjoint, set-valued labels to the lineages, which merge when the lineages coalesce. Denote by $\mathcal{P}$ the set of partitions of finite subsets of $\N$.

\begin{definition}\label{def:NSCPintro}
    A \emph{coalescent} is a $\mathcal{P}$-valued right-continuous Markov process $(\fpart_t)_{t \ge 0}$ whose transitions can be obtained by merging partition elements. It is called \emph{label invariant} if its law is invariant under changing the labels of the initial set of lineages.\footnote{In such a way that does not necessarily preserve their size. For example, the substitution $\left\{ 1,3 \right\} \to \left\{ 3,4,9 \right\} $ is allowed.}
    %and which is independent of particle labels in the sense that the counting process $(|\fpart_t|)_{t \ge 0}$ is also a Markov process.
\end{definition}
A precise definition is \cref{def:NSCP}. The label invariance condition is equivalent to requiring the block counting process $(|\fpart_t|)_{t \ge 0}$ to be Markovian. % second condition is equivalent to invariance of the process under reassigning labels of the initial set of lineages (not necessarily preserving their cardinality). %, including, say $\left\{ 1,2 \right\} \to \left\{ 3,5,9 \right\} $.
Another equivalent phrasing is that for any $n,m\in \N$ and $k_1 \ge \ldots \ge k_m \ge 2$ with $\sum_{i=1}^m k_i \le n$, every coalescence event in which $m$ disjoint sets of lineages of sizes $k_1$ to $k_m$ merge simultaneously while there are currently $n$ lineages---called an $(n,\vec{k})$-merger---happens at the same rate, which we denote by $\lank = \lambda_{n,k_1, \ldots ,k_m}$. % The Markov assumption is satisfied if the coalescent describes genealogies of a population in which subsequent generations are independent.

Not every coalescent in the sense of \cref{def:NSCPintro} arises from the genealogies of a population model, and a natural necessary condition is \emph{sampling consistency}: if we construct the genealogical tree corresponding to a sample of size $n+1$, then the induced genealogical tree for a subsample of size $n$ will have the same distribution as the tree obtained by constructing the coalescent directly from the subsample. %; see \cref{def:NSCP2}.
An equivalent definition is that there exists a coalescent started from all of $\N$ such that the coalescent started from a finite subset $A \subset \N$ is its restriction to $A$.

\begin{theorem}[\cite{DK99,pitmanlambda,sagitovlambda}]\label{thm:lambdaintro}
    A label invariant coalescent with no simultaneous mergers, that is $\lambda_{n,k_1, \ldots ,k_m} = 0$ whenever $m > 1$, is sampling consistent if and only if there exists a finite measure $\Lambda$ on $[0,1]$ such that \[
        \lambda_{n,k} = \int_0^1 p^{k-2} (1-p)^{n-k} \Lambda(\diff p),\qquad 2 \le k \le n.
    \] It is called the \emph{$\Lambda$-coalescent}.
\end{theorem}

The Kingman coalescent arises as a special case if $\Lambda$ is the unit mass at zero. If $\Lambda$ is the uniform measure on $[0,1]$, then the $\Lambda$-coalescent corresponds to the clustering process constructed by Bolthausen and Sznitman \cite{bolthausensznitman}.
A characterisation in the most general case of simultaneous multiple mergers is also known. Denote by $\triangle = \{\boldsymbol \xi = (\xi_1,\xi_2,\ldots ) \colon \xi_1 \ge \xi_2 \ge \ldots \ge 0, |\boldsymbol \xi| \le 1\}$ the infinite simplex, where $|\boldsymbol \xi| = \sum_{i=1}^\infty \xi_i$. Write $(u,v) = \sum_{i=1}^\infty u_iv_i$ for $u,v\in \triangle$.

\begin{theorem}[\cite{schweinsbergxi,haploid}]\label{thm:xiintro}
    A label invariant coalescent is sampling consistent if and only if there exists a finite measure $\Xi = a\delta_0 + \Xi_0$ on $\triangle$, where $\delta_0$ is the unit mass at zero and $\Xi_0$ has no atom at zero, such that $\lambda_{n,2} = a$ for all $n\ge 2$, and for any $(n,\vec{k}) \neq (n,2)$,
    \begin{equation}\label{eq:xirates}
        \lank = \int_\triangle \sum_{l=0}^s \binom{s}{l}  (1 - |\boldsymbol \xi|)^{s-l} \smashoperator{\sum_{i_1 \neq \ldots \neq i_{m+l}}} \xi_{i_1}^{k_1}\ldots \xi_{i_m}^{k_m} \xi_{i_m+1}\ldots \xi_{i_{m+l}} \frac{\Xi_0(\diff \boldsymbol \xi)}{(\boldsymbol \xi,\boldsymbol \xi)} ,%+ a \ind_{\{m=1,k_1=2\}},
    \end{equation}
    where $s = n - \sum_i k_i$. It is called the \emph{$\Xi$-coalescent}.
\end{theorem}

The $\Lambda$-coalescent arises as a special case if $\Xi$ is the pushforward of $\Lambda$ under the map $p \mapsto (p,0,\ldots )$ from $[0,1]$ to $\triangle$. We introduced sampling consistency as a necessary condition for a coalescent to describe the genealogies of some population, but it is also sufficient. M\"ohle and Sagitov \cite{haploid,sagitovxi2} give %necessary and sufficient
conditions for the genealogies of a haploid population evolving according to Cannings model---that is, a constant size population evolving in discrete generations with an exchangeable offspring mechanism---to converge as the population size tends to infinity when properly rescaled. The class of resulting coalescents is exactly the class of $\Xi$-coalescents.
%Necessary and sufficient
Conditions for the genealogies to converge to Kingman's coalescent were previously discovered by Kingman \cite{kingman1982b} with later improvements by M\"ohle \cite{moehle98,moehle99}. Analogous results for the $\Lambda$-coalescent were independently discovered by Donnelly and Kurtz~\cite{DK99} and Sagitov \cite{sagitovlambda}.

%hugely succesful bla. However, app limited because no space

\subsection{Spatial Coalescents}\label{sec:spatialcoalintro}
The $\Lambda$- and $\Xi$-coalescents %have been very successful, and
appear as a universal limiting object for the genealogies of a wide range of population models %DK99,pitmanlambda,haploid,BL2003,
\cite{bbschweinsberg,xiexamples1,schweinsbergGW,xiexamples2,xiexamples3,lambdadormancy}, including diploid populations \cite{diploid1,diploid2,diploid3} and populations where offspring sizes are not exchangeable~\cite{xicanningsnotexchangeable}.
However, they cannot fully describe the genealogies of a \emph{structured} population, where every individual has a \emph{type} or \emph{location} evolving in some space $E$. Typically, type is the appropriate interpretation in a genetic model, and location is appropriate in a model of a geographically dispersing population. We usually refer to spatial locations from now on, but the genetic interpretation remains valid throughout.
On top of the set of lineages $\fpart_t$, an associated coalescent process needs to keep track of the spatial location $\boldsymbol X_t(u)$ of every lineage $u\in \fpart_t$, which we summarise in a map $\boldsymbol X_t\colon \fpart_t \to E$.

\begin{definition}\label{def:coalintro}
    A \emph{spatial coalescent} is a right-continuous Feller Markov process $(\fpart_t,\boldsymbol X_t)_{t \ge 0}$ such that $\fpart_t \in \mathcal{P}$ and $\boldsymbol X_t \colon \fpart_t \to E$ for all $t \ge 0$, and all transitions in $(\fpart_t)$ can be obtained by merging partition elements. It is called \emph{label invariant} if its law is invariant under changing the labels of the initial set of lineages.
\end{definition}
A precise definition is in \cref{def:coal,def:labelinv}. The Feller property---which is automatic in the non-spatial setting---is a form of continuity in the initial condition that, together with right-continuity of sample paths, implies the strong Markov property~\cite[Thm.\ 8.3]{rogers}. The label invariance assumption is equivalent to the requirement that the empirical measure $\sum_{u\in \fpart_t} \delta_{\boldsymbol X_t(u)}$ is a Markov process, which was pointed out to us by F\'elix Foutel-Rodier.

%in the spatial setting, the picture is much less complete. Imagine, for example, a population model in which individuals move as independent BMs forward in time, while undergoing branching according to some unspecified mechanism.
%for example the class of models considered by DK falls into this category, containing the well known Fleming-Viot process as a special case.

%Describing the genealogies of such a population model essentially requires time reversing the process: Given a finite sample of the population model at present (with no knowledge of their history or the locations of the remaining population), how do their lineages move through space backwards in time? Since Brownian motion is reversible, they will presumably follow Brownian motions, but there must be a coupled drift because independent Brownian motions don't meet in d \ge 2, and even in d = 1 only in pairs. But what kind of drift leads to binary mergers in d \ge 2, or multiple or even simultaneous mergers in any dimension? These are interesting and basic, yet unanswered questions, and answers for a class of spatial population models will be a special case of our results.

The picture in the spatial setting is much less complete than in the non-spatial setting. Consider, for instance, a population model in which individuals move as independent Brownian motions on the $d$-dimensional (flat) torus $E = \mathbb{T}^d = \R^d / \Z^d$ forwards in time, while undergoing branching according to some unspecified (think general) mechanism. The well-known Fleming-Viot process and its $\Lambda$- and $\Xi$-generalisations \cite{DK96,BL2003,xiflemingviot}, as well as the class of models considered by Donnelly and Kurtz in their lookdown papers \cite{DK96,DK99}, when specialised to Brownian movement, fall into this category. % glossary for E
Describing the genealogies of such a population model essentially requires partially time reversing the process: given a finite sample from the population at present together with their spatial locations (with no knowledge of their history, or the spatial distribution of the remaining population), how do their lineages move through space backwards in time? Since Brownian motion is reversible, they will presumably follow Brownian motions, but there would have to be a coupled drift; if they were independent, lineages would never meet in dimensions $d \ge 2$, and even in $d = 1$ only in pairs. But what kind of drift leads to binary mergers in $d \ge 2$, or multiple, even simultaneous mergers in any dimension? These are interesting and basic questions, yet generally unanswered to the best of our knowledge. Our results will provide some answers (the curious reader may take a look at \cref{ex:driftintro}).

%Of course this is very difficult in general, but if the population is at stationarity, then its genealogies are time-homogeneous, and if the population size is constant we might hope for the genealogies to be Markov (this is certainly necessary bc it already is in the non-spatial case). In this paper, we find a universal class of spatial coalescent processes that describe the genealogies of population models falling into this category. To that end, we take an axiomatic approach.
%def

\subsection{The Brownian Spatial Coalescent}\label{sec:introBSC}
The genealogies of a spatial model of the kind considered in the previous paragraph will be very complicated to describe in general; they depend on the initial condition of the process, and wouldn't be Markovian without conditioning at least on the total population size process (as is the case in the non-spatial setting, where genealogies are Markovian after a time change that depends on the total population size process, see e.g.\ \cite{DK99}, Theorem\ 5.1, for the Kingman coalescent).
But if the population is at stationarity forwards in time, and the total population size is constant, then the genealogies are time homogeneous, and there is hope for them to be Markovian. If that is the case, they would be described by a spatial coalescent in the sense of \cref{def:coalintro}.

The goal of this project is to identify a universal class of spatial coalescents that describe the genealogies of constant size spatial population models at stationarity in which individuals follow independent Brownian motions. We will do so by axiomatically defining a class of spatial coalescents through a property that we would expect the genealogies of all such populations to satisfy.
The idea is as follows: if we condition a realisation of the forwards in time population model on the initial and final locations of all individuals, as well as the times and locations of branch events, % and the abstract (``genealogical'') tree, or forest in general, that encodes relationships between individuals,
then, regardless of the branching mechanism, all individuals follow independent Brownian bridges along the branches of the ``genealogical'' tree (or forest in general) forwards in time. Interpreting this backwards in time, it means for the associated coalescent that, conditional on the genealogical tree and the times and locations of what are now the coalescence events, lineages follow independent Brownian bridges along the branches of the tree \emph{backwards} in time.

Before turning this into a definition, we point out a subtle difficulty that arises when starting such a coalescent process from an initial configuration with more than one particle in the same location. If $d \ge 2$, then forwards in time, almost-surely the only occasion where any pair of particles would ever be in the same location is when they were just born in the same branching event. The same is true in $d = 1$ for three or more colliding particles. Thus the coalescent started from such a configuration would have to merge those particles instantaneously, which it cannot if we require right-continuous paths.
One way out would be to work with left-continuous processes, but that would leave many technical special cases in dealing with instantaneous mergers at time zero. The simpler solution is to exclude those points from the state space. %, and thus as possible initial conditions.
For $\fpart \in \mathcal{P}$ denote by $E^\fpart$ the set of maps $\fpart \to E$, and let
\begin{align}\label{eq:def:Ecirc}
    %E^\fpart_\circ \coloneqq \left\{ \boldsymbol x\in E^\fpart\colon \#\!\left\{ u\colon \boldsymbol x(u) = z \right\} \le 1 + \ind_{\left\{ d = 1 \right\} } \, \forall z\in E \right\},
    E^\fpart_\circ \coloneqq \left\{ \boldsymbol x\in E^\fpart\colon \#\!\left\{ \left\{ u,v \right\} \subset \fpart\colon u\neq v, \boldsymbol x_u = \boldsymbol x_v \right\} \le \ind_{\left\{ d = 1 \right\} } \right\}
\end{align}
denote the set of spatial decorations of $\fpart$ in which no (if $d\ge 2$) or at most one (if $d = 1$) pair of particles reside in the same location. We occasionally identify $E_\circ ^{\fpart}$ with an open subset of $E^{|\fpart|}$ by ordering the elements of $\fpart$ in ascending order of their minima (and similarly in other contexts). % if no confusion is possible.
%We can identify $E^\fpart_\circ $ with an open subset of $E^{|\fpart|}$ to define on it a topology and Lebesgue measure. % In particular, $E^\fpart_\circ $ is a Polish space w.r.t.\ a topology in which $\boldsymbol x_n \to \boldsymbol x$ iff $\boldsymbol x_n(u) \to \boldsymbol x(u)$ in $E$ for every $u\in \fpart$.
Let
\begin{equation}\label{eq:defX}
    \mathcal{X}\coloneqq \left\{ (\fpart,\boldsymbol x)\colon \fpart \in \mathcal{P}, \boldsymbol x\in E^\fpart_\circ  \right\}.
\end{equation}
Note that $\fpart$ is the domain of $\boldsymbol x$ for $(\fpart,\boldsymbol x) \in \mathcal{X}$, so we may occasionally just write $\boldsymbol x\in \mathcal{X}$. We denote the unconstrained state space as
\begin{equation}\label{eq:defXbar}
    \overline{\mathcal{X}} \coloneqq \left\{ (\fpart, \boldsymbol x)\colon \fpart \in \mathcal{P}, \boldsymbol x \in E^\fpart \right\} .
\end{equation}
Note that a spatial coalescent with state space $\mathcal{X}$ also defines a stochastic process $(\fpart_t,\boldsymbol X)_{t\ge 0}$ in the unconstrained state space $\overline{\mathcal{X}}$ as long as $(\fpart_0,\boldsymbol X_0)\in \mathcal{X}$.  % that is $\boldsymbol X_0 \in E^{\fpart_0}_\circ $.

\begin{definition}\label{def:BSCintro}
    A spatial coalescent with state space $\mathcal{X}$ is called a \emph{Brownian spatial coalescent} if, conditional on the genealogical forest, and times and locations of all coalescence events, lineages follow independent Brownian bridges along the branches of the forest, and Brownian motions after the last coalescence event in which they are involved.
\end{definition}
A precise definition is \cref{def:BSC}. We remark here that the only real obstacle in replacing Brownian motion with a much more general motion lies in finding the points that have to be excluded from the state space. We go into more detail in \cref{sec:generalmotion}.
%note that the definition does not assume anything about the branching mechanism, but as we will see later, Markovianity requires the branching to be oblivious to space!

\begin{remark}
    The defining property of a Brownian spatial coalescent should be satisfied by the genealogies of any population model in which individuals follow independent Brownian motions forwards in time, regardless of the branching mechanism. This includes spatial interactions such as competition or local regulation. However, it will turn out that assuming the Markov property of the genealogies eliminates \emph{all} branching mechanism that non-trivially depend on the spatial distribution of the population. See \cref{sec:interactivebranching} below.
\end{remark}

We do not yet require any kind of sampling consistency. If a parallel is drawn with the non-spatial setting, then the Brownian spatial coalescent should be compared with the non-spatial coalescent in the sense of \cref{def:NSCPintro}; both contain many examples that do not describe the genealogies of any population. Our first result is a full characterisation of the class of Brownian spatial coalescents, which shows that this comparison is more apt than one might expect. Write $\mergers$ for the set of tuples $(n,\vec{k}) = (n,(k_1, \ldots ,k_m))$ with $n \ge 2$, $1 \le m \le n$, $k_1 \ge \ldots \ge k_m \ge 1$ and $\sum_{i=1}^m k_i \le n$. Write $\mathcal{M}_F(A)$ for the set of finite measures on a measurable space $A$.
\begin{theorem}\label{thm:BSCintro}
    Every label invariant Brownian spatial coalescent is uniquely characterised by a family of measures $\boldsymbol \nu = (\nu_{n,k_1, \ldots ,k_m} \in \mathcal{M}_F(E^m))_{(n,\vec{k}) \in \mergers}$. We call it the \emph{Brownian spatial coalescent with transition measures $\boldsymbol \nu$}.
\end{theorem}

% description of law in terms of \nu_nk, definition of normalisation N^\nu

The transition measures are reminiscent of the transition rates of a non-spatial coalescent, and in fact if we formally replace $E$ with a singleton, then the Brownian spatial coalescent with transition measures $(\nunk)$ is just the non-spatial coalescent with transition rates $(\lank = |\nunk|)$ (see \cref{rem:nunk} (i) below). Here $|\nu|$ is the total mass of a measure $\nu$.

\begin{figure}
  \centering

  \begin{subfigure}[t]{0.4\textwidth}
    \centering
    \begin{tikzpicture}[scale=1.0]
      \coordinate (zero) at (-.5,-.5);
      \coordinate (t-l)  at (-.5,4.95);
      \coordinate (b-r)  at (3.3,-.5);

      \coordinate (Aa) at (0.25,0);
      \coordinate (Ab) at (2.5,0);
      \coordinate (A)  at (1.5,3.0);

      % axes
      \draw[->] (zero) -- (t-l);
      \draw[->] (zero) -- (b-r);
      \node[above] at (t-l) {\footnotesize time};
      \node[below] at (b-r) {\footnotesize space};

      \def\tick{.1}
      \draw (-.5+\tick,0) -- (-.5-\tick,0) node[left] {\footnotesize$0$};
      \draw[dotted,teal] (-.5,0) -- (Ab);

      \draw[dotted] (Aa) -- (Aa|-0,-.5);
      \draw (Aa|-0,-.5+\tick) -- (Aa|-0,-.5-\tick) node[below] {$x_1$};

      \draw[dotted] (Ab) -- (Ab|-0,-.5);
      \draw (Ab|-0,-.5+\tick) -- (Ab|-0,-.5-\tick) node[below] {$x_2$};

      % background (forest)
      \BBY[teal,opacity=.8]{Aa}{A}[948435]
      \BBY[teal,opacity=.8]{Ab}{A}[892459]
      \BBY[black,opacity=.4]{A}{A|-0,4.95}[8241356]

      \node[fill=white,xshift=-5pt,text=teal] at ($(Aa)!0.4!(A)$) {\footnotesize $p_\tau(\xi - x_1)$};
      \node[fill=white,xshift=5pt,text=teal] at ($(Ab)!0.3!(A)$) {\footnotesize $p_\tau(\xi - x_2)$};

      \draw[fill,gray] (Aa) circle[radius=2pt];
      \draw[fill,gray] (Ab) circle[radius=2pt];

      % time/space stamps
      \draw[densely dotted,thick] (A) -- (-.5,0|-A);
      \draw (-.5+\tick,0|-A) -- (-.5-\tick,0|-A) node[left] {$\tau$};

      \draw[densely dotted,thick] (A) -- (A|-0,-.5);
      \draw (A|-0,-.5+\tick) -- (A|-0,-.5-\tick) node[below] {$\xi$};

      \draw[fill] (A) circle[radius=2pt];
    \end{tikzpicture}
    % \caption{...}
  \end{subfigure}
  \hfill
  % --- right panel ---
  \begin{subfigure}[t]{0.59\textwidth}
    \centering
    \def\xscale{1.55}
    \begin{tikzpicture}[scale=1.0,xscale=\xscale]
      \coordinate (zero) at (-.5,-.5);
      \coordinate (t-l)  at (-.5,4.95);
      \coordinate (b-r)  at (3.2,-.5);

      \coordinate (Aa) at (0.0,0);
      \coordinate (Ab) at (1.25,0);
      \coordinate (B)  at (2.5,0);
      \coordinate (A)  at (0.75,2.625);
      \coordinate (AB) at (2.00,4.5);

      % axes
      \draw[->] (zero) -- (t-l);
      \draw[->] (zero) -- (b-r);
      \node[above] at (t-l) {\footnotesize time};
      \node[below] at (b-r) {\footnotesize space};

      \def\tick{.1}
      \draw (-.5+\tick,0) -- (-.5-\tick,0) node[left] {\footnotesize$0$};
      \draw[dotted,gray] (-.5,0) -- (B);

      \draw[dotted] (Aa) -- (Aa|-0,-.5);
      \draw (Aa|-0,-.5+\tick) -- (Aa|-0,-.5-\tick) node[below] {$x_1$};

      \draw[dotted] (Ab) -- (Ab|-0,-.5);
      \draw (Ab|-0,-.5+\tick) -- (Ab|-0,-.5-\tick) node[below] {$x_2$};

      \draw[dotted] (B) -- (B|-0,-.5);
      \draw (B|-0,-.5+\tick) -- (B|-0,-.5-\tick) node[below] {$x_3$};

      % background (forest)
      \BBY[teal,opacity=.8]{Aa}{A}[731475][0.04]
      \BBY[teal,opacity=.8]{Ab}{A}[801947][0.04]
      \BBY[teal,opacity=.8]{A}{AB}[280594][0.04]
      \BBY[teal,opacity=.8]{B}{AB}[572934][0.04]
      \BBY[black,opacity=.4]{AB}{AB|-0,4.95}[687990][0.04]

      \draw[fill,gray] (Aa) circle[radius=2pt];
      \draw[fill,gray] (Ab) circle[radius=2pt];
      \draw[fill,gray] (B)  circle[radius=2pt];

      % time/space stamps
      \draw[densely dotted,thick] (A) -- (-.5,0|-A);
      \draw (-.5+\tick,0|-A) -- (-.5-\tick,0|-A) node[left] {$\tau_1$};

      \draw[densely dotted,thick] (AB) -- (-.5,0|-AB);
      \draw (-.5+\tick,0|-AB) -- (-.5-\tick,0|-AB) node[left] {$\tau_2$};

      \draw[densely dotted,thick] (A) -- (A|-0,-.5);
      \draw (A|-0,-.5+\tick) -- (A|-0,-.5-\tick) node[below] {$\xi_1$};

      \draw[densely dotted,thick] (AB) -- (AB|-0,-.5);
      \draw (AB|-0,-.5+\tick) -- (AB|-0,-.5-\tick) node[below] {$\xi_2$};

      \node[fill=white,xshift=-10.5pt,text=teal] at ($(Aa)!0.5!(A)$) {\footnotesize $p_{\tau_1}(\xi_1 - x_1)$};
      \node[fill=white,xshift=12pt,text=teal] at ($(Ab)!0.3!(A)$) {\footnotesize $p_{\tau_1}(\xi_1 - x_2)$};
      \node[fill=white,xshift=-15pt,text=teal] at ($(A)!0.5!(AB)$) {\footnotesize $p_{\tau_2-\tau_1}(\xi_2 - \xi_1)$};
      \node[fill=white,xshift=15pt,text=teal] at ($(B)!0.4!(AB)$) {\footnotesize $p_{\tau_2}(\xi_2 - x_3)$};

      \draw[fill] (A)  circle[radius=2pt];
      \draw[fill] (AB) circle[radius=2pt];
    \end{tikzpicture}
    % \caption{...}
  \end{subfigure}

  \caption{%
      An illustration of two possible realisations of a Brownian spatial coalescent, together with the spatial factors appearing in~\cref{eq:Pxtwoparticles,eq:Pxthreeintro}.
  }
  \label{fig:introtreeexample}
\end{figure}

We refer to \cref{thm:BSC} for a formal version of \cref{thm:BSCintro} that applies without the assumption of label invariance, and gives a precise description of the law of a Brownian spatial coalescent with given transition measures $\boldsymbol \nu$ building on the notation introduced in \cref{sec:setup}.

Here, we give an informal description by example. We start with the simplest non-trivial initial condition, which is $n=2$ particles at locations $\boldsymbol x = (x_1,x_2)$. Then the only possible merge event is a binary ($k=2$) merger, and the associated transition measure is $\nu_{n,k}=\nu_{2,2}$. If $\nu_{2,2}=0$, then the two lineages just evolve as independent Brownian motions indefinitely. Otherwise, they almost surely merge. Conditional on the time $\tau$ and location $\xi$ of the merger, see \cref{fig:introtreeexample} (left), the law of the coalescent is that of two independent Brownian bridges from $x_1$ and $x_2$ to $\xi$ at time $\tau$, followed by a Brownian motion (recall \cref{def:BSCintro}), see \cref{fig:introtreeexample} (left). It remains to explain the law of $(\tau,\xi)$, which is, as a function of the initial locations $\boldsymbol x$,
\begin{equation}\label{eq:Pxtwoparticles}
    P^{\boldsymbol x}(\diff \tau, \diff \xi) \propto p_\tau(x_1 - \xi) p_\tau(x_2 - \xi) \e^{-\nut{2} \tau} \nu_{2,2}(\diff \xi) \diff \tau,
\end{equation}
where $\nut{2} = |\nu_{2,2}|$, and more generally $\nut{n} = \sum_{ \vec{k}} \binom{n}{k_1\ldots k_m} |\nunk|$. The normalisation constant is the total mass of the right-hand side, which depends on $\boldsymbol x$ and is denoted by $N^{\boldsymbol \nu}(\boldsymbol x)$. Note the similarity with the non-spatial case, where $P(\diff \tau) = \lambda_{2,2} \e^{-\lambda_{2}\tau} \diff \tau$; the spatial formula can be obtained from the non-spatial one by replacing transition rates with transition measures, and adding a spatial factor for each branch in the coalescent tree (or forest). This analogy generalises to all initial conditions. For example, with $n=3$ lineages at initial locations $\boldsymbol x = (x_1,x_2,x_3)$, there are four possible topological shapes $T$ of the tree: one where all three lineages merge simultaneously, and three with two subsequent binary mergers\footnote{We assume here that $\nut{3} > 0$ and $\nut{2} > 0$, otherwise it is possible that not all lineages merge and we have a genealogical forest rather than a tree.}. A realisation of the case where, schematically, $T = \treethree{black}$, is in \cref{fig:introtreeexample} (right). The law of the space- and time coordinates of the merge events in that case is
\begin{align}
    P^{\boldsymbol x} \left( T=\treethree{black}, \diff \tau_1,\diff \tau_2, \diff \xi_1,\diff \xi_2\right)
    &\propto \overbrace{\color{black}p_{\tau_1}(x_1 - \xi_1)\, p_{\tau_1}(x_2 - \xi_1)\, p_{\tau_2-\tau_1}(\xi_1 - \xi_2)\, p_{\tau_2}(\xi_2 - x_3)}^{\text{spatial factors}}\nonumber\\
    &\qquad\;\; \times\,\underbrace{\color{black}\nu_{3,2}(\diff \xi_1)\,\e^{-\nut{3} \tau_1}\, \nu_{2,2}(\diff \xi_2) e^{-\nut{2} (\tau_2 - \tau_1)} \diff \tau_1 \diff \tau_2}_{\text{analogous to non-spatial coalescent}}.
    \label{eq:Pxthreeintro}
\end{align}
Note again the analogy to the non-spatial coalescent, where by the law of competing exponentials the corresponding formula is $P(T=\treethree{black},\diff \tau_1, \diff \tau_2) = \lambda_{3,2}\e^{-\lambda_3 \tau_1} \lambda_{2,2} \e^{-\lambda_2(\tau_2-\tau_1)}\diff \tau_1 \diff \tau_2$. The normalisation for $P^{\boldsymbol x}$---a joint probability measure over the topology $T$ of the tree plus its time- and space coordinates (to which we will also refer collectively as a \emph{decorated} tree)---is again denoted by $N^{\boldsymbol \nu}(\boldsymbol x)$, and obtained by integrating the right-hand side of~\eqref{eq:Pxthreeintro} over $\tau_1,\tau_2,\xi_1,\xi_2$, and then summing the resulting quantity over the four possible values of $T$.

We note that this normalisation $N^{\boldsymbol \nu}(\boldsymbol x)$ is not finite for every collection $\boldsymbol \nu$ of finite measures; in fact, finiteness of $N^{\boldsymbol \nu}$ is a necessary and sufficient condition for a Brownian spatial coalescent with transition measures $\boldsymbol \nu$ to exist. Furthermore, continuity of $N^{\boldsymbol \nu}$ as a function of the initial condition $\boldsymbol x \in \mathcal{X}$ corresponds to the Feller property of the resulting coalescent; since we require the Feller property (\cref{def:coalintro}), we have the following characterisation. See \cref{thm:BSC} for the formal version.
%as a function of the initial condition $\boldsymbol x = (x_1, \ldots ,x_n) \in \mathcal{X}$ plays an important role. For example, it decides for a given family $\boldsymbol \nu$ of transition measures whether an associated Brownian spatial coalescent exists.

\begin{theorem}\label{thm:BSCintro2}
    For a given family $\boldsymbol \nu = (\nu_{n,k_1, \ldots ,k_m} \in \mathcal{M}_F(E^m))_{(n,\vec{k})\in \mergers}$, a Brownian spatial coalescent with transition measures $\boldsymbol \nu$ exists if and only if $N^{\boldsymbol \nu}\colon \mathcal{X} \to [0,\infty]$ is finite and continuous.
\end{theorem}
% See \cref{thm:BSC}.
% The main task in the proof of \cref{thm:BSCintro2} is to show that the process defined through the transition measures $\boldsymbol \nu$ is a Markov process, which is not obvious from its description.

A few remarks are in order.

\begin{remark}\label{rem:nunk}
    \begin{enumerate}
        \item If we formally replace $E$ with a singleton, and Brownian motion with the trivial Markov process on $E$, then the spatial factors in the formula for $P^{\boldsymbol x}$ vanish, the transition measures reduce to non-negative numbers $\lank$, and we exactly recover the law of a non-spatial coalescent with transition rates $\lank$.
        \item The way in which we have described the law of the Brownian spatial coalescent is not inherently Markovian. The Markov property in terms of this description states that, if we sample the decorated coalescence tree at time zero and start running the particles along the corresponding Brownian bridges, and at some time $t > 0$ we resample the decorated coalescence tree for the remaining lineages using their current locations as initial conditions, and run them along the Brownian bridges associated with the new tree, then this does not affect the overall law of the process.
        \item % Continuity of $N^{\boldsymbol \nu}$ corresponds to the Feller property of the associated coalescent; for the latter to be well-defined, finiteness of $N^{\boldsymbol \nu}$ would suffice.
We do not have a nice characterisation of the measures $\boldsymbol \nu$ for which $N^{\boldsymbol \nu}$ is finite and continuous, but a sufficient condition is that all transition measures are constant multiples of Lebesgue measure, see \cref{lem:NF}. % Since $N^{\boldsymbol \nu}(\boldsymbol x)$ is an integral over functions that are continuous in $\boldsymbol x$, dominated convergence gives a sufficient integrability condition. %mention that non-atomic is necessary? I didnt prove it here but in some scribbles I think
%The statement is not trivial, because the process defined through the transition measures $\boldsymbol \nu$ as described above does not obviously have the Markov property.
The function $N^{\boldsymbol \nu}$ also reveals a technical reason why it was necessary to exclude some points from the state space $\mathcal{X}$: If $\boldsymbol x$ approaches a point in the boundary of $\mathcal{X}$, then $N^{\boldsymbol \nu}(\boldsymbol x)$ diverges if the transition measure associated with the coalescent event that would have to happen instantaneously is non-zero, see also \cref{rem:Ndivergence} and the following remark.
\item
%\begin{example}\label{ex:wmformula}
    If we start from $n=2$ lineages at locations $\boldsymbol x = (x_1,x_2)$, and additionally
    %Starting from two lineages at locations $x_1$ and $x_2$, the probability that they coalesce at time $t > 0$ at location $z\in E$ is proportional to
    %\begin{equation}\label{eq:2merge}
        %\e^{-\nut{2} t} p_t(x_1-z) p_t(x_2-z) \nu_{2,2}(\diff z) \diff t.
    %\end{equation}
    $\nu_{2,2}(\diff z) = |\nu_{2,2}| \diff z$ (which will turn out to be a consequence of sampling consistency, see \cref{thm:samplingconintro}), then
    the probability that the two particles reach a common ancestor at time $\tau > 0$ is obtained by integrating \cref{eq:Pxtwoparticles} over $\xi$:
    \begin{equation}\label{eq:wmformula}
        P^{\boldsymbol x}(\diff \tau) 
        %= \frac{1}{N^{\boldsymbol \nu}(\boldsymbol x)}
        \propto
        |\nu_{2,2}| \e^{-\nut{2} \tau} p_{2\tau}(x_1-x_2) \diff \tau,
    \end{equation}
    %since $\nut{2} = |\nu_{2,2}|$.
    which is reminiscent of the Wright-Mal\'ecot formula. 
    %It follows from \cref{eq:2merge} that conditional on $t$, the coalescence location $z$ is equal in law to the midpoint of a Brownian bridge running from $x_1$ to $x_2$ in time $2t$.\footnote{This observation was made by Andra\v{z} Jelin\v{c}i\v{c} after listening to a talk about the project.}

    The normalisation is $N^{\boldsymbol \nu}(x_1,x_2) = \int_0^\infty  |\nu_{2,2}|\e^{-\nut{2} \tau} p_{2\tau}(x_1-x_2) \diff \tau$. If $r = |x_1-x_2|$ is small, then $N^{\boldsymbol \nu}(x_1,x_2)$ is of order $r^{2-d}$ if $d \ge 3$, of order $\log \frac{1}{r}$ if $d = 2$, and of order $1$ if $d = 1$. In particular, $N^{\boldsymbol \nu}(x_1,x_2)$ diverges as $|x_1-x_2| \to 0$ if and only if $d \ge 2$, that is if and only if a merge event of the two lineages would have to happen instantaneously when $x_1 = x_2$ (because, in $d \ge 2$, two particles would never be in the same location forwards in time except if they had just originated from the same birth event).
    %If the locations of more than two individuals are known, then similar calculations lead to higher-order versions of~\eqref{eq:wmformula}.
%\end{example}
    \end{enumerate}
\end{remark}

% The main work in the proof of \cref{thm:BSCintro} is to construct the measures $\nunk$ using only the defining property (\cref{def:BSCintro}) and the Markov property of a given Brownian spatial coalescent. The idea is to apply the Markov property at the (random) time of the first merge event, and proceed inductively over the number of initial particles, starting with two.
% sampling consistency (classical and new, reference to markov mapping theorem)

\subsection{Sampling Consistency}
Our main result is a characterisation of sampling consistency within the class of Brownian spatial coalescents, analogous to the characterisations of the $\Lambda$- and $\Xi$-coalescents in the non-spatial setting (\cref{thm:lambdaintro,thm:xiintro}). A characterisation of this form is, to our knowledge, without precedent in the literature on spatial coalescents.

Sampling consistency for spatial coalescents is classically taken to mean that the distribution of the coalescent started from some set of $n$ initial locations is the same as that of the coalescent started with an additional $n+1$'st particle at any fixed location and projected back onto the first $n$ particles. We would expect the genealogies of a forward in time population model to have this property, but only if knowing the location of an additional individual does not reveal any information about the population. This is the case, for example, if the population is spread homogeneously throughout space and time, as is the case for a number of spatial models whose genealogies have been studied successfully, which includes the spatial $\Lambda$-Fleming-Viot process, or populations living on a discrete lattice with positive population density at every deme. We review related literature in detail in \cref{sec:relatedliterature}. But in our case, where the spatial distribution of the population fluctuates in time and space driven by its forward dynamics, the location of the additional particle will reveal information about the entire population. This will, in general, influence the genealogy of the first $n$ particles. Instead, the additional particle has to be genuinely \emph{sampled} from the (stationary) spatial distribution of the population, \emph{conditional} on the $n$ lineages whose locations are already known. % We thus propose the following notion of sampling consistency.

\begin{definition}\label{def:samplingconintro}
    A spatial coalescent is \emph{sampling consistent} if, for every choice of initial locations $x_1$ to $x_n$, there exists a probability measure $\mu_{x_1, \ldots ,x_n}$ on $E$ such that the coalescent started from $x_1$ to $x_n$ has the same distribution as the coalescent induced on the first $n$ particles when starting with an additional particle $x_{n+1}$ sampled from $\mu_{x_1, \ldots ,x_n}$.
\end{definition}

See \cref{def:samplingcon} in the next section for a precise definition. We will prove that, as the motivation suggests, $\mu_{x_1, \ldots ,x_n}$ is the law of a sample from the stationary distribution of an associated forwards in time population model, conditional on having sampled $x_1$ to $x_n$ in $n$ previous samples. The familiar reader may also note that \cref{def:samplingconintro} can be cast in the language of the Markov mapping theorem: $\mu_{x_1,\ldots ,x_n}$ is the kernel that recovers the state of the ``larger'' Markov process (the coalescent started from $n+1$ particles) from the state of the projected one (the coalescent started from $n$ particles).

% characterisation of sampling consistency

\begin{theorem}\label{thm:samplingconintro}
    A label invariant Brownian spatial coalescent is sampling consistent if and only if there exists a finite measure $\Xi$ on $\triangle$ such that the transition measures satisfy \[
        \nunk(\diff \boldsymbol z) = \lank \diff \boldsymbol z,\qquad (n,\vec{k}) \in \mergers,
    \] where $(\lank)$ are the rates of the non-spatial $\Xi$-coalescent.
\end{theorem}
This gives rise to what we call the \emph{Brownian spatial $\Xi$-coalescent}.
If simultaneous mergers are not allowed, we get an analogous statement for what we call the \emph{Brownian spatial $\Lambda$-coalescent}.
If $\boldsymbol \nu$ are the transition measures of the Brownian spatial $\Xi$-coalescent, write $N^\Xi = N^{\boldsymbol \nu}$.

% forward models (Xi FV, and spatial cannings models)

\begin{proposition}\label{prop:musamplingintro}
    Let $\Xi \neq 0$ be a finite measure on $\triangle$. Then the family of probability measures associated with the Brownian spatial $\Xi$-coalescent through \cref{def:samplingconintro}, which we denote by $(\mu^\Xi_{x_1, \ldots ,x_n})$, is unique and
    \begin{equation}\label{eq:muNintro}
        \mu^{\Xi}_{x_1, \ldots ,x_n}(\diff y) = \frac{N^\Xi(x_1,\ldots ,x_n,y)}{N^\Xi(x_1, \ldots ,x_n)} \diff y.
    \end{equation}
    Furthermore, there exists a unique random probability measure $\mu^\Xi$ on $E$ such that
    \begin{equation}\label{eq:Emuintro}
        \E\left[ \mu^\Xi(\diff x_1) \ldots \mu^\Xi(\diff x_n)\right] = N^\Xi(x_1, \ldots ,x_n) \diff x_1\ldots \diff x_n
    \end{equation}
    as measures on $E^n$ for every $n\in \N$.
\end{proposition}
%We call $\mu$ the \emph{stationary measure associated to the Brownian spatial $\Xi$-coalescent}.
Note that \cref{prop:musamplingintro} reveals another meaning of the normalisation function $N^\Xi$: it describes the joint density of samples from a random realisation of $\mu^\Xi$.

\begin{remark}
    A spatial coalescent process is sampling consistent in the classical sense if and only if it is sampling consistent in the sense of \cref{def:samplingconintro} with respect to (w.r.t.)\ \emph{any} choice of probability measures $(\mu_{x_1, \ldots ,x_n})$. But since the probability measures in \cref{prop:musamplingintro} are unique, there exists no Brownian spatial coalescent, except the trivial one, which is sampling consistent in the classical sense. This may explain why no non-trivial, sampling consistent spatial coalescent in which an individual lineage follows a Brownian motion has been found to this day in the setting of constant population density (where we would expect associated coalescents to be sampling consistent in the classical sense), except in one dimension where independent Brownian motions meet.
\end{remark}

\Cref{eq:Emuintro,eq:muNintro} imply \[
    \E \left[ \mu^\Xi(\diff x_1)\ldots \mu^\Xi(\diff x_n) \right] =  \diff x_1 \mu^\Xi_{x_1}(\diff x_2) \ldots \mu^\Xi_{x_1, \ldots ,x_{n-1}}(\diff x_n),
\] that is, $\mu^\Xi_{x_1, \ldots ,x_n}$ is the distribution of a sample from $\mu^\Xi$ conditional on having already sampled $x_1$ to $x_n$ in $n$ previous, independent samples. This confirms the motivation behind \cref{def:samplingconintro}, provided we can find a forwards in time population model with stationary distribution $\mu^\Xi$ whose genealogies are described by the Brownian spatial $\Xi$-coalescent. This will be the $\Xi$-Fleming-Viot process.

Before showing that, we present a drift representation for the evolution of the Brownian spatial coalescent. The suggestion to look for a representation of this kind was made by Michal Bassan.
For a family of transition measures $\boldsymbol \nu$ we write $\boldsymbol \nu(\diff \boldsymbol z) = \boldsymbol \lambda \diff \boldsymbol z$ as a short hand for $\nunk(\diff \boldsymbol z) = \lank \diff \boldsymbol z$ for all $(n,\vec{k}) \in \mergers$.

\begin{theorem}\label{thm:driftintro}
    If a Brownian spatial coalescent with transition measures $\boldsymbol \nu$ of the form $\boldsymbol \nu(\diff \boldsymbol z) = \boldsymbol \lambda \diff \boldsymbol z$ is started from distinct initial locations $x_1$ to $x_n$, then the paths $\boldsymbol Z_t = (Z^1_t, \ldots ,Z^n_t)$ of the lineages until just before the first merge event (or until just before the first time that any pair of lineages meet if $d = 1$) are described by the following stochastic differential equation:
    \begin{equation}\label{eq:driftintro}
        \diff \boldsymbol Z_t = \diff \boldsymbol B_t + \nabla \log N^{\boldsymbol \nu}(\boldsymbol Z_t) \diff t,
    \end{equation}
    where $(\boldsymbol B_t)$ is an $nd$-dimensional standard Brownian motion on $E$ with periodic boundary conditions.
\end{theorem}

In particular, \cref{thm:driftintro} applies to the Brownian spatial $\Xi$-coalescent. It also holds for a general Brownian spatial coalescent provided $N^{\boldsymbol \nu}$ is differentiable and the order of the differentiation and the integrals defining $N^{\boldsymbol \nu}$ can be exchanged, in the way that is needed in the proof; see \cref{sec:prfdrift}. The following remark and example apply only to Brownian spatial coalescents of the form assumed in \cref{thm:driftintro}.

\begin{remark}\label{rem:driftintro}
    \begin{enumerate}
        \item If $d \ge 2$, % then almost surely the only time that lineages meet in a Brownian spatial coalescent is in the moment they coalesce. This means that
            then the drift representation determines the law of a Brownian spatial coalescent completely: run \cref{eq:driftintro} until its maximal existence time, where the drift diverges and some set of lineages collide. Coalesce lineages that are in the same location and restart \cref{eq:driftintro}. % with the remaining lineages.
            Repeat until only one lineage is left (or a number $k$ such that $\nut{k} = 0$).
        \item If $d = 1$, then lineages may meet in pairs without coalescing. At such points $N^{\boldsymbol \nu}$ is (continuous and) not differentiable, but it is possible to extend $\nabla \log N^{\boldsymbol \nu}$ to all of $\mathcal{X}$ in such a way that the drift representation \cref{eq:driftintro} holds up until just before the first merge event. But unlike in $d\ge 2$, binary mergers always happen before the maximal existence time of \cref{eq:driftintro}, making the construction in (i) invalid, except if binary mergers are impossible because $\nu_{k,2} = 0$ for the relevant values of $k$. % except if binary mergers are not possible, that is $\nu_{n,2} = 0$ for all $n\in \N$.
            We leave it open to find a description of the law with which binary mergers occur conditional on the solution of \cref{eq:driftintro}; calculations suggest that they do \emph{not} happen at a rate proportional to the collision local time of pairs of lineages, as is common in models of coalescing Brownian motions.
            %in the case of the Brownian spatial Kingman coalescent, for example, \cref{eq:driftintro} is well-defined for all times. We leave it open to describe

            %The occurrence of binary mergers is random conditional on the solution of \cref{eq:driftintro}, and some calculations suggest that they do \emph{not} happen at a rate proportional to their collision local time, as is common in models of coalescing Brownian motion.
    \end{enumerate}
\end{remark}

The function $N^{\boldsymbol \nu}$ is a sum over forests, whose number explodes combinatorially as the number of initial locations increases, but we can gain some interesting insights in simple special cases.
\begin{example}\label{ex:driftintro}
    Consider a Brownian spatial coalescent whose transition measures $\boldsymbol \nu$ are of the form $\boldsymbol \nu(\diff \boldsymbol z) = \boldsymbol \lambda \diff \boldsymbol z$.
    \begin{enumerate}
        \item If $\lambda_{2,2} > 0$ and the Brownian spatial coalescent is started from two lineages, then the drift symmetrically pulls them towards each other,\footnote{%
                More precisely, the drift acting on a lineage located at $x$ when there is another lineage at location $y$ points in the same direction as $\int_0^\infty \e^{-\nut{2} t}\nabla _x p_t(x-y) \diff t$. On $\R^d$, this is the same direction as $y-x$, but on the torus it is more complex because of the periodicity. The difference becomes negligible for small separation.%
            }
            and the strength of the drift as their separation $r$ tends to zero is of order $1$ if $d = 1$, otherwise of order $1 / r$. This is not surprising in light of the fact that the Brownian movement in $d \ge 2$ effectuates a repulsive drift of the same order $1 / r$.
        \item If $\lambda_{3,3} > 0$ and the Brownian spatial coalescent is started from three individuals, and conditioned on a triple merger, then the drift pulls each particle in the direction of the midpoint of the other two.\footnote{As for the binary merger, this is technically only true if we were on $\R^d$, but the difference is negligible for small separations of the three lineages.} Without the conditioning, the drift is mixed in a complicated way with pairwise binary attractions.
        \item If $\lambda_{4,2,2} > 0$ and $\lambda_{2,2} = 0$, and the Brownian spatial coalescent is started from four individuals at locations $x_1$ to $x_4$, and conditioned on a simultaneous binary merger of $x_1$, $x_2$, and $x_3$, $x_4$, then the drift symmetrically pulls $(x_1,x_3)$ and $(x_2,x_4)$ towards each other, and the strength of the drift is of order $1 / r$ where $r$ is their $2d$-dimensional Euclidean separation. We put $\lambda_{2,2} = 0$ because the drift is otherwise influenced in a complex way by the inevitable second merge event.
    \end{enumerate}
\end{example}

\begin{remark}
    It is known in a variety of settings that the movement of a single lineage $Z_t$ in a spatially evolving population is driven by an SDE of the form $\diff Z_t = \diff B_t + \nabla \log N(Z_t) \diff t$, where $N$ is something akin to a population density (for example in \cite{sylviedynamics} the population density $N$ is a centred one-dimensional Gaussian and the drift is $\nabla \log N (x) \propto -x$). In our case, we have more than one lineage, and the spatial distribution of the population is random, fluctuates, and does not admit a density except in one dimension (see e.g.\ \cite{etheridgesuperprocesses}, Chapter 2.9).
    Nevertheless, the same representation still holds by replacing $N$ with the joint density of an $n$-sample from the population's stationary distribution.
\end{remark}

\subsection{Population Models}\label{sec:populationmodelsintro}
The Brownian spatial $\Xi$-coalescent turns out to describe the genealogies of neutral population models with Brownian movement, in the limit of large, asymptotically constant population size. Here \emph{neutral} refers to the fact that the branching mechanism is oblivious to the locations (or types, depending on the interpretation) of the individuals. The scaling limits of such processes in the limit of large population size are the \emph{$\Xi$-Fleming-Viot processes} \cite{xiflemingviot,DK96,DK99}. We will first introduce this process and show that its genealogies are described by the Brownian spatial $\Xi$-coalescent, and then we use this result to prove that the same is true for a large class of neutral population models in the limit of large population size.

\subsubsection{$\Xi$-Fleming-Viot Process}\label{sec:xiflemingviot}
The \emph{$\Xi$-Fleming-Viot process} \cite{xiflemingviot} is a probability measure valued Markov process on $E$ (or on all of $\R^d$) that generalises the ``generalised (or $\Lambda$-)Fleming-Viot process'' coined in \cite{BL2003}. The $\Xi$-Fleming-Viot process is dual to the non-spatial $\Xi$-coalescent \cite{xiflemingviot}, in the same sense that the classical Fleming-Viot process is dual to Kingman's coalescent \cite{DK96}. % We will see that its genealogies at stationarity are described by the Brownian spatial $\Xi$-coalescent. In fact, the latter can be used to describe the full time reversal of the $\Xi$-Fleming-Viot process.

\medskip\paragraph{\textbf{Lookdown Construction}}
The $\Xi$-Fleming-Viot process is most easily defined using a lookdown construction \cite{DK96,DK99,xiflemingviot}.
 Let $\Xi = a \delta_0 + \Xi_0$ be a finite measure on $\triangle$, where $\delta_0$ is the unit mass at zero and $\Xi_0$ has no atom at zero.
 Let $\mathfrak{N}^{ij}$ for $i,j\in \N, i < j$ be a family of independent Poisson point processes on $\R$ with rate $a$, and $\mathfrak{M}$ a Poisson point process on $\R \times \triangle$, independent of the $\mathfrak{N}^{ij}$, with intensity measure \[
    \diff t \otimes \frac{\Xi_0(\diff \boldsymbol \xi)}{\left<\boldsymbol \xi,\boldsymbol \xi \right> }.
\] We define a dynamic on $E^\infty$, started from an initial configuration $\boldsymbol Y(0) = (Y_1(0),Y_2(0),\ldots )$, as well as counting processes $L^{ij}$ for $i,j\in \N$, $i < j$, $L^{l}_{J}$, for $l\in \N$, $J\subset [l]$ with $|J| \ge 2$, and $L^l_{J,k}$ for $l\in \N$, $J \subset [l]$ with $|J| \ge 2$ and $k \in \N$, as follows. All of the counting processes start at zero and have jumps of unit size.

\begin{enumerate}
    \item At a point $t$ of the process $\mathfrak{N}^{ij}$, the process $L^{ij}$ jumps , and the $j$'th level particle ``looks down'' to the $i$'th level particle and copies its location. All other particles are ``bumped up'' by one level, i.e.\ $Y_{k+1}(t) = Y_k(t-)$ for all $k \ge j$.
    \item At a point $(t,\boldsymbol \xi)$ of the process $\mathfrak{M}$, we independently assign every level to a basket with label $i\in \N$ with probability $\xi_i$, and to no basket with probability $1 - \sum_i \xi_i$. The particle with the lowest level in each basket is the ``parent'' of that basket.
    Every particle that has been assigned to a basket copies the location of the parent of that basket. All levels that are not in any basket assume the pre-reproduction locations of non-parental particles, retaining their order. See also Section~2.2 in~\cite{xiflemingviot}.

        Denote by $J_k$ for $k\in \N$ the set of levels in basket $k$, and $J = \bigcup_{k\ge 1} J_k$. Then, for $l\in \N$, if $J \cap [l] \neq \emptyset$, then $L^l_{J \cap [l]}$ jumps at time $t$, and for $k\in \N$ with $J_k \cap [l] \neq \emptyset$, the process $L^l_{J_k\cap [l], k}$ jumps at time $t$.
    \item All levels follow independent Brownian motions between reproductive events.
\end{enumerate}
% \begin{enumerate}
%     \item At a point $t$ of the process $\mathfrak{N}^{ij}$, the $j$'th level particle ``looks down'' to the $i$'th level particle and copies its location. The process $L^{ij}$ jumps at time $t$.
%     \item At a point $(t,\boldsymbol \xi)$ of the process $\mathfrak{M}$, we independently assign every level to a basket with label $i\in \N$ with probability $\xi_i$, and to no basket with probability $1 - \sum_i \xi_i$. Then every particle that has been assigned to a basket copies the location of the particle with the smallest level in the same basket.
%
%         Denote by $J_k$ for $k\in \N$ the set of levels in basket $k$, and $J = \bigcup_{k\ge 1} J_k$. Then, for $l\in \N$, if $J \cap [l] \neq \emptyset$, then $L^l_{J \cap [l]}$ jumps at time $t$, and for $k\in \N$ with $J_k \cap [l] \neq \emptyset$, the process $L^l_{J_k\cap [l], k}$ jumps at time $t$.
%     \item All levels follow independent Brownian motions between reproductive events.
% \end{enumerate}
In words, $L^{ij}$ keeps track of all birth events that involved only $i$ and $j$, $L^l_J$ keeps track of all birth events that involved, amongst levels in $[l]$, exactly those in $J$, and $L^l_{J,k}$ counts birth events in which, amongst levels in $[l]$, exactly those in $J$ were assigned to basket~$k$. See also eqs.~(2.18) and~(2.19) in~\cite{xiflemingviot} for precise definitions. We denote the collection of all of these counting processes by $\boldsymbol L$.
There may be an infinite number of lookdown events in finite time, but % the condition $\int_\triangle \sum_i \xi_i^2 \frac{\Xi_0(\diff \boldsymbol \xi)}{\left<\boldsymbol \xi,\boldsymbol \xi \right> } = |\Xi_0| < \infty$ ensures that
any fixed level is only hit by non-trivial reproductive events at a finite rate~\cite{xiflemingviot}.

The point in this construction lies in the fact that it preserves exchangeability: If $\boldsymbol Y(0) = (Y_1(0),Y_2(0),\ldots )$ is exchangeable, then so is $\boldsymbol Y(t) = (Y_1(t),Y_2(t),\ldots )$ for every $t > 0$ \cite{DK99,xiflemingviot}. We denote the associated de Finetti measure by \[
    \mathcal{Y}_t = \lim_{n\to \infty} \frac{1}{n}\sum_{i=1}^n \delta_{Y_i(t)},\quad t \ge 0.
\] The \emph{$\Xi$-Fleming-Viot process} started at $(\boldsymbol L(0), \boldsymbol Y(0),\mathcal{Y}_0)$ is a c\`adl\`ag modification of the process $(\boldsymbol L(t), \boldsymbol Y(t),\mathcal{Y}_t)_{t \ge 0}$ (only the third component has to be modified). For details of the construction see \cite{xiflemingviot}. (The term ``$\Xi$-Fleming-Viot process'' often just refers to the process $(\mathcal{Y}_t)$, but it is convenient to include the particle representation explicitly.)

\begin{remark}\label{rem:lambdafv}
    If simultaneous births to multiple parents are not allowed, then $\xi_2 = \xi_3 = \ldots = 0$ for $\Xi$-a.e.\ $\boldsymbol \xi \in \triangle$, and $\Xi$ reduces to a finite measure $\Lambda$ on $[0,1]$. % which is commonly denoted $\Lambda = a\delta_0 + \Lambda_0$, where $\Lambda_0$ has no atom at zero. % In terms of $\Lambda$, step (ii) above can be summarised as follows: At a point $(t,p)$ of a Poisson point process with intensity $\diff t \otimes \frac{\Lambda_0(\diff p)}{p^2}$ on $\R \times [0,1]$, select every level independently with probability $p$, and each selected particle copies the location of the selected particle with the lowest level.
    The resulting particle process is called the \emph{$\Lambda$-Fleming-Viot process}, whose lookdown construction was introduced in~\cite{DK99}.

    If, further, $\Lambda = \delta_0$, then $(\mathcal{Y}_t)$ is the well-known Fleming-Viot process, whose lookdown construction was introduced in~\cite{DK96}. The Fleming-Viot process can alternatively be obtained by conditioning a Dawson-Watanabe superBrownian motion on having constant, unit population size~\cite{Etheridge91,Perkins91}.
\end{remark}

\medskip\paragraph{\textbf{Genealogy}} The particle representation $(\boldsymbol L(t), \boldsymbol Y(t))$ can be used to define the genealogies of the $\Xi$-Fleming-Viot process. Let $T > 0$, and denote by $A_k(t)$ the level of the ancestor at time $T-t$ of the particle at level $k$ at time $T$ (assigned at jump times in such a way that $A_k$ is right-continuous). Let $\fpart_t$ for $t\in [0,T]$ be the partition of $\N$ induced by the equivalence relation \[
    i \sim_t j \iff A_i(t) = A_j(t).
\]
For $u\in\fpart_t$, write $A_u(t)$ for the common ancestor of the levels in the equivalence class $u$.
%Note that $A_i(t) = \min \left\{ j\colon i \sim_t j \right\} $ by the definition of the lookdown construction, so $(\fpart_t)$ contains all information about the processes $(A_i(t)\colon i\in \N)$.
Denote by
\begin{equation}\label{eq:xfromy}
    \boldsymbol X_t(u)
    %\boldsymbol X(t)_u
    = \lim_{s \uparrow T-t} Y_{A_u(t)}(s),\qquad t \in [0,T], u\in \fpart_t,
    %todo this display is a great example that it is a bit confusing whether and when time is an argument or a subscript
\end{equation}
the position at time $T-t$ of the common ancestor of the particles with levels in $u$ at time~$T$ (assigned in such a way that $\boldsymbol X$ is right-continuous).

For $\ell\in \N$ and $t \in [0,T]$, write $\fpart_t^\ell$ for the partition induced by $\fpart_t$ on $[\ell] = \{1, \ldots ,\ell\}$, and $\boldsymbol X_t^\ell$ for the restriction of $\boldsymbol X_t$ to $\fpart_t^\ell$. Then $(\fpart^\ell_t,\boldsymbol X^\ell_t)_{t \in [0,T]}$ describes the genealogies of a sample of $\ell$ individuals from the $\Xi$-Fleming-Viot process; it is a stochastic process with state space $\overline{\mathcal{X}}$ and initial condition $\fpart_0^\ell = \left\{ \left\{ 1 \right\} , \ldots ,\left\{ \ell \right\}  \right\} $, $\boldsymbol X_0^\ell(\{i\}) = Y_i(T-)$ for $i\in [\ell]$, where $Y_1(T-), \ldots ,Y_\ell(T-)$ are i.i.d.\ samples from $\mathcal{Y}_T$ (conditional on $\mathcal{Y}_T$, which is itself random).
% By standard properties of Brownian motion, if $(\fpart_0^n,\boldsymbol X_0^n) \in \mathcal{X}$, then almost-surely $(\fpart_t^n,\boldsymbol X_t^n) \in \mathcal{X}$ for all $t\in [0,T]$.

The $\Xi$-Fleming-Viot process must be at stationarity forwards in time in order for its coalescent to be time-homogeneous.
%The initial condition is random: $\fpart_0 = \left\{ \left\{ 1 \right\} , \ldots ,\left\{ n \right\}  \right\} $ by construction, and $\boldsymbol X^n_0(\left\{ i \right\} ) = x_i$ is the location of the $i$'th level particle at time zero, so $(x_1, \ldots ,x_n)$ is an i.i.d.\ sample from $\mathcal{Y}_0 \sim \mu^\Xi$.
\begin{proposition}\label{prop:xifvstationary}
    The $\Xi$-Fleming-Viot process has a unique stationary distribution, which is equal to the distribution of $\mu^\Xi$. % It can be constructed to run bi-infinitely at stationarity, which we denote by $(\boldsymbol L(t), \boldsymbol Y(t),\mathcal{Y}_t)_{t\in (-\infty,\infty)}$. todo remark that we can construct bi-infinite version
\end{proposition}
More precisely, at stationarity $\mathcal{Y}_t$ has the same distribution as $\mu^\Xi$ for all $t\ge 0$, and conditional on $\mathcal{Y}_t$, the vector $\boldsymbol Y(t)$ is an i.i.d.\ sequence of samples from $\mathcal{Y}_t$. Therefore, the joint law of the first $\ell$ levels $(Y_1(t), \ldots ,Y_\ell(t))$ at stationarity is given by \cref{eq:Emuintro} for any $\ell\in \N$; in particular, almost-surely $Y_1(T), \ldots ,Y_\ell(T)$ are pairwise distinct and so $(\fpart_0^\ell,\boldsymbol X_0^\ell) \in \mathcal{X}$. By standard properties of Brownian motion, this implies that also $(\fpart_t^\ell,\boldsymbol X_t^\ell) \in \mathcal{X}$ for all $t\in [0,T]$ almost-surely.
The following theorem is the main result of this section.

\begin{theorem}\label{thm:xigenealogiesintro}
    The law of $(\fpart^\ell_t,\boldsymbol X^\ell_t)_{t\in[0,T]}$ is that of a Brownian spatial $\Xi$-coalescent started from the random initial condition $(\fpart_0^\ell,\boldsymbol X_0^\ell)$.
\end{theorem}

In fact, the Brownian spatial $\Xi$-coalescent can be used to describe the full time-reversal of the $\Xi$-Fleming-Viot process.

\begin{theorem}\label{thm:timereversalintro}
    For any $\ell\in \N$, the law of the time reversal $(Y_1(T-t), \ldots ,Y_\ell(T-t))_{t \in[0,T]}$ at stationarity can be described as follows:
    \begin{enumerate}
        \item Sample initial points $y_1, \ldots ,y_\ell$ from a random realisation of $\mu^\Xi$.
        \item Evolve according to a Brownian spatial $\Xi$-coalescent.
        \item When a coalescence event occurs, resample points using the measures $(\mu^\Xi_{x_1, \ldots ,x_\ell})$ so that the total number of lineages stays constant.
    \end{enumerate}
\end{theorem}

A precise statement, including how levels are reassigned at coalescence--resampling events, can be found in \cref{sec:prfxifv}. As an example for rule (iii), if a ternary merger occurred and the remaining lineages are at locations $y_1, \ldots ,y_{\ell-2}$, we sample $y_{\ell-1}$ from $\mu^\Xi_{y_1, \ldots ,y_{\ell-2}}$, and $y_\ell$ from $\mu^\Xi_{y_1, \ldots ,y_{\ell-2},y_{\ell-1}}$, then resume the Brownian spatial $\Xi$-coalescent from $y_1, \ldots ,y_\ell$.

\begin{remark}
    A byproduct of \cref{prop:xifvstationary} is explicit formulas for samples from the stationary distribution of a $\Xi$-Fleming-Viot process through \cref{eq:Emuintro}. For instance, recalling \cref{rem:nunk} (iv), the density of a two-sample from the stationary distribution $\mu^\Xi$ of the $\Xi$-Fleming-Viot process is \[
        \E \left[ \mu^\Xi(\diff x_1) \mu^\Xi(\diff x_2) \right] \propto \left( \int_0^\infty \e^{-|\Xi| \tau} p_{2\tau}(x_1-x_2) \diff \tau \right) \diff x_1 \diff x_2,
    \] normalised such that the integral over $x_1$ and $x_2$ is one. Here we used that $\lambda_2 = \lambda_{2,2}$, which equals $|\Xi|$ by \cref{eq:xirates}. In particular, the distribution of a two-sample from a stationary $\Xi$-Fleming-Viot process only depends on $|\Xi|$. Formulas for higher order samples involve sums over an increasing number of trees, and depend on higher order rates of the $\Xi$-coalescent.
\end{remark}

\begin{remark}
    The stationary distribution $\mu^\Xi$ is not reversible in the sense that the process $(\boldsymbol Y(t),\mathcal{Y}_t)_{t\in [0,T]}$ running at stationarity has the same law as $(\boldsymbol Y(T-t),\mathcal{Y}_{T-t})_{t\in [0,T]}$. The Fleming-Viot process is only reversible in this sense if the underlying spatial motion (or, in the genetics setting, mutation process) is a pure jump process in which the destination of each jump is independent of the location just before the jump. See \cite{reversible1,reversible2,reversible3,reversible4,reversible5} for literature on this topic.
\end{remark}
% START

\subsubsection{Scaling Limits of Neutral Population Models}
For $n\in \N$ (the scaling parameter), let $N^n(t)$ be the total size of a population at time $t$, let $N_b^n(t)$ be the number of births up to and including time $t$, and let $N_d^n(t)$ denote the number of deaths, so \[
    N^n(t) = N^n(0) + N^n_b(t) - N^n_d(t).
\]
We assume that $N^n_b$, $N^n_d$, and hence $N^n$ are right-continuous. We assume that the model is \emph{neutral}, which means that at an event where some number $k$ of individuals die, each of the $\binom{N(t-)}{k} $ sets of individuals is equally likely to be selected. Similarly, at an event where some number $k$ of individuals simultaneously give birth to, say, $c_1, \ldots ,c_k$ children, then each of the $\binom{N(t-)}{k} k!$ ways of choosing parents $1$ through $k$ are equally likely. If a birth and a death event happen at the same time, we use the convention of~\cite{DK99} that the death event is processed before the parents of the birth event are chosen. We assume that each individual carries a type or location in $E$; offspring copy their parents' location at birth, and in between birth and death events, locations evolve as independent, $E$-valued Markov processes corresponding to a specified generator $G^n$.

Donnelly and Kurtz show in their seminal paper~\cite{DK99} that the possible scaling limits of such models, under the additional restriction that simultaneous births to multiple parents cannot occur, are exactly the $\Lambda$-Fleming-Viot processes. The generalisation to $\Xi$-Fleming-Viot processes is the topic of~\cite{xiflemingviot}. The main ingredient in their proofs is a lookdown construction for the pre-limiting population model that mirrors that of the $\Xi$-Fleming-Viot process.

\medskip\paragraph{\textbf{Lookdown Construction}}
The (modified) lookdown construction~\cite{DK99,xiflemingviot} for the population model described above is a way of labelling its individuals in a particular order, $\boldsymbol Y^n(t) = (Y^n_1(t), \ldots ,Y^n_{N^n(t)}(t))$. We summarise this construction here, and refer to Section~2 of~\cite{xiflemingviot}---and Section 1.2 of~\cite{DK99} for the special case where simultaneous births to multiple parents are not allowed---for a details. We describe the construction under the assumption that birth and death events cannot happen simultaneously, but it can be adapted to allow this under the convention that deaths are always processed first, see also~\cite{DK99}. For ease of notation, we suppress the dependence on $n\in \N$ for the following description.
\begin{enumerate}
    \item[(i)] If some number $k$ of individuals die, then the $k$ individuals with the largest levels are removed.
    \item[(ii)] If some number $k$ of individuals simultaneously give birth to some number $c_1,\ldots,c_k$ of children at time $t$, then $k$ disjoint sets of levels $J_1,\ldots,J_k \subset \{1,\ldots,N(t)\}$ with sizes $|J_i| = c_i + 1$ are sampled uniformly at random (from the \emph{post}-reproduction set of levels; remember that $N$ is right-continuous, so $N(t) = N(t-) + \sum_{i=1}^k c_i$). We may assume (by relabelling) that $l_1 = \min J_1 < \ldots < l_k = \min J_k$. The smallest level of the first family is guaranteed to satisfy $l_1 \le N(t-) - k + 1$---in particular it has been alive before the reproduction event---and it is the parent of the first family: $Y_j(t) = Y_{l_1}(t-)$ for $j\in J_1$. We ``bump up'' levels that where in $J_1 \setminus \{l_1\}$ before the reproduction event and repeat. That is, the (pre-reproduction level of the) parent of the second family is $\tilde{l_2} = l_2 - |J_1 \cap [l_2]|$, which is guaranteed to satisfy $\tilde{l_2} \le N(t-) - k + 2$. Continuing like this, we obtain parents $\tilde{l_i} \le N(t-) - k + i$ for $i \in [k]$ and put $Y_j(t) = Y_{\tilde{l_i}}(t-)$ for $j\in J_i$. Levels in $[N(t)] \setminus (J_1 \cup \ldots\cup J_k)$ copy the types of the non-parental pre-reproduction types, retaining their order.

    If $k=1$ and $c_1 = 1$, say $J_1 = \{i,j\}$ with $i < j$, then $L^{ij}$ jumps. Otherwise, for $l\in [N(t)]$ and $i\in [k]$, if $J_i \cap [l] \neq \emptyset$ then $L^l_{J_i \cap [l],i}$ jumps, and further, denoting $J = \bigcup_{i=1}^k J_i$, if $J \cap [l] \neq \emptyset$ then $L^l_{J\cap [l]}$ jumps.
    \item[(iii)] Between birth and death events, individuals evolve as independent Markov processes in $E$ corresponding to the generator $G$.
\end{enumerate}
For convenience, define $Y^n_i(t)$ for $i > N(t)$ by setting it equal to either $Y^n_i(s-)$ where $s = \sup \{u \in [0,t]\colon N^n(u) \ge i\}$ if $\max_{u\in [0,t]} N^n(u) \ge i$, and otherwise to $Y^n_i(0)$ (for this to be well-defined, an initial location has to be defined for all levels). This defines $\boldsymbol Y^n(t) = (Y^n_1(t),Y^n_2(t),\ldots) \in E^\infty$ for all $t \ge 0$. We further denote the collection of level processes by $\boldsymbol L^n$, as in Section~\ref{sec:xiflemingviot}

As before, the point of the lookdown construction lies in the following fact.
\begin{lemma}
    If $\boldsymbol Y^n(0)$ is exchangeable, then $\boldsymbol Y^n(t)$ is exchangeable for all $t \ge 0$.
\end{lemma}
\begin{proof}
    This was proved in the case where simultaneous births to multiple parents are not allowed in Theorem~1.1 of~\cite{DK99}, and in the case where they are allowed but population size is constant in Theorem~2.2 of~\cite{xiflemingviot}. Adapting their arguments to this setting, where simultaneous births to multiple parents are allowed and population size may be variable, is straight-forward.
\end{proof}

\medskip\paragraph{\textbf{Genealogy}}
Let $T > 0$. Just as for the $\Xi$-Fleming-Viot process, we can use the lookdown construction $(\boldsymbol L^n,\boldsymbol Y^n)$ to construct the genealogy of the first say $\ell\in \N$ levels of the population, $(\fpart_t^{n,\ell},\boldsymbol X_t^{n,\ell})_{t\in [0,T]}$, which is a $\overline{\mathcal{X}}$-valued stochastic process with initial condition $\fpart_0^{n,\ell} = \left\{ \left\{ 1 \right\} , \ldots ,\left\{ \ell \right\}  \right\} $ and $\boldsymbol X_0^{n,\ell}(\{i\}) = Y^n_i(T-)$.

\medskip\paragraph{\textbf{Limit of large population size}}
Donnelly and Kurtz~\cite{DK99} show (in the case where there are no simultaneous births to multiple parents, with generalisations due to~\cite{xiflemingviot}) that the possible scaling limits of $(\boldsymbol L^n,\boldsymbol Y^n)$, under an assumption of asymptotically constant normalised population size $\frac{1}{n} N^n(t)$, are the associated processes $(\boldsymbol L,\boldsymbol Y)$ of the $\Xi$-Fleming-Viot processes, and they provide conditions under which this convergence holds. The main result of this section is the following theorem, which states that the genealogies of the scaling limit can be shown to converge to the Brownian spatial $\Xi$-coalescent at no additional cost.

\begin{theorem}\label{thm:neutralgenealogy}
    Let $\Xi$ be a finite measure on $\Delta$, and $T > 0$. Denote by $\boldsymbol (\boldsymbol L(t),\boldsymbol Y(t))_{t\in [0,T]}$ the particle representation of the $\Xi$-Fleming-Viot process at stationarity.  % , and by $(\fpart_t,\boldsymbol X_t)_{t\in [0,T]}$ the
    Suppose that $(\boldsymbol L^n,\boldsymbol Y^n) \implies (\boldsymbol L,\boldsymbol Y)$ weakly. Then the weak limit \[
        (\fpart^{n,\ell},\boldsymbol X^{n,\ell}) \implies (\fpart^\ell, \boldsymbol X^\ell)
    \] exists, is almost-surely $\mathcal{X}$-valued, and has the law of a Brownian spatial $\Xi$-coalescent with random initial condition $\fpart^\ell_0 = \{\{1\}, \ldots ,\{\ell\}\}$, $\boldsymbol X^\ell_0(\{i\}) = Y_i(T-)$.
\end{theorem}

If the spatial motion of the prelimiting population is Brownian motion (rather than converging to it asymptotically), then the assumption of Theorem~\ref{thm:neutralgenealogy} simplifies by the following lemma.

\begin{lemma}\label{lem:simplerassumption}
    If $G^n$ is the generator of standard Brownian motion for all $n\in \N$, then $\boldsymbol L^n \implies \boldsymbol L$ and $\boldsymbol Y^n(0) \implies \boldsymbol Y(0)$ imply $(\boldsymbol L^n,\boldsymbol Y^n) \implies (\boldsymbol L, \boldsymbol Y)$.
\end{lemma}

%\Cref{thm:neutralgenealogy} essentially means that whenever convergence of the non-spatial genealogy of a neutral population model to the $\Xi$-coalescent (or its special cases the Kingman, Beta, Bolthausen-Sznitman, or $\Lambda$-coalescent) is proved using a lookdown construction, convergence of the spatial genealogies to the corresponding Brownian spatial coalescent follows for free as long as the forward-in-time spatial motions converge.
A recent example application of \cref{thm:neutralgenealogy} is Theorem~1.8 of~\cite{ruairi}.

\subsection{Related Literature}\label{sec:relatedliterature}
% cite THE COALESCENT STRUCTURE OF CONTINUOUS-TIME GALTON–WATSON TREE? (no it doesnt have space...)
The inclusion of space (or type) in population and coalescent models is the subject of a vast amount of literature. The classical approach is to consider populations that are subdivided into demes of large constant size, each situated at a vertex of a graph. Lineages can merge with other lineages that are currently in the same deme, and migrate between neighbouring demes. We refer to \cite{EBV2010} for a short review of classical results in this setting, and point out \cite{spatiallambda,smith2021} and references therein for recent results. If the population is thought to be (discretely) structured by type rather than space, then associated coalescents are commonly called \emph{multi-type coalescents}. Recent results using this terminology include \cite{multitypelambda,canningsmultitype,multitypepoisson}.

In reality, populations are often not subdivided but spread across a spatial continuum. One might be tempted to approximate this through deme-based models with small granularity, but since local population size can be small this would break the crucial assumption of large population size in every deme; see also \cite{E2008}, Section 5. %Therefore, new ideas are necessary that inherently incorporate continuous space.
A common approach that incorporates continuous space directly is to assume that genealogical trees can be constructed from independent Brownian motions (or more general spatial motions \cite{evans}) which coalesce when they meet, either instantaneously or at a rate proportional to their collision local time. This is inherently limited to one dimension, and a suggested extension to higher dimensions has lineages coalesce at rates depending on their separation. The position of the common ancestor is typically taken to be a Gaussian centred on the midpoint between the two lineages immediately before the coalescence event. Aside from the fact that this behaviour is biologically unnatural, %this process is not sampling consistent (in the classical or our sense), and
there is no corresponding forwards in time model for the evolution of the population \cite{EBV2010}.

A major breakthrough in the theory of spatial population models was the introduction of the Spatial $\Lambda$-Fleming-Viot process (SLFV) by Barton, Etheridge and V\'eber in \cite{E2008,EBV2010}, see also \cite{EBV2013} for a review. It models the genetic composition of a spatially structured population by a measure on $\R^d\times [0,1]$, where $\R^d$ is the geographical space and $[0,1]$ is the space of genetic types. The population density is modelled to be constant so that the spatial marginal is always Lebesgue measure. The population evolves through a sequence of events, each of which replaces a certain proportion of the population in a randomly chosen ball by offspring of an individual chosen at random within the same ball. If the radius is small, the event is interpreted as an ordinary reproduction event subject to local regulation, and if it is large one may think of it as an extinction--recolonisation event. The genealogies of the SLFV (also called its dual) have a simple description: backwards in time, whenever a lineage is hit by an event, with a certain probability (equal to the proportion of the population that was hit forwards in time) it jumps to a randomly chosen location within the same ball, otherwise it stays put. If a number of lineages are affected by the same event, they jump to the same location and coalesce. The dual of the SLFV was the first tractable model for the ancestry of a population evolving in a two-dimensional spatial continuum that was sampling consistent, had a natural corresponding forward in time population model, and in which independently evolving lineages fail to meet. A variant of the SLFV allows for more than one parent per event, leading to the possibility of simultaneous mergers in the dual.

%More recently, Freeman introduced the segregated $\Lambda$-coalescent \cite{F2015}. The underlying concept of space is very general, and the results demonstrate that spatial structure can engender some interesting behaviours not visible in the non-spatial $\Lambda$-coalescent. However, it does not seem to be clear whether an associated forward in time model exists, and the model does not include the possibility of simultaneous merge events.
%But as in the SLFV, population density is assumed constant, and lineages evolve in a sequence of correlated jumps.

Despite its success, the SLFV has weaknesses. It only models populations with perfectly homogeneous spatial distribution, which in real populations fluctuates strongly in time and both geographical and type space \cite{E2008}, Chapter 7. Secondly, individual lineages in the SLFV evolve through a sequence of correlated jumps, but in many applications it would be more natural if they followed Brownian motions, or the paths of other continuous Markov processes.
The $\Xi$-Fleming-Viot processes are a class of spatial population model that showcase both characteristics. %, and their genealogies at stationarity are described by the Brownian spatial coalescent.

%For a more general example, we can replace the non-spatial Fleming-Viot process by the $\Lambda$ (or generalised) Fleming-Viot process introduced in \cite{BL2003} %(see also E2008 for more discussion)
%to obtain what one might call a $\Lambda$-Fleming-Viot superprocess, whose genealogies are described by what we will call the ``Brownian spatial $\Lambda$-coalescent''. Note that non-spatial equivalents are already known: the duals of the ordinary and $\Lambda$-Fleming-Viot processes are respectively the Kingman and $\Lambda$-coalescent \cite{BL2003}.
%We further note that Brownian motion could, in principle, both in the definition of the Fleming-Viot superprocess and the coalescent theory we present in this paper, be replaced by another continuous Markov spatial motion on a suitable Polish space, see also \cref{rem:generalmotion}.

A different approach to spatial coalescents is to replace geographical space with a tree-like structure, such as the hierarchical group which is often used to mimic higher-dimensional spaces. A spatial variant of the $\Lambda$-coalescent on the hierarchical group is the subject of \cite{greven}. In \cite{F2015}, Freeman also works on a space with a hierarchical structure, on which he considers a coalescent model called the segregated $\Lambda$-coalescent.

% discussion of markov property

\subsection{Conclusion and Outlook}\label{sec:conclusion}

%We introduced the Brownian spatial coalescent, a class of coalescent processes in continuous space of any dimension in the most general case of simultaneous multiple mergers.
%It has the following three crucial properties, all of which are, to the best of our knowledge, novel in the literature on coalescent models in continuous space in $d \ge 2$.
%First, an individual lineage evolves as a Brownian motion. Second, they describe the genealogies of a class of spatial population models whose spatial distribution is non-homogeneous and fluctuates randomly in time and space, driven by its forward dynamics. Third and most outstandingly, they can be axiomatically characterised through a notion of sampling consistency in direct analogy to the non-spatial $\Lambda$- and $\Xi$-coalescents. % To our knowledge, no present coalescent model in continuous space in $d \ge 2$ satisfies only one of these three characteristics (barring those that do not describe the genealogies of any forwards in time population models).
We discuss some consequences of our results, and possible further research directions.

\subsubsection{Spatially Interactive Branching}\label{sec:interactivebranching}
The Brownian spatial coalescent was axiomatically defined through a property that the genealogies of any spatial population model in which individuals follow independent Brownian motions should satisfy. This includes models with spatially interactive branching mechanisms such as competition or local regulation.
Yet \cref{sec:populationmodelsintro} shows that the full set of sampling consistent Brownian spatial coalescents is exhausted by spatial models in which the branching mechanism is \emph{oblivious} to spatial locations. The only other assumptions made in the definition of Brownian spatial coalescents are time-homogeneity and the Markov property. If the spatial model has a stationary distribution, which will certainly be the case for some interesting spatially interactive branching mechanisms, then their genealogies will be time homogeneous, which means they \emph{cannot be Markovian}. Thus, describing the genealogies of populations whose branching mechanisms have interesting spatial interactions, at least in the present setting, requires the study of non-Markovian coalescent processes. Since we expect most of our results to remain true for much more general spatial motions (see next paragraph), this observation is not limited to population models with Brownian movement.

\subsubsection{General Spatial Motions}\label{sec:generalmotion}
The definition of the Brownian spatial coalescent and most theorems and proofs (with the exception of the drift representation \cref{thm:driftintro}) work with little modification when Brownian motion on the torus is replaced by another well-behaved Markov process $(Y_t)$ on some compact Polish space $E$ with a stationary distribution $\lambda$ and continuous transition densities $(q_t)$ w.r.t.\ $\lambda$, such that its time reversal w.r.t.\ $\lambda$ is also a well-behaved Markov process. Then $(q_t)$ takes the place of $(p_t)$, and $\lambda$ takes the place of Lebesgue measure in the main theorems.

\begin{example}
    Examples include random walks, $\alpha$-stable motion or diffusions on fractals. One could also take a cartesian product of a geographical and a type space, for example $E = \mathbb{T}^d \times \{1, \ldots ,k\}$, where the torus is interpreted as geographical space, and $\{1, \ldots ,k\}$ is type space. The motion could be a random walk on type space, and Brownian motion in geographical space, with a speed that may depend on the type.
\end{example}

The only major difficulty is understanding which points have to be excluded from the state space. One approach could be to exclude those points where the normalisation $N^{\boldsymbol \nu}$ diverges, but then it remains to show that these points will almost-surely never be reached by the evolution of the coalescent.
This comes down to characterising whether or not $k$ independent copies of the spatial motion have a positive probability of meeting simultaneously, and if so in which points.
In fact, by studying the normalisation of the ``$(Y_t)$-spatial $\delta_1$-coalescent'', in which only $(n,n)$-mergers for $n\in \N$ are possible, this gives rise to the following conjecture.

\begin{conjecture}
    Under some regularity conditions on $(Y_t)$, the following are equivalent for every $k\in \N$, $k \ge 2$, and $x_1, \ldots ,x_k \in E$.
    \begin{enumerate}
        \item If $k$ independent copies of $Y$ are started from initial locations $x_1$ to $x_k$, then with positive probability there exists a time at which all $k$ copies are in the same location.
        \item \[
            \int_0^1 \int_E \prod_{i=1}^k q_s(x_i,z)^k\lambda(\diff z) \diff s < \infty.
        \]
    \end{enumerate}
\end{conjecture}
A simple case is if the motion is a random walk on some finite graph $E$. Then no points have to be excluded, and all of our results (except \cref{thm:driftintro}) apply with obvious modifications.

\subsubsection{The Brownian Spatial Coalescent on $\R^d$}
We have some unpublished results about Brownian spatial coalescents evolving on all of $\R^d$, which can be defined analogously to \cref{def:BSCintro}. Corresponding to each is again a family of transition measures $\boldsymbol \nu = (\nunk)$, where $\nu_{n,k_1, \ldots ,k_m}$ is a locally finite measure on $(\R^d)^m$. If $\boldsymbol \lambda = (\lank)$ are the rates of a $\Xi$-coalescent, then $\nunk(\diff \boldsymbol z) = \lank \diff \boldsymbol z$ is again a valid choice, and the corresponding Brownian spatial coalescent describes the genealogies of a $\Xi$-Fleming-Viot process evolving on $\R^d$, at ``stationarity modulo translation'' (cmp.\ \cite{DK96}, Theorem 2.9).

If $d \ge 3$, then the class of Brownian spatial coalescents that describe the genealogies of some population is larger than that: for each Brownian spatial coalescent on $\R^d$ that we described in the previous paragraph, there exists another whose definition is the same except that all exponential factors in the densities on tree decorations (of the form $\e^{-\nut{n}(\tau - \tau')}$) are removed.
%In that case, there exists a class of Brownian spatial coalescents on $\R^d$ with $d \ge 3$ that has no exponential factors in the densities defining its law (which were previously of the form $\e^{-\nut{n} t}$) even if not all transition measures are zero.
If $\nu_{n,2}(\diff z) = \diff z$ for all $n\in \N$, and all other transition measures are equal to zero, then this coalescent describes the genealogies of the infinite-mass superBrownian motion on $\R^d$ at stationarity (which exists iff $d \ge 3$ \cite{etheridgesuperprocesses}, Chapter 2.7). It has the property that, from any initial configuration with at least two particles, the probability that no further merge event happens is strictly between zero and one. The existence of coalescents with this property on $\R^d$ is directly linked to the existence of infinite-mass stationary distributions for associated Dawson-Watanabe superprocesses on $\R^d$, which is known to correspond to transience of the underlying spatial motion \cite{transient1,transient2}. This connection can also be seen from the perspective of the Brownian spatial coalescent: Brownian motion is transient if and only if $\int_C p_t(x) \diff x$ is integrable over $t\in (0,\infty)$ for every compact $C \subset \R^d$ (this characterisation holds mutatis mutandis for more general spatial motions), which is certainly necessary for the normalisation $N^{\boldsymbol \nu}(\boldsymbol x)$ to be finite, since integrability at large $\tau$ is no longer guaranteed by the exponential factors.

Finally, it is not clear how to define sampling consistency on $\R^d$ in such a way that it captures the coalescents associated with infinite-mass populations, because their stationary distributions are random measures of infinite mass that cannot be normalised to be probability measures.

\subsection{Structure}
In \cref{sec:setup}, we set up the notation necessary to formalise some of the definitions and statements from the introduction, and introduce the framework used for the proofs, which are in \cref{sec:proofs}. Some technical proofs are postponed to the Appendix, most notably the proof of \cref{lem:NF}, which states that the function $N^{\boldsymbol \nu}$ is continuous under certain conditions; it is conceptually straight-forward but long and technical. % \Cref{app:measurability} contains proofs of a number of technical lemmas used in \cref{sec:setup,sec:proofs}, and \cref{app:NF} contains the proof of \cref{lem:NF}, which states that the normalisation function $N^{\boldsymbol \nu}$ is continuous under certain conditions;

\section{Setup}\label{sec:setup}
\subsection{Non-spatial Coalescents}\label{sec:non-spatial}

Recall that $\mathcal{P}$ denotes the (countable) set of partitions of finite subsets of $\N$, which we equip with the discrete topology. For $\fpart, \fpart' \in \mathcal{P}$ write $\fpart \le \fpart'$ if $\fpart'$ is a refinement of $\fpart$, and define \[
%, and $\mathcal{P}(\fpart) \coloneqq \left\{ \fpart'\in \mathcal{P}\colon \fpart' \le \fpart \right\}$. %Note that $\mathcal{P}(\fpart)$ stands in natural bijection with $\mathcal{P}_{\left\{ 1, \ldots ,|\fpart| \right\} }$, its maximal element is $\fpart$, and its minimal element is $\left\{\bigcup \fpart\right\}$, where $\bigcup \fpart = \bigcup_{B\in \fpart} B$. For example, $\mathcal{P}(\left\{ \left\{ 1 \right\} ,\left\{ 2,3 \right\}  \right\} )$ comprises $\left\{ \left\{ 1 \right\} ,\left\{ 2,3 \right\}  \right\} $ and $\left\{ \left\{ 1,2,3 \right\}  \right\} $.
    \Omega_0 \coloneqq \left\{ \fpart \in R([0,\infty),\mathcal{P})\colon \fpart_s \ge \fpart_t \forall 0 \le s \le t \right\},
\] where $R([0,\infty),\mathcal{P})$ is the space of right-continuous $\mathcal{P}$-valued paths. (Due to the monotonicity condition, every path in $\Omega_0$ must in fact also have left limits.)
We equip $\Omega_0$ with its Borel $\sigma$-algebra $\mathcal{F}^\fpart \coloneqq \sigma(\fpart_t\colon t\ge 0)$, and filtration $\mathcal{F}^\fpart_t \coloneqq \bigcap_{s > t} \sigma(\fpart_r\colon r \le s)$ for $t \ge 0$, where $(\fpart_{t})_{t\ge 0} \coloneqq \id_{\Omega_0}$ denotes the canonical process in $\Omega_0$.
We further equip $\Omega_0$ with the product topology, which is the smallest topology on $\Omega$ with respect to which $\boldsymbol X_t \colon \Omega \to \mathcal{X}$ is continuous for all $t \ge 0$.

\begin{definition}\label{def:NSCP}
    A \emph{(non-spatial) coalescent} is a time-homogeneous $\Omega_0$-valued Markov process $(\P^{\fpart})_{\fpart\in \mathcal{P}}$. It is called \emph{label invariant} if it is independent of particle labels in the sense that any transition that is the result of merging $m$ disjoint sets of lineages of sizes $k_1\ge \ldots \ge k_m$, while at $n$ particles, occurs at the same rate $\lambda_{n,k_1, \ldots ,k_m}$.
\end{definition}

%We denote the laws of an NSCP by $(\P^\fpart)_{\fpart\in \mathcal{P}}$.
In general, a coalescent is characterised by its transition rates $\lambda_{\fpart,\fpart'} \ge 0$ for $\fpart > \fpart'$, which induces a bijection between the set of coalescents and the set $\rates$ of non-negative arrays $\boldsymbol \lambda = (\lambda_{\fpart,\fpart'}\colon \fpart > \fpart')$.
For $\fpart \in \mathcal{P}$ and $u\subset \N$, write $\fpart \wo u \coloneqq \left\{ v \setminus u\colon v \in \fpart, v \neq u \right\}$ for the partition induced by $\fpart$ on $\N\setminus u$. %
\begin{definition}\label{def:NSCP2}
    A coalescent is called \emph{sampling consistent} if, for any $\fpart \in \mathcal{P}$ with $|\fpart| \ge 2$ and $u\in \fpart$, \[
        \P^{\fpart}((\fpart_t\wo u)_{t\ge 0}\in \cdot ) = \P^{\fpart \wo u}((\fpart_t)_{t\ge 0}\in \cdot ).
    \]
\end{definition}

\subsection{Spatial Coalescents}\label{sec:spatialcoal}

Recall that $E = \R^d / \Z^d$ denotes the $d$-dimensional square flat torus with periodic boundary conditions. It is a Polish space with the Euclidean metric
\begin{equation}\label{eqdef:rho}
    \rho(x + \Z^d,y+\Z^d) = \inf \{ |x - y + \vec{k}|\colon \vec{k}\in \Z^d\},\quad x,y\in \R^d. %glossary E,\rho
\end{equation}
Recall the definitions of $E^\fpart$ and $E^\fpart_\circ $ for $\fpart\in\mathcal{P}$ (see \cref{eq:def:Ecirc}), and $\mathcal{X}$. The identification of $E^\fpart_\circ $ with an open subset of $E^{|\fpart|}$ defines on it a topology and Lebesgue measure, and then $\mathcal{X}$ is a Polish space with the topology inherited from $E^\fpart_\circ ,\, \fpart\in \mathcal{P}$, in which $(\fpart_n,\boldsymbol x_n) \to (\fpart,\boldsymbol x)$ iff $\fpart_n = \fpart$ for all but finitely many $n$, and $\boldsymbol x_n \to \boldsymbol x$. % In particular, $E^\fpart_\circ $ is a Polish space w.r.t.\ a topology in which $\boldsymbol x_n \to \boldsymbol x$ iff $\boldsymbol x_n(u) \to \boldsymbol x(u)$ in $E$ for every $u\in \fpart$.
Since $\fpart = \dom(\boldsymbol x)$ is the domain of $\boldsymbol x$ for $(\fpart,\boldsymbol x) \in \mathcal{X}$, we are justified to occasionally just write $\boldsymbol x\in \mathcal{X}$. Denote
\begin{equation*}%\label{eq:defOmega}
    \Omega \coloneqq \left\{ (\fpart_t,\boldsymbol X_t) \in R([0,\infty), \mathcal{X})\colon \fpart_t \le \fpart_s \forall t \ge s \ge 0 \right\},
\end{equation*}
where $R([0,\infty),\mathcal{X})$ denotes the space of right-continuous functions $[0,\infty) \to \mathcal{X}$.
We equip $\Omega$ with the product topology as well as its Borel $\sigma$-algebra $\mathcal{F}^{\boldsymbol X} \coloneqq \sigma(\boldsymbol X_t\colon t\ge 0)$ and the filtration $\mathcal{F}^{\boldsymbol X}_t \coloneqq \bigcap_{s > t} \sigma(\boldsymbol X_r\colon r \le s)$ for $t \ge 0$, where $(\boldsymbol X_t)_{t\ge 0}\coloneqq (\fpart_t,\boldsymbol X_t)_{t\ge 0} \coloneqq \id_{\Omega}$ denotes the canonical process in $\Omega$. Write $C_b(\mathcal{X})$ for the space of bounded and continuous functions $\mathcal{X}\to \R$.

\begin{definition}\label{def:coal}
    A \emph{spatial coalescent} is a time-homogeneous $\Omega$-valued Markov process $(\P^{\boldsymbol x})_{\boldsymbol x\in \mathcal{X}}$ whose \emph{semigroup} $(\semi_t)_{t\ge 0}$, defined by
    \begin{equation}\label{eq:semi}
        \semi_t(\boldsymbol x,A) \coloneqq \P^{\boldsymbol x}(\boldsymbol X_t \in A),
    \end{equation}
    for $\boldsymbol x\in \mathcal{X}$, $t\ge 0$, and $A\subset \mathcal{X}$ measurable,
    has the property that $\semi_t f \in C_b(\mathcal{X})$ for all $f \in C_b(\mathcal{X})$, which we refer to as the \emph{Feller property}.
\end{definition}
%The Feller property in this context states that the semigroup, defined by
%for $\boldsymbol x\in \mathcal{X}$, $t\ge 0$, and $A\subset \mathcal{X}$ measurable, satisfies $\semi_t f \in C_b(\mathcal{X})$ for every $f\in C_b(\mathcal{X})$, the set of bounded continuous functions $\mathcal{X}\to \R$.
The Feller property---which is automatic in the non-spatial setting---is a form of continuity in the initial condition that, together with right-continuity of sample paths, implies the strong Markov property~\cite[Thm.\ 8.3]{rogers}:
%The Feller property is a form of continuity in the initial condition which is automatic in the non-spatial setting, and it implies the strong Markov property:
if $T$ is an $(\mathcal{F}^{\boldsymbol X}_t)$-stopping time, then \[
    \P^{\boldsymbol x} \left(T < \infty, (\boldsymbol X_{T+t})_{t\ge 0} \in \cdot  \,\middle\vert\, \mathcal{F}^{\boldsymbol X}_T\right) = \ind_{\{T < \infty\}}\P^{\boldsymbol X_T}((\boldsymbol X_t)_{t\ge 0} \in \cdot ),\qquad \text{$\P^{\boldsymbol x}$-a.s.}
\] for every $\boldsymbol x\in \mathcal{X}$.

The formal definition of label invariance in the spatial setting is a bit more technical than in the non-spatial setting, because a convenient formulation in terms of transition rates is not available. (Within the class of Brownian spatial coalescents, \cref{lem:labelinv} provides an analogous characterisation.) Suppose that $\fpart_0,\fpart_1 \in \mathcal{P}$ with $|\fpart_0| = |\fpart_1|$. Then the law of a label invariant spatial coalescent should not be affected if we change the labels of the initial set of lineages with a bijection $\iota\colon \fpart_0 \to \fpart_1$. To apply this change of labels to the entire process $(\fpart_t)_{t\ge 0}$, we extend $\iota$ to a map $\bigcup_{\fpart \le \fpart_0} \fpart \to \bigcup_{\fpart \le \fpart_1} \fpart$ with the property that $\iota(u \cup v) = \iota(u) \cup \iota(v)$ whenever $u,v,u\cup v$ are in the extended domain of $\iota$ (this extension is unique). To perform the change of labels in the spatial variables, we extend $\iota$ for every $\fpart \subset \bigcup_{\fpart' \le \fpart_0} \fpart'$ to a map $E^\fpart_\circ  \to E^{\iota(\fpart)}_\circ $ through $\iota(\boldsymbol x) = \boldsymbol x \circ \iota ^{-1}$. Denote by $\labelinv(\fpart_0,\fpart_1)$ the set of bijections $\iota\colon \fpart_0 \to \fpart_1$ that are extended %to the domain $\bigcup_{\fpart \le \fpart_0} \fpart \cup \bigcup_{\fpart \subset \bigcup_{\fpart' \le \fpart_0} \fpart'}  E^\fpart_\circ )$
in the way described above. %glossary

\begin{definition}\label{def:labelinv}
    A spatial coalescent $(\P^{\boldsymbol x})_{\boldsymbol x\in \mathcal{X}}$ is called \emph{label invariant} if for any two sets $\fpart_0,\fpart_1 \in \mathcal{P}$ of equal size, $\iota \in \labelinv(\fpart_0,\fpart_1)$, and any $\boldsymbol x \in E^{\fpart_0}_\circ $,
    \begin{equation}\label{eq:labelinv}
        \P^{\iota(\boldsymbol x)}((\boldsymbol X_t)_{t \ge 0} \in \cdot ) = \P^{\boldsymbol x}((\iota(\boldsymbol X_t))_{t \ge 0} \in \cdot ).
    \end{equation}
    %where we uniquely extended $\iota$ to a map $\bigcup_{\fpart \le \fpart_0} \fpart \to \bigcup_{\fpart \le \fpart_1} \fpart$ with $\iota(u \cup v) = \iota(u) \cup \iota(v)$; if $\fpart \le \fpart_0$ and $\boldsymbol x \in E_\circ^{\fpart}$, then $\iota(\boldsymbol x) \coloneqq \boldsymbol x \circ \iota^{-1} \in E^{\iota(\fpart)}_\circ $.
\end{definition}

%todo glossary of notation
% - [n] = {1, \ldots ,n}
% - \rates and \nurates
% - everywhere with a 'glossary' comment

%A label invariant spatial coalescent is determined by the laws $\P^{\boldsymbol x}$ where $(\fpart,\boldsymbol x) \in \mathcal{X}$ with $\fpart = \left\{ \{1\}, \ldots ,\{n\} \right\} $ for some $n\in \N$. % If $x_i \coloneqq \boldsymbol x(\{i\})$ for $i\in [n]$, we also write $\P^{x_1, \ldots ,x_n}$ for $\P^{\boldsymbol x}$. In this context we occasionally identify $\boldsymbol x \in \mathcal{X}$ with the tuple $(x_1, \ldots ,x_n) \in E^n$ if no confusion is possible.
%In this and similar contexts, we occasionally identify a tuple $\boldsymbol x = (x_1, \ldots ,x_n) \in E^n$ with the map $\boldsymbol x \in E^{\{\{1\}, \ldots ,\{n\}\}}$ defined by $\boldsymbol x(\{i\}) = x_i$ for $i\in [n]$. The subset of $E^n$ obtained by applying this identification to the elements of $E^{\{\{1\}, \ldots ,\{n\}\}}_\circ $ is denoted $E^n_\circ $. For $\boldsymbol x = (x_1, \ldots ,x_n) \in E^n$ and $y\in E$ write $\boldsymbol x y \coloneqq (x_1, \ldots ,x_n,y) \in E^{n+1}$. Then for $\boldsymbol x \in E^n_\circ $ write
%\begin{equation}\label{eq:Ex}
%    E_{\boldsymbol x} \coloneqq \left\{ y\in E\colon \boldsymbol xy \in E^{n+1}_\circ  \right\} .
%\end{equation}
%
We now make precise the notion of sampling consistency for spatial coalescents introduced informally in \cref{def:samplingconintro}.
If $\boldsymbol x\in E^A$ and $\boldsymbol y\in E^B$ for disjoint sets $A$ and $B$, we can join the two maps and denote the result by $\boldsymbol x\boldsymbol y\in E^{A\cup B}$, so
\begin{equation}\label{eq:joinmaps}
    \boldsymbol x \boldsymbol y \colon A \cup B \to E;\, u \mapsto
    \begin{cases}
        \boldsymbol x(u), & u \in A,\\
        \boldsymbol y(u), & u \in B.
    \end{cases}
\end{equation}
For $(\fpart,\boldsymbol x) \in \mathcal{P}$, $u\subset \N$ with $u \cap \fpart = \emptyset $ and $y\in E$, write $\boldsymbol x y \colon \fpart \cup \{u\} \to E$ for the extension of $\boldsymbol x$ to $\fpart \cup \{u\}$ with $u\mapsto y$, where $u$ is suppressed in the notation and will be clear from context. Then for $(\fpart,\boldsymbol x) \in \mathcal{X}$ and any such $u \subset \N$, and $\fpart'\in \mathcal{P}$ disjoint from $\fpart$, write
\begin{align}\begin{split}\label{eq:Ex}
    E_{\boldsymbol x} &\coloneqq \{ y\in E\colon \boldsymbol xy \in E^{\fpart \cup \{u\}}_\circ  \},\\
    E_{\boldsymbol x}^{\fpart'} &\coloneqq \{ \boldsymbol y \in E^{\fpart'}_\circ \colon \boldsymbol x\boldsymbol y \in E^{\fpart\cup\fpart'}_\circ \}.
\end{split}\end{align}
(The first definition does not depend on the choice of $u$.)
For $(\fpart,\boldsymbol x) \in \mathcal{X}$ and any $u\in \fpart$, define $(\boldsymbol x \wo u)\colon (\fpart \wo u) \to E$ by $v \setminus u \mapsto \boldsymbol x(v)$. Write $\mathcal{M}_1(S)$ for the set of probability measures on a measurable space $S$.
\begin{definition}\label{def:samplingcon}
    A spatial coalescent $(\P^{\boldsymbol x})_{\boldsymbol x\in \mathcal{X}}$ is called \emph{sampling consistent} if there exists a family of probability measures $(\mu_{\boldsymbol x} \in \mathcal{M}_1(E_{\boldsymbol x})\colon \boldsymbol x\in \mathcal{X})$ such that
    %for any $n\in \N$, $\boldsymbol x \in E^n_\circ $,
    %\begin{align}\label{eq:samplingcon}
    %    \int \P^{\boldsymbol x y}((\boldsymbol X_t \wo \{n+1\})_{t \ge 0} \in \cdot ) \mu_{\boldsymbol x}(\diff y) = \P^{\boldsymbol x}((\boldsymbol X_t)_{t\ge 0 } \in \cdot ).
    %\end{align}
    for any $\fpart \in \mathcal{P}$, $|\fpart| \ge 2$, and $u\in \fpart$, $\boldsymbol x\in E^{\fpart\wo u}_\circ $,
    \begin{align}\label{eq:samplingcon}
        \int_{E_{\boldsymbol x}} \P^{\boldsymbol x y}((\boldsymbol X_t \wo u)_{t\ge 0} \in \cdot ) \mu_{\boldsymbol x}(\diff y) = \P^{\boldsymbol x}((\boldsymbol X_t)_{t\ge 0} \in \cdot ).
    \end{align}
\end{definition}

With these definitions, the projection of a (label invariant or sampling consistent) spatial coalescent onto $\mathcal{P}$ is a (label invariant or sampling consistent) coalescent.

\subsection{Brownian Spatial Coalescents}
The law of a Brownian spatial coalescent will be defined using densities over time and space decorations of all possible shapes of the genealogical forest, as outlined in the introduction. We introduce the necessary notation.

\begin{definition}\label{def:forest}
    A \emph{forest} is a set whose members form a strictly decreasing (necessarily finite) sequence in $\mathcal{P}$.
\end{definition}

All of the following notation is illustrated in \cref{fig:treenotation}. The notation for a generic forest is $F = \{\fpart_0^F, \ldots ,\fpart_m^F\}$, where $\lf(F) \coloneqq \fpart_0^F$ is the set of leaves, $\fpart_i^F$ for $i \in [m]$ is the set of nodes (or vertices) immediately after the $i$'th merge event, and $\rt(F) \coloneqq \fpart_m^F$ is the set of roots. Call $F$ \emph{trivial} if $m = 0$, that is if $\rt(F) = \lf(F)$. Note that this notion of a forest encodes its topology \emph{and} the order (and simultaneousness) of merge events, but no branch lengths or spatial information. This is important because already in the non-spatial setting, forests with the same topology but different order of merge events may have different probability.
Write $\mathbb{F}$ for the set of all forests, and for $\fpart,\fpart'\in \mathcal{P}$ with $\fpart' \le \fpart$,
\begin{align*}
    \mathbb{F}(\fpart) &= \left\{ F\in \mathbb{F}\colon \lf(F) = \fpart \right\},\\
    \mathbb{F}(\fpart,\fpart') &= \left\{ F\in \mathbb{F}\colon \lf(F) = \fpart, \rt(F) = \fpart' \right\} .
\end{align*}
We call $F\in \mathbb{F}$ a \emph{tree} if $|\rt(F)| = 1$, and denote by $\mathbb{T}$ and $\mathbb{T}(\fpart)$ for $\fpart \in \mathcal{P}$ the set of all trees, and the set of trees with leaves $\fpart$, respectively.
Write $\nd(F) \coloneqq \bigcup_{i=0} ^m \fpart_i^F$ for the set of all nodes, and $\nd^\circ (F) \coloneqq \nd(F) \setminus \lf(F)$ for the non-leaf nodes of $F$, or equivalently the set of all nodes that correspond to a merge event. Note that the union defining $\nd(F)$ is not disjoint; in particular, e.g.\ $\{9\}$ in \cref{fig:treenotation} is not in $\nd^\circ (F)$ despite being a root.
%We occasionally identify $F$ with either of the sets $\left\{ \fpart_0^F, \ldots ,\fpart_m^F \right\} $ or $\left\{ (\fpart_0^F,\fpart_1^F), \ldots ,(\fpart_{m-1}^F,\fpart_m^F) \right\} $.
The maps \[
    \ch_F \colon \nd(F) \to 2^{\nd(F)},\qquad \pr_F\colon \nd(F) \to \nd(F) \cup \left\{ \emptyset  \right\}
\] assign to a node its children and its parent, respectively, where $\ch_F(u) \coloneqq \emptyset $ for $u\in \lf(F)$ and $\pr_F(u) \coloneqq \emptyset $ for $u\in \rt(F)$. \Cref{fig:treenotation} illustrates this notation.
\begin{figure}
    \centering
        \begin{tikzpicture}[scale=1.2]
    \coordinate (A) at (0,0);
    \coordinate (B) at (1,0);
    \coordinate (C) at (2.5,0);
    \coordinate (D) at (3.5,0);
    \coordinate (D1) at (3.5,1.5);
    \coordinate (D2) at (3.5,3);
    \coordinate (AB) at (0.5,1.5);
    \coordinate (ABC) at (1.75,3);

    \draw (A) node[below] (Alabel) {\small$\left\{ 1 \right\} $};
    \draw (B) node[below] (Blabel) {\small$\left\{ 3,4 \right\} $};
    \draw (C) node[below] (Clabel) {\small$\left\{ 7 \right\} $};
    \draw (D) node[below] (Dlabel) {\small$\left\{ 9 \right\} $};

    \draw (A) to[out=90,in=-90-30] (AB);
    \draw (B) to[out=90,in=-90+30] (AB);
    \draw (AB) to[out=60,in=-90-45] (ABC);
    \draw[name path=c] (C) to[out=90+0,in=-65] (ABC);
    \draw (D) to (D2);

    \draw (AB) node[left] (ABlabel) {\small$\left\{ 1,3,4 \right\} $};
    \draw (ABC) node[left] (ABClabel) {\small$\left\{ 1,3,4,7 \right\} $};
    \draw (D1) node[above right, black!10!gray] {\small$\left\{ 9 \right\} $};
    \draw (D2) node[above right, black!10!gray] (D2label) {\small$\left\{ 9 \right\} $};

    \draw[dotted,thick] (A) -- (4.5,0) node[right] {\small$\fpart_0^F = \left\{ \left\{ 1 \right\} ,\left\{ 3,4 \right\} ,\left\{ 7 \right\},\left\{ 9 \right\}   \right\} $};
    \draw[dotted,thick, name path = pi1] (AB) -- (4.5,1.5) node[right] {\small$\fpart_1^F = \left\{ \left\{ 1, 3,4 \right\} ,\left\{ 7 \right\}  ,\left\{ 9 \right\} \right\} $};
    \draw[dotted,thick] (ABC) -- (4.5,3) node[right] {\small$\fpart_2^F = \left\{ \left\{ 1,3,4, 7 \right\} ,\left\{ 9 \right\}  \right\} $};
    \node[
        draw=teal,
        thick,
        rounded corners=6pt,
        inner sep=6pt,
        fit=(ABC)(ABClabel)(D2)(D2label),
    ] (looproot) {};
    % Label to the left of the loop
    \node[teal, anchor=north west] at ($(looproot.north)+(0.0,-0.0)$) {$\rt(F)$};
 
    \node[
        draw=orange,
        thick,
        rounded corners=6pt,
        inner sep=6pt,
        fit=(AB)(ABC)(ABlabel)(ABClabel)
    ] (loopinner) {};
    
    % Label to the left of the loop
    \node[orange, anchor=west] at ($(loopinner.west)+(0.1,0.0)$) {$\nd^\circ(F)$};

    \coordinate (fake) at (-1.1,0);
    \node[
        draw=blue,
        thick,
        rounded corners=6pt,
        inner sep=6pt,
        fit=(A)(B)(C)(D)(Alabel)(Blabel)(Clabel)(Dlabel)(fake)
    ] (loopleaf) {};
    \node[blue, anchor=west] at ($(loopleaf.west)+(0.1,0.0)$) {$\lf(F)$};
    
    \draw[fill,black] (A) circle[radius=1.5pt];
    \draw[fill,black] (B) circle[radius=1.5pt];
    \draw[fill,black] (C) circle[radius=1.5pt];
    \draw[fill,black] (D) circle[radius=1.5pt];
    \draw[fill,gray] (D1) circle[radius=1.5pt];
    \draw[fill,gray] (D2) circle[radius=1.5pt];
    \draw[fill,black] (AB) circle[radius=1.5pt];
    \draw[fill,black] (ABC) circle[radius=1.5pt];
    %\node[name intersections={of=pi1 and c}] (I1) at (intersection-1);
    % \draw[fill,black,name intersections={of=pi1 and c}] (intersection-1) circle[radius=1.5pt];
    \draw[
        fill,
        gray,
        name intersections={of=pi1 and c}
    ] (intersection-1) circle[radius=1.5pt]
      node[above right, black!10!gray] {\small$\{7\}$};
\end{tikzpicture}
    \caption{Illustration of the notation used for forests. In this example, $\ch_F(\left\{ 1,3,4 \right\} ) = \left\{ \left\{ 1 \right\} ,\left\{ 3,4 \right\}  \right\} $, and $\pr_F(\left\{ 7 \right\}) = \left\{ 1,3,4,7 \right\}  $, and $\ch_F(\left\{ 9 \right\} ) = \pr_F(\left\{ 9 \right\} ) = \emptyset $. Nodes that appear more than once in the tree are drawn in black the first time they appear (starting at the leaves), and gray afterwards. Note in particular that $\{9\}$ is both a leaf and a root (but not an inner node). }
        \label{fig:treenotation}
\end{figure}

\begin{definition}\label{def:tau}
    A \emph{time decoration} of a forest $F$ is a map $\tau \colon F \to (0,\infty)$ with $\tau(\lf(F)) = 0$ and $\tau(\fpart) < \tau(\fpart')$ for $\fpart > \fpart'$. Write $\dct(F)$ for the set of time decorations of $F$, which we equip with the \emph{upper limit topology}, that is $\tau^n \to \tau$ in $\tm(F)$ if for every $\fpart \in F$, $\tau^n(\fpart) \to \tau(\fpart)$ \emph{and} there is $n_0\in \N$ such that $\tau^n(\fpart) \le \tau(\fpart)$ for all $n\ge n_0$. Write \[
        \dct(\mathbb{F}) = \left\{ (F,\tau) \colon F\in \mathbb{F}, \tau\in \dct(F) \right\}
    \] for the set of time decorated forests, and define $\dct(S)$ for $S \subset  \mathbb{F}$ analogously.
\end{definition}
Since $F = \dom(\tau)$ for $(F,\tau) \in \dct(\mathbb{F})$ we occasionally just write $\tau \in \dct(\mathbb{F})$.
If $F = \{\fpart_0^F, \ldots ,\fpart_m^F\} \in \mathbb{F}$ and $\tau\in \dct(F)$, then for $i\in [m-1]$ we write $\tau(\fpart_i^F)^+ \coloneqq \tau(\fpart_{i+1}^F)$, and $\tau(\fpart_m^F)^+ \coloneqq \infty $. We extend $\tau$ to $\nd(F)$ by writing 
\begin{equation}\label{eq:tau_u}
    \tau_u \coloneqq \min \left\{ \tau(\fpart)\colon u\in \fpart \in F \right\} ,\quad u\in \nd(F),
\end{equation} for the time of a nodes ``birth'', and we set $\tau_\emptyset = \infty$ so that $\tau_{\pr_F(u)} = \infty$ if $u$ has no parent. See \cref{fig:dctdcs} for an illustration.

If $F$ is trivial, then $\dct(F)$ is a singleton whose only element is defined by $\tau_{\lf(F)} = 0$, and we equip $\dct(F)$ with the measure that assigns mass one to the single member.
Otherwise, since the Borel $\sigma$-algebra of the upper limit topology coincides with that of the usual Euclidean topology, we can define Lebesgue measure on $\dct(F)$ using the obvious bijection with an open subset of $\R^m$. In both cases we denote the integral of a function $f$ against the measure on $\dct(F)$ by $\int f(\tau)\diff \tau$. The set $\dct(\mathbb{F})$ inherits its topology from $\dct(F)$, noting it can be identified with a discrete union.

\begin{definition}\label{def:xi}
    A \emph{spatial decoration} of a forest $F=(\fpart^F_0, \ldots ,\fpart^F_m)$ is a map $\xi\colon \nd^\circ (F) \to E$ such that $\xi\vert_{\fpart^F_i \setminus \fpart^F_{i-1}} \in E^{\fpart^F_i \setminus \fpart^F_{i-1}}_\circ $ for all $i=1, \ldots ,m$. Write $\dcs(F)$ for the set of spatial decorations of $F$.
\end{definition}
Note that $\xi\vert_{\fpart^F_i \setminus \fpart^F_{i-1}}$ are exactly the locations of the $i$'th set of (simultaneously occuring, if more than one) merge events, so this condition is ensuring that no two (or three, if $d = 1$) simultaneously occuring merge events are in the same location. (Note that when writing $A \setminus B$ we mean $A \cap B^c$, and $B$ does not necessarily have to be a subset of $A$.)
If $F$ is non-trivial, we define a topology and Lebesgue measure on $\dcs(F)$ using the obvious bijection with an open subset of $E^{|\nd^\circ (F)|}$. If $F$ is trivial then $\dcs(F)$ is a singleton comprised of the unique map $\xi\colon \emptyset \to E$, and we equip $\dcs(F)$ with the measure that assigns mass one to the single member. In both cases write $\int f(\xi) \diff \xi$ for the integral of a function $f$ against the measure on $\dcs(F)$.

The following maps assign to a path in $\Omega_0$ its associated undecorated and time decorated coalescence forest, respectively.
\begin{IEEEeqnarray*}{rClrCl}
    \fr&\colon&  \Omega_0  \to  \mathbb{F};&  \quad  (\fpart_t)  &\mapsto&  \left\{ \fpart_t\colon t \ge 0 \right\} \quad \text{(the range of the path $(\fpart_t)$)}\\
    \tm&\colon& \Omega_0 \to \dct(\mathbb{F});& \quad \omega = (\fpart_t) &\mapsto& \Big( \fr(\omega), \big[ \fpart \mapsto  \inf \left\{ t > 0\colon \fpart_t = \fpart \right\}  \big]  \Big)
\end{IEEEeqnarray*}
Both maps are measurable with respect to the product $\sigma$-algebra on $\Omega_0$, see Lemmas~\ref{lem:frmeasurable} and~\ref{lem:tmmeasurable}.
We can think of a (non-spatial) coalescent as a random, time decorated forest, and indeed $\tm$ is a bijective and in fact bimeasurable map. In particular, the laws $\P^{\fpart}$ of a coalescent are determined by the pushforward laws $\tm \# \P^\fpart$ on the set of time-decorated forests, which can be calculated explicitly using simple arguments of competing exponential clocks, see \cref{lem:ftm} below. We identify $\fr$ and $\tm$ with their extensions to $\Omega$ (obtained by composing with the projection $\Omega \to \Omega_0$).

In the spatial setting, we can assign to a path in $\Omega$ a forest decorated with space in addition to time coordinates. For $\omega = (\boldsymbol x_t) \in \Omega$ with $F = \fr(\omega)$, we write \[
    \sp(\omega) \coloneqq
    %\begin{cases}
    %    \nd^\circ (F) \to E,\\
    %    u\mapsto \boldsymbol x_{\tm(\omega)_u}(u)
    %\end{cases}
    %\in \dcs(F).
    \big[ u\mapsto \boldsymbol x_{\tm(\omega)_u}(u)\big] \in \dcs(F)
\] for the associated space decoration of $F$. We prove that this map is measurable with respect to the product $\sigma$-algebra on $\Omega$ in Lemma~\ref{lem:spmeasurable}. Let $\dcb(F) = \dct(F) \times \dcs(F)$ for $F\in \mathbb{F}$, and write
\begin{align*}
    \dcb (\mathbb{F}) &= \left\{ (F,\tau,\xi)\colon F\in \mathbb{F}, \tau \in \dct(F), \xi \in \dcs(F) \right\}
    %\dcb (\mathbb{F}) &= \{ (F,\tau,\xi,\boldsymbol x)\colon (F,\tau,\xi) \in \dcb(\mathbb{F}), \boldsymbol x \in E^{\lf(F)}_\circ  \}
\end{align*}
for the set of time and space decorated forests, and define $\dcb(S)$ for $S\subset \mathbb{F}$ analogously. We write $F^\star = (F,\tau,\xi)$ for a generic element of $\dcb(\mathbb{F})$. The topology on $\dcb(\mathbb{F})$ is inherited from that of $\dcb(F),\, F\in \mathbb{F}$, noting it can be identified with a discrete union.

Since $F = \dom(\tau)$ for $(F,\tau,\xi) \in \dcb(\mathbb{F})$ we occasionally just write $(\tau,\xi) \in \dcb(\mathbb{F})$. The following map assigns to a path in $\Omega$ its associated decorated forest, see \cref{fig:dctdcs} for an illustration.
%Define the measurable map (\cref{lem:decmeasurable})
\begin{equation}\label{eq:defdc}
    \dc \colon \Omega \to \dcb(\mathbb{F});\quad \omega = (\boldsymbol x_t) \mapsto \Big(\tm(\omega), \big[ u \mapsto \boldsymbol x_{\tm(\omega)_u}(u)\big] \Big).
\end{equation}
It encodes all information about the path $\omega$ except the spatial motion of lineages in between merge events. Brownian spatial coalescents are exactly those for which this motion is Brownian conditional on the decorated forest, which means that its laws $\P^{\boldsymbol x}$ are determined completely by the laws $\dc \# \P^{\boldsymbol x}$ on the space of decorated coalescence forests. To make this precise, we introduce a stochastic kernel $K_{\boldsymbol x}(F^\star,\cdot )\in \mathcal{M}_1(\Omega)$ describing the law of a Brownian spatial coalescent conditional on its decorated forest $F^\star \in \dcb(\mathbb{F})$ and initial condition $\boldsymbol x \in E^{\lf(F)}_\circ $ (see \cref{lem:kbr}). Then a spatial coalescent is a Brownian spatial coalescent if and only if $K_{\boldsymbol x}$ is a conditional probability of $\P^{\boldsymbol x}$ given $\dc$, that is if $\P^{\boldsymbol x} = (\dc\# \P^{\boldsymbol x}) \otimes K_{\boldsymbol x}$ for all $\boldsymbol x\in \mathcal{X}$. See \cref{def:BSC} below. Measurability of $\dc$ is proved in Lemma~\ref{lem:decmeasurable}.

The following maps capture the motion of particles along each branch of the coalescence forest:
\begin{IEEEeqnarray*}{rrCl}%\label{eq:pth}
    \pth^{F^\star}_u\colon &
    \left\{ \dc = F^\star \right\} &\,\to\, & R([\tau_u,\tau_{\pr_F(u)}),E),\\[3pt]
    & (\boldsymbol x_t) &\mapsto& (\boldsymbol x_t(u))_{\tau_u \le t < \tau_{\pr_F(u)}},
\end{IEEEeqnarray*}
for $F^\star = (F,\tau,\xi) \in \dc(\mathbb{F})$ and $u \in \nd(F)$, where $\left\{ \dc = F^\star \right\} = \left\{ \omega\in \Omega\colon \dc(\omega) = F^\star \right\} $. See again \cref{fig:dctdcs}. These maps are continuous (if the codomain is equipped with the product topology, like the domain).
\begin{figure}
    \centering
    \begin{tikzpicture}[baseline=(current bounding box.north),scale=1.0]
    \def\yscale{1.5}
    \begin{scope}[scale=0.7,shift={(.8,0)}]
        \coordinate (zero) at (-.5,-.5);
        \coordinate (t-l) at (-.5,3.6*\yscale);
        \coordinate (b-r) at (4.5,-.5);

        \coordinate (Aa) at (0,0);
        \coordinate (Ab) at (.75,0);
        \coordinate (Ba) at (1.75,0);
        \coordinate (Bb) at (2.25,0);
        \coordinate (Bc) at (3,0);
        \coordinate (C) at (3.7,0);
        \coordinate (C2) at (3.7,3.3*\yscale);

        \coordinate (A) at (0.375,1.5*\yscale);
        \coordinate (AB) at ({(.375+2.5)/2},3*\yscale);
        \coordinate (B) at (2.5,1.5*\yscale);

        %% coordinate axes
        \draw[->] (zero) -- (t-l);
        \draw[->] (zero) -- (b-r);
        \draw (t-l) node[above] {\footnotesize time};
        \draw (b-r) node[below left] {\footnotesize space};

        %% background
        \BBY{Aa}{A}[948435]
        \BBY{Ab}{A}[892459]
        \BBY{A}{AB}[349875]
        \BBY{Ba}{B}[349927]
        \BBY{Bc}{B}[432894]
        \BBY{B}{AB}[573984] %498729
        \BBY{C}{3.5,3.3*\yscale}[382949] %39847
        \BBY{Bb}{B}[98329] %549752
        \BBY{AB}{AB|-0,3.3*\yscale}[824356]

        \def\tick{.1}
        \draw[dotted,gray] (-.5,0) -- (C);
        \draw (-.5+\tick,0) -- (-.5-\tick,0) node[left] {\footnotesize$0$};
        \draw[fill] (Aa) circle[radius=2pt];
        \draw[fill] (Ab) circle[radius=2pt];
        \draw[fill] (Ba) circle[radius=2pt];
        \draw[fill] (Bb) circle[radius=2pt];
        \draw[fill] (Bc) circle[radius=2pt];
        \draw[fill] (C) circle[radius=2pt];
    \end{scope}

    \begin{scope}[shift={(-3+.3,-7)}]
        \def\yscale{1.5}
%\begin{tikzpicture}[baseline=(current bounding box.north),scale=1.2]
    \coordinate (zero) at (-.5,-.5);
    \coordinate (t-l) at (-.5,3.6*\yscale);
    \coordinate (b-r) at (4,-.5);

    \coordinate (Aa) at (0,0);
    \coordinate (Ab) at (.75,0);
    \coordinate (Ba) at (1.75,0);
    \coordinate (Bb) at (2.25,0);
    \coordinate (Bc) at (3,0);
    \coordinate (C) at (3.7,0);
    \coordinate (C2) at (3.7,3.3*\yscale);

    \coordinate (A) at (0.375,1.5*\yscale);
    \coordinate (AB) at ({(.375+2.5)/2},3*\yscale);
    \coordinate (B) at (2.5,1.5*\yscale);

    %% coordinate axes
    \draw[->] (zero) -- (t-l);
    \draw[->] (zero) -- (b-r);
    \draw (t-l) node[above] {\footnotesize time};
    \draw (b-r) node[below left] {\footnotesize space};

    %% background
        %% forest
    \BBY[gray,opacity=.4]{Aa}{A}[948435]
    \BBY[gray,opacity=.4]{Ab}{A}[892459]
    \BBY[gray,opacity=.4]{A}{AB}[349875]

    \BBY[gray,opacity=.4]{Ba}{B}[349927]
    \BBY[gray,opacity=.4]{Bc}{B}[432894]
    \BBY[gray,opacity=.4]{B}{AB}[573984] %498729
    \BBY[gray,opacity=.4]{C}{3.5,3.3*\yscale}[382949] %39847
    \BBY[gray,opacity=.4]{Bb}{B}[98329] %549752

    \BBY[gray,opacity=.4]{AB}{AB|-0,3.3*\yscale}[824356]

    %\draw (A) node[above left] {$u$};
    %\draw[fill,\ca] (A) circle[radius=2pt];
    %\node[\ca] at (0.3,2.7*\yscale) {$\pth_u$};

    %\draw (Bb) node[below right] {$x$};
    %\draw[fill,\cb] (Bb) circle[radius=2pt];
    %\node[fill=white,text=\cb] at (1.6,1*\yscale) {$\pth_x$};

    %\draw (C) node[above left] {$y$};
    %\draw[fill,\cc] (C) circle[radius=2pt];
    %\node[\cc] at (3.0,2.5*\yscale) {$\pth_y$};

    %% time and space stamps
    \def\tick{.1}
    %\draw[dotted] (Aa) -- (-.5,0);
    \draw[dashed,thick,orange] (B) -- (-.5,1.5*\yscale);
    \draw[dashed,thick,orange] (AB) -- (-.5,3*\yscale);

    \draw (-.5+\tick,0) -- (-.5-\tick,0) node[left] {\footnotesize$0$};
    \draw[dotted,gray] (-.5,0) -- (C);
    \draw[orange] (-.5+\tick,1.5*\yscale) -- (-.5-\tick,1.5*\yscale) node[left,rotate=90,anchor=south] {$\tau_u=\tau_w$};
    \draw[orange] (-.5+\tick,3*\yscale) -- (-.5-\tick,3*\yscale) node[left] {$\tau_v$};

    \draw[dashed,thick,teal] (A) -- (A|-0,-.5);
    \draw[dashed,thick,teal] (B) -- (B|-0,-.5);
    \draw[dashed,thick,teal] (AB) -- (AB|-0,-.5);

    \draw[teal] (A|-0,-.5+\tick) -- (A|-0,-.5-\tick) node[below] {$\xi_u$};
    \draw[teal] (AB|-0,-.5+\tick) -- (AB|-0,-.5-\tick) node[below] {$\xi_v$};
    \draw[teal] (B|-0,-.5+\tick) -- (B|-0,-.5-\tick) node[below] {$\xi_w$};

    %% forest
    \draw (A) node[above left] {$u$};

    \draw (B) node[above right] {$w$};

    \draw (AB) node[above right] {$v$};

    \draw[black] (Aa) to[out=90,in=-90-30] (A);
    \draw[black] (Ab) to[out=90,in=-90+30] (A);
    \draw[black] (A) to[out=90,in=-90-50] (AB);
    \draw[black] (Ba) to[out=90,in=-90-45] (B);
    \draw[black] (Bb) to[out=90,in=-90-15] (B);
    \draw[black] (Bc) to[out=90,in=-90+30] (B);
    \draw[black] (B) to[out=90,in=-40] (AB);
    \draw[black] (C) to (C2);
    \draw[black] (AB) -- (AB|-0,3.3*\yscale);

    \draw[fill] (A) circle[radius=2pt];
    \draw[fill] (B) circle[radius=2pt];
    \draw[fill] (AB) circle[radius=2pt];

    \draw[fill,gray] (Aa) circle[radius=2pt];
    \draw[fill,gray] (Ab) circle[radius=2pt];
    \draw[fill,gray] (Ba) circle[radius=2pt];
    \draw[fill,gray] (Bb) circle[radius=2pt];
    \draw[fill,gray] (Bc) circle[radius=2pt];
    \draw[fill,gray] (C) circle[radius=2pt];
    %\draw[fill,gray] (C2) circle[radius=2pt];
%\end{tikzpicture}
    \end{scope}
    \begin{scope}[shift={(3+.3,-7)}]
        \input{figures1/fig_pathnotation}
    \end{scope}

0   \draw[|->] (-.1,1) to[out=-90-35,in=90] node[above left] {$\dc$} (-.6,-1.3);
    \draw[|->] (4-.1,1) to[out=-90+35,in=90] node[above right] {$\pth$} (4+.4,-1.3);
\end{tikzpicture}
    \caption{At the top is an illustration of an element $\omega\in \Omega$. The map $\dc$ (bottom left) extracts the abstract coalescence forest, and times and spatial locations of the merge events. The $\pth$ maps (bottom right) extract the motion of particles along branches of the coalescence forest.}
    \label{fig:dctdcs}
\end{figure}
Denote the law of a standard Brownian motion in $E$ (with periodic boundary conditions) started at some $x$ at time $t \ge 0$ by $B^{(t,x)+}$, the law of a Brownian bridge started at time $s \ge 0$ at $x\in E$, ending at time $t > s$ at $y\in E$ by $B^{(s,x) \to (t,y)}$, the law of the same bridge followed by a Brownian motion started at time $t$ at $y$ by $B^{(s,x)\to(t,y)+}$ etc. Recall our notation for joining maps: if $\boldsymbol x\in E^{\lf(F)}_\circ $ and $\xi \in \dcs(F)$ then $\boldsymbol x \xi$ is the map $\nd(F) \to E$ that extends $\xi$ to $\lf(F)$ using $\boldsymbol x$.
%For $\boldsymbol x\in E_\circ ^{\lf(F)}$, write \[
%    (\boldsymbol x\xi)\colon \nd(F) \to E;\, u \mapsto
%    \begin{cases}
%        \boldsymbol x_u, & u \in \lf(F),\\
%        \xi_u, & u \in \nd^\circ (F).
%    \end{cases}
%\]
%\begin{lemma}\label{lem:kbr}
\begin{restatable}{lemma}{lemkbr}\label{lem:kbr}
    Given $F^\star = (F,\tau,\xi) \in \dc(\mathbb{F})$ and $\boldsymbol x \in E_\circ ^{\lf(F)}$, there is a unique law $\kbr_{\boldsymbol x}(F^\star,\cdot )\in \mathcal{M}_1(\Omega)$ under which $\boldsymbol X_0 = \boldsymbol x$ and $\dc = F^\star$ a.s., and \[
    \left( \pth^{F^\star}_u(\boldsymbol X)\colon u\in \nd(F) \right)
\] is a family of independent random variables such that the law of $\pth^{F^\star}_u(\boldsymbol X)$ is $B^{(\tau_u,(\boldsymbol x\xi)_u)+}$ if $\pr_F(u) = \emptyset $, and $B^{(\tau_u,(\boldsymbol x\xi)_u)\to(\tau_{\pr_F(u)},\xi_{\pr_F(u)})}$ if $\pr_F(u) \neq \emptyset $.
    There exists an extension of these laws to a family $(\kbr_{\boldsymbol x}(F^\star,\cdot )\colon \boldsymbol x\in \mathcal{X},F^\star \in \dcb(\mathbb{F}))$ such that
    \begin{align*}
        %\{ (\boldsymbol x,F^\star)\colon \boldsymbol x\in E_\circ ^{\lf(F)} \} &\to \mathcal{M}_1(\Omega);\quad %\hspace{3cm}\\
        \mathcal{X}\times \dcb(\mathbb{F}) \to \mathcal{M}_1(\Omega);\quad (\boldsymbol x,F^\star) \mapsto \kbr_{\boldsymbol x}(F^\star,\cdot )
    \end{align*}
    is continuous. In particular, $(\boldsymbol x,F^\star) \mapsto \kbr_{\boldsymbol x}(F^\star,A)$ is measurable for any measurable $A\subset \Omega$.
    %\[
    %        \mathcal{X} \times \dcb(\mathbb{F}) \to [0,1];\quad (\boldsymbol x,F^\star) \mapsto \kbr_{\boldsymbol x}(F^\star,A)
    %\] is measurable.
\end{restatable}
The proof of \cref{lem:kbr} is in Appendix~\ref{sec:lemkbr}. The exact definition of $\kbr_{\boldsymbol x}(F^\star,\cdot )$ when $\boldsymbol x \not\in E^{\lf(F)}_\circ $ is irrelevant, we only need this extension to ensure $K_{\boldsymbol x}(F^\star,\cdot )$ can be integrated over $F^\star \in \dcb(\mathbb{F})$ for fixed $\boldsymbol x\in \mathcal{X}$, like in \cref{eq:defBSC} below.
% when integrating $\kbr_{\boldsymbol x}(F^\star,\cdot )$ over a law $P(\diff F^\star)$ on $\dcb(\mathbb{F})$ that is concentrated on $F^\star$ with $\lf(F) = \dom(\boldsymbol x)$, as in \cref{eq:defBSC} below.

%Recall the notion of regular conditional probability: if $G$ is a Polish space, $P \in \mathcal{M}_1(G)$, and $X\colon G \to H$ is measurable into some measurable space $H$, then there exists a stochastic kernel $(K(x,\cdot ) \in \mathcal{M}_1(G))_{h \in H}$, called a \emph{regular conditional probability of $P$ given $X$} such that \[
%    P(\cdot ) = (X\#P) \otimes K = \int_H K(x,\cdot ) (X\# P)(\diff x).
%    \] Then $K(x,\cdot )$ is $X\#P$-almost everywhere unique.
%    %If $P$ itself is the law of another random variable $Y$ on some probability space $(S,\A,Q)$, we usually write $Q^{Y \,\vert\,X=x}$ for $ K(x,\cdot )$.

\begin{definition}\label{def:BSC}
    A spatial coalescent $(\P^{\boldsymbol x})_{\boldsymbol x\in \mathcal{X}}$ is called a \emph{Brownian spatial coalescent} if $K_{\boldsymbol x}$ is a conditional probability of $\P^{\boldsymbol x}$ given $\dc$, that is if
    \begin{equation}\label{eq:defBSC}
        %\P^{\boldsymbol x}(\cdot ) = (\dc\#\P^{\boldsymbol x}) \otimes K_{\boldsymbol x}(\cdot ) = \smashop{\int_{\dcb(\mathbb{F})}} K_{\boldsymbol x}(F^\star,\cdot ) (\dc\#\P^{\boldsymbol x})(\diff F^\star)
        \P^{\boldsymbol x}(\cdot ) = P^{\boldsymbol x} \otimes K_{\boldsymbol x}(\cdot ) = \smashop{\int_{\dcb(\mathbb{F})}} K_{\boldsymbol x}(F^\star,\cdot ) P^{\boldsymbol x}(\diff F^\star),
    \end{equation}
    for every $\boldsymbol x\in \mathcal{X}$, where $P^{\boldsymbol x} \coloneqq \dc \# \P^{\boldsymbol x} \in \mathcal{M}_1(\dcb(\mathbb{F}))$.
\end{definition}
In particular, a Brownian spatial coalescent is fully determined by the laws $P^{\boldsymbol x}$, so we can essentially think of it as a random, time and space decorated forest. This is a useful simplification, and the main work in the characterisation of Brownian spatial coalescents \cref{thm:BSC} is to understand which families of laws $(P^{\boldsymbol x})_{\boldsymbol x\in \mathcal{X}}$ give rise to a Markov process on $\Omega$ through \cref{eq:defBSC}.

We close by formalising the characterisation of a Brownian spatial coalescent in terms of its transition measures. Recall the heuristic description we provided in \cref{sec:introBSC}: for every tree (forest in general) we define a density on its spatial and time decorations, comprised of exponential factors for every merge event, and spatial factors associated with each branch of the forest. The density is w.r.t.\ Lebesgue measure in the time coordinates, and w.r.t.\ the transition measures in the spatial coordinates.

\begin{definition}\label{def:nurates}
    A family of \emph{transition measures} is a collection $\boldsymbol \nu = (\nu_{\fpart,\fpart'}\colon \fpart'<\fpart)$, where $\nu_{\fpart,\fpart'}$ is a finite measure on $E_\circ ^{\fpart'\setminus \fpart}$. Write $\nurates$ for the set of all families of transition measures.
\end{definition}

The spatial factors (recall \cref{fig:introtreeexample}) do not depend on the transition measures, and are given for a non-trivial forest $F \in \mathbb{F}$, $\tau \in \dct(F)$ and $\boldsymbol x \in E^{\lf(F)}_\circ $, by
\begin{equation}\label{eq:fsp}
    \fsp(\xi\,\vert\, \tau,\boldsymbol x) \coloneqq  \smashop{\prod_{u\in \nd(F)\setminus \rt(F) }} p(\tau_{\pr_F(u)} - \tau_u, (\boldsymbol x \xi)_u - \xi_{\pr_F(u)}),\qquad \xi \in \dcs(F).
\end{equation}
That is, we take the product of Brownian transition densities along each branch of the forest that either links to internal nodes, or a leaf to an internal node. See also \cref{fig:introtreeexample} in the introduction.
Here and in the following, we sometimes write $p(t,x)$ instead of $p_t(x)$ if it benefits legibility.
If $F$ is trivial we put $\fsp \vert_{\dcs(F)} \equiv 1$. The dependence on $F$ is implicit in $\xi$ and $\tau$, but we will occasionally write $\fsp^F$ to make it explicit.
%\begin{figure}
%    \centering
%    %\begin{tikzpicture}[scale=1.0]
%    %    \input{figures1/fig_ftm}
%    %\end{tikzpicture}
%    \begin{tikzpicture}[scale=1.0]
%        \input{figures1/fig_fsp}
%    \end{tikzpicture}
%    \caption{Illustration of $\fsp$ defined in \cref{eq:fsp}. Note that only spatial factors associated with three of the seven branches (those highlighted on the left) are written out in the formula on the right.}
%    %\caption{Illustration of $\ftm$ defined in \cref{eq:ftm} on the left, and of $\fsp$ defined in \cref{eq:fsp} on the right. Node labels represent spatial locations.}
%    \label{fig:ftmfsp}
%\end{figure}
Given transition measures $\boldsymbol \nu \in \nurates$ and $\fpart\in \mathcal{P}$, write \[
    \nut{\fpart} \coloneqq \sum_{\fpart'<\fpart} |\nu_{\fpart,\fpart'}|,
    \] which is zero if $|\fpart| = 1$.
    If the process is label invariant, then $\nut{\fpart}$ only depends on $n = |\fpart|$, and we write $\nut{n}$ like we did in the introduction. It turns out that, as in the non-spatial setting, if $\nut{\fpart} = 0$ then $\fpart$ is an absorbing state (in the sense that almost-surely no further coalescence events occur), and if $\nut{\fpart} > 0$ then almost-surely there is at least one more coalescence event. In particular, given an initial state $(\fpart,\boldsymbol x)\in \mathcal{X}$, only forests $F \in \mathbb{F}(\fpart)$ with $\nut{\rt(F)} = 0$ are possible. Given such a forest, the full (unnormalised) density on $\dcb(F)$ is
    \begin{equation}\label{eq:fnu}
        \fnu(\tau,\xi \,\vert\,\boldsymbol x) \coloneqq \fsp(\xi \,\vert\,\tau,\boldsymbol x) \smashop{\prod_{(\fpart,\fpart') \in F}} \e^{-\nut{\fpart}(\tau_{\fpart'}-\tau_\fpart)},\qquad (\tau,\xi) \in \dcb(F).
    \end{equation}
    Here we identified $F$ with the set $\{(\fpart_0^F,\fpart_1^F), \ldots ,(\fpart^F_{m-1},\fpart^F_m)\}$, which we will do in similar situations in the future if it is convenient and unambiguous.
    if $\nut{\rt(F)} > 0$, then we put $\fnu(\cdot \,\vert\,\boldsymbol x) \vert_{\dcb(F)}\equiv 0$, which defines $\fnu(\cdot \,\vert\,\boldsymbol x)$ on all of $\dcb(\mathbb{F})$. Note if $F$ is trivial and $\nut{\rt(F)} = 0$ then $\fnu(\cdot \,\vert\,\boldsymbol x) \vert_{\dcb(F)} \equiv 1$. The density $\fnu$ is w.r.t.\ Lebesgue measure in the time coordinates, and w.r.t.\ the relevant transition measures in the spatial coordinates: if $F\in \mathbb{F}$ is non-trivial, write
    \begin{equation}\label{eq:nu_F}
        \boldsymbol \nu_F(\diff \xi) \coloneqq \prod_{(\fpart,\fpart') \in F} \nu_{\fpart,\fpart'}(\diff \xi_{\fpart,\fpart'}),
    \end{equation}
where $\xi_{\fpart,\fpart'}\coloneqq \xi \vert_{\fpart'\setminus \fpart}$ for $\xi \in \dcs(F)$ and $(\fpart,\fpart') \in F$, so that $\boldsymbol \nu_F$ is a finite measure on $\dcs(F)$. If $F$ is trivial then $\boldsymbol \nu_F$ denotes the unique probability measure on the singleton $\dcs(F)$. It remains to define the normalisation. For $(\fpart,\boldsymbol x)\in \mathcal{X}$ and $F\in \mathbb{F}(\fpart)$ write
\begin{equation*}%\label{eq:defN}
    N^{\boldsymbol \nu}_F(\boldsymbol x) \coloneqq \smashoperator{\int_{\dcb(F)}} \fnu(\tau,\xi \,\vert\,\boldsymbol x) \boldsymbol \nu_F(\diff \xi) \diff \tau,\qquad
    N^{\boldsymbol \nu}(\boldsymbol x) \coloneqq \sum_{F\in \mathbb{F}(\fpart)} N_F^{\boldsymbol \nu}(\boldsymbol x),% \int_{\dcb(F)} \fnu(\tau,\xi \,\vert\,\boldsymbol x) \boldsymbol \nu_F(\diff \xi) \diff \tau.
\end{equation*}
defining functions respectively on $E^{\lf(F)}_\circ $ and $\mathcal{X}$. These quantities are not obviously, and will not generally, be finite. The following lemma gives a sufficient condition. If $\boldsymbol \lambda \in \rates$ then we formally write $\boldsymbol \nu(\diff \xi) = \boldsymbol \lambda \diff \xi$ if $\nu_{\fpart,\fpart'}(\diff \xi) = \lambda_{\fpart,\fpart'}\diff \xi$ for all $\fpart'<\fpart$. We write $\boldsymbol \nu(\diff \xi) \sim \diff \xi$ if there exists $\boldsymbol \lambda \in \rates$ with $\boldsymbol \nu(\diff \xi) = \boldsymbol \lambda \diff \xi$.

%\begin{lemma}\label{lem:NF}
\begin{restatable}{lemma}{lemNF}\label{lem:NF}
    If $\boldsymbol \nu(\diff \xi) \sim \diff \xi$ then $N^{\boldsymbol \nu}$ is continuous (in particular finite).
\end{restatable}

\begin{remark}\label{rem:Ndivergence}
    Under the assumption of \cref{lem:NF}, $N^{\boldsymbol \nu}(\boldsymbol x) \to \infty$ as $\boldsymbol x$ approaches $E^\fpart \setminus E^\fpart_\circ $ for some $\fpart\in \mathcal{P}$. % This is the technical reason that the process cannot be defined to start from initial conditions outside $E^\fpart_\circ $, and a consequence of the conceptual reason explained in the first paragraph of \cref{sec:spatialcoal}.
\end{remark}

Given a measure $P$ on $\dcb(\mathbb{F})$, it will be a convenient notation to write $P(F,\diff \tau,\diff \xi)$ for the measure on $\dcb(F)$ defined by $\int_A P(F,\diff \tau,\diff \xi) \coloneqq P(\left\{ F \right\} \times A)$ for $A \subset \dcb(F)$. Analogously in similar contexts. With this notation, a formal statement of \cref{thm:BSCintro,thm:BSCintro2} is as follows.

\begin{theorem}\label{thm:BSC}
    A spatial coalescent is a Brownian spatial coalescent if and only if there exists $\boldsymbol \nu\in \nurates$ such that $N^{\boldsymbol \nu}$ is continuous and
    \begin{equation}\label{eq:BSC}
        P^{\boldsymbol x}(F,\diff \tau,\diff \xi) = \frac{1}{N^{\boldsymbol \nu}(\boldsymbol x)} \fnu(\tau,\xi \,\vert\,\boldsymbol x) \boldsymbol \nu_F(\diff \xi) \diff \tau.
    \end{equation}
    We call it the \emph{Brownian spatial coalescent with transition measures $\boldsymbol \nu$}. In that case, $\boldsymbol x \mapsto \P^{\boldsymbol x}$ is continuous w.r.t.\ the topology of weak convergence.
\end{theorem}

Equation~\eqref{eq:BSC} formalises the examples~\cref{eq:Pxtwoparticles,eq:Pxthreeintro} from the introduction (recall also \cref{fig:introtreeexample} for illustration).

If a Brownian spatial coalescent is label invariant, then the family of transition measures reduces to a family $(\nunk \colon (n,\vec{k}) \in \mergers)$ as we wrote in the introduction, see \cref{lem:labelinv}.

\section{Proofs}\label{sec:proofs}

\subsection{Characterisation of Brownian Spatial Coalescents}
In this section we prove \cref{thm:BSC}.

\begin{figure}
    \centering
    \begin{tikzpicture}[scale=1.0]
    \coordinate (zero) at (-.5,-.5);
    \coordinate (t-l)  at (-.5,4.95); % 3.3*1.5
    \coordinate (b-r)  at (3.9,-.5);

    \coordinate (Aa) at (0.0,0);
    \coordinate (Ab) at (1.25,0);
    \coordinate (B)  at (2.5,0);
    \coordinate (A)  at (0.75,2.625); % 1.75*1.5
    \coordinate (AB) at (1.75,4.5);   % 3.0*1.5
    \coordinate (AA) at (1.85,4.5);   % 3.0*1.5
    \coordinate (BB) at (3.0,4.5);   % 3.0*1.5

    % axes
    \draw[->] (zero) -- (t-l);
    \draw[->] (zero) -- (b-r);
    \node[above] at (t-l) {\footnotesize time};
    \node[below] at (b-r) {\footnotesize space};

    \def\tick{.1}
    \draw (-.5+\tick,0) -- (-.5-\tick,0) node[left] {\footnotesize$0$};
    \draw[dotted,gray] (-.5,0) -- (B);

    \draw[dotted] (Aa) -- (Aa|-0,-.5);
    \draw (Aa|-0,-.5+\tick) -- (Aa|-0,-.5-\tick) node[below] {$x_1$};

    \draw[dotted] (Ab) -- (Ab|-0,-.5);
    \draw (Ab|-0,-.5+\tick) -- (Ab|-0,-.5-\tick) node[below] {$x_2$};

    \draw[dotted] (B) -- (B|-0,-.5);
    \draw (B|-0,-.5+\tick) -- (B|-0,-.5-\tick) node[below] {$x_3$};

    % background (forest)
    \BBY[gray,opacity=.4]{Aa}{A}[731475]
    \BBY[gray,opacity=.4]{Ab}{A}[801947]
    \BBY[orange,opacity=.8]{A}{AA}[280594]
    \BBY[orange,opacity=.8]{B}{BB}[172834]

    \draw[fill,gray] (Aa) circle[radius=2pt];
    \draw[fill,gray] (Ab) circle[radius=2pt];
    \draw[fill,gray] (B)  circle[radius=2pt];

    % time/space stamps (static version of the final overlay state)
    \draw[dotted, gray] (A)  -- (-.5,0|-A);
    \draw (-.5+\tick,0|-A) -- (-.5-\tick,0|-A) node[left,gray] {$\tau$};

    \draw[densely dotted,thick] (BB) -- (-.5,0|-BB);
    \draw (-.5+\tick,0|-BB) -- (-.5-\tick,0|-BB) node[left,orange] {$t$};

    \draw[dotted, gray] (A)  -- (A|-0,-.5);
    \draw (A|-0,-.5+\tick) -- (A|-0,-.5-\tick) node[below,gray] {\small $\xi$};

    \draw[densely dotted,thick] (AA) -- (AA|-0,-.5);
    \draw (AA|-0,-.5+\tick) -- (AA|-0,-.5-\tick) node[below,orange] {$y_1$};

    \draw[densely dotted,thick] (BB) -- (BB|-0,-.5);
    \draw (BB|-0,-.5+\tick) -- (BB|-0,-.5-\tick) node[below,orange] {$y_2$};

    \draw[fill,gray] (A)  circle[radius=2pt];
    \draw[fill] (AA) circle[radius=2pt];
    \draw[fill] (BB) circle[radius=2pt];

    \node[anchor=west,align=left] at (4.0, 3.0) {\small $f_{\boldsymbol \nu}(\tau,\xi,t,\boldsymbol y \,\vert\, \boldsymbol x) = \overbrace{\color{gray}p_{\tau}(\xi - x_1)p_{\tau}(\xi - x_2) \e^{-\nut{3} \tau }}^{=f_{\boldsymbol \nu}(\tau,\xi \,\vert\,\boldsymbol x)}$\\[5pt]
    \small $\hspace*{2.5cm}\times{\color{orange}\,\, \e^{-\nut{2}(t-\tau)} p_{t-\tau}(y_1-\xi) p_{t}(y_2-x_3) }$};
\end{tikzpicture}
    \caption{This figure illustrates the notation for $f_{\boldsymbol \nu}(\tau,\xi,t,\boldsymbol y \,\vert\,\boldsymbol x)$ introduced in~\eqref{eq:fnu_terminal_condition}.}
    \label{fig:fnu_terminal_condition}
\end{figure}

\subsubsection{``If'' Direction of \cref{thm:BSC}}
%The aim of this section is to prove the ``if'' direction of \cref{thm:BSC}. For that purpose,
Fix $\boldsymbol \nu \in \nurates$ throughout, denote by $(P^{\boldsymbol x})_{\boldsymbol x\in \mathcal{X}}$ the laws defined in \cref{eq:BSC}, and $\P^{\boldsymbol x} = P^{\boldsymbol x} \otimes K_{\boldsymbol x}$ for $\boldsymbol x\in \mathcal{X}$. We need to show that the coalescent process defined by $(\P^{\boldsymbol x})_{\boldsymbol x\in \mathcal{X}}$ is Markov. Recall that the associated semigroup is denoted $(\semi_t)_{t\ge 0}$, see \cref{eq:semi}. For fixed $\fpart'\le \fpart$ and $A \subset E^{\fpart'}_\circ $ we write \[
    \semi_t((\fpart,\boldsymbol x),(\fpart',A)) \coloneqq \semi_t((\fpart,\boldsymbol x), \left\{ (\fpart',\boldsymbol y)\colon \boldsymbol y\in A \right\} ),
\] which defines a sub-probability measure on $E^{\fpart'}_\circ $ with total mass $\P^{(\fpart,\boldsymbol x)}(\fpart_t = \fpart')$. We introduce some additional notation. Let $F\in \mathbb{F}$ be a (possibly trivial) forest, $\boldsymbol x\in E^{\lf(F)}_\circ $, $\tau\in \dct(F)$, $\xi\in \dcs(F)$, and $t > \tau_{\rt(F)}$, $\boldsymbol y \in E^{\rt(F)}_\circ $. Then define
\begin{align}\label{eq:fnu_terminal_condition}
    %\fsp(\xi,t,\boldsymbol y \,\vert\,\tau,\boldsymbol x) &\coloneqq \fsp(\xi \,\vert\,\tau,\boldsymbol x) \prod_{u\in \rt(F)} p_{t-\tau_u}(\boldsymbol y_u - (\boldsymbol x\xi)_u),\\
    %\fnu(\tau,t,\xi,\boldsymbol y \,\vert\,\boldsymbol x) & \coloneqq \fsp(\xi,t,\boldsymbol y \,\vert\,\tau,\boldsymbol x) \e^{-\nut{\rt(F)} (t-\tau_{\rt(F)})} \prod_{(\fpart,\fpart') \in F} \e^{-\nut{\fpart} (\tau_{\fpart'}-\tau_{\fpart})}
    \fnu(\tau,\xi,t,\boldsymbol y \,\vert\,\boldsymbol x) & \coloneqq \fnu(\tau,\xi \,\vert\,\boldsymbol x)\e^{-\nut{\rt(F)} (t-\tau_{\rt(F)})} \smashop{\prod_{u\in \rt(F)}} p_{t-\tau_u}(\boldsymbol y_u - (\boldsymbol x\xi)_u).
\end{align}
If $t \le \tau_{\rt(F)}$ we put $\fnu(\tau,\xi,t,\boldsymbol y \,\vert\,\boldsymbol x) \coloneqq 0$. Think of this as a version of $\fnu(\tau,\xi \,\vert\,\boldsymbol x)$ with ``terminal condition $\boldsymbol y$ at time $t$'', see \cref{fig:fnu_terminal_condition} for an illustration.
% Note that $(\ftm(\tau,t))_{\tau\in \dct(F)}$ for $F\in \mathbb{F}$ is the probability (density) that a non-spatial coalescent process with rates $\boldsymbol \lambda$, started from $\lf(F)$, transitions at times given by $\tau$ along $F$, and then remains in the state $\rt(F)$ until time $t$. % In particular, $\lim_{t \to \infty} \ftm(\tau,t) = \ftm(\tau)$.

\begin{lemma}\label{lem:semigroup}
    If $(\fpart,\boldsymbol x) \in \mathcal{X}$ and $\fpart' \le \fpart$, then
    \begin{align*}
        \semi_t((\fpart,\boldsymbol x),(\fpart',\diff \boldsymbol y)) = \frac{N^{\boldsymbol \nu}(\boldsymbol y)}{N^{\boldsymbol \nu}(\boldsymbol x)} \sum_{F\in \mathbb{F}(\fpart,\fpart')}\,\, \Bigg(\int_{\dcb(F)} \fnu(\tau,\xi,t,\boldsymbol y \,\vert\,\boldsymbol x) \boldsymbol \nu_F(\diff \xi) \diff \tau\Bigg)\diff \boldsymbol y.
    \end{align*}
\end{lemma}

We introduce notation for the proof. If $\fpart_1 \ge \fpart_2 \ge \fpart_3$ and $F \in \mathbb{F}(\fpart_1,\fpart_2), F' \in \mathbb{F}(\fpart_2,\fpart_3)$, then we write $F F' \coloneqq  F \cup F' \in \mathbb{F}(\fpart_1,\fpart_3)$ for the concatenation of $F$ and $F'$.
If further $\tau \in \dct(F)$ and $\tau' \in \dct(F')$, and $t > \tau_{\rt(F)}$, write
\begin{IEEEeqnarray}{rcCrCl}\label{eq:tauconcat}
    (\tau / t / \tau') &\colon& FF'\to (0,\infty);& \quad \fpart &\mapsto &
    \begin{cases}
        \tau(\fpart), & \fpart \in F,\\
        t + \tau'(\fpart), & \fpart \in F' \setminus \{\lf(F')\},
    \end{cases}%\\[5pt]
    %    \xi \xi' &\colon&\, \nd(FF')^\circ  \to E;& u &\mapsto&
    %\begin{cases}
    %    \xi_u, & u\in \nd(F)^\circ ,\\
    %    \xi'_u, & u \in \nd(F')^\circ ,
    %\end{cases}\nonumber
\end{IEEEeqnarray}
so that $(\tau/t / \tau') \in \dct(FF')$ and $\xi\xi' \in \dcs(FF')$ (the latter is defined in the sense of \cref{eq:joinmaps}). If $A$ is a statement, and $x$ is an expression that evaluates to a real number if $A$ is true (and may otherwise be ill-defined), then we use $[x]_A$ as short-hand for $x$ if $A$ is true, and $1$ otherwise.

\begin{proof}[Proof of \cref{lem:semigroup}]
    The possible values of $\fr$ of a path in $\mathcal{P}$ that starts at $\fpart$ and passes through $\fpart'$ are exactly $F F'$ for $F \in \mathbb{F}(\fpart,\fpart')$ and $F'\in \mathbb{F}(\fpart')$. If $\fr = FF'$, then $\tm$ has to be of the form $(\tau / t / \tau')$ for $\tau\in \dct(F)$ with $\tau_{\fpart'} < t$ and $\tau'\in \dct(F')$, and $\sp$ has to be of the form $\xi\xi'$ for $\xi\in \dcs(F)$ and $\xi'\in \dcs(F')$.
    %For any such choice, the probability density
    %\begin{align*}
    %    \P^{\boldsymbol x}(\fr = FF', \tm = (\tau / t / \tau'), & \sp = \xi\xi')\\
    %    &= \frac{N_{F F'}(\boldsymbol x)}{N_\fpart(\boldsymbol x)} \ftm^{FF'}(\tau / t / \tau') \fsp^{FF'}(\xi\xi' \,\vert\, (\tau / t / \tau'),\boldsymbol x).
    %\end{align*}
    Conditional on $\fr = FF'$, $\tm = (\tau / t / \tau')$, and $\sp = \xi\xi'$, which happens with probability density
    \begin{align}\label{eqprf:semi:1}
        \frac{1}{N^{\boldsymbol \nu}(\boldsymbol x)} \fnu((\tau / t / \tau'), \xi\xi' \,\vert\,\boldsymbol x)\boldsymbol \nu_F(\diff \xi) \boldsymbol \nu_{F'}(\diff \xi') \diff \tau\diff \tau'.
    \end{align}
    The probability density of $\boldsymbol X_t \in (\fpart', \diff \boldsymbol y)$ is, by \cref{eq:BSC} and definition of $\kbr_{\boldsymbol x}$ (see \cref{lem:kbr}),
    \begin{multline}\label{eqprf:semi:2}
        %\P^{\boldsymbol x}\big(\boldsymbol X_t \in (\fpart',\diff \boldsymbol y)&\,\big\vert\,\fr = FF', \tm = (\tau / t / \tau'), \sp = \xi\xi'\big)\\
        \kbr_{\boldsymbol x}\Big((FF',(\tau / t / \tau'),\xi\xi'), \{\boldsymbol X_t \in (\fpart',\diff \boldsymbol y)\}\Big) \\
        = \prod_{u \in \fpart'} p_{t-\tau_u}(y_u - (\boldsymbol x\xi)_u) \left[\frac{p_{\tau'_{\pr'(u)}}(y_u - \xi'_{\pr'(u)})}{p_{t + \tau'_{\pr'(u)}- \tau_u }((\boldsymbol x \xi)_u - \xi'_{\pr'(u)})}\right]_{\pr'(u) \neq \emptyset } \diff \boldsymbol y,
    \end{multline}
    where $\pr \coloneqq \pr_F$ and $\pr' \coloneqq \pr_{F'}$. Recall also that $\tau_u$ refers to the time of ``birth'' of a node $u$, see~\eqref{eq:tau_u} and the surrounding discussion. Multiplying \cref{eqprf:semi:1,eqprf:semi:2} gives, after some careful cancellations,
    \begin{multline*}
        \P^{\boldsymbol x}(\fr = FF', \tm \in \diff (\tau / t / \tau'),  \sp \in \diff (\xi\xi'), \boldsymbol X_t \in (\fpart',\diff \boldsymbol y)) \\
                                                                        = \frac{1}{N^{\boldsymbol \nu}(\boldsymbol x)} \fnu(\tau,\xi,t,\boldsymbol y \,\vert\,\boldsymbol x) \fnu(\tau',\xi' \,\vert\,\boldsymbol y) \boldsymbol \nu_F(\diff \xi) \boldsymbol \nu_{F'}(\diff \xi')\diff \tau\diff \tau'\diff \boldsymbol y.
    \end{multline*}
    If $F'$ is trivial and $\lambda_{\fpart'} > 0$, then the equality holds because both $\P^{\boldsymbol x}(\fr = FF')$ and thus the left-hand side (LHS), and $\fnu\vert_{\dcb(F')}$ and thus the right-hand side (RHS) are zero. Integrating over $\tau,\xi,\tau',\xi'$ gives
    \begin{align*}
        \P^{\boldsymbol x}(\fr = FF', \boldsymbol X_t \in (\fpart',\diff \boldsymbol y)) = \frac{N^{F'}_{\boldsymbol \nu}(\boldsymbol y)}{N^{\boldsymbol \nu}(\boldsymbol x)} \left( \int_{\dcb(F)} \fnu(\tau,\xi,t,\boldsymbol y \,\vert\,\boldsymbol x) \boldsymbol \nu_F(\diff \xi) \diff \tau\right) \diff \boldsymbol y,
    \end{align*}
    and then summing over $F$ and $F'$ gives the claim.
\end{proof}

\begin{lemma}\label{lem:markov}
    If $t, s > 0$, then $\semi_t \semi_s = \semi_{t+s}$.
\end{lemma}
A crucial ingredient in the proof is the following identity, which is an immediate consequence of the definitions. If $F, F'\in \mathbb{F}$ with $\rt(F) = \lf(F')$, $(\tau,\xi) \in \dcb(F)$, $(\tau',\xi') \in \dcb(F')$, $t,t'> 0$ with $\tau_{\rt(F)} < t$, and $\boldsymbol y' \in E^{\lf(F')}_\circ $, then
\begin{equation}\label{eq:fnu_conv}
    \int_{E^{\rt(F)}_\circ } \fnu(\tau,\xi,t,\boldsymbol y \,\vert\,\boldsymbol x) \fnu(\tau',\xi',t',\boldsymbol y' \,\vert\,\boldsymbol y) \diff \boldsymbol y = \fnu((\tau / t / \tau'),\xi\xi', t+t',\boldsymbol y' \,\vert\,\boldsymbol x).
\end{equation}

\begin{proof}[Proof of \cref{lem:markov}]
    Fix $(\fpart_1,\boldsymbol x)$, $t, s > 0$, and $\fpart_2\le \fpart_1$. Then by \cref{lem:semigroup},
    \begin{align}
        (\semi_t \semi_s)&((\fpart_1,\boldsymbol x),(\fpart_2,\diff \boldsymbol y)) / \diff \boldsymbol y \nonumber\\
        &= \sum_{\fpart_2\le \fpart'\le \fpart_1} \int_{\boldsymbol z \in E^{\fpart'}_\circ } \semi_s((\fpart',\boldsymbol z),(\fpart_2,\diff \boldsymbol y)) \semi_t((\fpart_1,\boldsymbol x),(\fpart',\diff \boldsymbol z)) / \diff \boldsymbol y\nonumber\\
        &= \sum_{\fpart_2\le \fpart'\le \fpart_1} \int \diff \boldsymbol z \left( \frac{N^{\boldsymbol\nu}(\boldsymbol y)}{N^{\boldsymbol\nu}(\boldsymbol z)} \sum_{F_2\in \mathbb{F}(\fpart',\fpart_2)} \int_{\dcb(F_2)} \fnu(\tau_2,\xi_2,s,\boldsymbol y \,\vert\,\boldsymbol z)\boldsymbol \nu_{F_2}(\diff \xi_2) \diff \tau_2 \right. \nonumber\\
        &\hspace{3cm} \times \left. \frac{N^{\boldsymbol\nu}(\boldsymbol z)}{N^{\boldsymbol\nu}(\boldsymbol x)} \sum_{F_1 \in \mathbb{F}(\fpart_1,\fpart')} \int_{\dcb(F_1)} \fnu(\tau_1,\xi_1,t,\boldsymbol z \,\vert\,\boldsymbol x)\boldsymbol \nu_{F_1}(\diff \xi_1) \diff \tau_1 \right)\nonumber \\
        &\stackrel{(\ref{eq:fnu_conv})}{=} \frac{N^{\boldsymbol\nu}(\boldsymbol y)}{N^{\boldsymbol\nu}(\boldsymbol x)} \sum_{\substack{\fpart_2\le \fpart'\le \fpart_1 \\ F_1 \in \mathbb{F}(\fpart_1,\fpart') \\F_2 \in \mathbb{F}(\fpart',\fpart_2)}}
        \int\limits_{\dcb(F_1)} \smashoperator[r]{\int\limits_{\dcb(F_2)}}\fnu((\tau_1 / t / \tau_2),\xi_1 \xi_2,t+s,\boldsymbol y \,\vert\,\boldsymbol x) \boldsymbol \nu_{F_1}(\diff \xi_1) \diff \tau_1\label{eqprf:markov:1}\\[-25pt]
        &\hspace*{10cm} \boldsymbol \nu_{F_2}(\diff \xi_2) \diff \tau_2.\nonumber
    \end{align}
    \vspace{.5cm}

    If $\fpart_2 \le \fpart'\le \fpart_1$ and $F_1\in \mathbb{F}(\fpart_1,\fpart')$, $F_2\in \mathbb{F}(\fpart',\fpart_2)$, then $\xi \in \dcs(F_1F_2)$ is always of the form $\xi = \xi_1\xi_2$ for $\xi_{1,2}\in \dcs(F_{1,2})$, and $\tau \in \dct(F_1F_2)$ is of the form $\tau = (\tau_1 / t / \tau_2)$ for $\tau_{1,2}\in \dct(F_{1,2})$ if and only if $\tau_{\fpart'} < t < \tau_{\fpart'}^+$. Thus the final integral on the RHS of \cref{eqprf:markov:1} is equal to
    \begin{equation*}
        \int_{\dcb(F_1F_2)} \ind_{\left\{ \tau_{\fpart'} < t < \tau_{\fpart'}^+ \right\} } \fnu(\tau,\xi,t+s,\boldsymbol y \,\vert\,\boldsymbol x)  \diff \tau \boldsymbol \nu_F(\diff \xi),
    \end{equation*}
    so
    \begin{align*}
        (\semi_t\semi_s) &((\fpart_1,\boldsymbol x),(\fpart_2,\diff \boldsymbol y)) / \diff \boldsymbol y\\
                         &= \frac{N^{\boldsymbol\nu}(\boldsymbol y)}{N^{\boldsymbol\nu}(\boldsymbol x)} \sum_{F\in \mathbb{F}(\fpart_1,\fpart_2)}\sum_{\fpart'\in F} \int_{\dcb(F)} \ind_{\left\{ \tau_{\fpart'} < t < \tau_{\fpart'}^+ \right\} } \fnu(\tau,\xi,t+s,\boldsymbol y \,\vert\,\boldsymbol x)\boldsymbol \nu_F(\diff \xi) \diff \tau\\
                         &= \frac{N^{\boldsymbol\nu}(\boldsymbol y)}{N^{\boldsymbol\nu}(\boldsymbol x)} \sum_{F\in \mathbb{F}(\fpart_1,\fpart_2)} \int_{\dcb(F)} \fnu(\tau,\xi,t+s,\boldsymbol y \,\vert\,\boldsymbol x)\boldsymbol \nu_F(\diff \xi) \diff \tau\\
         &= \semi_{t+s}((\fpart_1,\boldsymbol x),(\fpart_2,\diff \boldsymbol y)) / \diff \boldsymbol y.
    \end{align*}
\end{proof}

This proves that the Brownian spatial coalescent with transition measures $\boldsymbol \nu$ is a Markov process.

\begin{lemma}\label{lem:strongmarkov}
    The map $\boldsymbol x \mapsto \P^{\boldsymbol x}$ is continuous w.r.t.\ the topology of weak convergence of probability measures. In particular, $(\P^{\boldsymbol x})_{\boldsymbol x\in \mathcal{X}}$ has the Feller property.
\end{lemma}
\begin{proof}
    Suppose $O\subset \Omega$ is open, and $(\fpart_n,\boldsymbol x_n) \to (\fpart,\boldsymbol x)$ in $\mathcal{X}$, without loss of generality $\fpart_n = \fpart$ for all $n\in \N$. Then,
    \begin{align*}
        \P^{\boldsymbol x_n}(O)
        &= \int_{\dcb(\mathbb{F})} \kbr_{\boldsymbol x_n}(F^\star,O) P^{\boldsymbol x_n}(\diff F^\star)\\
        &= \sum_{F\in \mathbb{F}(\fpart)} \frac{1}{N^{\boldsymbol \nu}(\boldsymbol x_n)} \int_{\dcb(F)} \kbr_{\boldsymbol x_n}((F,\tau,\xi),O) \fnu(\tau,\xi \,\vert\,\boldsymbol x_n) \boldsymbol \nu_F(\diff \xi) \diff \tau.
    \end{align*}
    By \cref{lem:NF,lem:kbr}, and definition of $\fnu$, we have $N^{\boldsymbol \nu}(\boldsymbol x_n) \to N^{\boldsymbol \nu}(\boldsymbol x)$, $\fnu(\tau,\xi \,\vert\,\boldsymbol x_n) \to \fnu(\tau,\xi \,\vert\,\boldsymbol x)$, and $\varliminf_{n\to \infty} K_{\boldsymbol x_n}(F^\star,O) \ge K_{\boldsymbol x}(F^\star,O)$ for all $F^\star = (F,\tau,\xi) \in \dcb(\mathbb{F})$. Thus by Fatou's lemma,
    \begin{align*}
        \varliminf_{n\to \infty} \P^{\boldsymbol x_n}(O)
        &\ge \sum_{F\in \mathbb{F}(\fpart)} \frac{1}{N^{\boldsymbol \nu}(\boldsymbol x)} \int_{\dcb(F)} K_{\boldsymbol x}((F,\tau,\xi),O) \fnu(\tau,\xi \,\vert\,\boldsymbol x_n)\boldsymbol \nu_F(\diff \xi)\diff \tau
        = \P^{\boldsymbol x}(O).
    \end{align*}
    This proves that $\boldsymbol x \mapsto \P^{\boldsymbol x}$ is continuous by  the Portmanteau theorem.
\end{proof}

%For later reference, we remark the following consequence of \cref{lem:semigroup}.
%
%\begin{corollary}\label{cor:Nintegrable}
%    For every $\fpart,\fpart' \in \mathcal{P}$, $\fpart' \subsetneq \fpart$, and $\boldsymbol x\in E^{\fpart'}_\circ $, \[
%        \int\limits_{E^{\fpart \setminus \fpart'}_{\boldsymbol x}} N^{\boldsymbol \nu}(\boldsymbol x\boldsymbol y) \diff \boldsymbol y < \infty.
%    \] In particular, if $\boldsymbol x\in E^\fpart_\circ $ then $\int_{E^\fpart_\circ } N^{\boldsymbol \nu}(\boldsymbol x) \diff \boldsymbol x < \infty$.
%\end{corollary}
%\begin{proof}
%    By \cref{lem:semigroup}, for fixed $(\fpart,\boldsymbol x) \in \mathcal{X}$, \[
%        \semi_1((\fpart,\boldsymbol x),(\fpart,\diff \boldsymbol y)) = \frac{N^{\boldsymbol \nu}(\boldsymbol y)}{N^{\boldsymbol \nu}(\boldsymbol x)} \prod_{u\in \fpart} p_1(\boldsymbol x_u - \boldsymbol y_u) \diff \boldsymbol y.
%    \] The right-hand side is a finite measure on $E^\fpart_\circ $ (because the left-hand side is a sub-probability measure), $N^{\boldsymbol \nu}(\boldsymbol x)$ is some positive number independent of $\boldsymbol y$, and $\inf p_1 > 0$, so $N^{\boldsymbol \nu}(\boldsymbol y)$ must be integrable over $E^\fpart_\circ $, that is $\int_{E^\fpart_\circ } N^{\boldsymbol \nu}(\boldsymbol x) \diff \boldsymbol x < \infty$.
%
%    This implies, in the context of the more general assertion, that $\int_{E^{\fpart\setminus \fpart'}_{\boldsymbol x}} N^{\boldsymbol \nu}(\boldsymbol x\boldsymbol y) \diff \boldsymbol y < \infty$ for Lebesgue-a.e.\ $\boldsymbol x\in E^{\fpart'}_\circ $.
%\end{proof}

\subsubsection{``Only if'' Direction of \cref{thm:BSC}}

%The goal of this section is to prove the ``only if'' direction of \cref{thm:BSC}.
Fix a Brownian spatial coalescent process, that is a spatial coalescent whose laws satisfy 
\begin{equation}\label{eq:intro:sec:BSC}
    \P^{\boldsymbol x} = (\dc\#\P^{\boldsymbol x}) \otimes K_{\boldsymbol x},
\end{equation}
in which case we write $P^{\boldsymbol x} = \dc\#\P^{\boldsymbol x}$ for $\boldsymbol x\in \mathcal{X}$  (recall \cref{def:BSC} and \cref{eq:defBSC}). We prove that the only way in which this process can satisfy \eqref{eq:intro:sec:BSC} while simultaneously being a Markov process is if the laws $(P^{\boldsymbol x})$ are of the very specific form \cref{eq:BSC}.

The core idea behind the proof is to evaluate $\P^{\boldsymbol x}$ on different, carefully chosen events in two ways, once by using \cref{eq:intro:sec:BSC}, and once by using the fact that the process satisfies the Markov property. This will reveal different properties that $P^{\boldsymbol x}$ must satisfy, and carefully combining them will ultimately yield \cref{eq:BSC}.

For $F\in \mathbb{F}$, $\tau\in \dct(F)$ and $s \ge 0$ we write $\tau + s$ for the map $u \mapsto \tau_u + s$, and $\tau \ge s$ if $\tau = \tau'+s$ for some $\tau'\in \dct(F)$.
Recall that the semigroup associated with $(\P^{\boldsymbol x})$ is denoted $(\semi_t)_{t\ge 0}$, see \cref{eq:semi}.
\begin{lemma}\label{lem:bsc:1}
    Let $\fpart \in \mathcal{P}$. If $F=F_0 \in \mathbb{F}(\fpart)$ is the trivial forest, then for every $\boldsymbol x \in E^\fpart_\circ $,
    \begin{equation}\label{eqprf:U1:11}
        \semi_s(\boldsymbol x,\diff \boldsymbol y) \P^{\boldsymbol y}(\fr = F_0) = \P^{\boldsymbol x}(\fr = F_0) \prod_{u\in \fpart}p_s(\boldsymbol x_u - \boldsymbol y_u) \diff \boldsymbol y
    \end{equation}
    as measures on $E^\fpart_\circ $. More generally, for any forest $F\in\mathbb{F}(\fpart)$,
    \begin{align}\begin{split}\label{eqprf:U1:3}
        \semi_s(\boldsymbol x,\diff \boldsymbol y) & P^{\boldsymbol y}(F,\diff \tau,\diff \xi)\\
                                                   &= P^{\boldsymbol x}(F,\diff (\tau+s),\diff \xi) \prod_{u\in \fpart} p_s(\boldsymbol x_u-\boldsymbol y_u) \left[ \frac{p(\tau_{\pr(u)}, \xi_{\pr(u)} - \boldsymbol y_u)}{p(\tau_{\pr(u)}+s,\xi_{\pr(u)}-\boldsymbol x_u)} \right] _{\pr(u) \neq \emptyset } \diff \boldsymbol y,
    \end{split}\end{align}
    as an equality of measures on $\dct(F)\times \dcs(F)\times E^\fpart_\circ $ (where we abbreviated $\pr = \pr_F$).
\end{lemma}
\begin{proof}
    If $|\fpart| =1$, say $\fpart = \{u\}$, then the trivial forest $F_0$ is the only element of $\mathbb{F}(\fpart)$, so $\P^{(\fpart,\boldsymbol x)}(\fr = F) = 1$ for all $\boldsymbol x$. In particular, \eqref{eqprf:U1:11} reads $\semi_s(\boldsymbol x, \diff \boldsymbol y) = p_s(\boldsymbol x_u - \boldsymbol y_u)$, which is just stating that a Brownian spatial coalescent with a single particle follows a Brownian motion. Assume $|\fpart| > 1$ for the rest of the proof.

    For $s > 0$ and $\boldsymbol x\in E^\fpart_\circ $, we evaluate $\P^{\boldsymbol x}(\fr = F_0, \boldsymbol X_s \in \diff \boldsymbol y)$, which is a sub-probability measure on $E^\fpart_\circ $, in two ways. On the one hand, by applying the Markov property at time $s$ it equals $\semi_s(\boldsymbol x,\diff \boldsymbol y) \P^{\boldsymbol y}(\fr = F_0)$. On the other hand, by \cref{eq:intro:sec:BSC} it equals $\P^{\boldsymbol x}(\fr = F_0) K_{\boldsymbol x}(F_0,\boldsymbol X_s\in \diff \boldsymbol y)$. Expanding $K_{\boldsymbol x}(F_0,\boldsymbol X_s\in \diff \boldsymbol y)$ gives~\eqref{eqprf:U1:11} (recall that by \cref{lem:kbr}, $K_{\boldsymbol x}(F_0,\cdot )$ is the law of a collection of $|\fpart|$ independent Brownian motions).

    % \Cref{eq:U1:2} for the trivial forest $F_0$ is just $\nn_\fpart(\boldsymbol x) \P^{\boldsymbol x}(\fr = F_0) = \nn_\fpart(\boldsymbol y) \P^{\boldsymbol y}(\fr = F_0)$, which holds by definition of $\nn_\fpart$.

    Let now $F \in \mathbb{F}(\fpart)$ be non-trivial. For $s > 0$ and $\boldsymbol x\in E^\fpart_\circ $, we evaluate
    \begin{equation}\label{eqprf:U1:1}
        \P^{\boldsymbol x}(\fr = F, \tm \in \diff (\tau + s), \sp \in \diff \xi, \boldsymbol X_s \in \diff \boldsymbol y),
    \end{equation}
    which is a sub-probability measure on $\dcb(F) \times E^\fpart_\circ $, in two ways. On the one hand, by applying the Markov property at time $s$ it equals $\semi_s(\boldsymbol x,\diff \boldsymbol y) P^{\boldsymbol y}(F,\diff \tau,\diff \xi)$. On the other hand, by conditioning on $\dc = (F,\tau,\xi)$ and \cref{eq:defBSC}, the expression \cref{eqprf:U1:1} equals % precise markov property see handy 12 july 9:53
    \begin{align*}%\label{eqprf:U1:2}
        %\semi_s(\boldsymbol x,\diff \boldsymbol y) P^{\boldsymbol y}(F,\diff \tau,\diff \xi)
        P^{\boldsymbol x}(F,\diff (\tau+s),\diff \xi) K_{\boldsymbol x}((F,\tau+s,\xi), \{\boldsymbol X_s \in \diff \boldsymbol y\}) .
        %&= P^{\boldsymbol x}(F,\diff (\tau+s),\diff \xi) \prod_{u\in \fpart} p_s(\boldsymbol x_u-\boldsymbol y_u) \left[ \frac{p(\tau_{\pr(u)}, \xi_{\pr(u)} - \boldsymbol y_u)}{p(\tau_{\pr(u)}+s,\xi_{\pr(u)}-\boldsymbol x_u)} \right] _{\pr(u) \neq \emptyset } \diff \boldsymbol y,
    \end{align*}
    Expanding $K_{\boldsymbol x}((F,\tau+s,\xi),\boldsymbol X_s\in \diff \boldsymbol y)$ on the RHS yields \eqref{eqprf:U1:3}; recall for this step that a Brownian bridge $(W_s)_{0\le s\le t}$ starting at $x$ at time zero and ending at $z$ at time $t$ satisfies $\P(W_s \in \diff y) = \frac{p_s(x-y) p_{t-s}(y-z)}{p_t(x-z)} \diff y$.
\end{proof}

We state a useful consequence of \cref{lem:bsc:1}
\begin{lemma}\label{lem:forestpossible}
     For any $\fpart \in \mathcal{P}$ and any $F\in \mathbb{F}(\fpart)$, exactly one of the following hold:
    \begin{enumerate}
        \item $\P^{(\fpart,\boldsymbol x)}(\fr = F) > 0$ for all $\boldsymbol x\in E^\fpart_\circ $, in which case we call $F$ \emph{possible},
        \item $\P^{(\fpart,\boldsymbol x)}(\fr = F) = 0$ for all $\boldsymbol x\in E^\fpart_\circ $, in which case we call $F$ \emph{impossible}.
    \end{enumerate}   
    In fact, if a forest $F$ is possible, then $\P^{\boldsymbol x}(\fr = F, \tau \ge s) > 0$ for all $\boldsymbol x\in E^\fpart_\circ $ and $s \ge 0$.
\end{lemma}

The proof requires the following simple observation.
 \begin{lemma}\label{lem:Sspositive}
    If $(\fpart,\boldsymbol x) \in \mathcal{X}$, $t > 0$, then $\semi_t((\fpart,\boldsymbol x),(\fpart,\diff \boldsymbol y))$ is not zero as a measure on $E^{\fpart}_\circ $.
\end{lemma}
\begin{proof}
    Fix $(\fpart,\boldsymbol x) \in \mathcal{X}$. Since $(\fpart_t,\boldsymbol X_t)_{t \ge 0}$ is right-continuous, we have $\fpart_t \to \fpart_0 =\fpart$ a.s.\ as $t \to 0$, in particular in distribution, so $\semi_t(\boldsymbol x,E^\fpart_\circ ) = \P^{\boldsymbol x}(\fpart_t = \fpart) \to 1$ as $t \to 0$.
    This implies the claim for sufficiently small $t > 0$, and repeated applications of $\semi_{2t} = \semi_t \semi_t$ finish the proof.
\end{proof}

\begin{proof}[Proof of \cref{lem:forestpossible}]
    We first explain why the trivial forest $F_0$ is either possible or impossible. If $\P^{\boldsymbol y}(\fr = F_0) = 0$ on a set of positive Lebesgue measure, then integrating \cref{eqprf:U1:11} over $\boldsymbol y$ in this set gives zero on the LHS, and the integral on the RHS is zero only if $\P^{\boldsymbol x}(\fr = F_0) = 0$, so we must have $\P^{\boldsymbol x}(\fr = F_0) = 0$ for all $\boldsymbol x\in E^\fpart_\circ $ and $F_0$ is impossible. Now assume that $\P^{\boldsymbol y}(\fr = F_0) > 0$ Lebesgue-almost everywhere. Then integrating \cref{eqprf:U1:11} over $\boldsymbol y\in E^\fpart_\circ $ gives a positive number on the LHS by \cref{lem:Sspositive}, and $\P^{\boldsymbol x}(\fr = F_0)$ on the RHS, so $\P^{\boldsymbol x}(\fr = F_0) > 0$ for all $\boldsymbol x\in E^\fpart_\circ $ and $F_0$ is possible.  % In that case, \cref{eqprf:U1:11} implies \cref{eq:U1} with $\beta_\fpart = 0$ and $\nn_\fpart(\boldsymbol x) = 1 / \P^{\boldsymbol x}(\fr = F_0)$. % By contraposition, this also proves that if the claim holds with $\beta_{\fpart} > 0$ then the trivial forest must be impossible. %this is proved separately elsewhere

    We now explain why any forest $F\in \mathbb{F}(\fpart)$ is either possible or impossible. Suppose that $\P^{\boldsymbol y}(\fr = F) = 0$ for $\boldsymbol y$ in a set of positive Lebesgue measure. Integrating \eqref{eqprf:U1:3} over this set and $(\tau,\xi) \in \dc(F)$ gives zero on the LHS, and on the RHS a quantity that is zero if and only if $P^{\boldsymbol x}(F, \left\{ \tau \ge s \right\} ) = 0$, which converges to $\P^{\boldsymbol x}(\fr = F)$ as $s \to 0 $ by continuity from below, so $\P^{\boldsymbol x}(\fr = F) = 0$ for all $\boldsymbol x\in E^\fpart_\circ $ and $F$ is impossible. Now suppose that $\P^{\boldsymbol y}(\fr = F) > 0$ for Lebesgue almost-all $\boldsymbol y\in E^\fpart_\circ $. Then $\semi_s(\boldsymbol x,\diff \boldsymbol y)$ has a Lebesgue density. Indeed, integrating \cref{eqprf:U1:3} over $\boldsymbol y\in N$ for a Lebesgue null set $N$, and $(\tau,\xi) \in \dcb(F)$, gives zero on the RHS, and on the LHS it gives $\int_N \P^{\boldsymbol y}(\fr = F) \semi_s(\boldsymbol x,\diff \boldsymbol y)$, which was positive if $\semi_s(\boldsymbol x,N) > 0$. Then, for an arbitrary $s > 0$,
    \begin{align}\label{eqprf:U1:50}
        \P^{\boldsymbol x}(\fr = F)
        &\ge \P^{\boldsymbol x}(\fr = F, \tau \ge s)
        = \int_{E^\fpart_\circ } \P^{\boldsymbol y}(\fr = F) \semi_s(\boldsymbol x,\diff \boldsymbol y) > 0,
    \end{align}
    so $F$ is possible. In the second step, we applied the Markov property at time $s$, and the last expression is positive because $\P^{\boldsymbol y}(\fr = F) > 0$ for Lebesgue-a.e.\ $\boldsymbol y\in E^\fpart_\circ $, and $\semi_s(\boldsymbol x,\cdot )$ is non-zero (\cref{lem:Sspositive}) and has a Lebesgue density. \Cref{eqprf:U1:50} also shows that if $F$ is possible, then $\P^{\boldsymbol x}(\fr = F, \tau \ge s) > 0$ for all $\boldsymbol x\in E^\fpart_\circ $ and $s \ge 0$.
\end{proof}

\begin{lemma}\label{lem:U1:aux}
    For every $\fpart\in \mathcal{P}$ there are a unique $\la_\fpart \ge 0$, and a function $\nn_\fpart\colon E^\fpart_\circ  \to (0,\infty)$ unique up to a constant positive multiple such that for all $s > 0$ and $\boldsymbol x \in E^\fpart_\circ $,
    \begin{equation}\label{eq:U1}
        \semi_s((\fpart,\boldsymbol x),(\fpart,\diff \boldsymbol y)) =  \e^{-\la_\fpart s}\frac{\nn_\fpart(\boldsymbol y)}{\nn_\fpart(\boldsymbol x)} \prod_{u\in \fpart} p_s(\boldsymbol x_u - \boldsymbol y_u) \diff \boldsymbol y.
    \end{equation}
    If the trivial forest $F_0 \in \mathbb{F}(\fpart)$ is possible, then $\nn_\fpart(\boldsymbol x) \P^{\boldsymbol x}(\fr = F_0)$ is constant in $\boldsymbol x\in E^\fpart_\circ $ and positive. Furthermore, for every possible $F\in \mathbb{F}(\fpart)$, every $\boldsymbol x,\boldsymbol y\in E^\fpart_\circ $, and $s \ge 0$,
    \begin{equation}\label{eq:U1:2}
        \e^{\la_\fpart s} \frac{\nn_\fpart(\boldsymbol x) P^{\boldsymbol x}(F,\diff (\tau+s),\diff \xi)}{\smashop{\prod_{\substack{u\in \fpart\\ \pr_F(u) \neq \emptyset }}} p(\tau_{\pr_F(u)}+s, \xi_{\pr_F(u)}-\boldsymbol x_u)} = \frac{\nn_\fpart(\boldsymbol y) P^{\boldsymbol y}(F,\diff \tau,\diff \xi)}{\displaystyle\smashop{\prod_{\substack{u\in \fpart\\ \pr_F(u) \neq \emptyset }}}p(\tau_{\pr_F(u)}, \xi_{\pr_F(u)}-\boldsymbol y_u)}
    \end{equation}
    as measures on $\dcb(F)$.
\end{lemma}

It will turn out that $\la_\fpart = \nut{\fpart}$, and that $\nn_\fpart$ is (a constant multiple of) the normalisation $N^{\boldsymbol \nu}$ for the transition measures $\boldsymbol \nu \in \nurates$ that define this Brownian spatial coalescent (see \cref{thm:U2}).

\begin{proof}[Proof of \cref{lem:U1:aux}]
    Fix $\fpart \in \mathcal{P}$ throughout the proof.
    Uniqueness is straightforward: if \cref{eq:U1} holds with $\la$ and $\nn$, but also with $\la'$ and $\nn'$, then for every $s > 0$, $\boldsymbol x\in E^\fpart_\circ $, and Lebesgue-almost all $\boldsymbol y\in E^\fpart_\circ $ (with the exceptional null-set depending on $\boldsymbol x$ and $s$),
    \begin{equation}\label{eqprf:U1:12}
        \e^{-\la s} \frac{\nn(\boldsymbol y)}{\nn(\boldsymbol x)} = \e^{-\la' s} \frac{\nn'(\boldsymbol y)}{\nn'(\boldsymbol x)}.
    \end{equation}
    Letting $s \to 0$ in $\Q$, it follows that for every $\boldsymbol x$, and all $\boldsymbol y\not\in N_{\boldsymbol x}$ for a null-set $N_{\boldsymbol x}\subset E^\fpart_\circ $, we have $\nn(\boldsymbol y) / \nn(\boldsymbol x) = \nn'(\boldsymbol y) / \nn'(\boldsymbol x)$. For a fixed choice $\boldsymbol x_0$ and $a \coloneqq \nn(\boldsymbol x_0) / \nn'(\boldsymbol x_0) > 0$, this implies $\nn(\boldsymbol y) = a \nn'(\boldsymbol y)$ for all $\boldsymbol y\in E^\fpart_\circ  \setminus N_{\boldsymbol x_0}$. Then for any $\boldsymbol x\in E^\fpart_\circ $ we can choose $\boldsymbol y\in E^\fpart_\circ  \setminus (N_{\boldsymbol x_0}\cup N_{\boldsymbol x})$ to obtain $
        \nn(\boldsymbol x) = \frac{\nn(\boldsymbol y)}{\nn'(\boldsymbol y)} \nn'(\boldsymbol x) = a \nn'(\boldsymbol x).
        $ Back into \cref{eqprf:U1:12} this also implies $\la = \la'$. Conversely, if \cref{eq:U1:2} holds for some function $\nn_\fpart$, then it also holds for $c\nn_\fpart$ for any $c > 0$.

    If $|\fpart| = 1$ then \cref{eq:U1} holds for $\beta_\fpart = 0$ and $\nn_\fpart \equiv 1$ (because a single lineage follows a Brownian motion), and the only element of $\mathbb{F}(\fpart)$ is the trivial forest, which makes the remaining statements trivial. Assume $|\fpart| > 1$ for the rest of the proof. Let $F_0\in \mathbb{F}(\fpart)$ be the trivial forest.

    Let $s > 0$ and $\boldsymbol x\in E^\fpart_\circ $.
    If the trivial forest is possible, then \cref{eqprf:U1:11} implies \cref{eq:U1} with $\beta_\fpart = 0$ and $\nn_\fpart(\boldsymbol x) = 1 / \P^{\boldsymbol x}(\fr = F_0)$. % By contraposition, this also proves that if the claim holds with $\beta_{\fpart} > 0$ then the trivial forest must be impossible. %this is proved separately elsewhere
    \Cref{eq:U1:2} for the trivial forest $F_0$ is just $\nn_\fpart(\boldsymbol x) \P^{\boldsymbol x}(\fr = F_0) = \nn_\fpart(\boldsymbol y) \P^{\boldsymbol y}(\fr = F_0)$, which then holds by definition of $\nn_\fpart$.

    It remains to prove \cref{eq:U1} in the case where the trivial forest is impossible, as well as \cref{eq:U1:2} for non-trivial forests. Thus we may assume that there exists a possible non-trivial forest $F\in \mathbb{F}(\fpart)$ (otherwise there is nothing left to prove). Let such an $F$ be given, and define for every $\boldsymbol x \in E^\fpart_\circ $,
    \begin{align}\label{eqprf:U1:defG}
        G(\boldsymbol x,F,\diff \tau,\diff \xi) &\coloneqq \frac{P^{\boldsymbol x}(F,\diff \tau,\diff \xi)}{\smashop{\prod_{\substack{u\in \fpart\\ \pr_F(u) \neq \emptyset }}}p(\tau_{\pr_F(u)}, \xi_{\pr_F(u)}-\boldsymbol x_u)},
        %\frac{P^{\boldsymbol x}(F,\diff \tau,\diff \xi)}{\prod_{u\in \fpart,\pr(u) \neq \emptyset } p(\tau_{\pr(u)}, \xi_{\pr(u)}-\boldsymbol x_u)},\\
    \end{align}
    which is a non-zero measure on $\dcb(F)$ that is finite on $\{\tau \ge a\} \times \dcs(F) \subset \dcb(F)$ for every $a > 0$. Recall from \cref{lem:forestpossible} that $P^{\boldsymbol x}(F,\{\tau \ge a\}) > 0$ for every $a \ge 0$ and $\boldsymbol x\in E^\fpart_\circ $, so
    %Let $0 \le a < b$ such that $P^{\boldsymbol x}(F,\left\{ a < \tau \le b \right\} ) > 0$ for all $\boldsymbol x\in E^\fpart_\circ $, then
    \begin{equation*}%\label{eqprf:U1:3_1}
        G_F(\boldsymbol x,s) \coloneqq G(\boldsymbol x,F,\left\{ \tau > 1+s \right\} \times \dcs(F)) = \smashop{\int\limits_{\dct(F)}} \ind_{\left\{ \tau > 1+s \right\} } \smashop{\int\limits_{\dcs(F)}} G(\boldsymbol x,F,\diff \tau,\diff \xi) \in (0,\infty)
    \end{equation*}
    for all $\boldsymbol x\in E^\fpart_\circ $ and $s \ge 0$. By continuity from below for measures, $G_F(\boldsymbol x,\cdot )$ is right-continuous on $[0,\infty)$.
    Rearranging \cref{eqprf:U1:3} using \eqref{eqprf:U1:defG} gives
    \begin{equation}\label{eqprf:U1:4:0}
        G(\boldsymbol x,F,\diff (\tau+s),\diff \xi)  \diff \boldsymbol y = \frac{\semi_s(\boldsymbol x,\diff \boldsymbol y)}{\prod_{u\in \fpart} p_s(\boldsymbol x_u-\boldsymbol y_u)} G(\boldsymbol y,F,\diff \tau,\diff \xi).
    \end{equation}
    %If we multiply \cref{eqprf:U1:4} with $\prod_{u\in \fpart}p_s(\boldsymbol x_u-\boldsymbol y_u)$ and integrate over $\boldsymbol y\in E^\fpart_\circ $ and $ \left\{ \tau \ge 1 \right\} \times \dcs(F) $, then we obtain a positive number on the RHS (by \cref{eqprf:U1:3_1,lem:Sspositive}) and the finite number a finite number on the LHS (because $G(\boldsymbol x,F,\cdot )$ is finite on $\{\tau \ge 1\} \times \dcs(F) \subset \dcb(F)$), so both are positive and finite, in particular \[
    %   G_F(\boldsymbol x,s) \coloneqq G(\boldsymbol x,F,\left\{ \tau \ge 1 + s \right\} ) \in (0,\infty)
    %\smashoperator[r]{\int\limits_{\dcb(F)}} \ind_{\left\{ \tau \ge s \right\} } G(\boldsymbol x, F,\diff \tau,\diff \xi) > 0
    %\] for all $\boldsymbol x$ and $s \ge 0$.
    If we integrate this over $\left\{ \tau > 1 + t \right\} \times\dcs(F) \subset \dcb(F) $ for some $t \ge 0$, we obtain
    \begin{align}\label{eqprf:U1:5}
        \frac{\semi_s(\boldsymbol x,\diff \boldsymbol y)}{\prod_{u\in\fpart}p_s(\boldsymbol x_u-\boldsymbol y_u)} = \frac{G_F(\boldsymbol x,t+s )}{G_F(\boldsymbol y,t )}\diff \boldsymbol y % = g(s,\boldsymbol x,\boldsymbol y) \diff \boldsymbol y,
    \end{align}
    for $s > 0$. Since the LHS does not depend on $F$ or $t$, neither does the RHS, and we can equate it for $t = 0$ and $t \ge 0$ to obtain that for every $s > 0,t \ge 0$, $\boldsymbol x\in E^\fpart_\circ $,
    \begin{equation}\label{eqprf:U1:5.5}
        g(s,\boldsymbol x,\boldsymbol y) \coloneqq \frac{G_F(\boldsymbol x,s )}{G_F(\boldsymbol y,0) } = \frac{G_F(\boldsymbol x,t+s )}{G_F(\boldsymbol y,t )},
    \end{equation}
    for Lebesgue-almost all $\boldsymbol y\in E^\fpart_\circ $, and $g$ does not depend on $F$. By right-continuity of $G_F(\boldsymbol x,\cdot )$, the same is true for $s = 0$ (and any $t \ge 0$). Thus for $t, s \ge 0$ and $\boldsymbol x \in E^\fpart_\circ $, there is a Lebesgue-null set $N_{\boldsymbol x,s,t}\subset E^\fpart_\circ $ such that for all $\boldsymbol z\not\in N_{\boldsymbol x,s,t}$, and all $\boldsymbol y \in E^\fpart_\circ $,
    \begin{align*}
        g(s,\boldsymbol x,\boldsymbol z)g(t,\boldsymbol z,\boldsymbol y) = \frac{G_F(\boldsymbol x,s)}{G_F(\boldsymbol z,0)} \frac{G_F(\boldsymbol z,t)}{G_F(\boldsymbol y,0)} = \frac{G_F(\boldsymbol x,s+t)}{G_F(\boldsymbol z,t)} \frac{G_F(\boldsymbol z,t)}{G_F(\boldsymbol y,0)} = g(s+t,\boldsymbol x,\boldsymbol y).
    \end{align*}
    This implies\footnote{%
        The sceptical reader may apply Lemma~\ref{lem:glem} to $\log g(s,\boldsymbol x,\boldsymbol y)$, which is well-defined because $g(s,\boldsymbol x,\boldsymbol y) > 0$. The assumptions of Lemma~\ref{lem:glem} are satisfied with $b(\boldsymbol x) = \log G_F(\boldsymbol x,0)$, and $A_{\boldsymbol x} \coloneqq E^\fpart_\circ \setminus \bigcup_{s,t\in \Q \cap [0,\infty)} N_{\boldsymbol x,s,t}$.
    }
    that there exists $\la \in \R$ such that for $\boldsymbol x\in E^\fpart_\circ $, $s\in \Q \cap [0,\infty)$, and Lebesgue-almost all $\boldsymbol y\in E^\fpart_\circ $,
    \begin{equation}\label{eqprf:U1:6}
        g(s,\boldsymbol x,\boldsymbol y) = \e^{-\la s} \frac{G_F(\boldsymbol x,0)}{G_F(\boldsymbol y,0)} = \e^{-\la s} \frac{\nn_F(\boldsymbol y)}{\nn_F(\boldsymbol x)},
    \end{equation}
    where $\nn_F(\boldsymbol x) \coloneqq  G_F(\boldsymbol x,0)^{-1}$.
    If the trivial forest is impossible, in which case we have not yet found $\la_\fpart$ and $\nn_\fpart$, \cref{eqprf:U1:5,eqprf:U1:5.5,eqprf:U1:6} imply \cref{eq:U1} for $s\in \Q \cap (0,\infty)$ with $\la_\fpart \coloneqq \la$ and $\nn_\fpart \coloneqq \nn_F$ (the fact that this choice for $\la_\fpart$ and $\nn_\fpart$ satisfies~\eqref{eq:U1} implies by uniqueness that the construction would have yielded the same $\nn_F$ for any $F$, up to a constant multiple). In that case, if we assume $\la_\fpart < 0$, then integrating \cref{eq:U1} over $\boldsymbol y\in E^\fpart_\circ $ gives
    \begin{align*}
        1 \ge \int \semi_s(\boldsymbol x,\diff \boldsymbol y)
        %&= \e^{-\la_\fpart s} \frac{1}{\nn_\fpart(\boldsymbol x)} \\
         &\ge \e^{|\la_\fpart| s} \frac{1}{\nn_\fpart(\boldsymbol x)} (\min_E p_s)^{|\fpart|}\int \nn_\fpart(\boldsymbol y) \diff \boldsymbol y \tendsto{s\to \infty} \infty,
    \end{align*}
    a contradiction.
    We now explain why it sufficed to show \cref{eq:U1} for $s \in \Q_+ \coloneqq \Q\cap (0,\infty)$. Denote the RHS in \cref{eq:U1} for fixed $\boldsymbol x\in E^\fpart_\circ $ temporarily by $h(s,\boldsymbol y)\diff \boldsymbol y$, which has now been shown to be a finite measure on $E^\fpart_\circ $ for all $s\in \Q_+$ (because the LHS in \cref{eq:U1} is a finite measure). Then $\nn_\fpart$ is integrable because $\inf p_s > 0$ for any fixed $s\in \Q_+$, which implies that $h(s,\boldsymbol y)\diff \boldsymbol y$ is a finite measure for all $s > 0$, and enables an application of dominated convergence to show that $s\mapsto h(s,\boldsymbol y)\diff \boldsymbol y$ is continuous in the weak topology on $\mathcal{M}_F(E^\fpart_\circ )$. %old document has argument for Rd that doesn't need integrability of N
    The LHS of \cref{eq:U1} is right-continuous in $s$ w.r.t.\ the same topology because $(\fpart_t,\boldsymbol X_t)$ is almost-surely right-continuous, which implies that \cref{eq:U1} holds for all $s > 0$.

    It remains to prove \cref{eq:U1:2} for every non-trivial, possible forest $F\in \mathbb{F}(\fpart)$. Fix one, if it exists, abbreviate $\la = \la_\fpart$, $\nn = \nn_\fpart$, and recall the definition of $G(\boldsymbol x,F,\diff \tau,\diff \xi)$ from \cref{eqprf:U1:defG}. Then \cref{eq:U1:2} is equivalent to $\e^{\la s} \nn(\boldsymbol x)G(\boldsymbol x,F,\diff (\tau+s),\diff \xi) =\nn(\boldsymbol y) G(\boldsymbol y,F,\diff \tau,\diff \xi)$. Plugging \cref{eq:U1} into \cref{eqprf:U1:4:0} and rearranging shows that for every $\boldsymbol x\in E^\fpart_\circ $ and $s > 0$,
    \begin{equation}\label{eqprf:U1:31}
        \e^{\la s} \nn(\boldsymbol x)G(\boldsymbol x,F,\diff (\tau+s),\diff \xi) \diff \boldsymbol y = \nn(\boldsymbol y) G(\boldsymbol y,F,\diff \tau,\diff \xi) \diff \boldsymbol y
    \end{equation}
    as measures on $\dcb(F) \times E^\fpart_\circ$. Since $G(\boldsymbol x,F,\diff \tau,\diff \xi)$ is a locally finite measure on $\dcb(F)$, the LHS in \cref{eqprf:U1:31} is continuous in $s\in [0,\infty)$ w.r.t.\ the topology of vague convergence on the space of locally finite measures on $\dcb(F) \times E^\fpart_\circ$. % on the space of Radon measures on $E^\fpart_\circ  \times \dcb(F)$
    Thus we can let $s \to 0$ in \cref{eqprf:U1:31} and obtain that it also holds for $s = 0$. Since the $\sigma$-algebra on $\dcb(F)$ is countably generated, we can further find for every $\boldsymbol x\in E^\fpart_\circ $ and $s \ge 0$ a single null set $N_{\boldsymbol x,s}\subset E^\fpart_\circ $ such that for all $\boldsymbol y\in E^\fpart_\circ  \setminus N_{\boldsymbol x,s}$, \[
        \e^{\la s} \nn(\boldsymbol x) G(\boldsymbol x,F,\diff (\tau+s),\diff \xi) = \nn(\boldsymbol y) G(\boldsymbol y,F,\diff \tau,\diff \xi)
    \] as measures on $\dcb(F)$, which is \cref{eq:U1:2}.
    If $s = 0$ then this equality is symmetric in $\boldsymbol x$ and $\boldsymbol y$ and must thus already hold for \emph{all} pairs $\boldsymbol x,\boldsymbol y\in E^\fpart_\circ $.
    %Indeed, fix some $\boldsymbol x_0 \in E^\fpart_\circ $ so that $h(\boldsymbol x_0) \coloneqq \widetilde{N}(\boldsymbol x_0) H(\boldsymbol x_0,\cdot ) = h(\boldsymbol y)$ for all $\boldsymbol y \not\in N_{\boldsymbol x,0}$. But then if there was $\boldsymbol y_0$ with $h(\boldsymbol y_0) \neq h(\boldsymbol x_0)$, then $h(\boldsymbol y_0) \neq h(\boldsymbol y)$ for all $\boldsymbol y$, but by assumption $h(\boldsymbol y_0) = h(\boldsymbol y)$ for all $y \not\in N_{\boldsymbol y_0} $.
    Then for $s > 0$, $\boldsymbol x,\boldsymbol y\in E^\fpart_\circ $, and $\boldsymbol y'\not\in N_{\boldsymbol x,s}$,
    \begin{align*}
        \e^{\la s} \nn(\boldsymbol x) G(\boldsymbol x,F,\diff (\tau+s),\diff \xi)
        &= \nn(\boldsymbol y') G(\boldsymbol y',F,\diff \tau,\diff \xi) = \nn(\boldsymbol y) G(\boldsymbol y,F,\diff \tau,\diff \xi).
    \end{align*}
\end{proof}
For the rest of this section, we fix an arbitrary choice for $\nn_\fpart,\, \fpart\in \mathcal{P}$ (which, recall, is unique up to a constant multiple).
%{lem:glem} is now in the appendix

\begin{lemma}\label{lem:Ncontinuous} % we don't actually need this for (local) finiteness of the nu-measures! double check the proof below. if we really don't need it then we can scrap the St C_b \subset C_b assumption
    $\nn_\fpart$ is continuous for every $\fpart\in \mathcal{P}$.
\end{lemma}
\begin{proof}
    Fix $\fpart\in \mathcal{P}$. First note that \cref{eq:U1} implies integrability of $\nn_\fpart$ over $E^{\fpart}_\circ $, because the LHS is integrable, and $\prod_{u\in \fpart}p_s(\boldsymbol x_u-\boldsymbol y_u)$ can be uniformly bounded from below by a positive constant for fixed $s > 0$. Now suppose that $\boldsymbol x_n\to \boldsymbol x$ in $E^\fpart_\circ $, and define $f\in C_b(\mathcal{X})$ by $f(\fpart',\boldsymbol x) \coloneqq \ind_{\left\{ \fpart'= \fpart \right\} }$. Then $\semi_1 f \in C_b(\mathcal{X})$, and
    \begin{align*}
        (S_1 f)(\boldsymbol x_n)
        &= \int_{E^{\fpart}_\circ } \semi_t((\fpart,\boldsymbol x_n),(\fpart,\diff \boldsymbol y))
        = \e^{-\la_\fpart} \frac{1}{\nn_\fpart(\boldsymbol x_n)} \int \nn_\fpart(\boldsymbol y) \prod_{u\in \fpart} p_1(\boldsymbol y_u - \boldsymbol x_n(u)) \diff \boldsymbol y.
    \end{align*}
    By dominated convergence and integrability of $\nn_\fpart$, we can pull the limit $n\to \infty$ into the integral on the RHS, the LHS converges to $(\semi_1 f)(\boldsymbol x)$ because $\semi_1 f$ is continuous, and all involved quantities are positive,
    %it follows that \[
    %    \int \nn_\fpart(\boldsymbol y) \prod_{u\in \fpart} p_1(\boldsymbol y_u-\boldsymbol x_n(u)) \diff \boldsymbol y \tendsto{n\to \infty}  \int \nn_\fpart(\boldsymbol y) \prod_{u\in \fpart}p_1(\boldsymbol y_u-\boldsymbol x_u) \diff \boldsymbol y > 0,
    %\] and because $S_1f$ is continuous, \[
    %    (S_1 f)(\boldsymbol x_n) \to (S_1 f)(\boldsymbol x) = \e^{-\beta} \frac{1}{\nn_\fpart(\boldsymbol x)} \int \nn_\fpart(\boldsymbol y) \prod_{u\in \fpart} p_1(\boldsymbol y_u-\boldsymbol x_u)\diff \boldsymbol y.
    %\]
    so we must have $\nn_\fpart(\boldsymbol x_n) \to \nn_\fpart(\boldsymbol x)$.
\end{proof}

\begin{lemma}\label{lem:U1:aux2}
    Let $\fpart\in \mathcal{P}$. If $\la_\fpart > 0$ then the trivial forest is impossible, and if $\la_\fpart = 0$ then the trivial forest is the only possible forest and $\nn_\fpart$ is constant. In particular, a forest $F\in\mathbb{F}(\mathcal{P})$ with $\la_{\rt(F)} > 0$ is impossible.
\end{lemma}
\begin{proof}
    Fix $\fpart \in \mathcal{P}$ and $\boldsymbol x\in E^\fpart_\circ $, and let $F_0\in \mathbb{F}(\fpart)$ be the trivial forest. If we integrate \cref{eq:U1} over $\boldsymbol y\in E^\fpart_\circ $ and let $s \to \infty$, we obtain
    \begin{align}\begin{split}\label{eqprf:U1aux2:1}
        \P^{\boldsymbol x}(\fr = F_0)
        &= \lim_{s \to \infty}\frac{\e^{-\la_\fpart s}}{\nn_\fpart(\boldsymbol x)} \int \nn_\fpart(\boldsymbol y) \prod_{u\in \fpart}p_s(\boldsymbol x_u-\boldsymbol y_u) \diff \boldsymbol y
        = \frac{\ind_{\left\{ \la_\fpart = 0 \right\} }}{\nn_\fpart(\boldsymbol x)} \int \nn_\fpart(\boldsymbol y) \diff \boldsymbol y,
    \end{split}\end{align}
    where we used that $p_s \to 1$ uniformly as $s\to \infty$. If $\la_\fpart > 0$ then \cref{eqprf:U1aux2:1} implies $\P^{\boldsymbol x}(\fr = F_0) = 0$ for all $\boldsymbol x\in E^\fpart_\circ $, so $F_0$ is impossible. Otherwise, $\boldsymbol z \mapsto \nn_\fpart(\boldsymbol z) \P^{\boldsymbol z}(\fr = F_0)$ is constant and positive by \cref{lem:U1:aux}, so \cref{eqprf:U1aux2:1} turns into \[
        1 = \int \P^{\boldsymbol y}(\fr = F_0)^{-1} \diff \boldsymbol y,
    \] which implies that $\P^{\boldsymbol y}(\fr = F_0) = 1$ for almost-all $\boldsymbol y\in E^\fpart_\circ $, in particular at least one $\boldsymbol y_0\in E^\fpart_\circ $. Then $\P^{\boldsymbol y_0}(\fr = F) = 0$ for every $F\in \mathbb{F}(\fpart)\setminus \left\{ F_0 \right\} $, so all non-trivial forests are impossible. Thus $\P^{\boldsymbol y}(\fr = F_0) = 1$ for all $\boldsymbol y\in E^\fpart_\circ $, so $\nn_\fpart(\boldsymbol y) = \nn_\fpart(\boldsymbol y) \P^{\boldsymbol y}(\fr = F_0)$ is constant in $\boldsymbol y\in E^\fpart_\circ $.

    The final statement essentially follows from the (strong) Markov property. Let $F \in \mathbb{F}(\fpart)$ be a forest with $\la_{\rt(F)}>0$. Let $\boldsymbol x \in E^{\fpart}_\circ $, and recall that $T\coloneqq \tau_{\rt(F)}\in (0,\infty]$ denotes the random time first time at which $\fpart_t = \rt(F)$. Then,
    \begin{align*}
        \P^{(\fpart,\boldsymbol x)}(\fr = F) 
        &\le \P^{(\fpart,\boldsymbol x)}(T < \infty, \text{ no merge events after $T$})\\
        &= \E^{(\fpart,\boldsymbol x)} \bigg[ \ind_{\left\{ T < \infty \right\}  } \underbrace{\P^{(\rt(F),\boldsymbol X_T)} \left( \fr \text{ is trivial} \right)}_{=0}   \bigg] \\
        &= 0.
    \end{align*}
    In the second step we used the strong Markov property, and in the final step we used that $\la_{\rt(F)} > 0$ and hence, starting from $(\rt(F),\boldsymbol X_T)$, the trivial forest is impossible.
\end{proof}

\begin{theorem}\label{thm:U2}
    There exists a unique family $\boldsymbol \nu\in \nurates$ such that \cref{eq:BSC} holds for every $F\in \mathbb{F}$ and $\boldsymbol x\in E^{\lf(F)}_\circ $, that is \[
        P^{\boldsymbol x}(F, \diff \tau,\diff \xi) = \frac{1}{N^{\boldsymbol \nu}(\boldsymbol x)} \fnu(\tau,\xi \,\vert\,\boldsymbol x) \boldsymbol \nu_F(\diff \xi) \diff \tau
    \] as measures on $\dcb(F)$. For all $\fpart \in \mathcal{P}$, $\la_\fpart = \nut{\fpart}$, and $\nn_\fpart / N^{\boldsymbol \nu}$ is constant on $E^{\fpart}_\circ $.
\end{theorem}

That is, there are choices for $\nn_\fpart,\, \fpart\in \mathcal{P}$ such that $\nn_\fpart = N^{\boldsymbol \nu} \vert_{E^\fpart_\circ }$. If we prove \cref{thm:U2}, then we finished proving the ``only if'' direction of \cref{thm:BSC}.

We first prove a helpful consequence of \cref{lem:U1:aux}. If $F = \{\fpart,\fpart_1^F, \ldots ,\fpart_m^F\}\in \mathbb{F}$ with $m > 1$, $G\coloneqq \{\fpart_1^F, \ldots ,\fpart_m^F\}$, and $(\tau,\xi) \in \dcb(G)$, $s > 0$, $\boldsymbol z\in E^{\fpart_1^F\setminus \fpart}_\circ $, then we define $s\tau \in \dct(F)$ by $\fpart \mapsto s$ and $\fpart^F_i \mapsto s + \tau(\fpart^F_i)$ for $i\in [m]$, and $\boldsymbol z\xi \in \dcs(F)$ in the sense of \cref{eq:joinmaps}.
This defines a bijection
\begin{IEEEeqnarray*}{rCl}
    (0,\infty)\times E^{\fpart^F_1\setminus \fpart}_\circ \times \dcb(G) & \quad \to \quad & \dcb(F),\\
    (s,\boldsymbol z,\tau,\xi) & \mapsto & (s\tau,\boldsymbol z\xi).
\end{IEEEeqnarray*}

\def\hf{H^\circ}
\def\he{H^\circ}
\begin{lemma}\label{lem:HF}
    Let $F = \{\fpart,\fpart_1^F, \ldots ,\fpart_m^F\} \in \mathbb{F}$ be non-trivial, then
    \begin{equation*}%\label{eqdef:HF}
        H_F(\diff \tau,\diff \xi) \coloneqq \e^{\beta_{\fpart}\min \tau} \frac{\nn_\fpart(\boldsymbol x) P^{\boldsymbol x}(F,\diff \tau,\diff \xi)}{\smashop{\prod_{\substack{u\in \fpart \\ \pr_F(u) \neq \emptyset }}} p(\tau_{\pr_F(u)}, \xi_{\pr_F(u)}-\boldsymbol x_u)},
    \end{equation*}
    defines a finite measure on $\dcb(F)$ whose definition does not depend on $\boldsymbol x \in E^{\fpart}_\circ $. If $m = 1$, then there is a finite measure $\he_{\fpart,\fpart_1^F}$ on $E^{\fpart_1^F\setminus \fpart}_\circ $ such that \[
        H_F(\diff \tau,\diff \xi) = \he_{\fpart,\fpart_1^F}(\diff \xi)\diff \tau.
        \] If $m > 1$, so $G \coloneqq \{\fpart_1^F, \ldots ,\fpart_m^F\}$ is non-trivial, then there exists a finite measure $\hf_F$ on $E^{\fpart_1^F\setminus \fpart}_\circ \times \dcb(G)$ such that \[
        H_F(\diff (s\tau),\diff (\boldsymbol z\xi)) = \hf_F(\diff \boldsymbol z,\diff \tau,\diff \xi) \diff s.
    \]
\end{lemma}
\begin{proof}
    Abbreviate $\pr = \pr_F$ and $\ch = \ch_F$. If we multiply \cref{eq:U1:2} with $\e^{\la_\fpart \min \tau}$, we get
    \begin{equation}\label{eqprf:HF:1}
        \e^{\la_\fpart \min(\tau + s)} \frac{\nn_\fpart(\boldsymbol x) P^{\boldsymbol x}(F, \diff (\tau+s),\diff \xi)}{\smashop{\prod_{\substack{u\in \fpart\\ \pr(u) \neq \emptyset} }} p(\tau_{\pr(u)} + s,\xi_{\pr(u)} - \boldsymbol x_u)} = \e^{\la_\fpart \min \tau} \frac{\nn_\fpart(\boldsymbol y) P^{\boldsymbol y}(F, \diff \tau,\diff \xi)}{\smashop{\prod_{\substack{u\in \fpart \\ \pr(u) \neq \emptyset }} }p(\tau_{\pr(u)} ,\xi_{\pr(u)} - \boldsymbol y_u)}
    \end{equation}
    for every $\boldsymbol x,\boldsymbol y\in E^\fpart_\circ $ and $s \ge 0$ as measures on $\dcb(F)$. Putting $s = 0$ shows that the definition of $H_F$ does not depend on $\boldsymbol x$, and in general we can rewrite \cref{eqprf:HF:1} as
    \begin{equation}\label{eqprf:HF:2}
        H_F(\diff (\tau+s),\diff \xi) = H_F(\diff \tau,\diff \xi),\qquad s \ge 0,
    \end{equation}
    as measures on $\dcb(F)$.

    Now assume that $m = 1$, then $F = \{\fpart,\fpart_1^F\}$, and $H_F$ is a finite measure on $(0,\infty) \times E^{\fpart_1^F\setminus \fpart}_\circ $ which is translation-invariant in its first component. For $A\subset E^{\fpart_1^F\setminus \fpart}_\circ $ measurable, $H_F(\cdot ,A)$ is a translation-invariant measure on $(0,\infty)$, thus a multiple $\he_{\fpart,\fpart_1^F}(A) \in [0,\infty]$ of Lebesgue measure. Then \[
        \he_{\fpart,\fpart_1^F}(A) = H_F((1,2)\times A),\qquad A \subset E^{\fpart_1^F\setminus \fpart} \text{ measurable},
    \] which implies that $\he_{\fpart,\fpart_1^F}$ is a finite measure. Then $H_F(\diff s,\diff \xi) = \he_{\fpart,\fpart_1^F}(\diff \xi) \diff s $.

    Now suppose that $m > 1$, then \cref{eqprf:HF:2} implies that, for $A \subset E^{\fpart_1^F\setminus \fpart}_\circ \times \dcb(G)$ measurable, \[
        H_F\left(  \left\{ (s\tau,\boldsymbol z\xi) \in \dcb(F)\colon s \in \,\cdot \,, (\boldsymbol z,(\tau,\xi))\in A \right\}  \right)
    \] is a translation-invariant measure on $(0,\infty)$ and thus a multiple of Lebesgue measure, which implies with a similar argument that \[
        H_F(\diff (s\tau),\diff (\boldsymbol z\xi)) = \hf_F(\diff \boldsymbol z,\diff \tau,\diff \xi) \diff s
    \] for a finite measure $\hf_F$ on $E^{\fpart_1^F\setminus \fpart}_\circ  \times \dcb(G)$.
\end{proof}

\def\fmerge{\fpart_\text{m}}
\def\fnomerge{\fpart_\text{n}}
\begin{proof}[Proof of \cref{thm:U2}]
    We use an induction over $|\lf(F)|$ to prove existence of $\boldsymbol \nu$, and comment on uniqueness at the end of the proof. If $\fpart \in \mathcal{P}$ with $|\fpart| = 1$ then the only $F\in \mathbb{F}(\fpart)$ is trivial and \cref{eq:BSC} holds automatically. Furthermore $\la_\fpart = \nut{\fpart} = 0$ by definition, and both $N^{\boldsymbol \nu} \vert_{E^\fpart_\circ }$ (for any $\boldsymbol \nu\in \nurates$ by definition) and $\nn_\fpart$ are constant (\cref{lem:U1:aux2}). Let $\fpart \in \mathcal{P}$ with $|\fpart| \ge 2$ and suppose that $\nu_{\fpart',\fpart''}$ for $\fpart'' < \fpart'$ with $|\fpart'| < |\fpart|$ have already been found such that \cref{eq:BSC} holds for all $F\in \mathbb{F}$ with $|\lf(F)| < |\fpart|$, and such that for all $\fpart'\in \mathcal{P}$ with $|\fpart'| < |\fpart|$, $\la_{\fpart'} = \nut{\fpart'}$, and $\nn_{\fpart'} / N^{\boldsymbol \nu}$ is constant on $E^{\fpart'}_\circ $ (this is a well-defined statement because $N^{\boldsymbol \nu}$ on $E^{\fpart'}_\circ $ only depends on transition measures $\nu_{\fpart_1,\fpart_2}$ with $\fpart_1 \le \fpart'$).

    To complete the inductive step, we need to find finite measures $\nu_{\fpart,\fpart'}$ for $\fpart' \in \mathcal{P}$ with $\fpart' < \fpart$ such that $\nut{\fpart} = \la_\fpart$ and, for some $c_\fpart > 0$,
    \begin{equation}\label{eqprf:U2:100}
        P^{\boldsymbol x}(F,\diff \tau,\diff \xi) = \frac{c_\fpart}{\nn_\fpart(\boldsymbol x)} \fnu(\tau,\xi \,\vert\,\boldsymbol x) \boldsymbol \nu_F(\diff \xi) \diff \tau
    \end{equation}
    as measures on $\dcb(F)$ for every $F\in \mathbb{F}(\fpart)$ and $\boldsymbol x\in E^\fpart_\circ $ (recall~\eqref{eq:nu_F} for the connection between $\boldsymbol \nu_F$ and $\nu_{\fpart,\fpart'}$). Indeed, the left-hand side defines a probability measure on $\dcb(\mathbb{F}(\fpart))$, and $N^{\boldsymbol \nu}$ is, by definition, the correct normalisation, so $c_\fpart / \nn_\fpart = 1 / N^{\boldsymbol \nu}$ on $E^\fpart_\circ $.

    If $\la_{\fpart} = 0$, then all $F\in \mathbb{F}(\fpart)$ except for the trivial forest are impossible (\cref{lem:U1:aux2}), and \cref{eqprf:U2:100} holds for all $F\in \mathbb{F}(\fpart)$ by putting $\nu_{\fpart,\fpart'} = 0$ for all $\fpart' < \fpart$. Then $\nut{\fpart} = 0 = \la_\fpart$, and $\nn_\fpart$ is constant (\cref{lem:U1:aux2}). 

    Assume for the remainder of the inductive step that $\la_\fpart > 0$. Suppose for a moment that the construction in this inductive step will indeed yield $\la_\fpart = \nut{\fpart}$ (which we already know for all $\fpart' < \fpart$ by induction).
    Then by definition of~$f_{\boldsymbol \nu}$---recall~\eqref{eq:fnu} and the following paragraph---the RHS of~\eqref{eqprf:U2:100} is zero for forests $F$ with $\la_{\rt(F)} = \nut{\rt(F)} > 0$, and so is the LHS because such forests are impossible by \cref{lem:U1:aux2}.
    Therefore, when constructing the measures $\nu_{\fpart,\fpart'}$ for $\fpart'\in \mathcal{P}, \fpart' < \fpart$, with $\nut{\fpart} = \la_\fpart$, it suffices to check~\eqref{eqprf:U2:100} only for forests~$F$ with $\la_{\rt(F)}=0$.

    Let $F= \{\fpart,\fpart_1^F, \ldots ,\fpart_m^F\} \in \mathbb{F}$ be such a forest. In particular $ m \ge 1$, since $\la_\fpart > 0$.
    %Abbreviate $\nn = \nn_{\fpart}$, $\la = \la_{\fpart}$, $\pr = \pr_F$, $\ch = \ch_F$. rearranging the definition of $H_F$ gives \[
    %If $m = 1$ and $\la_{\fpart^F_1} > 0$, then $F = \{\fpart,\fpart^F_1\}$ is impossible, and both left- and right-hand side of \cref{eq:BSC} are zero (the latter because $\nut{\fpart^F_1} = \la_{\fpart^F_1} > 0$ by induction hypothesis).
    Suppose that $m=1$, so $\la_{\fpart^F_1} = 0$. We let $\nu_{\fpart,\fpart_1^F} = c_\fpart^{-1} \he_{\fpart,\fpart_1^F}$ for a $c_\fpart>0$ that we specify later. Then rearranging the definition of $H_F$ in \cref{lem:HF} implies
    \[
        P^{\boldsymbol y}(F,\diff \tau,\diff \xi) = \e^{-\la_\fpart \tau} \frac{c_\fpart}{\nn_\fpart(\boldsymbol y)} \prod_{\substack{u\in \fpart \\ \pr(u) \neq \emptyset } } p(\tau_{\pr(u)},\xi_{\pr(u)}-\boldsymbol y_u) \nu_{\fpart,\fpart_1^F}(\diff \xi) \diff \tau,
        \] where we abbreviated $\pr = \pr_F$ and $\ch = \ch_F$. Assuming that $\nut{\fpart} = \la_\fpart$ after choosing the remaining transition measures and $c_\fpart$, this is equivalent to \cref{eqprf:U2:100} (recall the relevant definitions from \cref{eq:fsp,eq:fnu,eq:nu_F}).

        Now suppose $m > 1$. The claim has already been proved for $G = \{\fpart_1^F, \ldots ,\fpart_m^F\}$ by induction hypothesis because $|\fpart_1^F| < |\fpart|$. If $G$ is impossible then we must have $\nu_{\fpart_i^F,\fpart_{i+1}^F} = 0$ for some $i\in [m-1]$ ($\la_{\rt(G)} = 0$ so there is no other way for the RHS in~\eqref{eqprf:U2:100} to be zero for $G$), and since $F$ is then also impossible, both left- and right-hand sides of \cref{eqprf:U2:100} for $F$ are zero. Suppose thus that $G$ is possible, so $\nu_{\fpart_i^F,\fpart_{i+1}^F} \neq 0$ for all $i \in [m-1]$. Let $\fmerge = \fpart\setminus \fpart_1^F$ and $\fnomerge = \fpart_1^F\cap \fpart$, the leaf nodes that did and did not merge in the first transition, respectively.
    Rearranging the definition of $H_F$ in \cref{lem:HF} gives
    \begin{align}\label{eqprf:U2:5}
        P^{\boldsymbol x}(F,\diff &(s\tau),\diff (\boldsymbol z\xi))\nonumber \\
                                          &= \frac{\e^{-\la_\fpart s}}{\nn_\fpart(\boldsymbol x)} H_F(\diff (s \tau),\diff (\boldsymbol z\xi)) \bigg[\smashop{\prod_{\substack{u\in \fpart \\ \pr(u) \neq \emptyset } }} p((s\tau)_{\pr(u)},(\boldsymbol z\xi)_{\pr(u)}-\boldsymbol x_u)\bigg] \nonumber \\
                                          &= \frac{\e^{-\la_\fpart s}}{\nn_\fpart(\boldsymbol x)} \hf_F(\diff \boldsymbol z,\diff \tau,\diff \xi) \diff s \bigg[\smashop{\prod_{\substack{u\in \fnomerge \\ \pr(u) \neq \emptyset } }} p(\tau_{\pr(u)}+s,\xi_{\pr(u)}-\boldsymbol x_u) \prod_{u\in \fmerge} p_s(\boldsymbol x_u-\boldsymbol z_{\pr(u)})\bigg].
    \end{align}

    We now evaluate  % actually if F is trivial there is something to show. should do induction over |\fpart|, and then for fixed \fpart look at all ensuing forests, and case distinguish whether there is a positive chance of any forest etc as in lem:U1. then choose a fixed but arbitrary non-trivial possible forest
    \begin{equation}\label{eqprf:U2:3}
        \P^{\boldsymbol x}(\fr = F, \tm \in \diff (s\tau), \sp \in \diff (\boldsymbol z\xi), \boldsymbol X_s \vert_{\fnomerge} \in \diff \boldsymbol y),
    \end{equation}
    a sub-probability measure on $\dcb(F) \times E^{\fnomerge}_\circ $, in two ways. On the one hand, by \cref{eq:intro:sec:BSC} it equals % exact strong markov property application handy 12 july 10:17
    \begin{multline}\label{eqprf:U2:4}
        P^{\boldsymbol x}(F,\diff (s\tau),\diff (\boldsymbol z\xi)) \kbr_{\boldsymbol x}((F,s\tau,z\xi),\{ \boldsymbol X_s \vert_{\fnomerge}\in \diff \boldsymbol y\})\\
        = P^{\boldsymbol x}(F,\diff (s\tau),\diff (\boldsymbol z\xi)) \prod_{u\in \fnomerge} p_s(\boldsymbol x_u-\boldsymbol y_u) \left[ \frac{p(\tau_{\pr(u)},\boldsymbol y_u - \xi_{\pr(u)})}{p(\tau_{\pr(u)}+s,\boldsymbol x_u - \xi_{\pr(u)})} \right] _{\pr(u) \neq \emptyset } \diff \boldsymbol y.
    \end{multline}
    Let $\tau_0 \coloneqq \inf \left\{ t > 0\colon \fpart_t \neq \fpart_0 \right\} $, and $\tau_{\fpart_1^F}\coloneqq \inf \left\{ t > 0\colon \fpart_t = \fpart_1^F \right\} $, which are both $(\mathcal{F}^{\boldsymbol X}_t)$-stopping times, and write \[
        P^{\boldsymbol x}_{\fpart,\fpart_1^F}(\diff s,\diff z,\diff \boldsymbol y) \coloneqq \P^{\boldsymbol x}(\tau_0=\tau_{\fpart_1^F}\in \diff s, \boldsymbol X_s \in \diff (\boldsymbol y\boldsymbol z)),
    \] a sub-probability measure on $(0,\infty)\times E^{\fpart^F_1}_\circ $ whose total mass is the probability that the first jump is $\fpart \to \fpart_1^F$ when starting at $\boldsymbol x$. For the second evaluation, by the strong Markov property at time $\tau_{\fpart_1^F}$, \cref{eqprf:U2:3} equals $P^{\boldsymbol x}_{\fpart,\fpart_1^F}(\diff s,\diff \boldsymbol z,\diff \boldsymbol y) P^{\boldsymbol y\boldsymbol z}(G,\diff \tau,\diff \xi)$, and equating with \cref{eqprf:U2:4} gives
    \begin{multline*}
        P^{\boldsymbol x}(F,\diff (s\tau),\diff (\boldsymbol z\xi))\diff \boldsymbol y \\
        = \frac{P^{\boldsymbol x}_{\fpart,\fpart_1^F}(\diff s,\diff \boldsymbol z,\diff \boldsymbol y)}{\prod_{u\in \fnomerge}p_s(\boldsymbol x_u-\boldsymbol y_u)} P^{\boldsymbol y\boldsymbol z}(G,\diff \tau,\diff \xi) \prod_{\substack{u\in \fnomerge \\ \pr(u) \neq \emptyset } } \frac{p(\tau_{\pr(u)} + s,\boldsymbol x_u - \xi_{\pr(u)})}{p(\tau_{\pr(u)},\boldsymbol y_u-\xi_{\pr(u)})}.
    \end{multline*}
    %Since the LHS is absolutely continuous w.r.t.\ Lebesgue measure in the $\boldsymbol y$ coordinate, so is $P^{\boldsymbol x}_{\fpart,\fpart_1}(\diff s,\diff \boldsymbol z,\diff \boldsymbol y)$ and we can write \[
    %    P^{\boldsymbol x}_{\fpart,\fpart_1}(\diff s,\diff \boldsymbol z,\diff \boldsymbol y) = P^{\boldsymbol x}_{\fpart,\fpart_1}(\boldsymbol y,\diff s,\diff \boldsymbol z) \diff \boldsymbol y,
    %\] where $P^{\boldsymbol x}_{\fpart,\fpart_1}(\boldsymbol y,\diff s,\diff \boldsymbol z)$ for $\boldsymbol y\in E^{\fnomerge}_\circ $ is a sub-probability measure on $(0,\infty) \times E^{\fmerge}_\circ $ such that $P^{\boldsymbol x}_{\fpart,\fpart_1}(\boldsymbol y, \left\{ \boldsymbol z\colon (\boldsymbol y\boldsymbol z)\not\in E^{\fpart_1}_\circ \right\}) = 0 $.
    Combining with \cref{eqprf:U2:5} gives, after some cancellation,
    \begin{multline*}
        P^{\boldsymbol x}_{\fpart,\fpart_1^F}(\diff s,\diff \boldsymbol z,\diff \boldsymbol y) P^{\boldsymbol y\boldsymbol z}(G,\diff \tau,\diff \xi)
        =  \frac{\e^{-\la_\fpart s}}{\nn_\fpart(\boldsymbol x)}\hf_F(\diff \boldsymbol z,\diff \tau,\diff \xi)\diff s\\
        \times \prod_{u\in \fnomerge}p_s(\boldsymbol x_u-\boldsymbol y_u) \prod_{u\in \fmerge} p_s(\boldsymbol x_u-\boldsymbol z_{\pr(u)}) \smashop{\prod_{\substack{u\in \fnomerge \\ \pr(u) \neq \emptyset } }} p(\tau_{\pr(u)},\boldsymbol y_u-\xi_{\pr(u)}) \diff \boldsymbol y.
    \end{multline*}%
    \def\kz{K}%
    for fixed $\boldsymbol x \in E^{\fpart}_\circ $. By \cref{eq:BSC} which already holds for $G$, $P^{\boldsymbol y\boldsymbol z}(G,\diff \tau,\diff \xi)$ has a strictly positive density $\frac{1}{N^{\boldsymbol \nu}(\boldsymbol y\boldsymbol z)}\fnu(\cdot \,\vert\, \boldsymbol y\boldsymbol z)\vert_{\dcb(G)}$ w.r.t.\ the finite measure $ \boldsymbol \nu_G(\diff \xi)\diff \tau$, so we can apply \cref{lem:prfunique:1} below with $x = (\boldsymbol z,\boldsymbol y)$ and $\zeta = (\tau,\xi)$,
    to obtain existence of a measure $\kz(\diff \boldsymbol y,\diff \boldsymbol z)$ on $E^{\fpart_1^F}_\circ $ (which may depend on $\boldsymbol x$ and $F$) with %beta and the already found nu measures are functions of the Markov process itself
    \begin{align}\label{eqprf:U2:6}
        \kz(\diff \boldsymbol y,\diff \boldsymbol z) \diff s
        &= \e^{\la_\fpart s} \nn_\fpart(\boldsymbol x) \frac{P^{\boldsymbol x}_{\fpart,\fpart_1^F}(\diff s,\diff \boldsymbol y,\diff \boldsymbol z)}{\prod_{u\in \fnomerge}p_s(\boldsymbol x_u-\boldsymbol y_u) \prod_{u\in \fmerge} p_s(\boldsymbol x_u - \boldsymbol z_{\pr(u)})}
    \end{align}
    as measures on $E^{\fpart_1^F}_\circ  \times (0,\infty)$, and
    \begin{align}\label{eqprf:U2:7}
        \kz(\diff \boldsymbol y,\diff \boldsymbol z) \boldsymbol \nu_G(\diff \xi) \diff \tau
        &= \hf_F(\diff \boldsymbol z,\diff \tau,\diff \xi)\frac{N^{\boldsymbol \nu}(\boldsymbol y\boldsymbol z)}{\fnu(\tau,\xi \,\vert\,\boldsymbol y\boldsymbol z)} \prod_{\substack{u\in \fnomerge \\ \pr(u) \neq \emptyset } } p(\tau_{\pr(u)},\boldsymbol y_u-\xi_{\pr(u)}) \diff \boldsymbol y,
    \end{align}
    as measures on $E^{\fpart_1^F}_\circ \times \dcb(G)$. From \cref{eqprf:U2:6} it follows that $\kz$ depends on $F$ only through $(\fpart,\fpart_1^F)$, and from \cref{eqprf:U2:7} it follows that $\kz$ does not depend on $\boldsymbol x$. In other words, had we chosen another $\boldsymbol x\in E^\fpart_\circ $ and another possible forest $F \in \mathbb{F}(\fpart)$ whose first transition is $\fpart\to \fpart_1^F$, we would have obtained the same measure $\kz = \kz_{\fpart,\fpart_1^F}$. % From \cref{eqprf:U2:6} it follows that $P^{\boldsymbol x}_{\fpart,\fpart_1^F}(1) = 0$ for some $\boldsymbol x$ iff $\kz \equiv 0$ iff $P^{\boldsymbol x}_{\fpart,\fpart_1^F}(1) = 0$ for all $\boldsymbol x$, in which case all forests $F' = (\fpart,\fpart_1^F,\ldots )$ are impossible and \cref{eq:BSC} holds if we put $\nu_{\fpart,\fpart_1^F} = 0$. Conversely if $F$ is possible then $P^{\boldsymbol x}_{\fpart,\fpart_1^F}(1) > 0$ so $\kz \neq 0$ Assume now that $\kz \neq 0$.
    %Then with this measure, \cref{eqprf:U2:6} holds for all $\boldsymbol x\in E^\fpart_\circ $, and $\cref{eqprf:U2:7}$ holds for all possible, non-trivial forests $F=(\fpart,\fpart_1^F, \ldots )\in \mathbb{F}(\fpart)$.
    It follows from \cref{eqprf:U2:7} that $K$ is absolutely continuous in its first component, that is there exist measures $\kz(\boldsymbol y,\diff \boldsymbol z)$ on $E^{\fpart_1^F\setminus \fpart}_{\boldsymbol y}\coloneqq \{\boldsymbol z\in E^{\fpart_1^F\setminus \fpart}_\circ\colon \boldsymbol y\boldsymbol z\in E^{\fpart_1^F}_\circ\} $ such that $\boldsymbol y \mapsto K(\boldsymbol y,A)$ is measurable for measurable $A$, and such that
    \begin{equation}\label{eqprf:U2:8}
        K(\boldsymbol y,\diff \boldsymbol z) \boldsymbol \nu_G(\diff \xi)\diff \tau = \hf_F(\diff \boldsymbol z,\diff \tau,\diff \xi)\frac{N^{\boldsymbol \nu}(\boldsymbol y\boldsymbol z)}{\fnu(\tau,\xi \,\vert\,\boldsymbol y\boldsymbol z)} \prod_{\substack{u\in \fnomerge \\ \pr(u) \neq \emptyset } } p(\tau_{\pr(u)}, \boldsymbol y_u - \xi_{\pr(u)})
    \end{equation}
    as measures on $E^{\fpart_1^F\setminus \fpart}_{\boldsymbol y}  \times \dcb(G)$ for Lebesgue-almost all $\boldsymbol y\in E^{\fnomerge}_\circ $. (The fact that we can choose a single null-set outside of which equality holds as measures is because the $\sigma$-algebra of $E^{\fpart_1^F\setminus \fpart}_{\boldsymbol y} \times \dcb(G)$ is countably generated.) If we rearrange \cref{eqprf:U2:8} for $\hf_F$ and plug it into \cref{eqprf:U2:5}, we obtain
    \begin{multline*}
        P^{\boldsymbol x}(F,\diff (s\tau),\diff (\boldsymbol z\xi))\\
                                  %&= \frac{\e^{-\la_\fpart s}}{\nn(\boldsymbol x)} \hf_F(\diff \boldsymbol z,\diff \tau,\diff \xi) \diff s \bigg[\smashop{\prod_{\substack{u\in \fnomerge \\ \pr(u) \neq \emptyset } }} p(\tau_{\pr(u)}+s,\xi_{\pr(u)}-\boldsymbol x_u) \prod_{u\in \fmerge} p_s(\boldsymbol x_u-\boldsymbol z_{\pr(u)})\bigg]\\
        = \frac{\e^{-\la_\fpart s}}{\nn_\fpart(\boldsymbol x)} \bigg[ \smashop{\prod_{\substack{u\in \fnomerge \\ \pr(u) \neq \emptyset }} } \frac{p(\tau_{\pr(u)}+s,\xi_{\pr(u)}-\boldsymbol x_u)}{p(\tau_{\pr(u)}, \xi_{\pr(u)}-\boldsymbol y_u)} \prod_{u\in \fmerge} p_s(\boldsymbol x_u-\boldsymbol z_{\pr(u)})\bigg]\fnu(\tau,\xi \,\vert\,\boldsymbol y\boldsymbol z) \\[-8pt]
                                  \times   \frac{K(\boldsymbol y,\diff \boldsymbol z)}{N^{\boldsymbol \nu}(\boldsymbol y\boldsymbol z)}\boldsymbol \nu_G(\diff \xi) \diff \tau\diff s
    \end{multline*}
    for Lebesgue-almost all $\boldsymbol y$. By the definition of $\fnu$ \cref{eq:fnu}, this is equal to
    \begin{align}\label{eqprf:U2:9}
        P^{\boldsymbol x}(F,\diff (s\tau),\diff (\boldsymbol z\xi)) = \frac{1}{\nn_\fpart(\boldsymbol x)} \fnu(s\tau,\boldsymbol z\xi \,\vert\, \boldsymbol x) \frac{K(\boldsymbol y,\diff \boldsymbol z)}{N^{\boldsymbol \nu}(\boldsymbol y\boldsymbol z)} \boldsymbol \nu_G(\diff \xi) \diff (s\tau).
    \end{align}
    Fix an arbitrary $\boldsymbol y_0\in E^{\fnomerge}_\circ $ for which this holds and put $\nu_{\fpart,\fpart_1^F}(\diff \boldsymbol z) \coloneqq c_\fpart^{-1} \frac{K(\boldsymbol y_0,\diff \boldsymbol z)}{N^{\boldsymbol \nu}(\boldsymbol y_0\boldsymbol z)}$ for the same $c_\fpart > 0$ we introduced before.
    If we divide by $N^{\boldsymbol \nu}(\boldsymbol y_0\boldsymbol z)$ in \cref{eqprf:U2:8} and integrate over $\xi \in \dcs(G)$ and $\tau \in \left\{ 1 \le \tau \le 2 \right\} \subset \dct(G)$ say it follows that $\nu_{\fpart,\fpart_1^F}$ is a finite measure.
    Now \cref{eqprf:U2:9} evaluated for $\boldsymbol y = \boldsymbol y_0$ becomes
    \begin{equation*}
        P^{\boldsymbol x}(F,\diff (s\tau),\diff (\boldsymbol z\xi)) = \frac{c_\fpart}{\nn_\fpart(\boldsymbol x)} \fnu(s\tau,\boldsymbol z\xi \,\vert\, \boldsymbol x) \boldsymbol \nu_F(\diff (\boldsymbol z\xi)) \diff (s\tau),
    \end{equation*}
    which implies \cref{eqprf:U2:100}.

    We have now defined a finite measure $\nu_{\fpart,\fpart'}$ for every $\fpart'\in \mathcal{P}$ with $\fpart' < \fpart$: if $\la_{\fpart'} = 0$, then $\nu_{\fpart,\fpart'}$ is defined in the ``$m=1$ case''. Otherwise, there exists a possible, non-trivial forest $G \in \mathbb{F}(\fpart')$ so that $\nu_{\fpart,\fpart'}$ is defined in the ``$m>1$ case'', in which we explained why the definition of $\nu_{\fpart,\fpart'}$ in that case does not depend on the choice of $G\in \mathbb{F}(\fpart')$.

    It remains to show that we can choose $c_\fpart > 0$ such that $\nut{\fpart} = \la_\fpart$. Since $\la_\fpart > 0$ there must be at least one non-trivial, possible forest $F\in \mathbb{F}(\fpart)$, in which case \cref{eqprf:U2:9} necessitates that the associated $\nu_{\fpart,\fpart_1^F}$ is non-zero, so $\nut{\fpart} > 0$. Thus there exists a unique choice for $c_\fpart > 0$ (which by definition is a linear factor in every non-zero $\nu_{\fpart,\fpart'}$) for which $\nut{\fpart} = \la_\fpart$.

    Finally, the construction of $\boldsymbol \nu$ satisfying \cref{eq:BSC} in this proof was unique. The sceptical reader can prove uniqueness of the transition measures $\boldsymbol \nu$ associated with a Brownian spatial coalescent directly with an induction of a similar structure to this proof.
\end{proof}

We used the following lemma in the proof.
\begin{lemma}\label{lem:prfunique:1}
    Suppose $\Omega_1,\Omega_2,\Omega_3$ are measurable spaces, and \[
        G(\diff s,\diff x) f(x, \zeta) \nu(\diff \zeta) = g(s,x) F(\diff x,\diff \zeta) \mu(\diff s),
        \] where $G$ and $F$ are measures on $\Omega_1\times \Omega_2$ and $\Omega_2\times \Omega_3$ respectively, $f\colon \Omega_2\times \Omega_3\to \R$ and $g\colon \Omega_1\times \Omega_2\to \R$ are strictly positive measurable functions, and $\nu$ and $\mu$ are measures on $\Omega_3$ and $\Omega_1$ respectively, such that there exist $A \subset \Omega_3$ and $B\subset \Omega_1$ with $\mu(A), \nu(B)\in (0,\infty)$. Then there is a measure $H$ on $\Omega_2$ such that \[
        \frac{G(\diff s,\diff x)}{g(s,x)} = H(\diff x) \mu(\diff s),\qquad \frac{F(\diff x,\diff \zeta)}{f(x,\zeta)} = H(\diff x) \nu(\diff \zeta),
    \] respectively as measures on $\Omega_1\times \Omega_2$ and $\Omega_2\times \Omega_3$. % If $F$ and $G$ are $\sigma$-finite then so is $H$. If $\Omega_1,\Omega_2,\Omega_3$ are topological spaces with the Borel $\sigma$-algebra, and $F,G$ are locally finite, then so is $H$.
\end{lemma}
\begin{proof}
    Rewrite the equality as an equality of measures
    \begin{equation}\label{eq:prfmslem:1}
        \frac{G(\diff s,\diff x)}{g(s,x)} \nu(\diff \zeta) = \frac{F(\diff x,\diff \zeta)}{f(x,\zeta)} \mu(\diff s).
    \end{equation}
    Let $A,B$ be sets of positive and finite measure w.r.t.\ $\mu$ and $\nu$ respectively (by assumption there exists at least one each), and define $H(\diff x)$ by integrating \cref{eq:prfmslem:1} over $s\in A,\zeta \in B$ and multiplying with $\frac{1}{\mu(A) \nu(B)}$, that is \[
        H(\diff x) \coloneqq \frac{1}{\mu(A)}\int_{s\in A} \frac{G(\diff s,\diff x)}{g(s,x)} = \frac{1}{\nu(B)} \int_{\zeta \in B} \frac{F(\diff x,\diff \zeta)}{f(x,\zeta)}.
    \] Note $H$ is independent of $B$ by the first and $A$ by the second equality. To prove the first equality of measures in the claim, take a set $A$ with $\mu(A) \in (0,\infty)$ and $C \subset \Omega_2$. Then,
    \begin{align*}
        \int_{x\in C} \int_{s\in A} H(\diff x) \mu(\diff s)
        = H(C) \mu(A)
        &= \mu(A) \int_C H(\diff x) = \int_{x\in C}\int_{s\in A} \frac{G(\diff s,\diff x)}{g(s,x)}.
    \end{align*}
    If $A$ is a $\mu$-null set, then by integrating \cref{eq:prfmslem:1} over it and some set of positive $\nu$-measure we find that $\int_A \frac{G(\diff s,\diff x)}{g(s,x)} = 0$, so equality holds in this case as well. Since $\mu$ is $\sigma$-finite, we find equality for all measurable sets $A$. The second equality follows similarly.
\end{proof}

Given a Brownian spatial coalescent, say with transition measures $\boldsymbol \nu\in \nurates$, recall that a forest $F\in \mathbb{F}$ is either possible, that is $\P^{\boldsymbol x}(\fr = F) > 0$ for all $\boldsymbol x\in E^{\lf(F)}_\circ $, or impossible, that is $\P^{\boldsymbol x}(\fr = F) = 0$ for all such $\boldsymbol x$. We call a merge event $(\fpart,\fpart')$ \emph{possible} if there exists a possible forest $F = \{\fpart,\fpart', \ldots \}\in \mathbb{F}$, that is if for every $\boldsymbol x\in E^{\fpart}_\circ $ there is a positive probability that the first merge event is $(\fpart,\fpart')$. We collect the following obvious but useful statement.

\begin{lemma}
    A forest $F\in \mathbb{F}$ is possible if and only if $\nu_{\fpart,\fpart'} \neq 0$ for all $(\fpart,\fpart') \in F$ and $\nut{\rt(F)} = 0$. A $(\fpart,\fpart')$-merger is possible if and only if $\nu_{\fpart,\fpart'} \neq 0$.
\end{lemma}

In particular, if $\nut{\fpart} > 0$ for every $\fpart\in \mathcal{P}$ with $|\fpart| \ge 2$, then the only possible forests are trees. This will be the case for sampling consistent Brownian spatial coalescents (except the trivial one), as we will see in \cref{sec:samplingcon}.

\subsection{Characterisation of Label Invariance}
In preparation for the proof of \cref{thm:samplingconintro}, this section is devoted to a characterisation of label invariance for Brownian spatial coalescents.

For $(n,\vec{k}) \in \mergers$ and $\fpart,\fpart'\in \mathcal{P}$ with $n = |\fpart|$ and $\fpart' < \fpart$, we call $(\fpart,\fpart')$ an  \emph{$(n,\vec{k})$-merger} if $\fpart'$ can be obtained from $\fpart$ by coalescing $m$ disjoint sets $I_1, \ldots ,I_m \subset \fpart$ of sizes $k_1$ to $k_m$ into one each. A non-spatial coalescent with transition rates $\boldsymbol \lambda \in \rates$ is label invariant if and only if there are numbers $\lank \ge 0$ for $(n,\vec{k}) \in \mergers$ such that $\lambda_{\fpart,\fpart'} = \lank$ whenever $(\fpart,\fpart')$ is an $(n,\vec{k})$ merger (recall \cref{def:NSCPintro} and the following paragraph). We show that a similar characterisation holds for Brownian spatial coalescents. For $m\in \N$ define \[
    E^m_\circ = \left\{ \boldsymbol x \in E^m\colon \left[ \{i\} \mapsto \boldsymbol x_i \right] \in E^{\{\{1\}, \ldots ,\{m\}\}}_\circ   \right\} .
\] In particular, $E^1_\circ  = E$.
If $(n,k_1, \ldots ,k_m) \in \mergers$, choose a bijection $\ell \colon [m] \to \fpart'\setminus \fpart$ such that $\ell(i)$ for $i \in [m]$ is the union of $k_i$ distinct elements of $\fpart$, which is unique up to a permutation of pairs $(i,j)$ with $k_i = k_j$. Define $\kappa_{\fpart,\fpart'}\colon E^{\fpart'\setminus \fpart}_\circ \to E^m_\circ $ by $\boldsymbol x \mapsto \boldsymbol x \circ \ell$ (here and in the following, we occasionally identify a vector $\boldsymbol x \in E^m$ with the map $[i \mapsto \boldsymbol x_i] \in E^{[m]}$). Recall the definition of $\labelinv(\fpart_0,\fpart_1)$ from the paragraph preceding \cref{def:labelinv}. We further extend a map $\iota \in \labelinv(\fpart_0,\fpart_1)$ to
\[
    \mathbb{F}(\fpart_0) \to \mathbb{F}(\fpart_1); \quad F = \{\fpart_0^F, \ldots ,\fpart_m^F\} \mapsto \{\iota(\fpart_0^F), \ldots ,\iota(\fpart_m^F)\},
\] and, for every $F \in \mathbb{F}(\fpart_0)$, to
%we further extend it to $\mathbb{F}(\fpart_0) \to \mathbb{F}(\fpart_1)$ for every $F\in \mathbb{F}(\fpart_0)$ to $\dcs(F)$ by $\xi \mapsto \xi \circ \iota^{-1}$, and to $\dct(F)$ by $\iota(\tau) = \tau \circ \iota^{-1}$.
\begin{align*}
    \dcs(F) \to \dcs(\iota(F));& \quad \xi \mapsto \xi \circ \iota^{-1},\\
    \dct(F) \to \dct(\iota(F));& \quad \tau \mapsto \tau \circ \iota^{-1}.
\end{align*}
Finally, this lets us define $\iota$ on $\dcb(\mathbb{F}(\fpart_0))$ by $(F,\tau,\xi) \mapsto (\iota(F),\iota(\tau),\iota(\xi))$.

%denote by $\nkmap_{\fpart,\fpart'}\colon E^{\fpart'\setminus \fpart}_\circ  \to E^m_\circ $ the map that sends $\boldsymbol x\in E^{\fpart'\setminus \fpart}_\circ $ to $\big[i \mapsto \boldsymbol x(I_i)\big]$. It is unique up to a permutation of pairs $(i,j)$ with $k_i = k_j$, and we fix an arbitrary choice for every $(\fpart,\fpart')$.

%\begin{lemma}\label{lem:labelinv}
\begin{restatable}{lemma}{lemlabelinv}\label{lem:labelinv}
    The Brownian spatial coalescent with transition measures $\boldsymbol \nu\in \rates$ is label invariant if and only if for all $\fpart_0,\fpart_1\in \mathcal{P}$ of equal size and $\iota \in \labelinv(\fpart_0,\fpart_1)$, \[
        \forall \fpart' < \fpart_0\colon  \iota \# \nu_{\fpart_0,\fpart'} = \nu_{\iota(\fpart_0),\iota(\fpart')}.
    \]  %and $\fpart' < \fpart_0$, $\iota\#\nu_{\fpart_0,\fpart'} = \nu_{\iota(\fpart_0),\iota(\fpart')}$.
    In that case, there exists for every $(n,\vec{k}) = (n,k_1, \ldots ,k_m) \in \mergers$ a finite measure $\nunk$ on $E^m_\circ $ such that for every $(n,\vec{k})$-merger $(\fpart,\fpart')$, \[
        \nkmap_{\fpart,\fpart'} \# \nu_{\fpart,\fpart'} = \nunk.
    \] If $k_i = k_j$, then $\nunk$ is symmetric in the $i$th and $j$th coordinate.
    %there is a family $\boldsymbol \nu = (\nu_{n,k_1, \ldots ,k_m} \in E^m_\circ \colon n\ge 2, k_1 \ge\ldots \ge k_m \ge 2, \sum_i k_i \le n)$ of finite measures such that, whenever $\fpart, \fpart'\in \mathcal{P}$ are such that $\fpart'$ can be obtained from $\fpart$ by merging $m$ disjoint sets of blocks of sizes $k_1$ to $k_m$ then $\nu_{\fpart,\fpart'} = \nu_{n,k_1, \ldots ,k_n}$.
%\end{lemma}
\end{restatable}

The proof is straightforward and in Appendix~\ref{sec:prooflabelinv}. \Cref{lem:labelinv} justifies that, for a label invariant Brownian spatial coalescent, we identify $\nu_{\fpart,\fpart'}$ for a $(n,\vec{k})$-merger $(\fpart,\fpart')$ with $\nunk$, and by the symmetry of $\nunk$ the choice of the underlying map $\kappa_{\fpart,\fpart'}$ is irrelevant.

\subsection{Characterisation of Sampling Consistency}\label{sec:samplingcon}

In this section we prove \cref{thm:samplingconintro}. Throughout this section, fix a label invariant Brownian spatial coalescent, say with transition measures $\boldsymbol \nu \in \nurates$. Recall in this context the measures $\nunk$, and write $\nut{n} \coloneqq \nut{\fpart}$ for any $\fpart\in \mathcal{P}$ with $|\fpart| = n$, and $\nut{1} \coloneqq 0$.
The following lemma is a core part of the proof. For two finite measures write $m_1\sim m_2$ if $m_1 = c m_2$ for some $c > 0$. %glossary
Also recall the notation introduced in \eqref{eq:Ex}.

\begin{lemma}\label{lem:nulambda}
    Let $\boldsymbol \nu(\diff \xi) = \boldsymbol \lambda \diff \xi$ for some $\boldsymbol \lambda \in \rates$. Then the following are equivalent
    \begin{enumerate}
        \item The non-spatial coalescent with transition rates $\boldsymbol \lambda$ is sampling consistent.
        \item  $\int_{E_{\boldsymbol x}} N^{\boldsymbol \nu}(\boldsymbol xy) \diff y < \infty$ for every $\boldsymbol x\in \mathcal{X}$, and the Brownian spatial coalescent with transition measures $\boldsymbol \nu$ is sampling consistent w.r.t.\ the probability measures $(\mu_{\boldsymbol x})$ satisfying  $\mu_{\boldsymbol x}(\diff y) \sim N^{\boldsymbol \nu}(\boldsymbol xy) \diff y$.
    \end{enumerate}
    In that case, $\int N^{\boldsymbol \nu}(\boldsymbol xy) \diff y = N^{\boldsymbol \nu}(\boldsymbol x)$.
    %\[
    %    \mu_{\boldsymbol x}(\diff y) = \frac{N^{\boldsymbol\nu}(\boldsymbol xy)}{N^{\boldsymbol\nu}(\boldsymbol x)}\diff y.
    %\]
\end{lemma}
%
%\begin{lemma}\label{lem:nulambda}
%    Let $\boldsymbol \nu(\diff \xi) = \boldsymbol \lambda \diff \xi$ for some $\boldsymbol \lambda \in \rates$, and $(\mu_{\boldsymbol x})$ be probability measures with $\mu_{\boldsymbol x}(\diff y) \sim N^{\boldsymbol\nu}(\boldsymbol xy) \diff y$. Then the Brownian spatial coalescent with transition measures $\boldsymbol \nu$ is sampling consistent w.r.t.\ $(\mu_{\boldsymbol x})$ if and only if the non-spatial coalescent with transition rates $\boldsymbol \lambda$ is sampling consistent. In that case, \[
%        \mu_{\boldsymbol x}(\diff y) = \frac{N^{\boldsymbol\nu}(\boldsymbol xy)}{N^{\boldsymbol\nu}(\boldsymbol x)}\diff y.
%    \]
%\end{lemma}

\Cref{lem:nulambda} is proved in \cref{sec:lemnulambda}, and implies the ``if'' direction of \cref{thm:samplingconintro}. For the ``only if'' direction it remains to show that sampling consistency of a Brownian spatial coalescent w.r.t.\ some family $(\mu_{\boldsymbol x})$ of probability measures implies $\boldsymbol \nu(\diff \xi) \sim \diff \xi$, and that $\mu_{\boldsymbol x}(\diff y) \sim N^{\boldsymbol\nu}(\boldsymbol xy) \diff y$. That is the content of \cref{sec:consequences}. In the remainder of this section, we set up the notation used in the proofs and make some technical observations related to sampling consistency.

Excluding from now the trivial case where $\nut{n} = 0$ for all $n \in \N$ (so $\nunk = 0$ for all $(n,\vec{k})$), sampling consistency necessitates $\nut{n} > 0$ for all $n \ge 2$, in particular that only trees are possible. This is part of \cref{lem:exmerge} below. We now present a characterisation of all trees $G$ from which a fixed tree $F$ can be ``subsampled'' by removing one leaf. This characterisation will be at the core of our main arguments.

\def\new{^{\scriptscriptstyle\oplus}}
\def\vnew{v_{\scriptscriptstyle\oplus}} %\encircle{$v$}
\def\unew{u_{\scriptscriptstyle\oplus}} %\encircle{$u$}}
\def\newsub{_{\scaleto{\oplus}{3.5pt}}} % for use in subscripts, for example \tau_{u\newsub}
\def\wnew{w_{\scriptscriptstyle\oplus}} %{\scriptstyle\pr(\vnew)}
\def\tnew{\tau_{\scriptscriptstyle\oplus}} %\tau_{\scriptscriptstyle\pr(\!\vnew\!)}
\def\xinew{\xi_{\scriptscriptstyle\oplus}}
\def\Fnew{G} %F\new
\def\Pnew{\fpart\new_{\Fnew}} % state of tree after merge with \vnew
\def\Pnewrs{\fpart\new_{F}} % same state projected on \fpart_0
\def\TTm{\mathbb{T}\new_\textrm{m}(F)} % trees extending F of multiple merge type
\def\TTb{\mathbb{T}\new_\textrm{b}(F)} % trees extending F of binary merge type
\def\TTbs{\mathbb{T}\new_\textrm{sb}(F)} % trees extending F of simultaneous binary merge type
\def\lanew{\la\new_{\Fnew}}
\newcommand{\kbrp}[1]{\kbr^{\textrm{#1}\scriptscriptstyle\oplus}}

For every $\fpart_0 \in \mathcal{P}$ we fix some $\vnew \subset \N$ disjoint from all members of $\fpart_0$ and put $\fpart_0\new \coloneqq \fpart_0 \cup \left\{ \vnew \right\} \in \mathcal{P}$. For $F\in \mathbb{T}(\fpart_0)$ and $\Fnew\in \mathbb{T}(\fpart_0\new)$ we say that $\Fnew$ \emph{extends} $F$ if $F = \Fnew\rsto{\fpart_0} \coloneqq \left\{ \fpart\wo \vnew\colon \fpart\in \Fnew \right\} $ %, that is $F$ is obtained from $\Fnew$ by ``deleting'' $\vnew$
, and write $\mathbb{T}\new(F)$ for the set of such $\Fnew$. In that case, if $(\tau,\xi) \in \dcb(\Fnew)$ then $\tau\rsto{F} \in \dct(F)$ and $\xi\rsto{F} \in \dcs(F)$ denote the decorations induced on $F$, respectively defined by $\fpart \wo \vnew \mapsto \min\{\tau(\fpart')\colon \fpart'\wo\vnew = \fpart\wo\vnew\}$ and $u \setminus \vnew \mapsto \xi_u$; write $\tnew \coloneqq \tau_{\pr_\Fnew(\vnew)}$ and $\xinew\coloneqq \xi_{\pr_{\Fnew}(\vnew)}$ for time and location of the first merge of $\vnew$, see \cref{fig:tree_induction}.

If $G^\star = (G,\tau,\xi)\in \dcb(\mathbb{T})$ and $G\rsto{\fpart_0} = F$ then we denote $G^\star \rsto{\fpart_0} = (G\rsto{\fpart_0},\tau\rsto{F},\xi\rsto{F})$. For $F^\star \in \dcb(\mathbb{T})$ write $\dcb\new(F^\star) = \{G^\star\colon G^\star\rsto{\fpart_0} = F^\star\}$, and for fixed $G$ extending $F$, $\dcb\new(F^\star \,\vert\,G) = \{(G,\tau,\xi) \in \dcb\new(F^\star)\}$.
Write $\Pnew = \max \{\fpart\in \Fnew\colon \vnew\not\in \fpart\} $ for the state of the tree immediately after the first merge involving $\vnew$, and $\Pnewrs \coloneqq \Pnew\wo \vnew$ for the same state projected on $\fpart_0$.
If $\Fnew$ extends $F$, it has to fall within exactly one of the following three classes (see also \cref{fig:tree_induction}).
\begin{figure}
    \centering
    \def\yscale{1.0}
\def\tick{.05}
\def\xroom{3.15}
\def\nodecol{black}
\begin{tikzpicture}[scale=1.2] %[baseline=(current bounding box.north),scale=1.2]
    \coordinate (zero) at (-.5,-.5);
    \coordinate (zero2) at (-.5+\xroom,-.5);
    \coordinate (zero3) at (-.5+2*\xroom,-.5);
    \coordinate (t-l) at (-.5,3.6*\yscale);
    \coordinate (t-l2) at (-.5+\xroom,3.6*\yscale);
    \coordinate (t-l3) at (-.5+2*\xroom,3.6*\yscale);
    \coordinate (b-r) at (9,-.5);

    \draw[->] (zero) -- (t-l);
    \draw[->] (zero2) -- (t-l2);
    \draw[->] (zero3) -- (t-l3);
    \draw[->] (zero) -- (b-r);
    \draw (t-l) node[above] {\footnotesize time};
    \draw (b-r) node[below left] {\footnotesize space};

    \begin{scope}[shift={(0,0)}]
        \coordinate (A) at (0,0);
\coordinate (Ba) at (1.25,0);
\coordinate (Bb) at (2,0);
\coordinate (C) at (0.75,0);

\coordinate (B) at ({(1.25+2)/2},1.5*\yscale);
\coordinate (AB) at ({(1.25+2)/4},3*\yscale);

%% tree connections
\draw[black] (A) to[out=90,in=-90-30] (AB);
\draw[black] (Ba) to[out=90,in=-90-35] (B);
\draw[black] (Bb) to[out=90,in=-90+35] (B);
\draw[black,name path = BtoAB] (B) to[out=90+15,in=-90+40] (AB);
\draw[black] (AB) -- (AB|-0,3.3*\yscale);

\def\ymerge{2.2}
\path[name path = horizontal] (A|-0,\ymerge*\yscale) -- (B|-0,\ymerge*\yscale);
\draw[orange,densely dotted,thick,name intersections={of=BtoAB and horizontal}] (C) to[out=90,in=-90-30] (intersection-1);

%% space and time labels
\draw[dotted,gray,name intersections={of=BtoAB and horizontal}] (-.5,\ymerge*\yscale) -- (intersection-1);
\draw (-.5+\tick,\ymerge*\yscale) -- (-.5-\tick,\ymerge*\yscale) node[left] {\footnotesize$\tnew\!$};

%\draw[dotted,gray] (-.5,0|-B) -- (B);
%\draw (-.5+\tick,0|-B) -- (-.5-\tick,0|-B) node[left] {\footnotesize$\tau(\Pnewrs)$};

\draw[dotted,gray,name intersections={of=BtoAB and horizontal}] (intersection-1|-0,-.5) -- (intersection-1);
\draw[name intersections={of=BtoAB and horizontal}] (intersection-1|-0,-.5+\tick) -- (intersection-1|-0,-.5-\tick) node[below] {\footnotesize$\xinew$};

%\draw[dotted,gray] (C) to (C|-0,-.5);
%\draw (C|-0,-.5+\tick) to (C|-0,-.5-\tick) node[below] {\footnotesize $y$};

%% node labels
\draw[orange] (C) node[below left] {\footnotesize $\vnew$};
\draw (B) node[above right] {\footnotesize $\unew$};
\draw (AB) node[above right] {\footnotesize $\wnew$};

%% node circles
\draw[fill,\nodecol] (A) circle[radius=1.5pt];
\draw[fill,\nodecol] (B) circle[radius=1.5pt];
\draw[fill,\nodecol] (AB) circle[radius=1.5pt];
\draw[fill,\nodecol] (Ba) circle[radius=1.5pt];
\draw[fill,\nodecol] (Bb) circle[radius=1.5pt];
\draw[fill,orange] (C) circle[radius=1.5pt];
\draw[fill,orange,name intersections={of=BtoAB and horizontal}] (intersection-1) circle[radius=1.5pt];
    \end{scope}
    \begin{scope}[shift={(\xroom,0)}]
        \coordinate (A) at (0,0);
\coordinate (Ba) at (1.25,0);
\coordinate (Bb) at (2,0);
\coordinate (C) at (0.75,0);

\coordinate (B) at ({(1.25+2)/2},1.5*\yscale);
\coordinate (AB) at ({(1.25+2)/4},3*\yscale);

%% tree connections
\draw[black,name path = AtoAB] (A) to[out=90,in=-90-30] (AB);
\draw[black] (Ba) to[out=90,in=-90-35] (B);
\draw[black] (Bb) to[out=90,in=-90+35] (B);
\draw[black] (B) to[out=90+15,in=-90+40] (AB);
\draw[black] (AB) -- (AB|-0,3.3*\yscale);

\path[name path = horizontal] (A|-0,1.5*\yscale) -- (B|-0,1.5*\yscale);
\draw[orange,densely dotted,thick,name intersections={of=AtoAB and horizontal}] (C) to[out=90,in=-90+45] (intersection-1);

%% space and time labels

%\draw[dotted,gray] (C) to (C|-0,-.5);
%\draw (C|-0,-.5+\tick) to (C|-0,-.5-\tick) node[below] {\footnotesize $y$};
\draw[dotted,gray] (-.5,1.5*\yscale) -- (B);
\draw (-.5+\tick,1.5*\yscale) -- (-.5-\tick,1.5*\yscale) node[left] {\footnotesize $\tnew\!\!$};
\draw[dotted,gray,name intersections={of=AtoAB and horizontal}] (intersection-1) to (intersection-1|-0,-.5);
\draw[name intersections={of=AtoAB and horizontal}] (intersection-1|-0,-.5+\tick) to (intersection-1|-0,-.5-\tick) node[below] {\footnotesize $\xinew$};

%% node labels
\draw[orange] (C) node[below right] {\footnotesize $\vnew$};
\draw (A) node[below left] {\footnotesize $\unew$};
%\draw[gray,name intersections={of=AtoAB and horizontal}] (intersection-1) node[above right] {\footnotesize $\wnew$};
\draw (AB) node[above right] {\footnotesize $\wnew$};

%% node circles
\draw[fill,\nodecol] (A) circle[radius=1.5pt];
\draw[fill,\nodecol] (B) circle[radius=1.5pt];
\draw[fill,\nodecol] (AB) circle[radius=1.5pt];
\draw[fill,\nodecol] (Ba) circle[radius=1.5pt];
\draw[fill,\nodecol] (Bb) circle[radius=1.5pt];
\draw[fill,orange] (C) circle[radius=1.5pt];
\draw[fill,orange,name intersections={of=AtoAB and horizontal}] (intersection-1) circle[radius=1.5pt];
    \end{scope}
    \begin{scope}[shift={(2*\xroom,0)}]
        \coordinate (A) at (0,0);
\coordinate (Ba) at (1.25,0);
\coordinate (Bb) at (2,0);
\coordinate (C) at (0.75,0);

\coordinate (B) at ({(1.25+2)/2},1.5*\yscale);
\coordinate (AB) at ({(1.25+2)/4},3*\yscale);

%% tree connections
\draw[black] (A) to[out=90,in=-90-30] (AB);
\draw[black] (Ba) to[out=90,in=-90-35] (B);
\draw[black] (Bb) to[out=90,in=-90+35] (B);
\draw[black] (B) to[out=90+15,in=-90+40] (AB);
\draw[black] (AB) -- (AB|-0,3.3*\yscale);

\def\ymerge{2.2}
\draw[orange,densely dotted,thick] (C) to[out=90,in=-90-70] (B);

%% space and time labels

%\draw[dotted,gray] (C) to (C|-0,-.5);
%\draw (C|-0,-.5+\tick) to (C|-0,-.5-\tick) node[below] {\footnotesize $y$};
\draw[dotted,gray] (-.5,1.5*\yscale) -- (B);
\draw (-.5+\tick,1.5*\yscale) -- (-.5-\tick,1.5*\yscale) node[left] {\footnotesize $\tnew\!\!$};
\draw[dotted,gray] (B) to (B|-0,-.5);
\draw (B|-0,-.5+\tick) to (B|-0,-.5-\tick) node[below] {\footnotesize $\xinew$};

%% node labels
\draw[orange] (C) node[below right] {\footnotesize $\vnew$};
%\draw[gray] (B) node[above right] {\footnotesize $\wnew$};

%% node circles
\draw[fill,\nodecol] (A) circle[radius=1.5pt];
\draw[fill,\nodecol] (B) circle[radius=1.5pt];
\draw[fill,\nodecol] (AB) circle[radius=1.5pt];
\draw[fill,\nodecol] (Ba) circle[radius=1.5pt];
\draw[fill,\nodecol] (Bb) circle[radius=1.5pt];

\draw[fill,orange] (C) circle[radius=1.5pt];
    \end{scope}
\end{tikzpicture}
    \caption{There are three ways to extend a fixed tree to an additional leaf. Either the leaf $\vnew$ merges binary with another node $\unew$, at a time at which no merge happens in the original tree (left), or simultaneously with an existing merge event (middle); or $\vnew$ joins an existing merge event (right). The additional degrees of freedom in the tree decoration given that of the underlying tree are, from right to left, none, the location $\xinew$, and the time $\tnew$ and location $\xinew$ of the additional merge event.}
    \label{fig:tree_induction}
\end{figure}

\begin{enumerate}
    \item \emph{\textbf{Multiple merge}}: $\vnew$ is part of a multiple merger (\Cref{fig:tree_induction} right). In this case we can identify $\dcb(\Fnew)$ and $\dcb(F)$, so $\dcb\new(F^\star \,\vert\,G)$ is a singleton.
    \item \emph{\textbf{Binary merge}}: $\vnew$ merges binary with a node $\unew\in \nd(\Fnew)$, and no other merge takes place simultaneously (\Cref{fig:tree_induction} left). In this case, for $F^\star = (F,\tau_0,\xi_0) \in \dcb(\mathbb{F})$ there is a bijection
        \begin{IEEEeqnarray}{rCl}
            \dcb\new(F^\star \,\vert\,G) & \quad \longleftrightarrow \quad & (\tau_0(\Pnewrs),\tau_0(\Pnewrs)^+) \times E,\label{eq:bijb}\\
            (\tau,\xi) & \mapsto & (\tnew,\xinew),\nonumber
        \end{IEEEeqnarray}
        whose inverse we denote $(s,z) \mapsto (s\new\tau_0,z\new\xi_0)$.
        %between $\dcs(\Fnew)$ and $\dcs(F)\times E$ given by $\xi \mapsto (\xi\rsto{F},\xinew)$, the inverse of which we denote by $(\xi,z) \mapsto z\xi$; and between $\dct(\Fnew)$ and the set of $(\tau,s)\in \dct(F)\times (0,\infty)$ with $\tau(\Pnewrs) < s < \tau(\Pnewrs)^+$, given by $\tau \mapsto (\tau\rsto{F},\tnew)$, the inverse of which we denote by $(\tau,s)\mapsto s\tau$.
        In this case $\wnew \coloneqq \pr_{\Fnew}(\unew\cup\vnew) \neq \emptyset $ except if the merge event involving $\vnew$ is the final one, that is $|\Pnew| = 1$.
    \item \emph{\textbf{Simultaneous binary merge}}: $\vnew$ merges binary with a node $\unew\in \nd(\Fnew)$ as part of a simultaneous merge event (\Cref{fig:tree_induction} middle). Then $\wnew \coloneqq \pr_{\Fnew}(\unew\cup\vnew) \neq \emptyset $, and we can identify $\dct(\Fnew)$ and $\dct(F)$, and for $F^\star = (F,\tau_0,\xi_0) \in \dcb(\mathbb{F})$ there is a bijection
        \begin{equation}\label{eq:bijbs}
            \dcb\new(F^\star \,\vert\,G) \longleftrightarrow E_{\xi_0\vert_{\Pnewrs}};\quad
            (\tau,\xi)  \mapsto \xinew,
        \end{equation}
        %\begin{IEEEeqnarray*}{rCl}
        %    \dcb\new(F^\star \,\vert\,G) & \quad \longleftrightarrow \quad & E_{\xi\vert_{\Pnewrs}},\\
        %    (\tau,\xi) & \mapsto & \xinew,
        %\end{IEEEeqnarray*}
        (recall \cref{eq:Ex})
        the inverse of which we denote by $z \mapsto (\tau_0,z\new\xi_0)$.
        %between $\dcs(\Fnew)$ and the set of $(\xi,z) \in \dcs(F)\times E$ with $z \in E_{\xi\vert_{\Pnewrs}}$ given by $\xi \mapsto (\xi\rsto{F},\xinew)$, the inverse of which we denote by $(\xi,z) \mapsto z\xi$. %the set of $(\xi,z)\in \dcs(F)\times E$ with $z\xi\vert_{\Pnewrs}\in E^{\Pnew}_\circ $ given by $\xi \mapsto (\xi\rsto{F},\xinew)$, the inverse of which we denote by $(\xi,z) \mapsto z\xi$.
\end{enumerate}
Denote the sets of such trees by $\TTm$, $\TTb$, and $\TTbs$, respectively. We do not explicitly denote dependence of $\unew$ and $\wnew $ etc.\ on $\Fnew$, but it will always be clear from context.

This allows us to formulate sampling consistency on a more precise technical level. If all transition measures are absolutely continuous (which we will prove is a consequence of sampling consistency in \cref{lem:gnkdz}), then with a slight abuse of notation, we write $\nunk(\diff \boldsymbol z) = \gnk(\boldsymbol z)\diff \boldsymbol z$ and $\boldsymbol \nu_F(\diff \xi) = \boldsymbol \nu_F(\xi)\diff \xi$ etc., as well as $P^{\boldsymbol x}(F^\star) \coloneqq  P^{\boldsymbol x}(F,\tau,\xi) \coloneqq \frac{1}{N^{\boldsymbol \nu}(\boldsymbol x)} \fnu(\tau,\xi \,\vert\,\boldsymbol x) \boldsymbol \nu_F(\xi)$, so that $P^{\boldsymbol x}(F,\diff \tau,\diff \xi) = P^{\boldsymbol x}(F,\tau,\xi) \diff \tau\diff \xi$.

\begin{lemma}\label{lem:consPxKx}
    Let $\boldsymbol \nu \in \nurates$ be a non-trivial family of transition measures that are absolutely continuous w.r.t.\ Lebesgue measure. % and such that $\nut{n} > 0$ for all $n\in \N$.
    Then the associated coalescent process is sampling consistent w.r.t.\ a family of probability measures $(\mu_{\boldsymbol x})_{\boldsymbol x\in \mathcal{X}}$ if and only if $\nut{n} > 0$ for all $n\ge 2$, and for every $\fpart_0\in \mathcal{P}$, $F\in \mathbb{T}(\fpart_0)$, $\boldsymbol x\in E^{\fpart_0}_\circ $, and Lebesgue-a.e.\ $F^\star = (F,\tau,\xi) \in \dcb(F)$,
    \begin{align}
        K_{\boldsymbol x}(F^\star,\cdot )
        &P^{\boldsymbol x}(F,\tau,\xi)
        =  K_{\boldsymbol x}(F^\star,\cdot )\smashoperator[l]{\sum_{G\in \TTm}} \int P^{\boldsymbol x y}(G,\tau,\xi) \mu_{\boldsymbol x}(\diff y)\nonumber\\
        & + \smashoperator[l]{\sum_{G\in \TTb}} \iiint K_{\boldsymbol xy}((G,s\new\tau,z\new\xi),\{\rs{\boldsymbol X}\in\cdot \}) P^{\boldsymbol x y}(G,s\new\tau,z\new\xi) \diff s\diff z \mu_{\boldsymbol x}(\diff y) \label{eq:consPxKx}\\
        & + \smashoperator[l]{\sum_{G\in \TTbs}} \iint K_{\boldsymbol xy}((G,\tau,z\new\xi),\{\rs{\boldsymbol X}\in\cdot \}) P^{\boldsymbol x y}(G,\tau,z\new\xi) \diff z\mu_{\boldsymbol x}(\diff y)\nonumber,
    \end{align}
    where $(\rs{\boldsymbol X}_t) \coloneqq (\boldsymbol X_t \wo \vnew)$, and domains of integration for $s$ and $z$ are as indicated by the bijections \textup{(\ref{eq:bijb})} and \cref{eq:bijbs}.
\end{lemma}
\begin{proof}
    The fact that sampling consistency implies $\nut{n} > 0$ for all $n\ge 2$ is in \cref{lem:exmerge} below.
    \Cref{eq:consPxKx} is a direct reformulation of \cref{eq:samplingcon} using \cref{eq:defBSC}, the characterisation of $\mathbb{T}\new(F)$ for $F\in \mathbb{T}$, and in order to get an equation pointwise for a.e.\ $F^\star$, the fact that $K_{\boldsymbol x}(F^\star, \{\dc \in A\} \cap (\cdot )) = \ind_{\{F^\star\in A\}} K_{\boldsymbol x}(F^\star,\cdot )$ for all $F^\star \in \dcb(\mathbb{T})$ and measurable $A \subset \dcb(F)$. That it suffices to consider trees is because $\nut{n} > 0$ for $n\ge 2$ makes all forests that are not trees impossible.
\end{proof}

Before moving on to the proof of \cref{lem:nulambda}, we give an explicit description of the law of a non-spatial coalescent when regarded as a random, time decorated forest, which will make it easier to connect with the laws $(P^{\boldsymbol x})$ of a Brownian spatial coalescent. It follows from a simple argument about competing exponential clocks.
\begin{lemma}\label{lem:ftm}
    Let $\boldsymbol \lambda \in \rates$ be the transition rates of a non-spatial coalescent $(\P^{\fpart})_{\fpart\in \mathcal{P}}$, and denote the total jump rate while at $\fpart$ by $\lambda_\fpart\ge 0$. We define $\ftm\colon \dct(\mathbb{F}) \to (0,\infty)$ by $\ftm\big\vert_{\dct(F)}\equiv 0$ if $\lambda_{\rt(F)} > 0$, otherwise
    \begin{equation*}%\label{eq:ftm}
        \ftm(F,\tau) \coloneqq \smashop{\prod_{(\fpart,\fpart') \in F}} \lambda_{\fpart,\fpart'} \e^{-\lambda_\fpart(\tau_{\fpart'}-\tau_\fpart)},\qquad \tau \in \dct(F),
    \end{equation*}
    which is $1$ if $F$ is trivial. Then for $\fpart\in \mathcal{P}$ and $F\in \mathbb{F}(\fpart)$,
    \begin{align*}
        (\tm\#\P^{\fpart})(F,\diff \tau) = \ftm(F,\tau)\diff \tau.
    \end{align*}
\end{lemma}

\subsubsection{Proof of \cref{lem:nulambda}}\label{sec:lemnulambda}

In this section we fix label invariant rates $\boldsymbol \lambda\in \rates$ and define $\boldsymbol \nu \in \nurates$ by $\nunk(\diff \boldsymbol z) \coloneqq \lambda_{n,\smash{\vec{k}}} \diff \boldsymbol z$. Then the laws $(P^{\boldsymbol x})$ that characterise the Brownian spatial coalescent process with transition measures $\boldsymbol \nu$ (recall \cref{eq:BSC}) take the simpler form
\begin{align}\label{eq:Pxlambda}
    P^{\boldsymbol x}(F,\tau,\xi) = \frac{1}{N^{\boldsymbol\nu}(\boldsymbol x)} \ftm(F,\tau) \fsp^F(\xi \,\vert\,\tau,\boldsymbol x),
\end{align}
where in this chapter we will explicitly denote dependence of $\ftm$ and $\fsp$ on $F$. We further remark that a characterisation similar to \cref{lem:consPxKx} holds for non-spatial coalescent processes, which are sampling consistent with transition rates $\boldsymbol \lambda$ if and only if $\lambda_n > 0$ for all $n\ge 2$ and
\begin{equation}\label{eq:consftm}
    \ftm(F,\tau) = \smashoperator[l]{\sum_{G\in \TTb}} \int \ftm(G,s\new\tau) \diff s + \smashoperator{\sum_{G\in \TTbs}} \ftm(G,\tau) + \smashoperator{\sum_{G\in \TTm}} \ftm(G,\tau)
\end{equation}
for every $\fpart_0\in \mathcal{P}$, $F\in \mathbb{T}(\fpart_0)$, and Lebesgue-a.e.\ $\tau \in \dct(F)$, where the RHS is simply a density of $\P^{\fpart_0\new}(\{(G,\tau')\colon G\rsto{\fpart_0} = F, \tau'\rsto{F} \in \cdot \}) $ on $\dct(F)$.

\begin{lemma}\label{lem:nulambda1}
    If the non-spatial coalescent with transition rates $\boldsymbol \lambda$ is sampling consistent, then $\int N^{\boldsymbol\nu}(\boldsymbol xy) \diff y = N^{\boldsymbol\nu}(\boldsymbol x)$ for all $\boldsymbol x\in \mathcal{X}$, and the Brownian spatial coalescent with transition measures $\boldsymbol \nu$ is sampling consistent w.r.t.\ the probability measures defined by $\mu_{\boldsymbol x} = \frac{N^{\boldsymbol \nu}(\boldsymbol xy)}{N^{\boldsymbol \nu}(\boldsymbol x)}\diff y $.
\end{lemma}
%As a preparation, we prove the following.
%\begin{lemma}
%    If $(\fpart_0,\boldsymbol x) \in \mathcal{P}$, $(F,\tau,\xi) \in \dcb(\mathbb{T}(\fpart_0))$, and $G\in \TTb$, then
%    \begin{multline}\label{eqprf:nulambda1:2}
%        \iint \fsp(z\new\xi \,\vert\,s\new\tau,\boldsymbol xy) B^{(\tau_{u\newsub},(\boldsymbol x\xi)_{u\newsub})\to (s,z) \to (\tau_{w\newsub},\xi_{w\newsub})}(\cdot ) \diff z\diff y \\
%        = \fsp(\xi \,\vert\,\tau,\boldsymbol x) B^{(\tau_{u\newsub},(\boldsymbol x\xi)_{u\newsub}) \to (\tau_{w\newsub},\xi_{w\newsub})}(\cdot ).
%    \end{multline}
%    In particular, $\int \fsp(z\new\xi \,\vert\,s\new\tau,\boldsymbol xy) \diff z\diff y = \fsp(\xi \,\vert\,\tau,\boldsymbol x)$.
%\end{lemma}
%\begin{proof}
%    This follows directly from \[
%        \fsp(z\new\xi \,\vert\,s\new\tau,\boldsymbol xy) = \fsp(\xi \,\vert\,\tau,\boldsymbol x) \frac{p_s(y-z) p_{s-\tau_{u\newsub}}(z-(\boldsymbol x\xi)_{u\newsub}) p_{\tau_{w\newsub}-s}(\xi_{w\newsub}-z)}{p_{\tau_{w\newsub}-\tau_{u\newsub}}(\xi_{w\newsub}-(\boldsymbol x\xi)_{u\newsub})}
%    \] and \cref{lem:bb}.
%\end{proof}
A central ingredient in this and many other proofs in this and the following section is the observation stated in the following lemma. It essentially says that, under some assumptions, if $(T,Z)$ is a $(0,\infty)\times E$ valued random variable such that the law of a Brownian bridge from $(0,0)$ to $(T,Z)$, followed by a Brownian motion starting at $(T,Z)$, is the same as that of a Brownian motion starting in $(0,0)$, then $Z \sim \mathcal{N}(0,T)$ conditional on $T$. It is proved in Appendix~\ref{sec:lembbproof}.
\begin{restatable}{lemma}{lembb}\label{lem:bb}
%\begin{lemma}\label{lem:bb}
    Suppose $x_0\in E$, $s_0 > 0$, $f\colon E \times (s_0,\infty) \to (0,\infty)$ is continuous, and $\mu$ is a finite measure on $E$, and that $\int_{s_0}^\infty \int_E f(x,s) \mu(\diff x) \diff s < \infty$. Then \[
        B^{(x_0,s_0)+}(\cdot ) \sim \iint B^{(x_0,s_0)\to(x,s)+}(\cdot ) f(x,s) \mu(\diff x) \diff s
    \] if and only if $f(x,s) \mu(\diff x) \sim p_{s-s_0}(x-x_0) \diff x$ for all $s > s_0$. If further $x_1\in E$, $s_1 > s_0$, and $f\colon E\times (s_0,s_1) \to (0,\infty)$ is continuous, then \[
    B^{(x_0,s_0)\to (x_1,s_1)}(\cdot ) \sim \iint B^{(x_0,s_0) \to (x,s)\to(x_1,s_1)}(\cdot ) f(x,s) \mu(\diff x) \diff s
\] if and only if $f(x,s) \mu(\diff x) \sim p_{s-s_0}(x-x_0) p_{s_1-s}(x_1-x) \diff x$ for all $s\in (s_0,s_1)$.
\end{restatable}

\begin{proof}[Proof of \cref{lem:nulambda1}]
    In that case $\nut{n} = \lambda_n > 0$ for $n\ge 2$, so the assumptions of \cref{lem:consPxKx} are satisfied and we have to confirm \cref{eq:consPxKx}. Note that passing to the total mass in \cref{eq:consPxKx} and then integrating over $(\tau,\xi) \in \dcb(F)$ and summing over $F\in \mathbb{F}(\fpart_0)$ gives $1 = \int \mu_{\boldsymbol x}(\diff y)$ (it equates to taking the total mass in \cref{eq:samplingcon}). In particular, if we prove \cref{eq:consPxKx} for any family of measures, in this case $\mu_{\boldsymbol x}(\diff y) = \frac{N^{\boldsymbol\nu}(\boldsymbol xy)}{N^{\boldsymbol\nu}(\boldsymbol x)}\diff y$, then they must already be probability measures (in particular finite), that is $\int N^{\boldsymbol\nu}(\boldsymbol xy) \diff y = N^{\boldsymbol\nu}(\boldsymbol x)$. Let $\fpart_0\in \mathcal{P}$, $F^\star = (F,\tau,\xi)\in \dcb(\mathbb{T}(\fpart_0))$ and $\boldsymbol x\in E^\fpart_\circ $. The laws on both sides of \cref{eq:consPxKx} are determined by their pushforward under the collection of maps, i.e.\ random variables $(\pth^{F^\star}_u)_{u\in \nd(F)}$. Under both laws this family of random variables is independent, so it suffices to show that \cref{eq:consPxKx} holds after applying a single pushforward $\pth_u\coloneqq \pth^{F^\star}_u$ for a fixed but arbitrary $u\in \nd(F)$. Assume $u$ is not the root, so $w\coloneqq \pr_F(u) \neq \emptyset $, otherwise the proof is similar. Then $\pth_u\#K_{\boldsymbol x}(F^\star,\cdot ) = B^{(\tau_u,(\boldsymbol x\xi)_u) \to (\tau_w,\xi_w)}$. If $G^\star = (G,s\new\tau,z\new\xi) \in \TTb$ with $\unew = u$ then
    \begin{align*}
        \pth_u \#K_{\boldsymbol xy}(G^\star,\{\rs{\boldsymbol X}\in \cdot \}) = B^{(\tau_u,(\boldsymbol x\xi)_u)\to(s,z)\to(\tau_w,\xi_w)},
    \end{align*}
    and
    \begin{align}
        &\iiint B^{(\tau_u,(\boldsymbol x\xi)_u)\to(s,z)\to(\tau_w,\xi_w)} P^{\boldsymbol xy}(G,s\new\tau,z\new\xi) \diff s\diff z\mu_{\boldsymbol x}(\diff y)\nonumber\\
               &= \frac{1}{N^{\boldsymbol\nu}(\boldsymbol x)} \int \ftm(G,s\new\tau) \left( \iint B^{(\tau_u,(\boldsymbol x\xi)_u)\to(s,z)\to(\tau_w,\xi_w)}(\cdot ) \fsp^G(z\new\xi \,\vert\,s\new\tau,\boldsymbol xy)\diff z\diff y\right) \diff s\label{eqprf:nulambda1:1}\\
               &= \frac{1}{N^{\boldsymbol\nu}(\boldsymbol x)} \left(\int \ftm(G,s\new\tau) \diff s\right) B^{(\tau_u,(\boldsymbol x\xi)_u) \to (\tau_w,\xi_w)}(\cdot ) \fsp^F(\xi \,\vert\,\tau,\boldsymbol x) ,\nonumber
    \end{align}
    where the second equality followed directly from \cref{lem:bb} because
    \begin{equation}\label{eqprf:nulambda1:3}
        \fsp^G(z\new\xi \,\vert\,s\new\tau,\boldsymbol xy) = \fsp^F(\xi \,\vert\,\tau,\boldsymbol x) \frac{p_s(y-z) p_{s-\tau_{u}}(z-(\boldsymbol x\xi)_{u}) p_{\tau_{w}-s}(\xi_{w}-z)}{p_{\tau_{w}-\tau_{u}}(\xi_{w}-(\boldsymbol x\xi)_{u})}.
    \end{equation}
    If $\unew \neq u$ then the pushforward of $K_{\boldsymbol xy}(G^\star,\{\rs{\boldsymbol X}\in \cdot \})$ under $\pth_u$ is the same as that of $K_{\boldsymbol x}(F^\star,\cdot )$, and then
    \begin{align*}
        \iiint P^{\boldsymbol xy}(G,s\new\tau,z\new\xi)\diff s\diff z\mu_{\boldsymbol x}(\diff y)
        &= \frac{1}{N^{\boldsymbol\nu}(\boldsymbol x)} \fsp^F(\xi \,\vert\,\tau,\boldsymbol x) \int \ftm(G,s\new\tau) \diff s
    \end{align*}
    follows from \cref{eqprf:nulambda1:1}. Similar calculations can be made for $G\in \TTbs$, and if $G\in \TTm$ then $\fsp^G(\xi \,\vert\,\tau,\boldsymbol xy) = \fsp^F(\xi \,\vert\,\tau,\boldsymbol x) p_{\tnew}(\xinew-y)$ so \[
        \int P^{\boldsymbol xy}(G,\tau,\xi) \mu_{\boldsymbol x}(\diff y) = \frac{1}{N^{\boldsymbol\nu}(\boldsymbol x)} \ftm(G,\tau) \fsp^F(\xi \,\vert\,\tau,\boldsymbol x).
    \] We can thus cancel $\frac{1}{N^{\boldsymbol\nu}(\boldsymbol x)} \fsp^F(\xi \,\vert\,\tau,\boldsymbol x) B^{(\tau_u,(\boldsymbol x\xi)_u \to (\tau_w,\xi_w)}$ to conclude that the pushforward of \cref{eq:consPxKx} under any $u\in \nd(F)$ is equivalent exactly to \cref{eq:consftm} and thus implied by sampling consistency of the non-spatial coalescent with transition rates $\boldsymbol \lambda$.
\end{proof}

\begin{lemma}\label{lem:nulambda2}
    If $\int N^{\boldsymbol \nu}(\boldsymbol xy) \diff y < \infty$ for all $\boldsymbol x\in \mathcal{X}$, and the Brownian spatial coalescent with transition measures $\boldsymbol \nu(\diff \xi) = \boldsymbol \lambda \diff \xi$ is sampling consistent w.r.t.\ the probability measures $(\mu_{\boldsymbol x})$ defined by $\mu_{\boldsymbol x}(\diff y) \sim N^{\boldsymbol \nu}(\boldsymbol xy) \diff y$, then the non-spatial coalescent with transition rates $\boldsymbol \lambda$ is sampling consistent and $\int N^{\boldsymbol\nu}(\boldsymbol xy) \diff y = N^{\boldsymbol\nu}(\boldsymbol x)$.
\end{lemma}
\begin{proof}
    The assumptions of \cref{lem:consPxKx} are satisfied, so \cref{eq:consPxKx} holds with \[
        \mu_{\boldsymbol x}(\diff y) = \frac{N^{\boldsymbol\nu}(\boldsymbol xy)}{\int N^{\boldsymbol\nu}(\boldsymbol xy')\diff y'}\diff y,\quad \boldsymbol x\in \mathcal{X}.
    \] By the same calculations made in the proof of \cref{lem:nulambda1}, this implies
    \begin{multline}\label{eqprf:nulambda2:1}
        \frac{1}{N^{\boldsymbol\nu}(\boldsymbol x)} \ftm(F,\tau) \\[-10pt]
        = \frac{1}{\int N^{\boldsymbol\nu}(\boldsymbol xy) \diff y} \left(\sum_{G\in \TTb} \int \ftm(G,s\new\tau) \diff s + \smashoperator{\sum_{G\in \TTbs}} \ftm(G,\tau) + \smashoperator{\sum_{G\in \TTm}} \ftm(G,\tau)\right)
    \end{multline}
    for every $\fpart_0\in \mathcal{P}$, $F\in \mathbb{T}(\fpart_0)$, and $\tau\in \dct(F)$. Integrating out $\tau\in \dct(F)$ gives $N^{\boldsymbol\nu}(\boldsymbol x) = \int N^{\boldsymbol\nu}(\boldsymbol xy) \diff y$, and then \cref{eqprf:nulambda2:1} turns into \cref{eq:consftm} which implies sampling consistency of the non-spatial coalescent with transition rates $\boldsymbol \lambda$.
\end{proof}

\Cref{lem:nulambda1,lem:nulambda2} together imply \cref{lem:nulambda}.

%\subsubsection{Sampling Consistency implies Assumption of \cref{lem:nulambda}}
\subsubsection{Consequences of Sampling Consistency}\label{sec:consequences}

In this section we fix some label invariant Brownian spatial coalescent, say with transition measures $\boldsymbol \nu \in \nurates$, and assume it is sampling consistent w.r.t.\ some family of probability measures $(\mu_{\boldsymbol x})_{\boldsymbol x\in \mathcal{X}}$. The goal is to prove $\boldsymbol \nu(\diff \xi) \sim \diff \xi$, and that $\mu_{\boldsymbol x}(\diff y) \sim N^{\boldsymbol\nu}(\boldsymbol xy) \diff y$ for every $\boldsymbol x\in \mathcal{X}$.

We write $(n,k_1, \ldots ,k_m) \le (n',k_1', \ldots ,k_{m'}')$ if there is a way to sample $n$ particles out of the $n'$ particles in an $(n',\vec{k}')$-merger restricted to which we observe an $(n,\vec{k})$-merger. That is, if $n \le n'$ and $m \le m'$ and there exists an assignment of $n$ particles to buckets of sizes $k_0',k_1', \ldots ,k_{m'}$ (where $k_0'=n' - \sum k_i'$) such that the numbers of particles in buckets $1$ to $m'$ have sizes $k_1$ to $k_m$ (ignoring empty buckets, and without maintaining order).

\begin{lemma}\label{lem:nkle}
    \begin{enumerate}
        \item If $(n,\vec{k}) \le (n',\vec{k}')$ and $(n',\vec{k}')$ is possible, then so is $(n,\vec{k})$.
        \item If $(n,\vec{k})$ is possible and $n' > n$, then there exists $\vec{k}'$ such that $(n',\vec{k}') \ge (n,\vec{k})$ and $(n',\vec{k}')$ is possible.
    \end{enumerate}
\end{lemma}
\begin{proof}
    Suppose $(n,\vec{k}) \le (n',\vec{k}')$ and the latter is possible. Let $\fpart, \fpart'\in \mathcal{P}$ with $|\fpart| = n$, $|\fpart'| = n'$ and $\fpart \subset \fpart'$. There exists a possible forest $F'\in \mathbb{F}(\fpart')$ that starts with an $(n',\vec{k}')$-merger such that the forest $F \coloneqq F'\rsto{\fpart}$ induced on $\fpart$ starts with an $(n,\vec{k})$-merger. It suffices to show that $F$ is possible. By sampling consistency, for some fixed $\boldsymbol x\in E^\fpart_\circ $ there is a probability measure $\mu$ on $E^{\fpart'\setminus \fpart}_\circ $, namely, for an arbitrary ordering $\fpart' \setminus \fpart = \left\{ u_1, \ldots ,u_l \right\} $, \[
        \mu(\diff \boldsymbol z) = \mu(\diff z_{u_1}, \ldots ,\diff z_{u_l}) = \mu_{\boldsymbol x}(\diff z_{u_1}) \mu_{\boldsymbol x z_{u_1}}(\diff z_{u_2}) \ldots \mu_{\boldsymbol x z_{u_1} \ldots z_{u_{l-1}}}(\diff z_{u_l}),
    \]  such that
    \begin{align*}
        \P^{\boldsymbol x}(\fr = F)
        &= \int \P^{\boldsymbol x\boldsymbol z}(\fr \vert_{\fpart}=F) \mu(\diff \boldsymbol z)
        \ge \int \P^{\boldsymbol x\boldsymbol z}(\fr = F') \mu(\diff \boldsymbol z) >0,
    \end{align*}
    so $F$ is possible. Now suppose $(n,\vec{k})$ is possible and $n' > n$. Let $\fpart,\fpart'\in \mathcal{P}$ with $|\fpart| = n$, $|\fpart'| = n'$ and $\fpart\subset \fpart'$, and let $F\in \mathbb{F}(\fpart)$ such that $F$ is possible and it starts with an $(n,\vec{k})$-merger. Then
    \begin{align*}
        0 < \P^{\boldsymbol x}(\fr = F)
        &= \int \P^{\boldsymbol x\boldsymbol z}(\fr\vert_{\fpart}=F) \mu(\diff \boldsymbol z) \le \sum_{\substack{F'\in \mathbb{F}(\fpart') \\ F'\vert_{\fpart} = F} } \int \P^{\boldsymbol x\boldsymbol z}(\fr = F') \mu(\diff \boldsymbol z).
    \end{align*}
    Thus there exists a possible forest $F'\in \mathbb{F}(\fpart')$ with $F'\vert_{\fpart} = F$, whose first merge event satisfies $(n',\vec{k}') \ge (n,\vec{k})$.
\end{proof}

\begin{lemma}\label{lem:exmerge}
    Exactly one of the following hold.
    \begin{enumerate}
        \item All mergers are impossible.
        \item All of the mergers $(n,n)$ for $n \ge 2$ are possible and all other mergers are impossible.
        \item For every $n\ge 3$ there exists a possible $(n,\vec{k})$-merger with $\sum_i k_i < n$, and $(2,2)$ is possible.
    \end{enumerate}
    In cases (ii) and (iii), $\nut{n} > 0$ for all $n\ge 2$, in particular the only possible forests are trees.
\end{lemma}
The extremal cases (i) and (ii) correspond in the non-spatial setting to the $\Lambda$-coalescent with $\Lambda$ the zero-measure, and a Dirac mass at $1$, respectively. \Cref{lem:exmerge} states that if neither (i) nor (ii) hold, then for every initial condition with at least three particles there is a positive probability that the first merge event leaves at least one particle untouched.

\begin{proof}[Proof of \cref{lem:exmerge}]
    Suppose that every merger that isn't of the form $(n,n)$ for $n \ge 2$ is impossible. Then either (i) holds, or there exists an $n\ge 2$ such that $(n,n)$ is possible. Then $(k,k) \le (n,n)$ is possible for all $2 \le k < n$. If $n' > n$, then $(n',\vec{k}')$ is possible for some $\vec{k}'$, but then it must be of the form $(n',n')$, so in fact $(n,n)$ is possible for all $n \ge 2$ and all other merge events are impossible, so (ii) holds. In particular $\nut{n} > 0$ for all $n\ge 2$.

    Now assume that there exists $n_0 \ge 3$ such that a merge event $(n_0,\vec{k}_0) \neq (n_0,n_0)$ is possible, and we want to show (iii) holds. Then $(2,2) \le (n_0,\vec{k}_0)$ is possible, and for any $n > n_0$, by \cref{lem:nkle}(ii) there exists a possible merge event $(n,\vec{k}) \ge (n_0,\vec{k}_0)$, which cannot be $(n,n)$ because otherwise $(n_0,\vec{k}_0) = (n_0,n_0)$. Thus if $n\ge 3$ arbitrary, we can take $N \coloneqq  (3n)\vee n_0$ and a possible merge event $(N,k_1, \ldots ,k_m) \neq (N,N)$. If $m = 1$, then $k_1 < N$ so $(n,k_1 \wedge (n-1)) \le (N,\vec{k})$ is possible by \cref{lem:nkle}(i). If $m \ge 2$, then we construct a merger $(n,\vec{k}') \le (N,\vec{k})$ with $\sum_i k_i' < n$ by assigning one particle to the smallest bucket of size $k_m$, and spread all other particles over the remaining buckets, which works as long as $n-1 \le N-k_m$, but if this wasn't true then $N \ge m  k_m \ge 2(N-n+1)$ and thus $n \ge N / 2$, a contradiction.
\end{proof}

We assume from now that we are in case (iii). In case (i) there is nothing to show, and in case (ii) proofs are easier than and can be directly adapted from those that follow. We will from now explicitly denote dependence of $\fnu$ on $F$ by writing $\fnu^F$. The following lemma serves as a start to an inductive argument.
%Denote the law of a Brownian motion in $E$ started at some $x$ at time $t \ge 0$ by $B^{(t,x)+}$, the law of a Brownian bridge started at time $s \ge 0$ at $x\in E$, ending at time $t > s$ at $y\in E$ by $B^{(s,x) \to (t,y)}$, the law of the same bridge followed by a Brownian motion started at time $t$ at $y$ by $B^{(s,x)\to(t,y)+}$ etc.

%For a random variable $Y$ and a finite measure $\mu$ on the domain of $Y$ we write $Y \sim \mu(\diff y)$ if $\mu / |\mu|$ is the law of $Y$, for example a standard normal $Y$ satisfies $Y \sim \e^{-y^2 / 2}\diff y$.

\begin{lemma}\label{lem:nu1}
    There is $\lambda_{2,2} > 0$ such that $\nu_{2,2}(\diff z) = \lambda_{2,2}\diff z$, and for every $x\in E$, \[
        \mu_x(\diff y) \sim N^{\boldsymbol \nu}(xy) \diff y.
    \]
\end{lemma}

%marker
\begin{proof}[Proof of \cref{lem:nu1}]
    Let $\fpart_0 = \{u,v\}\in \mathcal{P}$ and $x \in E^{\{u\}} $. Abbreviate $\mu(\diff y) \coloneqq N^{\boldsymbol \nu}(xy)^{-1} \mu_x(\diff y)$ and $\nu\coloneqq \nu_{2,2}$, so that $\mu(\diff y),\nu(\diff y) \sim \diff y$ are to be shown. The law of $\boldsymbol X_t(u)$ under $\P^{x}$ is $B^{(0,x)+}$. If $y\in E^{\left\{ v \right\} }$, then started from $xy$ the only possible tree is $F = \left\{ \fpart_0,\left\{ uv \right\}  \right\} $, and conditional on $\dc = (\tau,\xi) \in \dcb(F)$, the law of $\rs{\boldsymbol X}_t(u)$ (recall $\rs{\boldsymbol X}$ from \cref{lem:consPxKx}) under $\P^{xy}$ is $B^{(0,x)\to(\tau,\xi)+}$. Thus by sampling consistency,
    \begin{align*}
        B^{(0,x)+}(\cdot )
        &= \int_E \int_{\dcb(F)} B^{(0,x)\to (\tau,\xi)+}(\cdot ) P^{xy}(F,\diff \tau, \diff \xi)\mu_x(\diff y), \nonumber\\
        &= \int\limits_E \int\limits_E \int\limits_0^\infty B^{(0,x)\to (\tau,\xi)+}(\cdot ) \fnu^F(\tau,\xi \,\vert\,xy) \diff \tau \nu(\diff \xi) \mu(\diff y),\label{eqprf:nu1:1}
    \end{align*}
    which by \cref{lem:bb} and definition of $\fnu^F(\cdot \,\vert\,xy)$ implies \[
        \left( \int_E p_s(x-z) p_s(y-z) \mu(\diff y) \right) \nu(\diff z) \sim p_s(x-z) \diff z
    \] %where we are identifying $E^{\{v\}}$ and $E^{\{u\}}$ with $E$.
    for Lebesgue-a.a.\ $s > 0$. That is, for such $s > 0$ there is a constant $c(s) > 0$ such that $\left( \int p_s(y-z) \mu(\diff y) \right) \nu(\diff z) = c(s) \diff z$. Then $\nu$ must have a positive Lebesgue-density which we can write as $1 / g$ for some measurable $g\colon E\to (0,\infty)$, and then $T_s \mu = c(s) g$ Lebesgue-a.e.\ for Lebesgue-a.a.\ $s>0$, where $(T_t)_{t\ge 0}$ denotes the heat semigroup. By \cref{lem:Tsmu} below there is a constant $c > 0$ such that $g \equiv c$ a.e.\ and $\mu(\diff z) = c \diff z$.
\end{proof}

\begin{lemma}\label{lem:Tsmu}
    Suppose that $\mu$ is a finite, non-zero measure on $E$, $g \colon E\to (0,\infty)$ is measurable, $N \subset (0,\infty)$ is a Lebesgue-null set and for $s \in (0,\infty)\setminus N$ there is $c(s) > 0$ such that $T_s \mu = c(s) g$ Lebesgue-a.e. Then there is a constant $c > 0$ such that $g(z) = c$ a.e.\ and $\mu(\diff z) = c \diff z$.
\end{lemma}
\begin{proof}
    All equalities of functions $E\to \R$ in this proof are meant Lebesgue-a.e. First note that $g$ is bounded because $T_s\mu$ is bounded for any $s > 0$. Let $s_0 > 0$, put $f\coloneqq T_{s_0} \mu = c(s_0) g$, and $\widetilde{c}(s) \coloneqq \frac{c(s+s_0)}{c(s_0)}$ whenever $s \in S\coloneqq  (0,\infty)\setminus (N-s_0) $. Then $T_s f = c(s+s_0) g = \widetilde{c}(s) f$ for all $s \in S$, which by the semigroup property implies $\widetilde{c}(t+s) = \widetilde{c}(s) \widetilde{c}(t)$ whenever $s,t,s+t\in S$. This implies that there exists $\alpha \in \R$ with $\widetilde{c}(s) = \e^{\alpha s}$ for all $s$ in some dense subset $A$ of $S$. Indeed, the set $N' \coloneqq \bigcup_{r\in \Q_+} \frac{N-s_0}{r}$ is null and its complement $S'\subset S$ has the property that for all $s\in S'$ and $r\in \Q_+$ also $rs \in S'$. Fix some $s_1\in S'$, then $\widetilde{c}(\frac{p}{q}s_1)^q = \widetilde{c}(ps_1) = \widetilde{c}(s_1)^p$ for all $p,q\in \N$, that is $\widetilde{c}(rs_1) = \widetilde{c}(s_1)^r$ for all $r\in \Q_+$, which implies the claim with $\alpha = \log c(s_1)/s_1$ and $A = s_1\Q_+$. We have proved that $T_s g = \e^{\alpha s}g$ for all $s\in A$. %The LHS is pointwise continuous in $s$ by dominated convergence because $g$ is bounded so denseness of $A$ implies that $T_s g = \e^{\alpha s}g$ for all $s > 0$.
    The LHS tends to the constant $c\coloneqq \int g(x) \diff x$ uniformly on $E$ as $s \to \infty$, which necessitates that $\alpha = 0$ and that $g\equiv c$ a.e. Then $(T_s \mu)(\diff z) = c \diff z$ as measures for all $s\in (0,\infty)\setminus N$, and the LHS tends to $\mu$ weakly as $s \to \infty$, so $\mu(\diff z) = c \diff z$.
\end{proof}

% This is where the \new notation used to be introduced

\begin{lemma}\label{lem:gnkdz}
    %If $\fpart\in \mathcal{P}$, $\boldsymbol x\in E^{\fpart}_\circ $ and $y\in E_{\boldsymbol x}$, then
    %\begin{equation}\label{eq:samplingcons_Px}
    %    P^{\boldsymbol x}(\cdot ) = \int P^{\boldsymbol xy}(\dc\rsto{\fpart} \in \cdot ) \mu_{\boldsymbol x}(\diff y).
    %\end{equation}
    %Furthermore,
    All measures $\nunk$ are absolutely continuous w.r.t.\ Lebesgue measure.
    %For every $(n,\vec{k})$ there exists a measurable non-negative function $\gnk \colon E^m_\circ \to [0,\infty)$ such that $\nunk(\diff \boldsymbol z) = \gnk(\boldsymbol z)\diff \boldsymbol z$.
\end{lemma}
\begin{proof}
    We prove the claim with an induction over $n \ge 2$, the start of which is due to \cref{lem:nu1}. Let $n \ge 3$, assume that the claim is proved for $(n',\vec{k}')$ for all $n' < n$, and take $(n,k_1, \ldots ,k_m)$ possible. Let $\fpart_0\in \mathcal{P}$ with $|\fpart_0| = n -1$, and $\fpart_0\new$ etc.\ as above. If $k_1 \ge 3$ then $(n-1,\vec{k}') \coloneqq (n-1,k_1-1,k_2, \ldots ,k_m) \le (n,\vec{k})$ is possible and there are possible trees $F \in \mathbb{T}(\fpart_0)$, $G\in \mathbb{T}(\fpart\new)$ with $G\rsto{\fpart_0}=F$ whose first merge events are respectively $(n-1,\vec{k}')$ and $(n,\vec{k})$. If $\xi \in \dcs(F)$ write $\xi^{(0)} \coloneqq \xi\vert_{\fpart ^{1}_F \setminus \lf(F)}$ for the location(s) of the first merge event, and similarly for $\xi\in \dcs(G)$. Then, for fixed $\boldsymbol x\in E^{\fpart_0}_\circ $ and $y\in E_{\boldsymbol x}$, $P^{\boldsymbol x}(F,\xi^{(0)}\in \cdot )$ and $P^{\boldsymbol x y}(G,\xi^{(0)}\in \cdot )$ have positive densities w.r.t.\ $\nu_{n-1,\smash{\vec{k}'}}$ and $\nunk$, respectively. Indeed,
    \begin{align*}
        P^{\boldsymbol x}(F,\xi^{(0)}\in \diff \boldsymbol z) = \frac{1}{N^{\boldsymbol\nu}(\boldsymbol x)} \bigg(\int\limits_{\dct(F)} \smashoperator[r]{\int\limits_{\dcs(F\setminus \fpart_0)}} \fnu^F(\tau,\xi' \boldsymbol z \,\vert\,\boldsymbol x) \diff \tau\,\boldsymbol \nu_{F\setminus \fpart_0}(\diff \xi')\bigg) \nu_{n-1,\smash{\vec{k}'}}(\diff \boldsymbol z),
    \end{align*}
    and similarly for $P^{\boldsymbol xy}(G,\xi^{(0)}\in \diff\boldsymbol z)$. Since $P^{\boldsymbol x}(F,\xi^{(0)} \in \cdot ) \ge \int P^{\boldsymbol x y}(G,\xi^{(0)}\in \cdot  ) \mu_{\boldsymbol x}(\diff y)$ by sampling consistency, this implies that $\nunk$ is absolutely continuous w.r.t.\ $\nu_{n-1,\smash{\vec{k}'}}$ and thus has a Lebesgue density. If $k_1 = \ldots = k_m = 2$, then $(n-1,\vec{k}') \coloneqq (n-1,k_1, \ldots ,k_{m-1}) \le (n,\vec{k})$ is possible, and by an analogous argument we obtain that $\nunk$ is absolutely continuous w.r.t.\ Lebesgue measure in the first $m-1$ arguments, and thus by symmetry (\cref{lem:labelinv}) in all arguments.
\end{proof}

We have now proved the assumptions of \cref{lem:consPxKx}, of which we use the following corollary.
Recall from the paragraph preceeding \cref{lem:consPxKx} that we will now slightly abuse notation by writing $\nunk(\diff \boldsymbol z) = \gnk(\boldsymbol z)\diff \boldsymbol z$ and $\boldsymbol \nu_F(\diff \xi) = \boldsymbol \nu_F(\xi)\diff \xi$ etc.

\begin{lemma}\label{lem:samplingconKx}
    Let $\fpart_0\in \mathcal{P}$, $F\in \mathbb{T}(\fpart_0)$, and $\boldsymbol x\in E^{\fpart_0}_\circ $, then
    \begin{align}
        K_{\boldsymbol x}(F^\star,\cdot ) \sim &\sum_{G\in \TTb} \iiint K_{\boldsymbol x y}((G,s\new\tau,z\new\xi),\{\rs{\boldsymbol X}\in\cdot\} ) P^{\boldsymbol xy}((G,s\new\tau,z\new\xi)) \diff s\diff z\mu_{\boldsymbol x}(\diff y)\nonumber\\
                                               &+ \sum_{G\in \TTbs} \iint K_{\boldsymbol x y}((G,\tau,z\new\xi),\{\rs{\boldsymbol X}\in\cdot\} ) P^{\boldsymbol xy}((G,\tau,z\new\xi)) \diff z\mu_{\boldsymbol x}(\diff y)\label{eq:samplingconKx}
    \end{align}
    for Lebesgue-a.e.\ $F^\star = (F,\tau,\xi) \in \dcb(F)$, where domains of integration for $s$ and $z$ are as indicated by the bijections \textup{(\ref{eq:bijb})} and \cref{eq:bijbs}. For fixed $F^\star$, if $G \in \TTb$ and $y\in E_{\boldsymbol x}$,
    \begin{equation}\label{eq:Pxyb}
        P^{\boldsymbol x y}((G,s\new\tau,z\new\xi)) \sim \frac{\e^{-\lanew s}}{N^{\boldsymbol\nu}(\boldsymbol xy)} p_s(z-y) p_{s-\tau_{u\newsub}}(z-(\boldsymbol x\xi)_{u\newsub}) \left[ p_{\tau_{w\newsub} -s}(\xi_{w\newsub}-z) \right] _{\wnew\neq \emptyset} \boldsymbol \nu_G(z\new\xi)
    \end{equation}
    where $\lanew = \nut{|\Pnew|} - \nut{|\Pnew|+1}$, and if $G\in \TTbs$,
    \begin{equation}\label{eq:Pxybs}
        P^{\boldsymbol x y}((G,\tau,z\new\xi)) \sim \frac{1}{N^{\boldsymbol\nu}(\boldsymbol xy)}p_{\tnew}(z-y) p_{\tnew-\tau_{u\newsub}}(z-(\boldsymbol x\xi)_{u\newsub}) p_{\tau_{w\newsub}-\tnew}(\xi_{w\newsub}-z) \boldsymbol \nu_G(z\new\xi).
    \end{equation}
    The constants in both cases only depend on $F^\star$ and $\boldsymbol x$.
\end{lemma}
\begin{proof}
    The first claim follows directly from \cref{eq:consPxKx} (note that the first term on the RHS of \eqref{eq:consPxKx} is itself $\sim K_{\boldsymbol x}(F^\star,\cdot )$ and can therefore be absorbed into the LHS of \eqref{eq:samplingconKx}), and
    %Let $A\subset \dcb(F)$ measurable, then $K_{\boldsymbol x}(F^\star, \,\cdot \,\cap \{\dc\in A\}) = \ind_{\{F^\star\in A\}} K_{\boldsymbol x}(F^\star,\cdot )$, so by sampling consistency,
    %\begin{multline*}
    %    \smashoperator{\int_{\dcb(F)}} \ind_{\{F^\star \in A\}} K_{\boldsymbol x}(F^\star,\cdot ) P^{\boldsymbol x}(\diff F^\star)\\[-20pt]
    %    = \smashoperator[l]{\sum_{G\in \mathbb{T}\new(F)}}\int\mu_{\boldsymbol x}(\diff y)\smashoperator{\int_{\dcb(G)}} \ind_{\{G^\star\!\rsto{\fpart_0} \in A\}}K_{\boldsymbol x y}(\Fnew^\star,\{\rs{\boldsymbol X} \in \cdot \}) P^{\boldsymbol xy} (\diff \Fnew^\star).
    %    %&= \iint \ind_{(F,\tau,\xi) \in A} K_{\boldsymbol xy}(G^\star,\cdot ) P^{\boldsymbol x y}((G,s\tau,z\xi)) \diff s\diff z \diff \tau\diff \xi \mu_{\boldsymbol x}(\diff y) + \ldots
    %\end{multline*}
    %On the left we write $P^{\boldsymbol x}(\diff F^\star) =  P^{\boldsymbol x}((F,\tau,\xi)) \diff \tau\diff \xi$ and $ \ind_{\{F^\star\in A\}} = \ind_{\{(F,\tau,\xi) \in A\}}$. On the right, if $G\in \TTb$ then $G^\star \in \dcb(G)$ is of the form $(G,s\new\tau,z\new\xi)$, see \cref{eq:bijb}, and we can write $P^{\boldsymbol xy}(\diff G^\star) = P^{\boldsymbol xy}((G,s\new\tau,z\new\xi)) \diff s\diff z\diff \tau\diff \xi$, and $ \ind_{\{G^\star\!\rsto{\fpart_0} \in A\}} = \ind_{\{(F,\tau,\xi) \in A\}}$. Similarly for $G\in \TTbs$, and if $G\in \TTm$ then $K_{\boldsymbol xy}(G^\star,\{\rs{\boldsymbol X}\in\cdot \}) = K_{\boldsymbol x}(F^\star,\cdot )$ for any $y\in E_{\boldsymbol x}$.
    \cref{eq:Pxyb,eq:Pxybs} follow from the definition of $\fnu$ by discarding terms that don't depend on $s$, $z$, $y$, and $z$, $y$, respectively.
\end{proof}

\begin{lemma}\label{lem:mux}
    For every $\boldsymbol x\in \mathcal{X}$, $\mu_{\boldsymbol x}(\diff y) \sim N^{\boldsymbol\nu}(\boldsymbol x y) \diff y$.
\end{lemma}
\begin{proof}
    Let $n\ge 3$, $\fpart_0\in \mathcal{P}$ with $|\fpart_0| = n$, and $\boldsymbol x\in E^{\fpart_0}_\circ $. By \cref{lem:exmerge} there exists a possible tree $F = \{\fpart_0, \ldots ,\fpart_m = \left\{ \unew \right\} \}$ such that $G \coloneqq \{\fpart_0\cup\left\{ \vnew \right\} , \ldots ,\fpart_{m-1} \cup \left\{ \vnew \right\}, \left\{ \unew,\vnew \right\} , \left\{ \unew \cup \vnew \right\} \}$ is possible. Note $G\rsto{\fpart_0} = F$. In words, in $G$ the first $n-1$ particles merge according to $F$ and then the root of $F$ merges with the $n$th particle $\vnew$. Fix $F^\star=(F,\tau,\xi)\in \dcb(F)$ for which \cref{eq:samplingconKx} holds, and denote by $s_0\coloneqq \tau(\unew)$, $z_0\coloneqq \xi(\unew)$ time and place of the birth of $F$'s root. Then the law of $\boldsymbol Y_{\bullet} \coloneqq \boldsymbol X_{s_0+\bullet}(\unew)$, the motion of $F$'s root, under $K_{\boldsymbol x}(F^\star,\cdot )$ is $B^{(s_0,z_0)+}$. For any $G'^\star \in \dcb\new(F^\star)$ with $G'\neq G$ the law of $\boldsymbol Y$ under $K_{\boldsymbol x}(G'^\star,\{\rs{\boldsymbol X}\in\cdot \})$ is the same, because $\vnew$ merges into $F$ before $s_0$. If $G^\star = (G,s\new\tau,z\new\xi) \in \dcb\new(F^\star \,\vert\,G)$, then the law of $\boldsymbol Y$ under $K_{\boldsymbol x}(G^\star,\{\rs{\boldsymbol X}\in\cdot \})$ is $B^{(s_0,z_0)\to(s,z)+}$, so by \cref{lem:samplingconKx},
    \begin{align*}
        B^{(s_0,z_0)+}(\cdot )
        &\sim \int_{E_{\boldsymbol x}} \int_E \int_{s_0}^\infty B^{(s_0,z_0)\to(s,z)+}(\cdot ) P^{\boldsymbol x y}((G,s\new\tau,z\new\xi)) \diff s\diff z\mu_{\boldsymbol x}(\diff y)\\
        &\sim \iiint B^{(s_0,z_0)\to(s,z)+}(\cdot ) p_s(z-y) p_{s-s_0}(z-z_0) \e^{-\nut{2} s} \diff s\diff z \mu(\diff y),
    \end{align*}
    where we put $\mu(\diff y) \coloneqq \frac{1}{N^{\boldsymbol\nu}(\boldsymbol x y)}\mu_{\boldsymbol x}(\diff y)$ and used that $\nu_{2,2}(\diff z) \sim \diff z$ is already known (that is, $\nu_{2,2}(z)$ is constant in $z$). By \cref{lem:bb}, this implies that, for Lebesgue-a.a.\ $s \in (s_0,\infty)$,
    \begin{align*}
        \left(\int_{E_{\boldsymbol x}} p_s(z-y) p_{s-s_0}(z-z_0) \mu(\diff y) \right) \diff z\sim p_{s-s_0}(z-z_0) \diff z,
    \end{align*}
    that is $T_s \mu$ is a.e.\ constant for a.a.\ $s>0$, and thus $\mu(\diff z) \sim \diff z$ by \cref{lem:Tsmu}.
\end{proof}

\begin{figure}
    \centering
    \def\yscale{1.2}
\def\tick{.05}
\def\xroom{3.8}
\def\nodecol{black}
\def\noderadius{1.2pt}
\begin{tikzpicture}[scale=1.5] %[baseline=(current bounding box.north),scale=1.2]
    \coordinate (zero) at (-.5,-.5);
    \coordinate (zero2) at (-.5+\xroom,-.5);
    \coordinate (t-l) at (-.5,2.4*\yscale);
    \coordinate (t-l2) at (-.5+\xroom,2.4*\yscale);
    \coordinate (b-r) at (-.5+2*\xroom+.5,-.5);

    \draw[->] (zero) -- (t-l);
    \draw[->] (zero2) -- (t-l2);
    \draw[->] (zero) -- (b-r);
    \draw (t-l) node[above] {\footnotesize time};
    \draw (b-r) node[below left] {\footnotesize space};

    \begin{scope}[shift={(0,0)}]
        \coordinate (bl) at (0,-.3);
\coordinate (tl) at  (0,2*\yscale);
\coordinate (br) at (2.7,-.3);
\coordinate (tr) at (2.7,2*\yscale);

\draw[black, thin] (bl)--(tl);
\draw[black, thin] (bl)--(br);
\draw[black, thin] (tr)--(tl);
\draw[black, thin] (tr)--(br);
\coordinate (Aa) at (.15,0);
\coordinate (Ab) at (.5,0);
\coordinate (Ac) at (.75,0);
\coordinate (A) at (.5,1.5*\yscale);
\coordinate (A2) at (.4,2*\yscale);

\coordinate (Ba) at (1.25,0);
\coordinate (Bb) at (2,0);
\coordinate (B) at ({(1+1.75)/2},1.5*\yscale);
\coordinate (B2) at (B|-0,2*\yscale);

\coordinate (C) at (2.5,0);

%% tree connections
\draw[black] (Aa) to[out=90,in=-90-20] (A);
\draw[black] (Ab) to[out=90,in=-90] (A);
\draw[black] (Ac) to[out=90,in=-90+15] (A);
\draw[black] (A) to[out=90,in=-90+20] (A2);
\draw[black] (Ba) to[out=90,in=-90-10] (B);
\draw[black] (B) to[out=90,in=-90] (B2);
\draw[gray, very thick, name path = PATH] (Bb) to[out=90+5,in=-90+30] (B);

\node[gray] (Y) at (1.8,1.3*\yscale) {\footnotesize$\boldsymbol Y_\bullet$};

\def\ymerge{1.0}
\path[name path = horizontal] (0,\ymerge*\yscale) -- (C|-0,\ymerge*\yscale);
\draw[orange,densely dotted,thick,name intersections={of=PATH and horizontal}] (C) to[out=90+20,in=-90+60] (intersection-1);

%% space and time labels
\draw[dotted,gray,name intersections={of=PATH and horizontal}] (-.5,\ymerge*\yscale) -- (intersection-1);
\draw (-.5+\tick,\ymerge*\yscale) -- (-.5-\tick,\ymerge*\yscale) node[left] {\footnotesize$s$};

\draw[dotted,gray] (-.5,1.5*\yscale) -- (B);
\draw (-.5+\tick,1.5*\yscale) -- (-.5-\tick,1.5*\yscale) node[left] {\footnotesize$s_1$};

%\draw[dotted,gray] (-.5,0|-B) -- (B);
%\draw (-.5+\tick,0|-B) -- (-.5-\tick,0|-B) node[left] {\footnotesize$\tau(\Pnewrs)$};

\draw[dotted,gray,name intersections={of=PATH and horizontal}] (intersection-1|-0,-.5) -- (intersection-1);
\draw[name intersections={of=PATH and horizontal}] (intersection-1|-0,-.5+\tick) -- (intersection-1|-0,-.5-\tick) node[below] {\footnotesize$z$};

\draw[dotted,gray] (B) -- (B|-0,-.5+\tick);
\draw (B|-0,-.5+\tick) -- (B|-0,-.5-\tick) node[below] {\footnotesize$z_1$};

%\draw[dotted,gray] (C) to (C|-0,-.5);
%\draw (C|-0,-.5+\tick) to (C|-0,-.5-\tick) node[below] {\footnotesize $y$};

%% node labels
\draw[orange] (C) node[below] {\footnotesize $\vnew$};
\draw (Bb) node[below] {\footnotesize $\unew$};
%\draw[gray,name intersections={of=PATH and horizontal}] (intersection-1) node[above right] {\footnotesize $\wnew$};

%% node circles
\draw[fill,\nodecol] (A) circle[radius=\noderadius];
\draw[fill,\nodecol] (Aa) circle[radius=\noderadius];
\draw[fill,\nodecol] (Ab) circle[radius=\noderadius];
\draw[fill,\nodecol] (Ac) circle[radius=\noderadius];
\draw[fill,\nodecol] (B) circle[radius=\noderadius];
\draw[fill,\nodecol] (Ba) circle[radius=\noderadius];
\draw[fill,\nodecol] (Bb) circle[radius=\noderadius];
\draw[fill,orange] (C) circle[radius=\noderadius];
\draw[fill,orange,name intersections={of=PATH and horizontal}] (intersection-1) circle[radius=\noderadius];
    \end{scope}
    \begin{scope}[shift={(\xroom,0)}]
        \coordinate (bl) at (0,-.3);
\coordinate (tl) at  (0,2*\yscale);
\coordinate (br) at (3.2,-.3);
\coordinate (tr) at (3.2,2*\yscale);

% bounding box
\draw[black, thin] (bl)--(tl);
\draw[black, thin] (bl)--(br);
\draw[black, thin] (tr)--(tl);
\draw[black, thin] (tr)--(br);

% coordinate definitions
\coordinate (Aa) at (.15,0);
\coordinate (Ab) at (.35,0);
\coordinate (Ac) at (.75,0);
\coordinate (A) at (.5,1.0*\yscale);
\coordinate (A2) at (.7,2*\yscale);

\coordinate (Ba) at (1.25,0);
\coordinate (Bb) at (1.55,0);
\coordinate (B) at (1.4,1.0*\yscale);
\coordinate (B2) at (1.2,2*\yscale);

\coordinate (Ca) at (2.0,0);
\coordinate (Cb) at (2.35,0);
\coordinate (Cb2) at (2.6,1.0*\yscale);
\coordinate (C) at (2.15,1.7*\yscale);
\coordinate (C2) at (2.1,2*\yscale);

\coordinate (D) at (3,0);

%% tree connections
\draw[black] (Aa) to[out=90-5,in=-90-25] (A);
\draw[black] (Ab) to[out=90,in=-90-10] (A);
\draw[black] (Ac) to[out=90,in=-90+20] (A);
\draw[black] (A) to[out=90,in=-90-20] (A2);

\draw[black] (Ba) to[out=90,in=-90-15] (B);
\draw[black] (Bb) to[out=90,in=-90+15] (B);
\draw[black] (B) to[out=90,in=-90+20] (B2);

\draw[black] (Ca) to[out=90,in=-90-15] (C);
\draw[gray,very thick, name path = pa] (Cb) to[out=90,in=-90] (Cb2);
\draw[gray,very thick, name path = pb] (Cb2) to[out=90,in=-90+45] (C);
\draw[black] (C) to[out=90,in=-90+15] (C2);

\draw[orange, densely dotted, thick] (D) to[out=90+10,in=-90+40] (Cb2);

\path[name path = hora] (0,.5) -- (D|-0,.5);
\draw[teal, densely dotted, name intersections={of=pa and hora}] (D) to[out=90+50,in=-90+35] (intersection-1);

\path[name path = horb] (0,1.5) -- (D|-0,1.5);
\draw[teal, densely dotted, name intersections={of=pb and horb}] (D) to[out=90,in=-90+60] (intersection-1);

\node[gray] (Y) at (2.6,1.6*\yscale) {\footnotesize$\boldsymbol Y_\bullet$};

%% space and time labels
\draw[dotted,gray] (-.5,1.0*\yscale) -- (Cb2);
\draw (-.5+\tick,1.0*\yscale) -- (-.5-\tick,1.0*\yscale) node[left] {\footnotesize$s_1$};

\draw[dotted,gray] (-.5,1.7*\yscale) -- (C);
\draw (-.5+\tick,1.7*\yscale) -- (-.5-\tick,1.7*\yscale) node[left] {\footnotesize$s_2$};

\draw[dotted,gray] (Cb2|-0,-.5) -- (Cb2);
\draw (Cb2|-0,-.5+\tick) -- (Cb2|-0,-.5-\tick) node[below] {\footnotesize$z$};

\draw[dotted,gray] (B) -- (B|-0,-.5+\tick);
\draw (B|-0,-.5+\tick) -- (B|-0,-.5-\tick);

\draw[dotted,gray] (A) -- (A|-0,-.5+\tick);
\draw (A|-0,-.5+\tick) -- (A|-0,-.5-\tick);

\draw[decorate, decoration = {brace, mirror, raise=5pt}] (A|-0,-.5) -- node[below,yshift=-10pt] {\footnotesize $\boldsymbol z_1$} (B|-0,-.5);

\draw[dotted,gray] (C) -- (C|-0,-.5+\tick);
\draw (C|-0,-.5+\tick) -- (C|-0,-.5-\tick) node[below] {\footnotesize$z_2$};

%% node labels
\draw[orange] (D) node[below] {\footnotesize $\vnew$};
\draw (Cb) node[below] {\footnotesize $\unew$};
%\draw[gray,name intersections={of=PATH and horizontal}] (intersection-1) node[above right] {\footnotesize $\wnew$};

%% node circles
\draw[fill,\nodecol] (A) circle[radius=\noderadius];
\draw[fill,\nodecol] (Aa) circle[radius=\noderadius];
\draw[fill,\nodecol] (Ab) circle[radius=\noderadius];
\draw[fill,\nodecol] (Ac) circle[radius=\noderadius];
\draw[fill,\nodecol] (B) circle[radius=\noderadius];
\draw[fill,\nodecol] (Ba) circle[radius=\noderadius];
\draw[fill,\nodecol] (Bb) circle[radius=\noderadius];
\draw[fill,\nodecol] (Ca) circle[radius=\noderadius];
\draw[fill,\nodecol] (Cb) circle[radius=\noderadius];
\draw[fill,orange] (Cb2) circle[radius=\noderadius];
\draw[fill,\nodecol] (C) circle[radius=\noderadius];
\draw[fill,orange] (D) circle[radius=\noderadius];

\draw[fill,teal, name intersections={of=pa and hora}] (intersection-1) circle[radius=\noderadius];
\draw[fill,teal, name intersections={of=pb and horb}] (intersection-1) circle[radius=\noderadius];
    \end{scope}
\end{tikzpicture}
    \caption{Illustrations for the proof of \cref{lem:nuconst2}.}
    \label{fig:prf:nuconst2}
\end{figure}

\begin{lemma}\label{lem:nuconst2}
    For every $(n,\vec{k}) = (n,k_1, \ldots ,k_m)\in \mergers$ and $i\in [m]$, $\gnk$ is constant in the $i$'th argument if $k_i = 2$. In particular, $\nu_{n,2}(\diff z) = \lambda_{n,2}\diff z$ for some $\lambda_{n,2}\ge 0$ for each $n \ge 2$.
\end{lemma}
\begin{proof}
    We first prove $\nu_{n,2}(\diff z) \sim \diff z$ for all $n \ge 2$. This is already proved for $n = 2$, so let $n \ge 3$, $(\fpart_0,\boldsymbol x)\in \mathcal{X}$ with $|\fpart_0| = n-1$, and $F\in \mathbb{T}(\fpart_0)$ a possible tree. Let $\unew \in \fpart_0$ be one of the leaves involved in the first merge event, and denote by $G \in \mathbb{T}(\fpart_0\new)$ the tree in which an additional leaf $\vnew$ merges with $\unew$ before the first merge of $F$, so $G\rsto{\fpart_0} = F$. Fix $F^\star=(F,\tau,\xi) \in \dcb(F)$ for which \cref{eq:samplingconKx} holds, and put $s_1 \coloneqq \tau(\pr_F(\unew))$, $z_1 \coloneqq \pr_F(\unew)$, the time and place of the first merge of $\unew$ in $F$. See \cref{fig:prf:nuconst2} (left). Then the law of $(\boldsymbol Y_r \coloneqq \boldsymbol X_{r}(\unew))_{0\le r \le s_1}$ under $K_{\boldsymbol x}(F^\star,\cdot )$ is $B^{(0,x_0)\to(s_1,z_1)}$ where $x_0 \coloneqq \boldsymbol x(\unew)$.
    And if $G^\star=(G,s\new\tau,z\new\xi)\in \dcb\new(F^\star \,\vert\,G)$, then the law of $\boldsymbol Y$ under $K_{\boldsymbol x}(G^\star,\{\rs{\boldsymbol X}\in\cdot \})$ is $B^{(0,x_0)\to(s,z)\to(s_1,z_1)}$. Thus by \cref{lem:samplingconKx} and \cref{lem:mux}, for some $\la\in \R$,
    \begin{align*}
        B^{(0,x_0)\to(s_1,z_1)}(\cdot )
        &\sim \int_{E_{\boldsymbol x}} \int_E \int_{0}^{s_1} B^{(0,x_0)\to(s,z)\to(s_1,z_1)}(\cdot ) P^{\boldsymbol x y}((G,s\new\tau,z\new\xi)) \diff s \diff z \mu_{\boldsymbol x}(\diff y)\\
        &\sim \iiint B^{(0,x_0)\to(s,z)\to(s_1,z_1)}(\cdot ) \e^{-\la s} p_s(z-y) p_s(z-x_0) \\[-5pt]
        &\hspace{6cm} p_{s_1-s}(z_1-z) \nu_{n,2}(z) \diff s\diff z\diff y\\
        &= \iint B^{(0,x_0)\to(s,z)\to(s_1,z_1)}(\cdot ) \e^{-\la s} p_s(z-x_0) p_{s_1-s}(z_1-z) \nu_{n,2}(z) \diff z \diff s,
    \end{align*}
    which by \cref{lem:bb} implies that $\nu_{n,2}(z)$ is constant and therefore $\nu_{n,2}(\diff z) \sim \diff z$.

    Now let $n \ge 3$ and $(n,k_1, \ldots ,k_m)$ possible with $m\ge 2$ and $k_m = 2$. Take $\fpart_0\in \mathcal{P}$ with $|\fpart_0| = n-1$, and let $G\in \mathbb{T}(\fpart_0\new)$ be a tree in which the first merge event is $(n,\vec{k})$, and in which the parent of the two leaves $\unew,\vnew$ comprising the merge associated to $k_m$, is involved in the merge event following the first. See \cref{fig:prf:nuconst2} (right). Let $F\coloneqq G \wo \vnew \in \mathbb{T}(\fpart_0)$, so $G\rsto{\fpart_0} = F$. Fix $F^\star = (F,\tau,\xi) \in \dcb(F)$ for which \cref{eq:samplingconKx} holds, and denote $s_1\coloneqq \tau(\fpart^F_1)$, $\boldsymbol z_1 \coloneqq \xi\vert_{\fpart^F_1 \setminus \fpart_0}$, $s_2\coloneqq \tau(\pr_F(\unew))$, $z_2\coloneqq \xi(\pr_F(\unew))$, then the law of $(\boldsymbol Y_r\coloneqq \boldsymbol X_r(\unew))_{0\le r \le s_2}$ under $K_{\boldsymbol x}(F^\star,\cdot )$ is $B^{(0,x_0)\to(s_2,z_2)}$, where $x_0 = \boldsymbol x(\unew)$. If $G'^\star \in \dcb\new(F^\star)$ with $G'\neq G$, then the law of $\boldsymbol Y$ under $K_{\boldsymbol x y}(G'^\star,\{\rs{\boldsymbol X}\in\cdot \})$ is the same,
    except in the two cases where $G'$ merges $\vnew $ with $\unew$ in $(0,s_1)$ or in $(s_1,s_2)$ (see teal-coloured parts of \cref{fig:prf:nuconst2}). In the former case,
    \begin{align*}
        \iiint & \boldsymbol Y\# K_{\boldsymbol x y}(G'^\star,\{\rs{\boldsymbol X}\in\cdot \}) P^{\boldsymbol x y}((G,s\new\tau,z\new\xi)) \diff s\diff z \mu_{\boldsymbol x}(\diff y)\\
        &\sim \iiint B^{(0,x_0)\to(s,z)\to(s_2,z_2)}(\cdot ) \e^{-\la s} p_s(z-y) p_s(z-x_0) p_{s_2-s}(z_2-z) \diff s\diff z\diff y\\
        &= \int_0^{s_1} B^{(0,x_0)\to(s,z)\to(s_2,z_2)}(\cdot ) \e^{-\la s} p_s(x_0-z) p_{s_2-s}(z-z_2) \diff z \diff s\\
        &\sim B^{(0,x_0) \to (s_2,z_2)}(\cdot ),
    \end{align*}
    where we used \cref{lem:bb}, and similarly if the merge is in $(s_1,s_2)$. Now, if $G^\star = (G,\tau,z\new\xi) \in \dcb\new(F^\star \,\vert\,G)$ (see the orange coloured part of \cref{fig:prf:nuconst2}, right), then the law of $\boldsymbol Y$ under $K_{\boldsymbol x y}(G^\star,\{\rs{\boldsymbol X}\in\cdot \})$ is $B^{(0,x_0) \to (s_1,z) \to (s_2,z_2)}$, so by \cref{lem:samplingconKx},
    \begin{align*}
        B^{(0,x_0) \to (s_2,z_2)}(\cdot )
        &\sim \int_{E_{\boldsymbol x}} \int_{E_{\boldsymbol z_1}} B^{(0,x_0)\to(s_1,z)\to(s_2,z_2)}(\cdot ) P^{\boldsymbol x y}((G,\tau,z\new\xi)) \diff z \mu_{\boldsymbol x}(\diff y)\\
        &\sim \iint B^{(0,x_0)\to(s_1,z)\to(s_2,z_2)}(\cdot ) p_{s_1}(z-x_0) p_{s_1}(z-y)\\[-5pt]
        &\hspace{6cm} p_{s_2-s_1}(z_2-z) \nu(\boldsymbol z_1,z) \diff z \diff y\\
        &\sim \iint B^{(0,x_0)\to(s_1,z)\to(s_2,z_2)}(\cdot ) p_{s_1}(z-x_0) p_{s_2-s_1}(z_2-z) \nu(\boldsymbol z_1,z) \diff z.
    \end{align*}
    Evaluating both laws at time $s_1$ gives that $\gnk(\boldsymbol z_1,\cdot )$ is (a.e.)\ constant. By symmetry of $\gnk$ it must then be constant in all arguments $i$ with $k_i = 2$.
\end{proof}

This means we can write $\nunk(\boldsymbol z) = \nunk( z_1, \ldots , z_{j-1})$ if $(n,k_1, \ldots ,k_m)\in \mergers$ with $m\ge 2$ and $k_j=\ldots =k_m=2$. The following lemma finishes this section.  %This concludes the proofs that make use of arguments based on \cref{lem:samplingconKx}.

\begin{lemma}\label{lem:nuconst}
    $\boldsymbol \nu(\diff \xi) \sim \diff \xi$.
\end{lemma}
\begin{proof}%continue here and make this proof a bit nicer using the techniques from above
    We first prove that for every $(n,\vec{k})\in \mergers$ there exist non-negative constants $(\cnk^{(j)})_{0\le j\le m}$ and $\cnk^\text{sb}$ (whose dependence on $(n,\vec{k})$ we suppress) such that
    \begin{equation}\label{eqprf:nuconst:1}
        \nunk(\boldsymbol z) = \cnk^{(0)} \nu_{n+1,\smash{\vec{k}}}(\boldsymbol z) + \cnk^\text{sb} \nu_{n+1,k_1, \ldots ,k_m,2}(\boldsymbol z) + \sum_{j=1}^m \cnk^{(j)} \nu_{n+1,k_1, \ldots ,k_j+1, \ldots ,k_m}(\boldsymbol z),
    \end{equation}
    for a.e.\ $\boldsymbol z$, and $\cnk^{(j)}> 0$ for $j\in [m]$. Fix $(n,\vec{k})$ and let $F$ be a possible tree that starts with an $(n,\vec{k})$-merger, and put $F^\circ \coloneqq F\setminus \lf(F)$. By passing to the total mass in \cref{eq:consPxKx}, for a.e.\ $\boldsymbol z\in E^{\fpart^F_1\setminus \fpart^F_0}_\circ $, $\xi\in \dcs(F^\circ )$, and $\tau \in \dct(F)$,
    \begin{align}
        P^{\boldsymbol x}((F,\tau,\boldsymbol z\xi))
        &= \smashoperator[l]{\sum_{G\in \TTb}} \iiint P^{\boldsymbol x y}((G,s\new\tau,z\new\boldsymbol z\xi)) \diff s\diff z \mu_{\boldsymbol x}(\diff y) \nonumber\\
        & + \smashoperator[l]{\sum_{G\in \TTbs}} \iint P^{\boldsymbol x y}((G,\tau,z\new\boldsymbol z\xi)) \diff z\mu_{\boldsymbol x}(\diff y) + \smashoperator[l]{\sum_{G\in \TTm}} \int P^{\boldsymbol x y}((G,\tau,\boldsymbol z\xi)) \mu_{\boldsymbol x}(\diff y).\label{eqprf:nuconst:2}
    \end{align}
    For the remainder of this proof put $f_F(\tau) \coloneqq \prod_{(\fpart,\fpart') \in F} \e^{-\nut{\fpart}(\tau_{\fpart'}-\tau_\fpart)}$ and analogously for other trees, so that \[
        \fnu^F(\tau,\boldsymbol z\xi \,\vert\,\boldsymbol x) = \frac{1}{N^{\boldsymbol\nu}(\boldsymbol x)} f_F(\tau) \fsp^F(\boldsymbol z\xi \,\vert\,\tau,\boldsymbol x).
    \] Further recall that $P^{\boldsymbol x}((F,\tau,\boldsymbol z\xi)) = \fnu^F(\tau,\boldsymbol z\xi \,\vert\,\boldsymbol x) \boldsymbol \nu_F(\diff (\boldsymbol z\xi))$. We will now divide \cref{eqprf:nuconst:2} by $\fsp^F(\boldsymbol z\xi \,\vert\,\tau,\boldsymbol x)$ and then integrate it over $\xi$ and $\tau$. On the left we get $\frac{|\boldsymbol \nu_{F^\circ }|}{N^{\boldsymbol\nu}(\boldsymbol x)} \left(\int f_F(\tau) \diff \tau\right) \nunk(\boldsymbol z)$. Now let $G\in \TTb$ such that the binary merge is before the first merge of $F$, then the corresponding summand in \cref{eqprf:nuconst:2} is $\lambda_{2,2} / \int N^{\boldsymbol\nu}(\boldsymbol x y)\diff y$ times
    \begin{multline*}
        %\iiint P^{\boldsymbol xy}((G,s\new\tau,z\new\boldsymbol z\xi))\diff s\diff z\mu_{\boldsymbol x}(\diff y) =
        \int \left( \iint \fsp^G(z\new\boldsymbol z\xi \,\vert\,s\new\tau,\boldsymbol xy) \diff y\diff z\right) f_G(s\new\tau) \nunk(\boldsymbol z) \boldsymbol \nu_{F^\circ }(\xi) \diff s\\
        = \fsp^F(\boldsymbol z\xi \,\vert\,\tau,\boldsymbol x) \left(\int f_G(s\new\tau) \diff s\right) \nunk(\boldsymbol z) \boldsymbol \nu_{F^\circ }(\xi),
    \end{multline*}
    an equality we've used before, cmp.\ \cref{eqprf:nulambda1:3}. Dividing by $\fsp^F(\boldsymbol z\xi \,\vert\,\tau,\boldsymbol x)$ and integrating over $\xi$ and $\tau$ gives a positive (because $F^\circ $ is possible) multiple of $\nunk(\boldsymbol z)$. Similarly, if the merge is after the first of $F$, we obtain a non-negative (because then $G$ may be impossible) multiple of $\nu_{n+1,\smash{\vec{k}}}(\boldsymbol z)$. If $G\in \TTbs$ then we get non-negative multiples of either $\nunk(\boldsymbol z)$ or $\nu_{n,k_1, \ldots ,k_m,2}(\boldsymbol z)$, and if $G\in \TTm$ then we get either a non-negative multiple of $\nu_{n+1,\smash{\vec{k}}}(\boldsymbol z)$, or a positive multiple of $\nu_{n+1,k_1\ldots k_j+1\ldots k_m}(\boldsymbol z)$ for some $j\in [m]$, and for each such $j$ there is a corresponding $G\in \TTm$.

    This proves the claim surrounding \cref{eqprf:nuconst:1}, from which we conclude with an induction over $k$ in the statement ``Every $\nunk$ is constant in the $i$th argument if $k_i \le k$''. For $k = 2$ this is \cref{lem:nuconst2}, now suppose it has been proved for some fixed $k \ge 2$. Take some $(n,\vec{k})$ for which there is $j\in [m]$ with $k_j = k+1$ and $k_i \le k$ for $i > j$. Put $\vec{k}' \coloneqq (k_1, \ldots k_{j-1},k_j-1,k_{j+1}, \ldots ,k_m)$, and invoke \cref{eqprf:nuconst:1} with $\nu_{n,\smash{\vec{k}'}}$ on the LHS. Then $\nunk$ appears as a summand on the RHS (perhaps multiple times), and all other summands and the LHS are, by the induction hypothesis, constant in the $j$th to $m$th arguments. Hence also $\nunk$ must be constant in the $j$th argument, and by symmetry in all arguments with $k_i = k+1$.
\end{proof}

%\begin{lemma}\label{lem:fspproj}
%    If $(F,\tau,\xi) \in \dcb(\mathbb{T})$, $G\in \TTb$, $G'\in \TTbs$, $s > 0$ with $s\new\tau\in \dct(G)$, \[
%        \iint \fsp^G(z\new\xi \,\vert\,s\new\tau,\boldsymbol xy) \diff z\diff y = \iint \fsp^{G'}(z\new\xi \,\vert\,\tau,\boldsymbol xy) \diff z\diff y = \fsp^F(\xi \,\vert\,\tau,\boldsymbol x).
%    \]
%\end{lemma}
%\begin{proof}
%    If $G\in \TTb$ and $\unew,\vnew$ etc.\ as usual, then put $s_0\coloneqq \tau(\unew)$, $z_0\coloneqq \xi(\unew)$, $s_1\coloneqq \tau(\pr_F(\unew)) \in [0,\infty]$ and, if $s_1<\infty$, $z_1 \coloneqq \xi(\pr_F(\unew))$. In that case,
%    \begin{align*}
%        \fsp^G(z\new\xi \,\vert\,s\new\tau, \boldsymbol x) = \fsp^F(\xi \,\vert\,\tau,\boldsymbol x) \frac{p_{s-s_0}(z-z_0) p_{s_1-s}(z_1-z) p_s(z-y)}{p_{s_1-s_0}(z_1-z_0)},
%    \end{align*}
%    so indeed $\iint \fsp^G(z\new\xi \,\vert\,s\new\tau,\boldsymbol x)\diff z\diff y =  \fsp^F(\xi \,\vert\,\tau,\boldsymbol x)$. The proof if $s_1 = \infty$, or $G\in \TTbs$, is analogous.
%\end{proof}

\subsection{Drift Representation}\label{sec:prfdrift}
In this section we prove the drift representation \cref{thm:driftintro}.

\begin{proof}[Proof of \cref{thm:driftintro}]
    Denote by $(\fpart_t,\boldsymbol X_t)_{t \ge 0}$ the Brownian spatial coalescent with transition rates $\boldsymbol \nu(\diff \boldsymbol z)  = \boldsymbol \lambda \diff \boldsymbol z$ for some $\boldsymbol \lambda \in \rates$, started from $\fpart_0 = \{\{1\}, \ldots ,\{n\}\}$ and $\boldsymbol X_0 = \boldsymbol x \in E^{\fpart_0}_\circ $. Denote $\boldsymbol Z_t = (Z_t^1, \ldots ,Z_t^n) = (\boldsymbol X_t(\{1\}), \ldots ,\boldsymbol X_t(\{n\}))$ as long as $t$ is smaller than the random time of the first merge event, that is for $t < \inf \{ s > 0\colon \fpart_s = \fpart_0\}$.

    We perform a slightly informal generator calculation, and leave technical details to the interested reader. The (time-dependent) generator of a Brownian bridge on the torus $E$, starting at $x_0$ at time $t_0$, going to $x_1$ at time $t_1$, is given by \[
        A_s f(x) = \frac{1}{2}\Delta f(x) + \nabla_x \log p_{t_1-s}(x_1-x) \cdot \nabla f(x),\qquad s \in [t_0,t_1),
    \] see e.g.~(1.2) in~\cite{bbgenerator}. Then the generator of $\boldsymbol Z$ is
    \begin{align*}
        A& f(z_1, \ldots ,z_n) \\
         &= \sum_{i=1}^n \frac{1}{N^{\boldsymbol \nu}(\boldsymbol z)} \,\,\,\smashoperator{\int\limits_{\dcb(\mathbb{F}(\fpart_0))}}\,\,\, \left( \frac{1}{2}\nabla_{z_i}^2 f(\boldsymbol z) + \Big[ \nabla _{z_i} \log p(\tau_{\pr_F(\{i\})},\xi_{\pr_F(\{i\})} - z_i)\Big]_{\pr_F(\{i\}) \neq \emptyset } \cdot \nabla_{z_i} f(\boldsymbol z) \right) \\[-7pt]
        &\hspace{7cm} \times \ftm(F,\tau) \fsp(\xi \,\vert\,\tau,\boldsymbol z) \diff \tau\diff \xi\\[5pt]
        &= \frac{1}{2} \Delta f(\boldsymbol z) + \frac{1}{N^{\boldsymbol \nu}(\boldsymbol z)} \sum_{i=1}^n \nabla_{z_i} f(\boldsymbol z) \cdot \smashoperator{\int\limits_{\dcb(\mathbb{F}(\fpart_0))}} \Big[ \nabla _{z_i} \log p(\tau_{\pr_F(\{i\})},\xi_{\pr_F(\{i\})} - z_i)\Big]_{\pr_F(\{i\}) \neq \emptyset }\\[-7pt]
        &\hspace{7cm} \ftm(F,\tau) \fsp(\xi \,\vert\,\tau,\boldsymbol z) \diff \tau\diff \xi.
    \end{align*}
    If $(F,\tau,\xi) \in \dcb(\mathbb{F}(\fpart_0))$ with $\pr_F(\{i\}) \neq \emptyset $, then $\fsp(\xi \,\vert\,\tau,\boldsymbol z)$ depends on $z_i$ only through a factor $p(\tau_{\pr_F(\{i\})},\xi_{\pr_F(\{i\})} - z_i)$, so \[
        (\nabla _{z_i} \log p(\tau_{\pr_F(\{i\})},\xi_{\pr_F(\{i\})} - z_i)) \fsp(\xi \,\vert\,\tau,\boldsymbol z) = \nabla _{z_i} \fsp(\xi \,\vert\,\tau,\boldsymbol z).
    \] If $\pr_F(\{i\}) = \emptyset $, then $\fsp(\xi \,\vert\,\tau,\boldsymbol z)$ is independent of $z_i$, so in any case
    \begin{align*}
        A f(z_1, \ldots ,z_n)
        &= \frac{1}{2}\Delta f(\boldsymbol z) + \frac{1}{N^{\boldsymbol \nu}(\boldsymbol z)} \sum_{i=1}^n \nabla_{z_i} f(\boldsymbol z)\cdot \smashoperator{\int_{\dcb(\mathbb{F}(\fpart_0))}} \ftm(F,\tau) \nabla _{z_i} \fsp(\xi \,\vert\,\tau,\boldsymbol z) \diff \tau \diff \xi\\
        &= \frac{1}{2}\Delta f(\boldsymbol z) + \frac{1}{N^{\boldsymbol \nu}(\boldsymbol z)} \sum_{i=1}^n \nabla_{z_i} f(\boldsymbol z) \cdot \nabla_{z_i} N^{\boldsymbol \nu}(\boldsymbol z)\\
        &= \frac{1}{2}\Delta f(\boldsymbol z) + \nabla f(\boldsymbol z) \cdot \nabla \log N^{\boldsymbol \nu}(\boldsymbol z).
    \end{align*}
To pull the derivative $\nabla _{z_i}$ out of the integral in the second step (and thereby showing differentiability of $N^{\boldsymbol \nu}$), we require statements analogous to Lemmas~\ref{lem:NFinnerintegralcontinuous} and~\ref{lem:app:NF}
    with $\fsp(\xi \,\vert\,\tau,\boldsymbol x)$ replaced by $\nabla _{x_i} \fsp(\xi \,\vert\,\tau,\boldsymbol x)$. This is true using similar methods, with the modification that in $d = 1$, points $\boldsymbol x$ which have two identical coordinates have to be excluded (as is already the case in $d \ge 2$); recall also \cref{rem:driftintro} (ii) on this issue.
\end{proof}

%A Brownian bridge from $(0,x_0)$ to $(t,x_1)$ on the torus when rollout out into $\R^d$ is the same as pulling a location first from $x_1 + 2\pi\Z^d$ with probability proportional to $p_t(2\pi k + x_1 - x_0)$, then an ordinary brownian bridge. So the drift part of the generator is actually
%\begin{align*}
%    \frac{1}{\sum_{k\in \Z^d} p_t(2\pi k + x_1 - x_0)}\sum_{k\in \Z^d} p_t(2\pi k + x_1 - x_0) \frac{2\pi k + x_1 - x_0}{t} = \frac{\nabla_{x_0} p_t(x_1-x_0)}{p_t(x_1-x_0)}
%\end{align*}
%
%\begin{lemma}
%    The generator of a Brownian bridge on the torus $E$, starting at $x_0$ at time $t_0$, going to $x_1$ at time $t_1$, is given by \[
%        A_s f(x) = \frac{1}{2}\Delta f(x) + \frac{\nabla_x p_{t_1-s}(x_1-x)}{p_{t_1-s}(x_1-x)} \cdot \nabla f(x),\qquad s \in [t_0,t_1).
%    \]
%\end{lemma}
%\begin{proof}
%    Let $q_t \colon \R^d \to [0,\infty)$ be defined by $q_t(x) = (2\pi t)^{-d / 2} \e^{-|x|^2 / (2t)}$, so that \[
%        p_t(x) = \sum_{k\in \Z^d} q_t(x+2\pi k),\qquad x\in E.
%    \] etc.
%\end{proof}

\subsection{Associated Population Models}

This section contains proofs of the results presented in \cref{sec:populationmodelsintro} of the introduction.

\subsubsection{$\Xi$-Fleming-Viot Process}\label{sec:prfxifv}

Fix a non-zero measure $\Xi$ on $\triangle$, and recall \cref{sec:populationmodelsintro}. We begin by constructing a stationary, bi-infinite version of the particle representation $(\boldsymbol Y(t))$ of the $\Xi$-Fleming-Viot process. Let $\mathfrak{N}^{ij}$ for $i,j\in \N$, $i < j$, and $\mathfrak{M}$ be Poisson point processes as described in \cref{sec:populationmodelsintro}.
For $i,j\in \N$ with $i < j$, denote by $(\tau^j_l)_{l\in \Z}$ the ordered set of times where level $j$ copies the location of (``looks down'' to) level $i$, either because $i$ was the parent of $j$ in a reproductive event, or because level $i$ was ``bumped up'' to level $j$ due to a reproductive event occuring among levels smaller than $i$. By definition of the lookdown construction, $(\tau^j_l)_{l\in \Z}$ is itself a Poisson point process with constant, finite intensity,
%Recall that for $i,j\in \N$ and $i < j$, level $j$ ``looks down'' to level $i$ at points of $\mathfrak{N}^{ij}$, and at points of $\mathfrak{M}$ if $i$ and $j$ are in the same basket, and that basket contains no levels smaller than $i$. In particular, the ordered set of times $(\tau^j_l)_{l\in \Z}$ where level $j$ looks down to some level $i < j$ is itself a Poisson point process with finite intensity,
so almost-surely countably infinite and without accumulation points. Denote by $k^j_l \in \{1, \ldots ,j-1\}$ for $l\in \Z$ the level to which particle $j$ looks down at time $\tau^j_l$.

%For $i,j\in \N$ and $i < j$, let $\boldsymbol \tau^{ij} = (\tau^{ij}_k)_{k\in \Z}$ denote the ordered set of times where level $j$ looks down to level $i$. This happens at a point of $\mathfrak{N}^{ij}$, and at a point of $\mathfrak{M}$ if $i$ and $j$ are in the same basket, and that basket contains no levels smaller than $i$. In particular, $\boldsymbol \tau^{ij}$ is itself a Poisson point process on $\R$ with constant, finite intensity, so almost-surely the set of those times is infinite and with no accumulation points.

Let $(B^1(t))_{t \in (-\infty,\infty)}$ be a Brownian motion running at stationarity on $E$, and \[
    (B^j_l(t)\colon j\in \N, j\ge 2, l\in \Z)_{t\ge 0}
\] be an independent family of Brownian motions on $E$, independent of $B^1$ and the Poisson point processes, and all starting at zero. To construct the process $(\boldsymbol Y(t))_{t\in (-\infty,\infty)}$, we put $Y_1 = B^1$, and then construct $Y_j$ for $j \ge 2$ inductively. Suppose $(Y_1(t), \ldots ,Y_{j-1}(t))_{t \in (-\infty,\infty)}$ has already been constructed for some $j \ge 2$. Then for every $l\in \Z$ put %$t\in (-\infty,\infty)$ let $l\in \Z$ be the unique number for which $t \in [\tau^j_l,\tau^j_{l+1})$, and put \[
\[
    Y_j(t) = Y_{k_l^j}(\tau^j_l) + B^j_l(t-\tau^j_l),\qquad t\in [\tau^j_l,\tau^j_{l+1}).
\] In words, whenever $j$ looks down to some level $i < j$ at some time $\tau$, we let $Y_j$ start at the location of $Y_i$ at time $\tau$ and evolve as an independent Brownian motion until the next time it looks down. Since these lookdown times are an infinite, discrete set, this defines $Y_j(t)$ for all $t\in (-\infty,\infty)$. For any fixed $t_0 \in \R$, the law of the evolution of the process $(\boldsymbol Y(t))_{t \ge t_0}$ is as described in \cref{sec:populationmodelsintro}, and since this entire construction is invariant under constant time shifts, the distribution of $\boldsymbol Y(t)$ and $\boldsymbol Y(s)$ is the same for any fixed $s,t\in (-\infty,\infty)$. % stationary dist unique? this construction is unique (for any bi-infinite stationary dist first level has to be stationary BM and the remaining construction must be the same) so at least for every bi-infinite stationary version the dist has to be the same. probably there are some general ergodicity theorems that say it has a unique stationary, but whatever. Later on I just talk about "the" stationary dist of Y

We will now prove \cref{thm:timereversalintro}, which will imply both \cref{prop:xifvstationary} and \cref{thm:xigenealogiesintro}. The statement is trivial for $\ell = 1$ (because both the first level of the $\Xi$-Fleming Viot process and the Brownian spatial $\Xi$-coalescent with a single particle are just an $E$-valued Brownian motion, and a single sample from $\mu^\Xi$ is uniform), so fix $\ell\ge 2$ for the rest of \cref{sec:prfxifv}, and put $\fpart_0 = \left\{ \left\{ 1 \right\} , \ldots ,\left\{ \ell \right\}  \right\} $. We present a formal construction of the process described in \cref{thm:timereversalintro}. It requires the following measurable maps (their existence follows from a basic measure theoretic argument, see Lemma~\ref{lem:hmeasurable}):
\begin{enumerate}
    \item a map $h_0\colon E^\ell  \to E^\ell $ such that the pushforward of Lebesgue measure $\diff \boldsymbol y$ on $E^\ell$ under $h_0$ is $N^\Xi(\boldsymbol y)\diff \boldsymbol y$.
    \item for every $m \in \left\{ 1, \ldots ,\ell-1 \right\} $ a map $h_m \colon E^m \times E^{\ell-m} \to E^{\ell-m}$ such that for every $\boldsymbol x\in E^m_\circ $, the pushforward of Lebesgue measure $\diff \boldsymbol y$ on $E^{\ell-m}$ under $h_m(\boldsymbol x,\cdot )$ is $\frac{N^\Xi(\boldsymbol x\boldsymbol y)}{N^\Xi(\boldsymbol x)} \diff \boldsymbol y$, and $h_m(\boldsymbol x,\cdot ) = \id_{E^{\ell-m}}$ if $\boldsymbol x\in E^m \setminus E^m_\circ $.
    \item for every $F\in \mathbb{T}(\fpart)$ a map $h_F\colon E^\ell  \times \dcb(F) \to \dcb(F)$ such that for every $\boldsymbol x\in E^\ell_\circ $, the pushforward of $\ftm(F,\tau) \diff \tau \diff \xi$ under $h_F(\boldsymbol x,\cdot )$ is $\frac{1}{N^\Xi_F(\boldsymbol x)} \ftm(F,\tau) \fsp(\xi \,\vert\,\tau,\boldsymbol x) \diff \tau \diff \xi$, and $h_F(\boldsymbol x,\cdot ) = \id_{\dcb(F)}$ if $\boldsymbol x\in E^\ell \setminus E^\ell_\circ $.
\end{enumerate}
The definitions of $h_m(\boldsymbol x,\cdot )$ and $h_F(\boldsymbol x,\cdot )$ if $\boldsymbol x\in E^m\setminus E^m_\circ $ and $\boldsymbol x\in E^\ell\setminus E^\ell_\circ $, respectively, is irrelevant as long as all maps are measurable.

%For $F = \{\fpart_0,\fpart_1^F,\ldots ,\fpart^F_m\}\in \mathbb{T}(\fpart_0)$, let $\overline{\dcs(F)}$ denote the set of maps $\xi\colon \nd^\circ (F) \to E$, which contains $\dcs(F)$ (recall \cref{def:xi}) as a dense, open subset, and is isomorphic to some power of $E$. Furthermore there is an obvious bijection between $\dct(F)$ and $(0,\infty)^m$. Thus we can apply a straight-forward adaptation of the Homeomorphic Measures Theorem \cite{homeo}, Theorem 9.1, due to Oxtoby and Ulam, to find the following maps:
%\begin{enumerate}

Now let $(S, \A, \P)$ be a probability space that supports the following random variables, all independent.
\begin{enumerate}
    \item For every $m\in \left\{ 0, \ldots ,\ell-1 \right\} $ an i.i.d.\ sequence $(\omega_m^{(i)})_{i\in \N_0}$ of uniform $E^{\ell-m}$-valued random variables.
    \item For every $F\in \mathbb{T}(\fpart_0)$ an i.i.d.\ sequence $(\omega_F^{(i)})_{i\in \N_0}$ of random variables in $\dcb(F)$ with distribution $\ftm(F,\tau)\diff \tau\diff \xi$.
    \item For every non-empty $u\subset [\ell]$ and $i\in \N_0$, a $d$-dimensional standard Brownian bridge $\omega_B^{(i)} \in C([0,1],E)$ with $\omega_B^{(i)}(0) = \omega_B^{(i)}(1) = 0$.
    \item An i.i.d.\ sequence $(\omega_U^{(i)})_{i\in \N_0}$ of uniform $[0,1)$-random variables.
\end{enumerate}
We further fix an ordering $\mathbb{T}(\fpart_0) = \left\{ F^{(1)}, \ldots ,F^{(|\mathbb{T}(\fpart_0)|)} \right\} $ of $\mathbb{T}(\fpart_0)$.

\begin{definition}\label{def:nthlevelresampling}
    Let $\boldsymbol \nu \in \nurates$. The \emph{$\ell$'th level Brownian spatial coalescent with resampling} associated with transition measures $\boldsymbol \nu$ is a c\`agl\`ad (left-continuous with right-limits) $E^\ell$-valued process $\boldsymbol Z^\ell(t) = (Z^\ell_1(t), \ldots ,Z^\ell_n(t))$, $t\ge 0$, defined on $S$ as follows. Given a realisation $\omega \in S$:
    \begin{enumerate}
        \item Put $\zeta^{(0)} \coloneqq  h_0(\omega_0^{(0)})$, and $t^{(0)} \coloneqq  0$. (At the end of the construction, $\zeta^{(i)} = \boldsymbol Z^\ell(t^{(i)}+)$.)
        \item Let $i \in \N_0$ and assume that $0 = t ^{(0)} < \ldots < t ^{(i)}$ and $\zeta^{(0)},\ldots, \zeta^{(i)}$ are defined and $\boldsymbol Z^\ell(t)$ has been constructed for $t\in [0,t ^{(i)})$. Then:
            \begin{enumerate}
                \item[(ii.1)] If $\zeta^{(i)} \in E^\ell\setminus E^\ell_\circ $ put $j = 1$, otherwise let $j\in \left\{ 1, \ldots ,|\mathbb{T}(\fpart_0)| \right\} $ be the unique number such that \[
                    \sum_{j'< j} N^\Xi_{F^{(j')}}(\zeta^{(i)}) \le \omega_U^{(i)} N^\Xi(\zeta^{(i)}) < \sum_{j'\le j} N^\Xi_{F^{(j')}}(\zeta^{(i)}).
                \] Put $F \coloneqq F^{(j)}$.
                \item[(ii.2)] Let $(\tau,\xi) = h_F(\zeta^{(i)}, \omega_F^{(i)}) \in \dcb(F)$, and $t^{(i+1)} \coloneqq  \tau(\fpart^{F}_1)$ the time of the first merge event. For $t\in (t ^{(i)},t ^{(i+1)}]$, and $l \in [\ell]$, put
                    \begin{equation*}
                        Z^\ell_l(t) \coloneqq  \zeta^{(i)}_l + \frac{t - t ^{(i)}}{\tau_{\pr_F(\{l\})}} (\xi_{\pr_F(\left\{ l \right\} )} - \zeta^{(i)}_l) + \sqrt{\tau_{\pr_F(\{l\})}} \omega_B^{(i)}\left( \frac{t- t ^{(i)}}{\tau_{\pr_F(\{l\})}} \right),
                    \end{equation*}
                    which is a Brownian bridge started from $\zeta^{(i)}_l$ at time $t ^{(i)}$, going to $\xi_{\pr_F(\{l\})}$ at time $t ^{(i)} + \tau_{\pr_F(\{l\})}$, and stopped at time $t ^{(i+1)}$.
                \item[(ii.3)] Let $\fpart_1^F = \{u_1,\ldots,u_k\}$ with $l_1 = \min u_1 < \ldots < l_k = \min u_k$.
                (Thinking forwards in time, where this corresponds to a reproduction event, $l_1,\ldots,l_k$ are the \emph{post}-reproduction levels of the individuals whose \emph{pre}-reproduction levels were $1,\ldots, k$, whose number of offspring, including themselves, is $|u_1|,\ldots,|u_k|$.)
                For every $j\in[k]$, put $\zeta^{(i+1)}_j = Z^\ell_{l_j}(t^{(i+1)}-)$, and \[
                    \zeta^{(i+1)}_{(k+1\ldots \ell)} = h_{k}(\zeta^{(i+1)}_{(1\ldots k)},\omega^{(i+1)}_{\ell-k}),
                \]
                where $\zeta^{(i+1)}_{(a\ldots b)} \coloneqq (\zeta^{(i+1)}_a,\zeta^{(i+1)}_{a+1},\ldots,\zeta^{(i+1)}_b)$.

                %\item[(ii.3)] Let $I = \left\{ \min u \colon u\in \fpart^F_1 \right\} $ be the set of levels that were either not involved in the first merge event, or were minimal among the set of levels with which they merged. Let $J = [n]\setminus I$. For every $l\in I$, put $\zeta^{(i+1)}_l = Z^n_l(t^{(i+1)}-)$, and \[
                %        \zeta_J^{(i)} \coloneqq  h_{|J|}(\zeta_I^{(i)}, \omega_{|J|}^{(i+1)}),
                %\] where $\zeta^{(i+1)}_J = (\zeta^{(i+1)}_j\colon j\in J)$, and $\zeta^{(i+1)}_J$ similarly.
            \end{enumerate}
    \end{enumerate}
    This defines a measurable map from $S$ into the space of c\`agl\`ad paths $[0,\infty)\to E^\ell$.
\end{definition}

%By properties of Brownian motion, $\boldsymbol Z^n(t+) \in E^n_\circ $ for all $t\ge 0$ with probability one (we take right-limits because $\boldsymbol Z^n(t) \not\in E^n_\circ $ at jump times).
All definitions dealing with cases where $\boldsymbol Z^\ell(t+) \in E^\ell\setminus E^\ell_\circ $ for some $t \ge 0$, which happens with probability zero, are only in place for the map $S \to D([0,\infty),E^\ell)$ to be well-defined, and do not affect the law of $\boldsymbol Z^\ell$.

\begin{remark}
    If the Brownian spatial coalescent with transition measures $\boldsymbol \nu$ is sampling consistent (that is, associated to some finite measure $\Xi$ on $\triangle$), then for every $m < \ell$, the laws of $(Z^m_1, \ldots ,Z^m_m)$ and $(Z^\ell_1, \ldots ,Z^\ell_m)$ are the same, and we expect that a variant of the Kolmogorov extension theorem can be used to construct an $E^\infty$-valued process $\boldsymbol Z(t) = (Z_1(t),Z_2(t),\ldots )$ whose first $\ell$ levels have the same law as $\boldsymbol Z^\ell$ for every $\ell\in \N$. Making this precise is not necessary to prove our results.
\end{remark}

From now, let $\boldsymbol Z^\ell$ be the $\ell$'th level Brownian spatial $\Xi$-coalescent with resampling.

\begin{theorem}\label{thm:xifvreversal}
    The laws of $(Z^\ell_1(t), \ldots ,Z^\ell_\ell(t))_{t \ge 0}$ and $(Y_1(-t), \ldots ,Y_\ell(-t))_{t \ge 0}$ are the same. In particular, the stationary distribution of $\boldsymbol Y$ is an i.i.d.\ sample from a random realisation of $\mu^\Xi$.
\end{theorem}
\begin{proof}
    It suffices to show that for every $T > 0$, the laws of $(Z^\ell_1(t), \ldots ,Z^\ell_\ell(t))_{t\in [0,T]}$, and $(Y_1(T-t), \ldots ,Y_\ell(T-t))_{t \in [0,T]}$, which coincides with that of $(Y_1(-t), \ldots ,Y_\ell(-t))_{t\in [0,T]}$, are the same. Denote by $\Q^{\boldsymbol y}$ for $\boldsymbol y\in E^\ell$ the probability measure on $D([0,T],E^\ell)$ describing the law of $(Y_1(t), \ldots ,Y_\ell(t))_{t \in [0,T]}$ started from (or conditioned on) $(Y_1(0), \ldots ,Y_\ell(0)) = \boldsymbol y$. Then we show
    \begin{align}\label{eq:prftimerev:1}
        \P ((Z^\ell_1(t),\ldots ,Z^\ell_\ell(t))_{t\in [0,T]} \in \cdot ) = \int_{E^\ell} \Q^{\boldsymbol y}((Y_1(T-t), \ldots ,Y_\ell(T-t))_{t\in [0,T]} \in \cdot ) N^\Xi(\boldsymbol y) \diff \boldsymbol y.
    \end{align}
    In particular, this implies that \[
        N^\Xi(\boldsymbol z) \diff \boldsymbol z = \P((Z^\ell_1(0), \ldots ,Z^\ell_\ell(0)) \in \diff \boldsymbol z) = \int_{E^\ell} \Q^{\boldsymbol y}((Y_1(T), \ldots ,Y_\ell(T)) \in \diff \boldsymbol z) N^\Xi(\diff \boldsymbol y) \diff \boldsymbol y,
    \] that is, $N^\Xi(\boldsymbol y) \diff \boldsymbol y$ (which equals $\E \left[ \mu^\Xi(\diff y_1)\ldots \mu^\Xi(\diff y_\ell) \right] $, recall \cref{eq:Emuintro}) is the stationary distribution of $(Y_1(t), \ldots ,Y_\ell(t))$. Plugging this back into \cref{eq:prftimerev:1} gives
    \begin{align*}
        \P((Z^\ell_1(t), \ldots ,Z^\ell_\ell(t))_{t\in [0,T]} \in \cdot ) = \P((Y_1(T-t), \ldots ,Y_\ell(T-t))_{t\in [0,T]}\in \cdot ),
    \end{align*}
    which is the claim.

    We now prove \cref{eq:prftimerev:1}. For simplicity we assume that $\xi_j = 0$ for $j\ge 2$ for $\Xi$-a.e.\ $\boldsymbol \xi \in \triangle$, in which case the evolution of $\boldsymbol Y$ can be described in terms of a finite measure $\Lambda$ on $[0,1]$, see \cref{rem:lambdafv}. The general case requires more notation but no different ideas. Consider the event where $(Y_1(t), \ldots ,Y_\ell(t))$, restricted to the first $\ell$ levels, starts at $\boldsymbol y^0 = (y^0_1, \ldots ,y^0_\ell)$, undergoes a branching event at time $s\in (0,T)$ involving $k \ge 2$ levels $J = \left\{ j_1<\ldots <j_k \right\} $, where $Y^\ell(s-) = \boldsymbol y^s = (y^s_1, \ldots ,y^s_\ell)$, and has no further branching events until time $T$ where it ends in $Y^\ell(T) = \boldsymbol y^T = (y^T_1, \ldots ,y^T_\ell)$. For $j\in [\ell]$, denote by $\underline{j}\in [\ell]$ the level to which $j$ looks down at this reproduction event (or $\underline{j} = j$ if $j$ does not look down), namely \[
        \underline{j} =
        \begin{cases}
            j, & j < j_1,\\
            j_1, & j\in J,\\
            j - |J \cap [j]| + 1, & \text{else}.
        \end{cases}
    \] Then $\{\underline{j}\colon j\in [\ell]\} = \{1,\ldots,m\}$ where $m = \ell - k+1$.

    The rate at which a branching event evolving exactly indices $J$ happens is $\lambda_{\ell,k}$, the associated rate of the $\Lambda$-coalescent. Indeed, if $k > 2$ then the rate is \[
        \int p^k (1-p)^{\ell-k} \frac{\Lambda(\diff p)}{p^2} = \lambda_{\ell,k},
    \] and if $k = 2$ it is $\Lambda(\{0\}) = \int p^{k-2} (1-p)^{\ell-k} \Lambda(\diff p) = \lambda_{\ell,2}$. Thus the probability (density) of the entire event is
    \begin{multline}\label{eq:prfxifv:1}
        \underbrace{N^\Lambda(\boldsymbol y^0) \diff \boldsymbol y^0}_{\text{initial sample}} \times
        \underbrace{\prod_{j=1}^\ell p_s(y^0_j-y^s_j)\diff y^s_j }_{\text{spatial movement in $(0,s)$}}
        \times \underbrace{\lambda_{\ell,k} \e^{-\lambda_\ell s}}_{\text{branching event}} \\
        % \times \underbrace{\prod_{j\in J} p_{T-s}(y^s_{j_1} - y^T_j) \diff y^T_j\prod_{j\not\in J} p_{T-s}(y^s_j - y^T_j) \diff y^T_j}_{\text{spatial movement in $(s,T)$}}
        \times \underbrace{\prod_{j=1}^\ell p_{T-s}(y^s_{\underline{j}} - y^T_j) \diff y^T_j}_{\text{spatial movement in $(s,T)$}}
        \times \underbrace{\e^{-\lambda_\ell(T-s)}}_{\text{no branching events in $(s,T)$}} .
    \end{multline}
    The process $((Z^\ell_1(T-t), \ldots ,Z^\ell_\ell(T-t)))_{t\in [0,T]}$ lies in the same event if and only if
    \begin{enumerate}
        \item the initial state is $\boldsymbol Z^\ell(0) = \boldsymbol y^T$,
        \item the first coalescence event of the Brownian spatial $\Lambda $-coalescent started from $\boldsymbol y^T$ at time $0$ is a multiple merger of the lineages with labels $J$ at time $T-s$, at locations $\boldsymbol Z^\ell_{(1\ldots m)}((T-s)+)=\boldsymbol y^{s}_{(1\ldots m)}$ (by Definition~\ref{def:nthlevelresampling}, the lineage with initial label $\{j\}$ will be at location $y^s_{\underline{j}}$ at time $T-s$.)
        \item the resampling at time $T-s$, given the locations $\boldsymbol y^{s}_{(1\ldots m)}$ yields $\boldsymbol y^{s}_{(m+1\ldots \ell)}$,
        \item the Brownian spatial $\Lambda$-coalescent started from $\boldsymbol y^s$ at time $T-s$ has no coalescence events until time $T$, where it ends in state $\boldsymbol y^0$.
    \end{enumerate}
    The probability density for this is (recall \cref{lem:semigroup})
    \begin{align*}
        \underbrace{N^\Lambda(\boldsymbol y^T) \diff \boldsymbol y^T}_{\text{(i)}} \times  \underbrace{\frac{N^\Lambda(\boldsymbol y^{s}_{(1\ldots m)})}{N^\Lambda(\boldsymbol y^T)}\lambda_{\ell,k}\e^{-\lambda_\ell(T-s)}
        %\prod_{j\in J} p_{T-s}(y^s_{j_1} - y^T_j)  \prod_{j\not\in J} p_{T-s}(y^s_j - y^T_j) \diff \boldsymbol y^{s,1}}_{\text{(ii)}} \\
        \prod_{j=1}^\ell p_{T-s}(y^s_{\underline{j}} - y^T_j) \diff \boldsymbol y^{s}_{(1\ldots m)}}_{\text{(ii)}} \\
        \times \underbrace{\frac{N^\Lambda(\boldsymbol y^s) \diff \boldsymbol y^{s}_{(m+1\ldots \ell)}}{N^\Lambda(\boldsymbol y^{s}_{(1\ldots m)})}}_{\text{(iii)}} \times \underbrace{\frac{N^\Lambda(\boldsymbol y^0)}{N^\Lambda(\boldsymbol y^s)}\e^{-\lambda_\ell s} \prod_{j=1}^\ell p_s(y_j^0-y^s_j)\diff y_j^0}_{\text{(iv)}},
    \end{align*}
    which is equal to \cref{eq:prfxifv:1}. Conditional on this event, both the realisation of $\boldsymbol Z^\ell$ and $\boldsymbol Y$ are obtained by sampling Brownian bridges between $\boldsymbol y^0$ and $\boldsymbol y^s$, and between $\boldsymbol y^s$ and $\boldsymbol y^T$. The same argument applies (with no modifications but more notation) to all possible realisations.
\end{proof}

This proves \cref{prop:xifvstationary,thm:timereversalintro}. To deduce \cref{thm:xigenealogiesintro}, we need to show that the coalescent process obtained from $\boldsymbol Z^\ell$ by projecting onto the trajectories of only the initial $\ell$ particles and their ancestors (equivalently by forgetting about all resampled particles and their ancestors), has the law of a Brownian spatial coalescent. This essentially follows from the definitions, and the strong Markov property of the Brownian spatial coalescent: First, observe that the statement is true by definition if we stop at the time $t ^{(1)}$ of the first coalescence--resampling event. Using an inductive argument, we can then assume the claim is true started from the smaller number of particles remaining after the first merge event. Then the strong Markov property of the Brownian spatial coalescent applied at time $t^{(1)}$, and sampling consistency imply that the overall law is the same.

\subsubsection{Scaling Limits of Neutral Population Models}
Let $T > 0$, and denote the space of possible realisations of $(\boldsymbol L(t))_{t\in [0,T]}$ by
\[
    \mathcal{L}([0,T]) \coloneqq \left\{ (\boldsymbol L(t))_{t\in [0,T]} \colon  \boldsymbol L(0) \equiv 0, \textrm{$\boldsymbol L$ is increasing and all jumps have unit size} \right\},
\] which is a closed subset of the space $D([0,T],\N^\infty)$ of c\'adl\'ag paths $[0,T] \to \N^\infty$ with the Skorokhod topology. Then the construction of the coalescent $(\fpart^{\ell}_t,\boldsymbol X^{\ell}_t)$ from the particle representation $(\boldsymbol L,\boldsymbol Y)$ described in Section~\ref{sec:xiflemingviot} defines maps
\begin{align}\label{eq:defcoalmaps}
    \iota_\fpart\colon \mathcal{L}([0,T]) \to D([0,T],\mathcal{P}_\ell);\quad \boldsymbol L \mapsto (\fpart_t^\ell)
\end{align}
and
\begin{align}
    \iota_{\fpart,\boldsymbol X}\colon \mathcal{L}([0,T]) \times D([0,T],E^\infty) \to D([0,T],\overline{\mathcal{X}});\quad (\boldsymbol L,\boldsymbol Y) \mapsto (\fpart_t^\ell,\boldsymbol X^\ell_t).
\end{align}
where $\mathcal{P}_\ell$ denotes the set of partitions of $[\ell]$, and we suppress the dependence of $\iota_\fpart$ and $\iota_{\fpart,\boldsymbol X}$ on $T$ and $\ell$. The main ingredient in the proofs of both Theorem~\ref{thm:neutralgenealogy} and Lemma~\ref{lem:simplerassumption} is that $\iota_\fpart$ and $\iota_{\fpart,\boldsymbol X}$ are continuous.

\begin{lemma}\label{lem:iotapi}
    $\iota_\fpart$ is continuous.
\end{lemma}
\begin{proof}
    Suppose that $\boldsymbol L^n \to \boldsymbol L$ in $\mathcal{L}([0,T])$, and put $\fpart^n \coloneqq \iota_\fpart(\boldsymbol L^n)$, $\fpart \coloneqq \iota_\fpart(\boldsymbol L)$. By otherwise restricting to $n\ge n_0$ for some $n_0\in \N$, we may assume without loss of generality that $\boldsymbol L^n$ and $\boldsymbol L$ have the same (potentially simultaneous) jumps, and in the same order, but at potentially different times. This implies by definition of $\iota_\fpart$ that $\fpart^n$ and $\fpart$ also have the same jumps in the same order. For $t\in [0,T]$, we write $\bar{t} \coloneqq T-t$.

    Let $t_0\in [0,T]$ and $t_n \to t_0$. We assume for simplicity that $t_0 \in (0,T)$, otherwise the following arguments require straightforward modifications. Firstly, we have to show that $\fpart^n(t_n) \in \{\fpart(t_0), \fpart(t_0-)\}$ for all but finitely many $n\in \N$.
    There exists $\varepsilon > 0$ such that $\boldsymbol L$ has no jumps in $(\bar{t}_0-2\varepsilon,\bar{t}_0+2\varepsilon) \setminus\left\{ \bar{t}_0 \right\} $. Then, for all $n\ge n_0$, say,
    \begin{align*}
        \boldsymbol L^n(\bar{t}_0 - \varepsilon) =
        %\boldsymbol L(\bar{t}_0-\varepsilon) =
        \boldsymbol L(\bar{t}_0-),\quad \text{and} \quad
        \boldsymbol L^n(\bar{t}_0 + \varepsilon) =
        %\boldsymbol L(\bar{t}_0+\varepsilon) =
        \boldsymbol L(\bar{t}_0).
    \end{align*}
    If $\boldsymbol L$ is continuous at $\bar{t}_0$, then $\boldsymbol L^n \equiv \boldsymbol L(\bar{t}_0)$ on $(\bar{t}_0-\varepsilon,\bar{t}_0+\varepsilon)$, and hence $\fpart^n \equiv \fpart(t_0)$ on $(t_0-\varepsilon,t_0+\varepsilon)$, and we are done.
    If $\boldsymbol L$ is discontinuous at $\bar{t}_0$, then for each $n\ge n_0$ there exists a unique $r_n \in (t_0-\varepsilon,t_0+\varepsilon)$ such that $\boldsymbol L^n \equiv \boldsymbol L(\bar{t}_0-)$ on $(\bar{t}_0-\varepsilon, \bar{r}_n)$ and $\boldsymbol L^n \equiv \boldsymbol L(\bar{t}_0)$ on $[\bar{r}_n,\bar{t}_0+\varepsilon)$. Therefore,
    \begin{align}\begin{split}\label{eq:prf:pi map continuous:1}
        \fpart^n \equiv \fpart(t_0-)\quad &\text{on}\quad (t_0-\varepsilon,r_n),\\
        \fpart^n \equiv \fpart(t_0)\phantom{-}\quad &\text{on}\quad [r_n,t_0+\varepsilon)
    \end{split}\end{align}
    In particular, $\fpart^n(t_n) \in \{\fpart(t_0-),\fpart(t_0)\}$ for all $n\ge n_0$.

    If $\boldsymbol L$ is discontinuous at $\bar{t}_0$, we need to show additionally that if $\fpart^n(t_n) = \fpart(t_0)$ for almost all---without loss of generality for all---$n\in \N$, and $t_n \le s_n \to t_0$, then also $\fpart_n(s_n) =\fpart(t_0)$ for almost all $n\in \N$. If $n$ is sufficiently large that $t_n,s_n \in (t_0-\varepsilon,t_0+\varepsilon)$ then $\fpart^n(t_n)=\fpart(t_0)$ implies by~\eqref{eq:prf:pi map continuous:1} that $\bar{r}_n \ge \bar{t}_n$, and therefore also $\bar{r}_n \ge \bar{s}_n > \bar{t}_0 - \varepsilon$, and therefore also $\fpart^n(s_n) = \fpart(t_0)$ by~\eqref{eq:prf:pi map continuous:1}.

    Finally, if $\fpart^n(t_n) = \fpart(t_0-)$ for almost all $n\in \N$, and $t_n \ge s_n \to t_0$, then a similar argument implies that, for sufficiently large $n\in \N$, we have $\bar{r}_n \le \bar{t}_n \le \bar{s}_n < \bar{t_0} + \varepsilon$, and therefore also $\fpart^n(s_n) = \fpart(t_0-)$.
    %$\fpart_t \in \bigcap_{\varepsilon > 0} \mathcal{F}(\boldsymbol L_s \colon s \in (\bar{t},\bar{t}+\varepsilon))$
\end{proof}

\begin{lemma}\label{lem:iotapix}
    $\iota_{\fpart,\boldsymbol X}$ is continuous.
\end{lemma}
\begin{proof}
    Let $(\boldsymbol L^n,\boldsymbol Y^n) \implies (\boldsymbol L,\boldsymbol Y)$, and write $(\fpart^n,\boldsymbol Y^n) = \iota_{\fpart,\boldsymbol X}(\boldsymbol L^n,\boldsymbol Y^n)$ and $(\fpart,\boldsymbol Y) = \iota_{\fpart,\boldsymbol X}(\boldsymbol L,\boldsymbol Y)$.
    Suppose that $t_n \to t_0 \in (0,T)$, and again write $\bar{t} = T-t$ for $t\in [0,T]$. We assume that $\boldsymbol Y^n$ is continuous at $\bar{t}_n$ for all $n\in \N$, otherwise the proof requires a small and straightforward modification (replacing $\bar{t}_n$ with $\bar{t}_n + \varepsilon_n$ for a suitable sequence $\varepsilon_n \to 0$).

    We already know that $\fpart^n(t_n) \in \{\fpart(t_0-),\fpart(t_0)\}$ for almost all (wlog for all) $n\in \N$ by Lemma~\ref{lem:iotapi}. Suppose that $\fpart^n(t_n) = \fpart(t_0)$ infinitely often. For such $n$, if $u\in \fpart(t_0)$, we have $A^n_u(t_n) = A_u(t_0) \eqqcolon a(u)$, and recalling~\eqref{eq:xfromy},
    \begin{align*}
        |X^n_u(t_n) - X_u(t_0-)| & \vee |X^n_u(t_n) - X_u(t_0)|\\
        &= |Y^n_{a(u)}(\bar{t}_n-) - Y_{a(u)}(\bar{t}_0)| \vee |Y^n_{a(u)}(\bar{t}_n-) - Y_{a(u)}(\bar{t}_0-)|\\
        &= |Y^n_{a(u)}(\bar{t}_n) - Y_{a(u)}(\bar{t}_0)| \vee |Y^n_{a(u)}(\bar{t}_n) - Y_{a(u)}(\bar{t}_0-)|\\
        &\tendsto{} 0,\quad n\to \infty,
    \end{align*}
    where we used that $\boldsymbol Y^n$ is continuous at $\bar{t}_n$ for all $n$.
    Similarly if $\fpart^n(t_n) = \fpart(t_0-)$ infinitely often.

    Now suppose that $\boldsymbol X^n(t_n) \to \boldsymbol X(t_0)$, and $t_n \le s_n \to t_0$. We have to show that  $\boldsymbol X^n(s_n) \to \boldsymbol X(t_0)$. First of all, $\boldsymbol X^n(t_n) \to \boldsymbol X(t_0)$ necessitates that $\fpart^n(t_n) = \fpart(t_0)$ for almost all (wlog all) $n\in \N$,
    Since $t_n \le s_n \to t_0$, this implies that also $\fpart^n(s_n) = \fpart(t_0)$ for almost all (wlog all) $n\in \N$.
    For $u\in \fpart(t_0)$, \[
        \left| X^n_u(s_n) - X_u(t_0) \right| = \left| Y^n_{a(u)}(\bar{s}_n-) - Y^n_{a(u)}(\bar{t}_0-) \right|
    \] We know that
    \begin{align*}
        0 &= \lim_{n\to \infty} \left| X^n_u(t_n) - X_u(t_0) \right| \\
        &= \lim_{n\to \infty} \left| Y^n_{a(u)}(\bar{t}_n-) - Y^n_{a(u)}(\bar{t}_0-) \right| \\
        &= \lim_{n\to \infty} \left| Y^n_{a(u)}(\bar{t}_n) - Y^n_{a(u)}(\bar{t}_0-) \right| .
    \end{align*}
    Since $\boldsymbol Y^n \to \boldsymbol Y$ and $\bar{t}_n \ge \bar{s}_n \to \bar{t}_0$, this implies that
    \begin{align*}
        0 &= \lim_{n\to \infty}  \left| Y^n_{a(u)}(\bar{s}_n) - Y^n_{a(u)}(\bar{t}_0-) \right|\\
          &= \lim_{n\to \infty}  \left| Y^n_{a(u)}(\bar{s}_n-) - Y^n_{a(u)}(\bar{t}_0-) \right|\\
          &= \lim_{n\to \infty} \left| X^n_u(s_n) - X_u(t_0) \right|.
    \end{align*}
    Similarly, if $\boldsymbol X^n(t_n) \to \boldsymbol X(t_0-)$ and $t_n \ge s_n \to t_0$, then $\boldsymbol X^n(s_n) \to \boldsymbol X(t_0-)$.
\end{proof}

\begin{proof}[Proof of Theorem~\ref{thm:neutralgenealogy}]
    By Lemma~\ref{lem:iotapix} and the continuous mapping theorem, the weak convergence $(\boldsymbol L^n,\boldsymbol Y^n) \implies (\boldsymbol L,\boldsymbol Y)$ implies $(\fpart^{n,\ell},\boldsymbol X^{n,\ell}) \implies (\fpart^\ell,\boldsymbol X^\ell)$, where $(\fpart^\ell,\boldsymbol X^\ell)$ is the genealogy of the $\Xi$-Fleming Viot process at stationarity. By Theorem~\ref{thm:xigenealogiesintro}, the latter is almost-surely $\mathcal{X}$-valued and has the law of a Brownian spatial $\Xi$-coalescent as claimed.
\end{proof}

\begin{proof}[Proof of Lemma~\ref{lem:simplerassumption}]
    By Skorokhod's representation theorem, we can jointly realise $(\boldsymbol L^n,\boldsymbol Y^n(0))$ for $n\in \N$ and $(\boldsymbol L, \boldsymbol Y(0))$ on some probability space such that $\boldsymbol L^n \to \boldsymbol L$ almost-surely and $\boldsymbol Y^n(0) \to \boldsymbol Y(0)$ almost-surely. Then it is straightforward to couple $\boldsymbol Y^n(t)$ and $\boldsymbol Y(t)$ for $t > 0$ in such a way that also $\boldsymbol Y^n \to \boldsymbol Y$ almost-surely (e.g.\ by constructing them from a common collection of independent Brownian motions), and therefore $(\boldsymbol L^n,\boldsymbol Y^n) \implies (\boldsymbol L,\boldsymbol Y) $ by a second application of the Skorokhod representation theorem.
\end{proof}

\section*{Acknowledgements}
% \addcontentsline{toc}{section}{Acknowledgements}
I would like to thank my supervisor Alison Etheridge for numerous helpful discussions, valuable feedback, guidance, and encouragement. I would further like to thank Michal Bassan and F{\'e}lix Foutel-Rodier for many insightful conversations. I further acknowledge support by the EPSRC grant EP/W523781/1.

% \printbibliography %[heading=bibintoc]
% \addcontentsline{toc}{section}{Bibliography}
%
% \begin{subappendices}
%     \input{appendices1.tex}
% \end{subappendices}

\appendix
\section{Maps and Spaces}
In this section we prove measurability of the maps $\fr, \tm, \sp,$ and $\dc$. Recall that $\Omega$ and $\Omega_0$ are equipped with the product topology and (consequently) the product $\sigma$ algebra.
%In this section we prove that the state spaces $\Omega$ and $\Omega_0$ are Polish, and some of the maps defined in section 2 of the main paper are measurable.

%\begin{lemma}\label{lem:OMMpolish}
%    $\Omega_0$ and $\Omega$ are closed subsets of $D([0,\infty),\mathcal{P})$ and $D([0,\infty),\mathcal{X})$, respectively. In particular, $\Omega_0$ and $\Omega$ are Polish.
%\end{lemma}
%\begin{proof}
%    Suppose $(\fpart^n_s) \in \Omega_0,\, n\in \N,$, and $(\fpart^n_s) \to (\fpart_s)\in D([0,\infty),\mathcal{P})$, and fix $t > 0$. Since $(\fpart_s)$ is c\`adl\`ag and has only countably many discontinuities, we can find $u < t < r$ such that $\fpart_u = \fpart_{t-}$ and $\fpart_t = \fpart_r$ and $(\fpart_s)$ is continuous at $u$ and $r$. Then $\fpart^n_t \to \fpart_t$ and $\fpart^n_r \to \fpart_r$, so for sufficiently large $n$, \[
%        \fpart_{t-} = \fpart_u = \fpart^n_u \ge \fpart^n_r = \fpart_r = \fpart_t.
%    \] This proves that $(\fpart_s) \in \Omega_0$.
%
%    If $(\fpart^n_s,\boldsymbol x^n_s) \in \Omega,\,n\in \N$, and $(\fpart^n_s,\boldsymbol x^n_s) \to (\fpart_s,\boldsymbol x_s) \in D([0,\infty),\mathcal{X})$, then $(\fpart^n_s) \to (\fpart_s)$ in $D([0,\infty),\mathcal{P})$, so $(\fpart_s) \in \Omega_0$, so $(\fpart_s,\boldsymbol x_s) \in \Omega$.
%\end{proof}
\begin{lemma}\label{lem:frmeasurable}
    For $N \in \N$ define $\fr^{(N)}\colon \Omega_0 \to \mathbb{F}$ by \[
    \fr^{(N)} (\omega) \coloneqq \left\{ \omega(t) \colon t \in \mathcal{T}_N \right\} ,
\] where $\mathcal{T}_N \subset (0,\infty)$ are chosen such that $\mathcal{T}_{N+1} \supset \mathcal{T}_N$ for all $N\in \N$, and $\mathcal{T}_N$ is an $(1 / N)$-cover of $[0,N]$, that is $\min_{t\in \mathcal{T}_N} |t - s| < 1 / N $ for all $s \in [0,N]$.
Then $\fr^{(N)}(\omega) \to \fr(\omega)$ as $N\to \infty$ in $\mathbb{F}$ for all $\omega\in \Omega$. Since $\fr^{(N)}$ is measurable, this implies that $\fr$ is measurable.
\end{lemma}
\begin{proof}
    Let $\omega\in \Omega$ with $\fr(\omega) = (\fpart_0^F, \ldots ,\fpart_m^F)$ for some $m\in \N$, and denote the jump times of $\omega$ by $0 \eqqcolon t_0 < t_1 < \ldots < t_m$. Then there exists $N\in \N$ such that $(t_{i-1}, t_i) \cap \mathcal{T}_N \neq \emptyset$ for all $i \in [N]$, and therefore $\fr^{(N)}(\omega) = \fr(\omega)$. Since the sequence $(\mathcal{T}_N)$ is nested, $\fr^{(N')}(\omega) = \fr(\omega)$ for all $N' \ge N$.
\end{proof}

\begin{lemma}\label{lem:tmFmeasurable}
    For $F \in \mathbb{F}$, the map \[
        \tm_F\colon \Omega_0 \supset\left\{ \fr = F \right\} \to \dct(F);\quad (\fpart_t) \mapsto \big[ \fpart \mapsto  \inf \left\{ t > 0\colon \fpart_t = \fpart \right\}  \big]
    \] is measurable.
\end{lemma}
\begin{proof}
    Say $F = (\fpart_0^F, \ldots ,\fpart_m^F)$ with $m \ge 1$. By definition of the topology on $\dct(F)$, measurability of $\tm_F$ is equivalent to measurability of $\tm_F(\cdot )(\fpart^F_i)$ for all $i\in [m]$. Then for $i \in [m]$ and $t > 0$, \[
        \left\{ \omega\colon \tm_F(\omega)(\fpart^F_i) \le t \right\} = \left\{ \omega(t) \le \fpart^F_i \right\} = \bigcup_{\fpart \le \fpart^F_i} \left\{ \omega(t) = \fpart \right\} ,
    \] is measurable as a finite union of measurable sets.
\end{proof}

With this definition, $\tm\colon \Omega_0 \to \dct(\mathbb{F})$ is given by $\omega \mapsto (\fr(\omega), \tm_{\fr(\omega)}(\omega))$.

\begin{lemma}\label{lem:tmmeasurable}
    The map $\tm\colon \Omega_0 \to \dct(\mathbb{F})$ is bijective and bimeasurable.
\end{lemma}
\begin{proof}
    Bijectivity is obvious. By definition of the topology on $\dct(\mathbb{F})$, measurability of $\tm$ is equivalent to measurability of $\tm_F$ for all $F \in \mathbb{F}$, which was proved above. For measurability of the inverse, recall that the topology on $\Omega_0$ is generated by cylinders, i.e.\ sets of the form $\left\{ \omega\colon \omega(t) = \fpart \right\} $ for $t \ge 0$ and $\fpart \in \mathcal{P}$. The preimage of such a set under the inverse of $\tm$ is
    \begin{align*}
        \bigcup_{\substack{F_1\in \mathbb{F} \\ \rt(F_1) = \fpart} } \bigcup_{\substack{F_2\in \mathbb{F} \\ \lf(F_2) = \fpart} } \left\{ (F_1F_2,(\tau_1 / t / \tau_2)) \colon \tau_1 \in \dct(F_1), \tau_2\in \dct(F_2), \tau_1(\rt(F_1)) \le t \right\}
    \end{align*}
    (recall the definition of $(\tau_1 / t / \tau_2)$ from \eqref{eq:tauconcat} in the main paper), % (\ref{eq:tauconcat}))
    which is a countable union over sets that are measurable by the definition of the topologies on $\dct(\mathbb{F})$ and $\dct(F),\, F \in \mathbb{F}$.
\end{proof}

\begin{lemma}\label{lem:spmeasurable}
    The map $\sp\colon \Omega \to \bigcup_{F\in \mathbb{F}} \dcs(F)$ is measurable.
\end{lemma}
\begin{proof}
    It is enough to show that for any $F\in \mathbb{F}$ and $u\in \nd^\circ(F) $, the map $\left\{ \fr = F \right\} \to E$ given by $\omega = (\boldsymbol x_t) \mapsto \boldsymbol x_{\tm_F(\omega)_u}(u)$ is measurable. It can be written as the composition of the three maps
    \begin{IEEEeqnarray*}{rCcCcCl}
        \left\{ \fr = F \right\} & \quad \to \quad & \left\{ \fr = F \right\} \times (0,\infty) & \quad \to \quad & \mathcal{X} & \quad \to \quad & E\\[-2pt]
        \cline{1-7}\\[-10pt]
        \omega & \mapsto & (\omega,\tm_F(\omega)(u)) & & & & \\[3pt]
               & & (\omega,t) & \mapsto & \omega(t) & & \\[-3pt]
               & & & & (\fpart,\boldsymbol x) & \mapsto & 
        \begin{cases}
            \boldsymbol x(u), & u \in \fpart,\\
            x_0, & \text{else},
        \end{cases}
    \end{IEEEeqnarray*}
    for some fixed, arbitrary $x_0\in E$. The first map is measurable by \cref{lem:tmFmeasurable}, the second by definition of the product topology, and the third because it is continuous.
\end{proof}

\begin{lemma}\label{lem:decmeasurable}
    The map $\dc\colon \Omega \to \dcb(\mathbb{F})$ is measurable.
\end{lemma}
\begin{proof}
    Follows from \cref{lem:tmFmeasurable,lem:spmeasurable}.
    %By definition of the topology on $\dcb(\mathbb{F})$, it suffices to show that for fixed $F$, the map \[
    %    \left\{ \fr = F \right\} \to \dct(F) \times \dcs(F);\, \omega \mapsto (\tm_F(\omega),\sp(\omega))
    %\] is measurable, which was proved in \cref{lem:tmFmeasurable,lem:spmeasurable}.
\end{proof}

\section{Disintegration of Brownian Motion (Lemma~\ref{lem:bb})}\label{sec:lembbproof}

A central ingredient in the characterisation of sampling consistency in Section~\ref{sec:proofs} of the main paper was Lemma~\ref{lem:bb}. It essentially says that, under some assumptions, if $(T,Z)$ is a $(0,\infty)\times E$ valued random variable such that the law of a Brownian bridge from $(0,0)$ to $(T,Z)$, followed by a Brownian motion starting at $(T,Z)$, is the same as that of a Brownian motion starting in $(0,0)$, then $Z \sim \mathcal{N}(0,T)$ conditional on $T$ (see \cref{fig:lem:bb} for an illustration).
Recall that for finite measures $m_1,m_2$, we write $m_1 \sim m_2$ if $m_1 = c m_2$ for some $c > 0$. To keep the appendix self-contained, we restate the Lemma here.
\lembb*
%\begin{lemma}\label{lem:bb}
%    Suppose $x_0\in E$, $s_0 > 0$, $f\colon E \times (s_0,\infty) \to (0,\infty)$ is continuous, and $\mu$ is a finite measure on $E$, and that $\int_{s_0}^\infty \int_E f(x,s) \mu(\diff x) \diff s < \infty$. Then \[
%        B^{(x_0,s_0)+}(\cdot ) \sim \iint B^{(x_0,s_0)\to(x,s)+}(\cdot ) f(x,s) \mu(\diff x) \diff s
%    \] if and only if $f(x,s) \mu(\diff x) \sim p_{s-s_0}(x-x_0) \diff x$ for all $s > s_0$. If further $x_1\in E$, $s_1 > s_0$, and $f\colon E\times (s_0,s_1) \to (0,\infty)$ is continuous, then \[
%    B^{(x_0,s_0)\to (x_1,s_1)}(\cdot ) \sim \iint B^{(x_0,s_0) \to (x,s)\to(x_1,s_1)}(\cdot ) f(x,s) \mu(\diff x) \diff s
%\] if and only if $f(x,s) \mu(\diff x) \sim p_{s-s_0}(x-x_0) p_{s_1-s}(x_1-x) \diff x$ for all $s\in (s_0,s_1)$.
%\end{lemma}
\pgfmathsetseed{824459}
\begin{figure}
    \begin{center}
    \begin{tikzpicture}
        \def\s{0.02}
        \def\sig{0.2}
        \draw[gray] (0,0)
        \foreach \x in {1,...,125}
        {   to ++(\s,rand*\sig)
        }
        node (BB) {}
        \foreach \x in {1,...,125}
        {   to ++(\s,rand*\sig)
        }
        node (A) {}
        \foreach \x in {251,...,500}
        {
            -- ++(\s,rand*\sig)
        };
        %\draw[dashed] let \p1 = (A) in node (\x1,\y1) {$T$};
        \draw[->] (0,-1) -- (0,4);
        \draw[-|] (0,0) -- (A |- 52,0);
        \draw(A |- 52,0) node[below] {$T$};
        \draw[->] (A |- 52,0) -- (10,0);
        \draw[dashed] (A |- 52,0) -- (A);
        \draw[dashed] (0, 52 |- A) -- (A);
        \draw(0,52|-A) node[left] {$Z$};
        \draw[fill] (A) circle[radius=2pt];
        \draw (BB) node[fill=white, above] {Brownian Bridge};
        \draw (7.8,2.5) node[fill=white] {Brownian Motion};
    \end{tikzpicture}
    \end{center}
    \caption{Illustration for \cref{lem:bb}.}
    \label{fig:lem:bb}
\end{figure}

\begin{proof}
    The ``if'' direction is clear in both cases: for fixed $s > s_0$, if we sample $z\sim \mathcal{N}(0,s)$ and, independent of $z$, a Brownian bridge from $(s_0,x_0)$ to $(s,z)$ followed by a Brownian motion started from $(s,z)$, then the law is the same as that of a Brownian motion started from $(s_0,x_0)$. Similarly for the second claim. From now assume that $d = 1$, so $E$ is the one-dimensional torus. The proof is analogous in higher dimensions. We start with the first statement. Without loss of generality $x_0 = s_0 = 0$ and $\iint f(x,s) \mu(\diff x) \diff s = 1$. A convenient reformulation is obtained by putting $g(s) \coloneqq \int f(x,s) \mu(\diff x)$ and $\nu_s(\diff x) \coloneqq f(x,s) \mu(\diff x) / g(s)$. Dominated convergence can be applied to show that $g$ is continuous, because $f$ assumes a unique finite maximum on the compact set $E \times [s,t]$ for any $0 < s < t$. By a similar argument, $\nu_s$ is continuous in $s$ w.r.t.\ the topology of weak convergence of probability measures. An argument bijecting continuous paths in $E$ and $\R$ shows that it suffices to prove the analogous statement in $\R$, where the probability measure $\nu_s$ on the torus is replaced by the probability measure $\frac{\rfy{p}_s(x)}{p_s(\pi(x))}\nu_s(\diff \pi(x))$ on $\R$, which is continuous in $s$ w.r.t.\ all Wasserstein distances. Here $\pi\colon \R \to E$ is the natural projection characterised by $\pi(x) - x \in \Z$, and $\rfy{p}_s$ is the heat kernel on $\R$. %details commented out.
    %details 1: Suppose $\mu$ is a prob distribution on $E$, and then consider the probability distribution on $\R^d$ that we get by sampling $x$ from $\nu$ and then a point from $x+2\pi k$ w.p.\ proportional to $p_t(x+2\pi k)$. So the probability of getting $x+2\pi k$ for $x\in E$, $k\in \Z^d$ (every point in $\R^d$ can be uniquely written like that) is $\mu(\diff x) p_t(x+2\pi k) / \sum_l p_t(x+2\pi l)$. Now suppose that is equal to $p_t(x+2\pi k) \diff x$, then it means that $\mu(\diff x) = (\sum_l p_t(x+2\pi l)) \diff x = p_t ^{\text{torus}}(x) \diff x$, which is what we want.
    %details 2: It suffices to prove the analogous statement in $\R$. To see this first note that there is a unique bijection $\Phi$ between $C([0,\infty),E)$ and $C([0,\infty),\R)$ with the property that $y(t) = \pi(\Phi(y)(t))$ for all $y\in C([0,\infty),E)$ and $t \ge 0$, where $\pi \colon \R \to E$ is the map characterised by $\pi(x) - x \in \Z$ for $x\in \R$. Then $\Phi$ maps Brownian motions/bridges onto each other. Let $Y$ be a random path in $E$ whose law is given by the RHS above, that is $B^{(x_0,s_0) \to (x,s)+}$ where $(x,s) \sim f(x,s)\mu(\diff x)\diff s$. Then $\Phi(Y)$ is a random path in $\R$ with law $B^{(x_0,s_0) \to (x',s)+}$ where $(x,s) \sim f(x,s) \mu(\diff x) \diff s$, conditional on which $x'$ is a random element of the discrete set $\pi^{-1}(x)$ with $\P(x' = y) \sim p_s(x_0-y)$ for $y\in \pi^{-1}(x)$. Then we will know that $x' \sim \mathcal{N}(x_0,s)$ conditional on $s$ in $\R$, so $x\sim \mathcal{N}(x_0,s)$ conditional on $s$ in $E$.
    
    So let $g\colon (0,\infty) \to (0,\infty)$ with $\int g(s) \diff s = 1$ be continuous, $\nu_s$ for $s > 0$ be probability measures on $\R$, continuous in $s$ w.r.t.\ the $1$-Wasserstein distance, such that \[
        B^{(0,0)+}(\cdot ) = \iint B^{(0,0)\to (x,s)+}(\cdot ) \nu_s(\diff x) g(s) \diff s.
    \] Denote by $\varphi_s(u) \coloneqq \int \e^{\mathrm{i} u x}\nu_s(\diff x)$ the characteristic function of $\nu_s$, which is bounded and differentiable in $u$, and both $\varphi_s$ and $\partial_u \varphi_s$ are jointly continuous in $u$ and $s$ by assumptions on $\nu_s$. This will justify all interchanges of differentiation and integration that follow. Evaluating the equation at time $t > 0$ and taking characteristic functions on both sides gives, after a short calculation, 
    \begin{align*}
        \e^{-t u^2 / 2}
        &= \int_0^t \varphi_{s}(u) \e^{-(t-s)u^2 / 2} g(s) \diff s + \int_t^\infty \varphi_s(u t / s) \e^{-t(s-t) u^2 / (2s)} g(s) \diff s.
    \end{align*}
    Cancelling $\e^{-t u^2 / 2}$ on both sides and taking a derivative in $t$ gives 
        \begin{align}\begin{split}\label{eqprf:bb:1}
            0 &= \int_t^\infty \partial_t \left( \varphi_s\left( {\textstyle\frac{ut}{s}} \right)  \e^{(ut)^2 / (2s)} \right) g(s)\diff s \\
                &= \frac{u}{t} \partial_u \left( \int_t^\infty \varphi_s\left( {\textstyle\frac{ut}{s}} \right) \e^{(ut)^2 / (2s)} g(s) \diff s \right).
        \end{split}\end{align}
        This implies that $\int_t^\infty \varphi_s(ut / s) \e^{(ut)^2 / (2s)} g(s) \diff s$ is constant in $u$, so letting $u \to 0$ for fixed $t > 0$ implies by dominated convergence that \[
        \int_t^\infty g(s) \diff s = \int_t^\infty \varphi_s(ut / s) \e^{(ut)^2 / (2s)}g(s) \diff s,
        \] for every $u,t > 0$. Taking a derivative in $t$ for fixed $u$ gives, reusing \cref{eqprf:bb:1}, \[
        -g(t) = -g(t) \varphi_t(u) \e^{u^2 t / 2} + \underbrace{\int_t^\infty \partial_t \left( \varphi_s(ut / s) \e^{(ut)^2 / (2s)} \right) g(s) \diff s}_{= 0} ,
    \] that is $\varphi_t(u) = \e^{-u^2 t / 2}$ and therefore $\nu_t = \mathcal{N}(0,t)$.

    Consider the second statement with $x_0=x_1=s_0=0$ and $s_1 = 1$, then by similar arguments we need to prove that 
    \begin{equation}\label{eqprf:bb:2}
        B^{(0,0) \to (0,1)}(\cdot ) = \int_0^1 \int_\R B^{(0,0) \to (x,s)\to (1,0)}(\cdot ) \nu_s(\diff x) g(s)\diff s
    \end{equation}
    implies $\nu_s = \mathcal{N}(0, s(1-s))$, with assumptions on $\nu_s$ and $g$ analogous to before. Consider the measurable map $C([0,1],\R) \to C([0,\infty),\R)$ that maps a path $y$ to $t \mapsto (1+t) y(\frac{t}{1+t})$; it sends $B^{(0,0)\to (0,1)}$ to $B^{(0,0)+}$, and $B^{(0,0) \to (x,s) \to (0,1)}$ to $B^{(0,0) \to (\frac{x}{1-s}, \frac{s}{1-s})+}$ for $x\in\R$, $s\in (0,1)$. The easiest way to see this is that the pushforward in both cases is a Gaussian process with the correct mean and covariance. Applying this map to \cref{eqprf:bb:2} and substituting $u = \frac{s}{1-s}$ gives \[
        B^{(0,0)+}(\cdot ) = \int_0^\infty \int_\R B^{(0,0) \to ((1+u)x,u)+}(\cdot ) \nu_{u / (1+u)}(\diff x) \frac{g(\frac{u}{1+u})}{(1+u)^2}\diff u.
    \] By what we've already proved this implies $\nu_{u / (1+u)}(\diff x) \sim p_u(\frac{x}{1+u})\diff x$, that is \[
    \nu_s(\diff x) \sim p_{s / (1-s)}(\frac{x}{1-s}) \diff x \sim p_{s(1-s)}(x)\diff x.
    \] 
\end{proof}

\section{Characterisation of Label Invariance (Lemma~\ref{lem:labelinv})}\label{sec:prooflabelinv}

Recall the following characterisation of label invariance of a Brownian spatial coalescent, stated as Lemma~\ref{lem:labelinv} in the main paper.

\lemlabelinv*
%\begin{lemma}\label{lem:labelinv}
%    The Brownian spatial coalescent with transition measures $\boldsymbol \nu\in \rates$ is label invariant if and only if for all $\fpart_0,\fpart_1\in \mathcal{P}$ of equal size and $\iota \in \labelinv(\fpart_0,\fpart_1)$, \[
%        \forall \fpart' < \fpart_0\colon  \iota \# \nu_{\fpart_0,\fpart'} = \nu_{\iota(\fpart_0),\iota(\fpart')}.
%    \]
%    In that case, there exists for every $(n,\vec{k}) = (n,k_1, \ldots ,k_m) \in \mergers$ a finite measure $\nunk$ on $E^m_\circ $ such that for every $(n,\vec{k})$-merger $(\fpart,\fpart')$, \[
%        \nkmap_{\fpart,\fpart'} \# \nu_{\fpart,\fpart'} = \nunk.
%    \] If $k_i = k_j$, then $\nunk$ is symmetric in the $i$th and $j$th coordinate.
%\end{lemma}
\begin{proof}
    Let $\boldsymbol \nu \in \nurates$ be the transition measures of a Brownian spatial coalescent.
    We assume that $\nut{\fpart} > 0$ for all $\fpart \in \mathcal{P}$ with $|\fpart| \ge 2$ to slightly simplify the proof. Let $\fpart_0,\fpart_1 \in \mathcal{P}$ be of equal size.
    Note that the right-hand condition is equivalent by a simple induction to $\iota\# \boldsymbol \nu_F = \boldsymbol \nu_{\iota(F)}$ for all $\iota\in \labelinv(\fpart_0,\fpart_1)$ and $F\in \mathbb{F}(\fpart_0)$. Let $\iota\in \labelinv(\fpart_0,\fpart_1)$, then \eqref{eq:labelinv} from the main paper (the equation characterising label invariance) %\cref{eq:labelinv}
    for $\boldsymbol x\in E^{\fpart_0}_\circ $ is equivalent to $P^{\iota(\boldsymbol x)}(F^\star \in \cdot ) = P^{\boldsymbol x}(\iota(F^\star) \in \cdot )$, that is \[
        \fnu(\tau,\xi \,\vert\,\iota(\boldsymbol x)) \boldsymbol \nu_{F}(\diff \xi) \diff \tau = \fnu(\iota^{-1}(\tau),\iota^{-1}(\xi) \,\vert\,\boldsymbol x) (\iota\#\boldsymbol \nu_{\iota^{-1}(F)})(\diff \xi) \diff \tau
    \] as measures on $\dcb(F)$ for every $F\in \mathbb{F}(\fpart_1)$, which in turn is equivalent to
    \begin{align}\begin{split}\label{eqprf:labelinv:1}
        %&\hspace{2cm}\nut{\rt(F)} = 0 \quad \text{iff}\quad \nut{\iota^{-1}(\rt(F))} = 0, \quad \text{and otherwise}\\
        \prod_{(\fpart,\fpart')\in F} \e^{-\nut{\fpart}(\tau_{\fpart'}-\tau_{\fpart})} \boldsymbol \nu_F(\diff \xi) \diff \tau = \prod_{(\fpart,\fpart') \in F} \e^{-\nut{\iota^{-1}(\fpart)}(\tau_{\fpart'}-\tau_{\fpart})} (\iota \# \boldsymbol \nu_{\iota^{-1}(F)})(\diff \xi) \diff \tau
    \end{split}\end{align}
    for all $F\in \mathbb{F}(\fpart_1)$. Clearly $\boldsymbol \nu_F = \iota \# \boldsymbol \nu_{\iota^{-1}(F)}$ for all $F\in \mathbb{F}(\fpart_1)$ implies \cref{eqprf:labelinv:1}. If label invariance holds,
    %then \cref{eqprf:labelinv:1} implies that $\boldsymbol \nu_F = 0$ iff $\iota \# \boldsymbol \nu_{\iota^{-1}(F)} = 0$. Otherwise,
    then integrating out $\xi \in \dcs(F)$ and varying $\tau$ in \cref{eqprf:labelinv:1} implies $\nut{\fpart} = \nut{\iota^{-1}(\fpart)}$ for all $(\fpart,\fpart') \in F$, which back into \cref{eqprf:labelinv:1} implies $\iota\# \boldsymbol \nu_{\iota^{-1}(F)} = \boldsymbol \nu_F$.

    Now suppose label invariance holds, and let $n\ge 2$ and $k_1\ge\ldots \ge k_m \ge 2$ with $\sum_i k_i \le n$. For existence of $\nunk$ it suffices to show that $\kappa_{\fpart,\fpart'}\#\nu_{\fpart,\fpart'}$ is the same for every $(n,\vec{k})$-merger $(\fpart,\fpart')$. Let $(\fpart_0,\fpart'_0)$ and $(\fpart_1,\fpart_1')$ be $(n,\vec{k})$-mergers, and abbreviate $\kappa_i \coloneqq \kappa_{\fpart_i,\fpart_i'}$ and $\nu_i \coloneqq \nu_{\fpart_i,\fpart_i'}$ for $i=1,2$. Fix some $\iota\in \labelinv(\fpart_0, \fpart_1)$ for which $\iota(\fpart_0') = \fpart_1'$, then
    \begin{align}\label{eqprf:labelinv:2}
        \kappa_1\# \nu_1
        = \kappa_1 \# (\iota \# \nu_0)
        = (\kappa_1 \circ \iota) \# \nu_0.
        %= (\kappa_1\circ \iota \circ \kappa_0^{-1})\# (\kappa_0\#\nu_0).
    \end{align}
    Here $\kappa_1\circ \iota\colon E^{\fpart_0'\setminus \fpart_0}_\circ  \to E^m_\circ $ depends on the choices made for $\kappa_0$ and $\iota$, but is always of the same form as $\kappa_0$, that is $\boldsymbol x \mapsto \boldsymbol x \circ \ell$ for some bijection $\ell\colon [m] \to \fpart'\setminus \fpart$ such that $\ell(i)$ is the union of $k_i$ disjoint elements of $\fpart_0$. If we show that $(\,\cdot \, \circ \ell)\# \nu_0$ is the same for each such choice for $\ell$, then both equality of \cref{eqprf:labelinv:2} to  $\kappa_0\#\nu_0$, and the claimed symmetry of $\nunk$ follow. If two such bijections $\ell,\ell'\colon [m] \to \fpart'\setminus \fpart$ are given, then there is a permutation $\pi$ of $\fpart'\setminus \fpart$ such that $\pi \circ \ell = \ell'$, and $\pi(u)$ for $u\in \fpart'\setminus \fpart$ is the result of merging as many blocks from $\fpart$ as $u$. % Then it suffices to show that $(\,\cdot \,\circ \pi) \# \nu_0 = \nu_0$.
\end{proof}

\section{Conditional Kernel of the Brownian Spatial Coalescent (Lemma~\ref{lem:kbr})}\label{sec:lemkbr}
Recall Lemma~\ref{lem:kbr} from the main paper, which states that the kernel $K_{\boldsymbol x}(F^\star,\cdot )$, which describes the law of the Brownian spatial coalescent started at $\boldsymbol x$ conditioned on $\dc = F^\star$, is well-defined and well-behaved.
\lemkbr*
%\begin{lemma}\label{lem:kbr}
%    Given $F^\star = (F,\tau,\xi) \in \dc(\mathbb{F})$ and $\boldsymbol x \in E_\circ ^{\lf(F)}$, there is a unique law $\kbr_{\boldsymbol x}(F^\star,\cdot )\in \mathcal{M}_1(\Omega)$ under which $\boldsymbol X_0 = \boldsymbol x$ and $\dc = F^\star$ a.s., and \[
%    \left( \pth^{F^\star}_u(\boldsymbol X)\colon u\in \nd(F) \right)
%\] is a family of independent random variables such that the law of $\pth^{F^\star}_u(\boldsymbol X)$ is $B^{(\tau_u,(\boldsymbol x\xi)_u)+}$ if $\pr_F(u) = \emptyset $, and $B^{(\tau_u,(\boldsymbol x\xi)_u)\to(\tau_{\pr_F(u)},\xi_{\pr_F(u)})}$ if $\pr_F(u) \neq \emptyset $.
%    There exists an extension of these laws to a family $(\kbr_{\boldsymbol x}(F^\star,\cdot )\colon \boldsymbol x\in \mathcal{X},F^\star \in \dcb(\mathbb{F}))$ such that
%    \begin{align*}
%        %\{ (\boldsymbol x,F^\star)\colon \boldsymbol x\in E_\circ ^{\lf(F)} \} &\to \mathcal{M}_1(\Omega);\quad %\hspace{3cm}\\
%        \mathcal{X}\times \dcb(\mathbb{F}) \to \mathcal{M}_1(\Omega);\quad (\boldsymbol x,F^\star) \mapsto \kbr_{\boldsymbol x}(F^\star,\cdot )
%    \end{align*}
%    is continuous. In particular, $(\boldsymbol x,F^\star) \mapsto \kbr_{\boldsymbol x}(F^\star,A)$ is measurable for any measurable $A\subset \Omega$.
%\end{lemma}
\begin{proof}[Proof of \cref{lem:kbr}]
    Uniqueness follows because the collection of maps $(\pth^{F^\star}_u\colon u\in \nd(F))$ uniquely determines an element of $\left\{ \dc =F^\star \right\} $.
    For existence and the remaining properties, let \[
        \Big((B_t^u)_{t\in [0,1]}\colon u\subset \N\Big),\qquad \Big( (W_t^u)_{t\ge 0} \colon u\subset \N\Big),
        \] be independent families of, respectively, i.i.d.\ standard Brownian bridges and i.i.d.\ standard Brownian motions, defined on some common probability space $(S,\mathcal{A},\P)$. For any $\boldsymbol x$ and $F^\star = (F,\tau,\xi) \in \dcb(\mathbb{F})$, we define a random element $(\fpart_t,\boldsymbol X_t) \in \left\{ \dc = F^\star \right\} $ with $\boldsymbol X_0 = \boldsymbol x$ in the following way: $(\fpart_t) \coloneqq  \tm^{-1}(F,\tau)$, and if $u\in \nd(F)$ and $t \in [\tau_u,\tau_{\pr_F(u)})$, \[
        \boldsymbol X_t(u) \coloneqq 
        \begin{cases}
            (\boldsymbol x \xi)_u + (\xi_{\pr_F(u)} - (\boldsymbol x\xi)_u) B^u_{(t - \tau_u) / (\tau_{\pr_F(u)} - \tau_u)}, & \pr_F(u) \neq \emptyset ,\\
            (\boldsymbol x \xi)_u + W^u_{t - \tau_u}, & \text{ else}.
        \end{cases}
    \] Then $\boldsymbol X_t \in E^{\fpart_t}$ for all $t \ge 0$, and by standard properties of Brownian motion, there is a $\P$-null set $N_{\boldsymbol x,F^\star}\in \A$ outside of which $\boldsymbol X_t \in E^{\fpart_t}_\circ $ for all $t \ge 0$. On $N_{\boldsymbol x,F^\star}$, we let $(\fpart_t,\boldsymbol X_t)$ be defined by $(\fpart_t) = \tm^{-1}(F,\tau)$, and $\boldsymbol X_t$ linearly interpolates between nodes of the forest (any other particular element of $\left\{ \dc = F^\star \right\} $ starting at $\boldsymbol x$ would do). Then the law $\kbr_{\boldsymbol x}(F^\star,\cdot )$ of $(\fpart_t,\boldsymbol X_t)$ has the stated properties.

    Let $F = (\fpart_0, \ldots ,\fpart_m) \in \mathbb{F}$ be fixed, $\tau^n \to \tau$ in $\dct(F)$ (recall from Definition~\ref{def:tau} that this is in the upper limit topology), $\xi^n \to \xi$ in $\dcs(F)$, and $\boldsymbol x_n \to \boldsymbol x$ in $E_\circ ^{\lf(F)}$, $F^\star_n \coloneqq (F,\tau^n,\xi^n)$, $F^\star \coloneqq (F,\tau,\xi)$, and denote the associated paths by $(\fpart^n_t,\boldsymbol X^n_t),\,n\in \N,$ and $(\fpart_t,\boldsymbol X_t)$. This gives a coupling of the laws $\kbr_{\boldsymbol x_n}(F^\star_n,\cdot ),\, n\in \N$, and $\kbr_{\boldsymbol x}(F^\star,\cdot )$ on the probability space $(S, \A, \P)$, so almost-sure convergence on $S$ implies weak convergence of the laws. Put $\tau_i \coloneqq \tau(\fpart_i)$ for $i\in \left\{ 0, \ldots ,m \right\} $, and $\tau_{m+1}\coloneqq \infty$, and similarly for $\tau_i^n,\, n\in \N$. 
    We show that $(\fpart^n,\boldsymbol X^n) \to (\fpart,\boldsymbol X)$ in $\Omega$ holds outside $N_{\boldsymbol x,F^\star}\cup\bigcup_{n\in \N} N_{\boldsymbol x_n,F^\star_n}$.
    Let $t > 0$. By construction, the only discontinuities of $(\fpart,\boldsymbol X)$ are the jump times of $(\fpart_t)$. If $t$ is no such jump time, say $\tau_i < t < \tau_{i+1}$ for some $i \in \left\{ 0, \ldots ,m \right\} $, then for large $n$ also $\tau^n_i < t < \tau^n_{i+1}$, so $\fpart^n(t) = \fpart_i = \fpart(t)$. If $t$ is a jump time, say $t = \tau_i$ for some $i \in [m]$, then (recalling that $\dct(F)$ is equipped with the upper limit topology) $\tau_{i}^n \le \tau_i = t < \tau_{i+1}^n$, and hence $\fpart^n(t) = \fpart(t)$, for sufficiently large $n$.
    Now say $\fpart \coloneqq \fpart(t)$, and without loss of generality $\fpart^n(t) = \fpart$ for all $n\in \N$.
    It remains to show that then $\boldsymbol X^n(t)_u \to \boldsymbol X(t)_u$ for any $u\in \fpart$. If $\pr_F(u) = \emptyset $, then
    \begin{align*}
        \left| \boldsymbol X^n(t)_u - \boldsymbol X(t)_u \right| 
        %&= \left| \xi^n_u + W^u_{t_n - \tau^n_u} - \xi_u - W^u_{t - \tau_u} \right| \\
         &\le \left| (\boldsymbol x_n\xi^n)_u - (\boldsymbol x\xi)_u \right| + \left| W^u_{t-\tau_u^n} - W^u_{t-\tau_u} \right| 
        \tendsto{n\to \infty} 0,
    \end{align*}
    and similarly if $\pr_F(u) \neq \emptyset $.

    We have proved that whenever $F_n^\star \to F^\star$ in $\dcb(\mathbb{F})$, and $\boldsymbol x_n \to \boldsymbol x$ in $E^{\lf(F)}_\circ $, then $\kbr_{\boldsymbol x_n}(F^\star_n,\cdot ) \to \kbr_{\boldsymbol x}(F^\star,\cdot )$ weakly. Now fix arbitrary choices $(\boldsymbol x_F \in E^{\lf(F)}_\circ)_{F\in \mathbb{F}} $, and put for $F\in \mathbb{F}$ and $\boldsymbol x \in \mathcal{X}$ with $\boldsymbol x \not\in E^{\lf(F)}_\circ $, \[
        \kbr_{\boldsymbol x}(F^\star,\cdot ) \coloneqq \kbr_{\boldsymbol x_F}(F^\star,\cdot ).
    \] This defines $\kbr_{\boldsymbol x}(F,\cdot )$ for any $\boldsymbol x\in \mathcal{X}$ and $F \in \mathbb{F}$. Now let $F^\star_n \to F^\star$ in $\dcb(\mathbb{F})$, and $\boldsymbol x_n \to \boldsymbol x$ in $\mathcal{X}$. If $\boldsymbol x\in E_\circ ^{\lf(F)}$, then for large $n$ we must have $\dom(\boldsymbol x_n) = \dom(\boldsymbol x) = \lf(F)$ and $F_n = F$, so also $\boldsymbol x_n \in E_\circ ^{\lf(F_n)}$, in which case $\kbr_{\boldsymbol x_n}(F^\star_n,\cdot ) \to \kbr_{\boldsymbol x}(F^\star,\cdot )$ is already proved. If $\boldsymbol x \not\in E_\circ ^{\lf(F)}$, then by the same argument also $\boldsymbol x_n \not\in E_\circ ^{\lf(F_n)}$ and $F_n = F$ for large $n$, so \[
        \kbr_{\boldsymbol x_n}(F^\star_n,\cdot ) = \kbr_{\boldsymbol x_{F_n}}(F^\star_n,\cdot ) = \kbr_{\boldsymbol x_F}(F^\star_n,\cdot ) \tendsto{n\to \infty} \kbr_{\boldsymbol x_F}(F^\star,\cdot ) = \kbr_{\boldsymbol x}(F^\star,\cdot ),
    \] by what we have already proved, since $\boldsymbol x_F \in E^{\lf(F)}_\circ = E^{\lf(F_n)}_\circ $ for large $n$.

    It remains to show that $(\boldsymbol x,F^\star) \mapsto \kbr_{\boldsymbol x}(F^\star,A)$ is measurable for measurable $A \subset \Omega$. The set of such $A$ is a $\lambda$-system, so by the $\pi$-$\lambda$ lemma it suffices to show the claim for $A$ taken from a $\pi$-system that generates the $\sigma$-algebra on $\Omega$, such as the family of closed sets. By what we have already showed and the Portmanteau theorem, if $A \subset \Omega$ is closed, and $\boldsymbol x_n \to \boldsymbol x$ and $F^\star_n \to F^\star$, \[
        \varlimsup_{n\to \infty} \kbr_{\boldsymbol x_n}(F^\star_n,A) \le \kbr_{\boldsymbol x}(F^\star,A),
    \] which implies that $(\boldsymbol x,F^\star) \mapsto \kbr_{\boldsymbol x}(F^\star,A)$ is upper semi-continuous and hence measurable.
\end{proof}

%\section{Proof of Lemma~\ref{lem:NF}}\label{app:NF}
\section{$N^{\nu}$ is finite and continuous (Lemma~\ref{lem:NF})}
In this section we prove Lemma~\ref{lem:NF} from the main paper.
\lemNF*
Throughout the section, fix $\boldsymbol \nu \in \nurates$ with $\boldsymbol \nu(\diff \xi) = \boldsymbol \lambda \diff \xi$ for some $\boldsymbol \lambda \in \rates$, that is $\nu_{\fpart,\fpart'}(\diff \xi) = \lambda_{\fpart,\fpart'} \diff \xi$ for every $\fpart,\fpart'\in \mathcal{P}$, $\fpart' < \fpart$. Then $N^F_{\boldsymbol \nu}$ for $F\in \mathbb{F}$ is a map $E^{\lf(F)}_\circ \to (0,\infty)$ defined by \[
    N^F_{\boldsymbol \nu}(\boldsymbol x) = \int\limits_{\dct(F)} \int\limits_{\dcs(F)} \ftm(\tau) \fsp(\xi \,\vert\,\tau,\boldsymbol x) \diff \xi \diff \tau.
\] The goal of this section is to show that $N^F_{\boldsymbol \nu}$ is finite and continuous.
\begin{lemma}\label{lem:NFinnerintegralcontinuous}
    Let $F\in \mathbb{F}$. For fixed $\tau \in \dct(F)$, the map \[
        E^{\lf(F)}_\circ \to (0,\infty);\quad \boldsymbol x \mapsto \smashoperator[r]{\int_{\dcs(F)}} \fsp(\xi \,\vert\,\tau,\boldsymbol x) \diff \xi
    \] is continuous.
\end{lemma}
\begin{proof}
    For fixed $\tau \in \dct(F)$, $\fsp(\xi \,\vert\,\tau,\boldsymbol x)$ is continuous in $\boldsymbol x$ for every $\xi \in \dcs(F)$, and bounded uniformly in $\boldsymbol x$ and $\xi$, so the statement follows from the dominated (or bounded) convergence theorem.
\end{proof}

Let $\fpart\in \mathcal{P}$. Then $E^{\fpart}_\circ $ is an open subset of $E^\fpart$, so the distance between some $\boldsymbol x\in E^\fpart$ to $E^\fpart \setminus E^\fpart_\circ $ is well-defined and denoted $\varepsilon_d(\boldsymbol x)$. More explicitly,
\begin{equation}\label{eq:epsd}
    \varepsilon_d(\boldsymbol x) =
    \begin{cases}
        \min_{u\neq v\in \fpart} \rho(\boldsymbol x_u , \boldsymbol x_v), & d \ge 2,\\
        \max_{u_0,v_0\in \fpart} \min_{u\neq v\in \fpart, \{ u,v \} \neq \left\{ u_0,v_0 \right\} } \rho(\boldsymbol x_{u},\boldsymbol x_{v}), & d = 1,
    \end{cases}
\end{equation}
where $\rho$ denotes the Euclidean metric on $E$, see \eqref{eqdef:rho}.
If $d\ge 2$ then $\varepsilon_d(\boldsymbol x)$ is the smallest, if $d = 1$ it is the second smallest distance among all pairs of two particles. % Then $\varepsilon_d$ is continuous, $\varepsilon_d(\boldsymbol x) \ge 0$, and equality holds iff $\boldsymbol x \not\in E^\fpart_\circ $.

\begin{lemma}\label{lem:app:NF}
    Let $F \in \mathbb{F}$. Then there exists $C = C(F,\boldsymbol \lambda,d) > 0$ such that, for any $A \subset E^{\lf(F)}_\circ $,
    \begin{equation}\label{eq:NFboundtorus}
        \smashoperator{\int_{\dct(F)}} \ftm(\tau) \sup_{\boldsymbol x\in A} \bigg( \smashoperator[r]{\int_{\dcs(F)}} \fsp(\xi \,\vert\,\tau,\boldsymbol x) \diff \xi \bigg) \diff \tau
        \le C \bigg[ 1 + \left( \inf_{\boldsymbol x\in A} \varepsilon_d(\boldsymbol x) \right) ^{-C} \bigg] .
        \end{equation}
        In particular, $N_{\boldsymbol \nu}^F \colon E^{\lf(F)}_\circ \to (0,\infty)$ is continuous and, specialising $A$ in \cref{eq:NFboundtorus} to a singleton, \[
        N^F_{\boldsymbol \nu}(\boldsymbol x) \le C \left( 1 + \varepsilon_d(\boldsymbol x)^{-C} \right) ,\qquad \boldsymbol x\in E^{\lf(F)}_\circ .
    \]
\end{lemma}

The continuity of $N^F_{\boldsymbol \nu}$ follows from \cref{eq:NFboundtorus} and \cref{lem:NFinnerintegralcontinuous} by the dominated convergence theorem. We start by proving a version of~\eqref{eq:NFboundtorus} where the torus is replaced by $\R^d$, where a certain heat kernel trick can be used that isn't available on the torus. Then we transfer the result back to the torus. Write $\rfy{E} = \R^d$, and denote by $\rfy{p}_t \colon \R^d \to (0,\infty)$ for $t > 0$ the heat kernel, defined by \[
    \rfy{p}_t(x) = (2\pi t)^{-d / 2} \exp\left(-\frac{|x|^2}{2t}\right),\qquad x\in \R^d.
\] Define $\rfy{E}^{\fpart}_\circ $ for $\fpart\in \mathcal{P}$ analogously to $E^{\fpart}_\circ $, and for a forest $F\in \mathbb{F}$, denote by $\rfy{\dcs}(F)$ the set of maps $\xi\colon \nd^\circ (F) \to \rfy{E}$ such that $\xi\big\vert_{\fpart} \in \rfy{E}^\fpart_\circ $ for all $\fpart \in F$. For $\xi\in \rfy{\dcs}(F)$, $\tau\in \dct(F)$ and $\boldsymbol x\in \rfy{E}^{\lf(F)}_\circ $, write
\begin{equation*}
    \rfy{\fsp}(\xi\,\vert\, \tau,\boldsymbol x) \coloneqq  \smashop{\prod_{u\in \nd(F)\setminus \rt(F) }} \rfy{p}(\tau_{\pr_F(u)} - \tau_u, (\boldsymbol x \xi)_u - \xi_{\pr_F(u)}).
\end{equation*}
For every $x,y\in E$ there exists a $\vec{k}(x,y) \in \{-1,0\}$ (not necessarily unique) such that $\rho(x,y) = |x-y+\vec{k}(x,y)|$ (on the left is the Euclidean metric on the torus $E$, on the right is the Euclidean norm on $\R^d$). Here and in the following, if we treat an element $x\in E$ as an element of $\R^d$, then we identify it with its unique representative in $[0,1)^d$. Then, for some $C = C(d) > 0$,
\begin{align}\label{eq:app:ptbound}
    p_t(x-y)
    &\le C(1 + \sqrt{t} )^d \rfy{p}_t(x-y+\vec{k}(x,y))\nonumber \\
    &\le C(1+\sqrt{t} )^d \smashop{\sum_{\vec{k}\in \{-1,0,1\}^d}} \rfy{p}_t(x - y + \vec{k}),
\end{align}
where the first inequality follows e.g.\ from \cite{heatkerneltorus}, Theorem 8.

%For every $x\in E$ there exists a $\vec{k}(x) \in \{-1,0\}^d$ (not necessarily unique) such that $\left\|x\right\| = |x +  \vec{k}(x)|$ (on the left is the Euclidean norm on the torus $E$, on the right is the Euclidean norm on $\R^d$). Here and in the following, if we treat an element $x\in E$ as an element of $\R^d$, then we identify it with its unique representative in $[0,)^d$. Then, for some $C = C(d) > 0$,
%\begin{equation*}
%    p_t(x) \le C(1 + \sqrt{t} )^d \rfy{p}_t(x +  \vec{k}(x)) \le C (1 + \sqrt{t} )^d \smashop{\sum_{\vec{k} \in \{-1,0\}^d}} \rfy{p}_t(x +  \vec{k}),
%\end{equation*}
%where the first inequality follows e.g.\ from \cite{heatkerneltorus}, Theorem 8. Similarly if $x,y\in E$, then
%\begin{equation}\label{eq:app:ptbound}
%    p_t(x-y) \le C(1+\sqrt{t} )^d \sum_{\vec{k}\in \{-1,0,1\}^d} \rfy{p}_t(x - y +  \vec{k}).
%\end{equation}

\begin{lemma}
    For every $F = (\fpart^F_0, \ldots ,\fpart^F_m) \in \mathbb{F}$ there is $C = C(d,F) > 0$ such that for every $\tau\in \dct(F)$ and $\boldsymbol x\in E^{\lf(F)}_\circ $,
    \begin{equation}\label{eq:appNF:fsp1}
        \fsp(\xi \,\vert\,\tau,\boldsymbol x) \le C \left(1 + \sqrt{\max \tau} \right)^C \smashop{\sum_{\substack{\vec{\boldsymbol \ell}\colon \nd^\circ (F) \to  \{-m, \ldots ,m\}^{d} \\ \vec{\boldsymbol k} \colon \lf(F) \to \{-m, \ldots ,m\}^d}  }} \rfy{\fsp}(\xi +  \vec{\boldsymbol \ell}  \,\vert\,\tau,\boldsymbol x +  \vec{\boldsymbol k})
        \end{equation}
        for all $\xi \in \dcs(F)$, and
        \begin{equation}\label{eq:appNF:fsp2}
            \smashop{\int_{\dcs(F) }} \fsp(\xi \,\vert\,\tau,\boldsymbol x) \diff \xi \le C (1 + \sqrt{\max \tau})^C \smashoperator[l]{\sum_{\vec{\boldsymbol k} \colon \lf(F) \to \{-m, \ldots ,m\}^d}}\,\, \int\limits_{\rfy{\dcs}(F) } \rfy{\fsp}(\xi \,\vert\,\tau,\boldsymbol x +  \vec{\boldsymbol k}) \diff \xi.
        \end{equation}
\end{lemma}
\begin{proof}
    Denote by $C > 0$ a constant depending only on $d$ and $F$ whose value may increase from line to line. Let $v\in \rt(F)$ with $\ch_F(v) \neq \emptyset $, and suppose for simplicity that $\ch_F(u) \neq \emptyset $ for all $u\in \ch_F(v)$. Then, using \cref{eq:app:ptbound},
    \begin{align*}
        \prod_{u\in \ch_F(v)} p_{\tau_v-\tau_u}(\xi_u - \xi_v)
        &\le \prod_{u\in \ch_F(v)} \left( C (1 + \sqrt{\tau_v - \tau_u}  ) ^d \sum_{\vec{k}_u\in \{-1,0,1\}^d} \rfy{p}_{\tau_v-\tau_u}(\xi_u - \xi_v +  \vec{k}_u)\right)\\
        &\le C (1 + \sqrt{\max \tau} )^{C} \prod_{u\in \ch_F(v)} \sum_{\vec{k}_u \in \{-1,0,1\}^d} \rfy{p}_{\tau_v-\tau_u}[(\xi_u +  \vec{k}_u) - \xi_v].
    \end{align*}
    We put $\vec{\boldsymbol \ell}_u \coloneqq \vec{k}_u$ for $u\in \ch_F(v)$. Now let $u\in \ch_F(v)$ and suppose again for simplicity that $\ch_F(w) \neq \emptyset $ for all $w\in \ch_F(u)$. Then
    \begin{align*}
        \prod_{w\in \ch_F(u)}& p_{\tau_u - \tau_w}(\xi_w - \xi_u)\\
        &= \prod_{w\in \ch_F(u)} p_{\tau_u - \tau_w}[\xi_w +  \vec{\boldsymbol \ell}_u - (\xi_u +  \vec{\boldsymbol \ell}_u)]\\
        &\le C (1 + \sqrt{\max \tau} )^C \smashoperator[l]{\prod_{w\in \ch_F(u)}} \smashoperator[r]{\sum_{\vec{k}_w \in \{-1,0,1\}^d}} \rfy{p}_{\tau_u - \tau_w} [(\xi_w +  \vec{k}_w +  \vec{\boldsymbol \ell}_u) - (\xi_u +  \vec{\boldsymbol \ell}_u)].
    \end{align*}
    We put $\vec{\boldsymbol \ell}_w \coloneqq \vec{k}_w + \vec{\boldsymbol \ell}_u$ for $w\in \ch_F(u)$. Proceeding inductively in this way down towards the leaves yields \cref{eq:appNF:fsp1}, from which the second claim easily follows (note that the integral on the left-hand side is w.r.t.\ Lebesgue measure on the torus, and on the right-hand side w.r.t.\ Lebesgue measure on all of $\R^d$).

    %details for the second claim: we first apply the bound inside the integral (which is still wrt Lebesgue measure on torus) so the integral turns into an integral against lebesgue measure on a torus exactly shifted by vec{l}, which can be upper bounded by an integral against the full lebesgue measure, making l vanish. So we get the same for every l, and C absorbs the combinatorial number (2m+1)^|nd(F)|.
    %the constant C absorbs the combinatorial constant from the sum
\end{proof}

This reduction lets us make use of the following statement. Define $\rfy{\varepsilon_d}(\boldsymbol x)$ for $\boldsymbol x\in \rfy{E}^\fpart$ for some $\fpart\in \mathcal{P}$ analogously to $\varepsilon_d$, by replacing $\rho(\cdot ,\cdot )$ in \cref{eq:epsd} with the Euclidean distance $|\cdot-\cdot  |$ on $\R^d$.

\begin{lemma}\label{lem:NFRd}
    Let $F \in \mathbb{F}$ and $K > 0$. Then there exists $C = C(F,\boldsymbol \lambda,d, K) > 0$ such that, for any $A \subset \rfy{E}^{\lf(F)}_\circ $,
    \begin{equation*}%\label{eq:NFboundRd}
        \smashoperator{\int_{\dct(F)}} \ftm(\tau) (1 + \sqrt{\max \tau} )^K \sup_{\boldsymbol x\in A} \bigg( \smashoperator[r]{\int_{\rfy{\dcs}(F)}} \rfy{\fsp}(\xi \,\vert\,\tau,\boldsymbol x) \diff \xi \bigg) \diff \tau
        \le C \bigg[ 1 + \left( \inf_{\boldsymbol x\in A} \rfy{\varepsilon_d}(\boldsymbol x) \right) ^{-C} \bigg] .
    \end{equation*}
\end{lemma}

The proof of \cref{lem:NFRd} hinges on an explicit representation of $\int \rfy{\fsp}(\xi \,\vert\,\tau,\boldsymbol x) \diff \xi$ based on the following heat kernel trick.
\begin{lemma}\label{lem:heatkerneltrick}
    Let $x_1, \ldots ,x_n \in \R^d$ and $s_1, \ldots ,s_n > 0$. Then, \[
        \prod_{i=1}^n \rfy{p}_{s_i}(x_i-z) = \frac{\rfy{p}_s(z-\overline{x})}{\rfy{p}_s(0)} \prod_{i=1}^n \rfy{p}_{s_i}(x_i - \interp),
    \] where $\frac{1}{s} = \sum_{i=1}^n \frac{1}{s_i}$ and $\interp = \sum_{i=1}^n \frac{s}{s_i}x_i$.
\end{lemma}
\begin{proof}
    By definition of the heat kernel,
    \begin{align}\label{eqprf:NFaux:1}
        \prod_{i=1}^n \rfy{p}_{s_i}(x_i-z)
        &= \left(\prod_{i=1}^n (2\pi s_i)^{-d / 2}\right) \exp \left( -\frac{1}{2} \sum_{i=1}^n \frac{(x_i-z)^2}{s_i} \right) ,
    \end{align}
    and
    \begin{align}\begin{split}\label{eqprf:NFaux:2}
        \sum_{i=1}^n \frac{(x_i-z)^2}{s_i}
        = \frac{z^2}{s} -2 \frac{\interp z}{s}+ \sum_{i=1}^n \frac{x_i^2}{s_i}
        &= \frac{1}{s}(z-\interp)^2 -\frac{1}{s}\interp^2 + \sum_{i=1}^n \frac{x_i^2}{s_i}\\
        &= \frac{1}{s}(z-\interp)^2 + \sum_{i=1}^n \frac{(x_i-\interp)^2}{s_i},
    \end{split}\end{align}
    where in the last step we used that
    \begin{align*}
        \sum_{i=1}^n \frac{(x_i-\interp)^2}{s_i} = \sum_{i=1}^n \frac{x_i^2}{s_i} - 2 \interp \sum_{i=1}^n \frac{x_i}{s_i} + \frac{1}{s} \interp^2 = \sum_{i=1}^n \frac{x_i^2}{s_i} - \frac{1}{s}\interp^2.
    \end{align*}
    Plugging \cref{eqprf:NFaux:2} back into \cref{eqprf:NFaux:1} finishes the proof.
\end{proof}

\begin{lemma}\label{lem:NFaux}
    For a non-trivial forest $F\in \mathbb{F}$ and $\boldsymbol x \in \rfy{E}^{\lf(F)}_\circ $, $\tau \in \dct(F)$,
    \begin{align}\label{eq:NFaux}
        \int_{\rfy{\dcs}(F)} \rfy{\fsp}(\xi \,\vert\,\tau,\boldsymbol x) \diff \xi = \prod_{v\in \nd^\circ (F)} \frac{1}{\rfy{p}_{r_v}(0)} \prod_{u\in \ch_F(v)} \rfy{p}_{r_u + \tau_v - \tau_u} (\interp_v -\interp_u),
    \end{align}
    where $\interp_u = \boldsymbol x_u$ and $r_u = 0$ for $u\in \lf(F)$ and, inductively for $v\in \nd^\circ (F)$, \[
        \frac{1}{r_v} \coloneqq \sum_{u\in \ch_F(v)} \frac{1}{r_u+\tau_v - \tau_u},\qquad \interp_v \coloneqq \sum_{u\in \ch_F(v)} \frac{r_v}{r_u+\tau_v - \tau_u} \interp_u.
    \] The following hold.
    \begin{enumerate}
        \item If $v\in \nd^\circ (F)$ with $\ch(v) \subset \lf(F)$, then \[
                r_v = \frac{\tau_v}{|\ch_F(v)|},\qquad \interp_v = \frac{1}{|\ch_F(v)|} \sum_{u\in \ch_F(v)} \boldsymbol x_u.
        \]
        \item There is $c = c(F) > 0$ such that $c \tau_v \le r_v < \tau_v$ for all $v\in \nd^\circ (F)$.
        \item If $v\in \nd^\circ (F)$ then \[
                \max_{u\in \ch(v)} |\interp_u-\interp_v| \ge \frac{1}{2}\max_{u,u'\in \ch(v)}|\interp_u-\interp_{u'}|.
        \]
    \end{enumerate}
\end{lemma}
\begin{proof}
    \Cref{eq:NFaux} follows by applying \cref{lem:heatkerneltrick} inductively from leaves to root, which is tedious but straightforward.

    \begin{enumerate}
        \item is obvious from the definitions.
        \item If $v$ is not a leaf and $r_u \le \tau_u$ holds for all children, then
            \begin{align*}
                \frac{1}{r_v} \ge \sum_{u\in \ch_F(v)} \frac{1}{\tau_u + \tau_v - \tau_u} = \frac{|\ch_F(v)|}{\tau_v} \ge \frac{2}{\tau_v} > \frac{1}{\tau_v}.
            \end{align*}
            If $v$ is not a leaf and $r_u \ge c \tau_u$ holds for all children, then
            \begin{align*}
                \frac{1}{r_v} \le \sum_{u\in \ch_F(v)} \frac{1}{c \tau_u + \tau_v - \tau_u} \le \sum_{u\in \ch_F(v)} \frac{1}{c(\tau_u + \tau_v - \tau_u)} = \frac{|\ch_F(v)|}{c} \frac{1}{\tau_v}.
            \end{align*}
        \item If $v\in \nd^\circ (F)$ and $u_1,u_2\in \ch(v)$, then \[
                |\interp_{u_1}-\interp_{u_2}| \le |\interp_{u_1} - \interp_v| + |\interp_v - \interp_{u_2}| \le 2 \max_{u\in \ch(v)}|\interp_u-\interp_v|.
        \]
    \end{enumerate}
\end{proof}

\begin{proof}[Proof of \cref{lem:NFRd}]
    Fix a non-trivial forest $F=(\fpart^F_0, \ldots ,\fpart^F_m)\in \mathbb{F}$, and abbreviate $\ch = \ch_F$, $\pr = \pr_F$ and, for $\tau\in \dct(F)$ and $i\in [m]$, $\tau_i = \tau_{\fpart^F_i}$. Throughout the proof, $c,C > 0$ are constants that only depend on $F$, $d$, $\boldsymbol \lambda$, and $K$, and whose value may respectively decrease and increase from line to line. First let $d \ge 2$, then for $\boldsymbol x \in E^{\lf(F)}_\circ $ and $\tau\in \dct(F)$, using \cref{eq:NFaux},
    \begin{align}\label{eqprf:NFlem:4}
        \int \rfy{\fsp}(\xi \,\vert\,\tau,\boldsymbol x) \diff \xi
        %&\le C \prod_{v\in \nd^\circ (F)} r_v^{d / 2} \prod_{u\in \ch(v)} \rfy{p}_{r_u+\tau_v-\tau_u}(\interp_v-\interp_u)\\
         &\le C \prod_{i=1}^m \underbrace{\smashoperator[r]{\prod_{v\in \fpart^F_i\setminus \fpart^F_{i-1}}}\, r_v^{d / 2} \prod_{u\in \ch(v)} \rfy{p}_{r_u+\tau_v-\tau_u}(\interp_v-\interp_u)}_{\eqqcolon G_i(\tau,\boldsymbol x)}.
    \end{align}
    Denote $G^{(i)}(\tau,\boldsymbol x) \coloneqq \prod_{j=i+1}^m G_j(\tau,\boldsymbol x)$, and ${\ftm}^{(i)}(\tau) \coloneqq \prod_{j=i+1}^{m} \lambda_{\fpart^F_{j-1},\fpart^F_j} \e^{-\lambda_{\fpart^F_j}(\tau_j-\tau_{j-1})}$ for $i\in \left\{ 0, \ldots ,m \right\} $, so ${\ftm}^{(0)}= \ftm$ and ${\ftm}^{(m)} \equiv 1$.

    We show inductively that for $i\in \left\{ 1, \ldots ,m \right\} $ there exists $b_i \ge 0$ such that
    \begin{align}\label{eqprf:NFlem:1}
        \int {\ftm}^{(i)}(\tau) (1+\sqrt{\tau_m} )^K\sup_{\boldsymbol x\in A}G^{(i)}(\tau,\boldsymbol x) \diff \tau_{i+1} \ldots \diff \tau_m \le C \tau_1^{-b_i} (1 + \tau_i)^C.
    \end{align}
    For $i = m$ the LHS is equal to $(1+\sqrt{\tau_m} )^K \le 2^K(1 + \tau_m)^K$ and we can choose $b_m = 0$. Suppose we have proved the claim for some fixed $i\in \left\{ 2, \ldots ,m \right\} $, then by induction hypothesis
    \begin{multline*}
        \int {\ftm}^{(i-1)}(\tau) (1+\sqrt{\tau_m} )^K \sup_{\boldsymbol x\in A} G^{(i-1)}(\tau,\boldsymbol x) \diff \tau_i \ldots \diff \tau_m\\
        \le C \tau_1^{-b_i} \int\limits_{\tau_{i-1}}^{\infty} \diff \tau_i\, (1 + \tau_i)^{C}\e^{-\lambda_i(\tau_i - \tau_{i-1})}\sup_{\boldsymbol x\in A}\,\smashoperator[r]{\prod_{v\in \fpart^F_i\setminus \fpart^F_{i-1}}}\, r_v^{d / 2}\prod_{u\in \ch(v)}\rfy{p}_{r_u+\tau_v-\tau_u}(\interp_v-\interp_u),
    \end{multline*}
    where we abbreviated $\lambda_i \coloneqq \lambda_{\fpart^F_i}$. We bound $r_v \le \tau_v = \tau_i$, and $\rfy{p}_{r_u+\tau_v-\tau_u}(\interp_v-\interp_u) \le C (r_u+\tau_v-\tau_u)^{-d / 2}$, and by \cref{lem:NFaux},
    \begin{equation}\label{eqprf:NFlem:3}
        r_u + \tau_v - \tau_u \ge c\tau_u + \tau_v-\tau_u \ge c\tau_v \ge c\tau_1.
    \end{equation}
    Let $l$ be the total number of children of nodes $v\in \fpart^F_i \setminus \fpart^F_{i-1}$. Then our bound is no larger than $\tau_1^{-b_i}$ times
    \begin{align*}
        C \int\limits_{\tau_{i-1}}^{\infty}& \diff \tau_i \, (1 +\tau_i)^{C}\e^{-\lambda(\tau_i-\tau_{i-1})} \tau_i^{|\fpart^F_i\setminus \fpart^F_{i-1}|d / 2} (c\tau_1)^{-ld / 2}\\
                                           &\le C \tau_1^{-ld / 2} \int_0^\infty \e^{-\lambda_i s}(1+\tau_{i-1} + s)^C \diff s\\
                                           &\le C \tau_1^{-ld / 2} (1 + \tau_{i-1})^C,
    \end{align*}
    which completes the inductive step. Using \cref{eqprf:NFlem:1} with $i = 1$, recalling that $r_u = \tau_u = \tau_1$ for $u\in \lf(F)$, and putting $s \coloneqq \tau_1, b\coloneqq b_1, \lambda\coloneqq \lambda_1$, \cref{eqprf:NFlem:4} becomes
    \begin{align*}
        \int \ftm(\tau) & (1+\sqrt{\tau_m} )^K \sup_{\boldsymbol x\in A} \int \rfy{\fsp}(\xi \,\vert\,\tau,\boldsymbol x)\diff \xi\diff \tau\\
                        &\le C \int_0^\infty  s^{-b} (1+ s)^C \e^{-\lambda s} \sup_{\boldsymbol x\in A}\,\smashoperator[r]{\prod_{v\in \fpart^F_1\setminus \fpart^F_{0}}}\,  s^{d / 2}\prod_{u\in\ch(v)}\rfy{p}_{r_u+ s_v- s_u}(\interp_v-\interp_u)\diff s\\
                        &\le C \int_0^\infty  s^{-b} (1+ s)^C \e^{-\lambda s} \sup_{\boldsymbol x\in A}\prod_{v\in \fpart^F_1\setminus \fpart^F_{0}}\prod_{u\in\ch(v)}\rfy{p}_s(\interp_v-\interp_u)\diff s.
    \end{align*}
    Fix any $v_0\in \fpart^F_1\setminus \fpart^F_0$ and pick, for fixed $\boldsymbol x\in A$, $u_0\in \ch(v)$ for which $|\interp_{v_0}-\boldsymbol x_{u_0}|$ is largest, so that $|\boldsymbol x_u - \boldsymbol x_{u'}| \le 2 |\boldsymbol x_{u_0}-\interp_{v_0}|$ for all $u,u'\in \ch(v)$, in particular $2|\boldsymbol x_{u_0}-\interp_{v_0}| \ge \inf_{\boldsymbol x\in A} \min_{u\neq v}|\boldsymbol x_u-\boldsymbol x_v| \eqqcolon D$. Bound $\rfy{p}_s(\interp_v - \interp_u) \le C s^{-d / 2}$ if $(u,v) \neq (u_0,v_0)$, so we can further bound by
    \begin{align}\label{eqprf:NFlem:2}
        %N_F(\boldsymbol x)
        C \int_0^\infty  s^{-b'} (1+ s)^C \e^{-\lambda  s} \e^{-c D^2 /  s } \diff  s \le C \left( 1 + D^{-2b'} \right) ,
    \end{align}
    for some $b' > b$, using \cref{lem:NFaux:1}.

    Now let $d = 1$. The approach above fails when the first transition $\fpart^F_0\to\fpart^F_1$ is a binary merger of two identical particles, say $u$ and $u'$ with parent $v = u\cup u'$, in which case $\interp_v = \interp_u = \interp_{u'} = \boldsymbol x_u$, so $\rfy{p}_{\tau_1}(\interp_v-\interp_u) = \rfy{p}_{\tau_1}(\interp_v-\interp_{u'}) = (2\pi \tau_1)^{-d / 2}$, and the final integral (of the form \cref{eqprf:NFlem:2} but without the factor $\e^{-cD^2 / s}$) diverges. We split $A$ into a finite number of sets according to which pair of leaves are closest to each other (and potentially identical): \[
        A_{uv} \coloneqq \Big\{ \boldsymbol x\in A\colon |\boldsymbol x_u-\boldsymbol x_v| = \,\smashoperator[l]{\min_{\substack{u',v'\in \lf(F)\\ \text{distinct}}}}|\boldsymbol x_{u'}-\boldsymbol x_{v'}| \Big\},
    \] for distinct $u,v\in\lf(F)$. There might be overlap due to ties, but it will always be true that in $A_{uv}$ all particle locations are distinct except possibly $\boldsymbol x_u = \boldsymbol x_v$, and \[
    \forall \boldsymbol x\in A_{uv}\colon \rfy{\varepsilon_1}(\boldsymbol x) = \,\, \smashoperator[l]{\min_{\substack{u'\neq v'\in \lf(F)\\ \{u',v'\} \neq \{u,v\}}}} |\boldsymbol x_{u'}-\boldsymbol x_{v'}|.
\] Since we can bound $\sup_{\boldsymbol x\in A} \le \sum_{u \neq v} \sup_{\boldsymbol x \in A_{uv}}$, it suffices to restrict our attention to $A_{u_0v_0}$ for fixed but arbitrary, distinct $u_0,v_0\in \lf(F)$. With a slight abuse of notation we denote $A = A_{u_0v_0}$ from now. If the first merge event is not a pure binary merge of $u_0$ and $v_0$, that is $\fpart^F_1\setminus \fpart^F_0 \neq \left\{ u_0\cup v_0 \right\} $, then we can copy the proof of the $d\ge 2$ case, and in the final step we choose $v\in \fpart^F_1\setminus \fpart^F_0$ with $v\neq u_0\cup v_0$, so there are $u,u'\in \ch(v)$ with $\left\{ u,u' \right\} \neq \left\{ u_0,v_0 \right\} $ and thus $|\boldsymbol x_u-\boldsymbol x_{u'}| \ge \rfy{\varepsilon_1}(\boldsymbol x)$ for all $\boldsymbol x\in A$. Now suppose the first event is a binary merge of $u_0$ and $v_0$, that is $\fpart^F_1\setminus \fpart^F_0$ is a singleton set containing their parent $w_0\coloneqq u_0\cup v_0$, so
    \begin{align}\label{eqprf:NFlem:5}
    G_1(\tau,\boldsymbol x)
    %&= \prod_{v\in \fpart^F_1\setminus \fpart^F_0} r_v^{d / 2} \prod_{u\in \ch(v)} \rfy{p}_{\tau_1}(\interp_u - \interp_v) \\
     &= (\tau_1 / 2)^{1 / 2} \rfy{p}_{\tau_1}(\boldsymbol x_{u_0} - \interp_{w_0}) \rfy{p}_{\tau_1}(\boldsymbol x_{v_0}-\interp_{w_0})
    \le C \tau_1^{-1 / 2}
    \end{align}
    for all $\boldsymbol x\in A$. Then as before, we can prove inductively that for $i\in \left\{ 2, \ldots ,m \right\} $ there is $b_i \ge 0$ with
    \begin{equation}\label{eqprf:NFlem:6}
        \int {\ftm}^{(i)}(\tau) (1+\sqrt{\tau_m} )^K \sup_{\boldsymbol x \in A} G^{(i)}(\tau,\boldsymbol x) \diff \tau_{i+1}\ldots \diff \tau_m \le C \tau_{2}^{-b_i} (1 + \tau_i)^C,
    \end{equation}
    where we only need to replace the lower bound $c\tau_1$ by $c\tau_2$ in \cref{eqprf:NFlem:3}. Let $s\coloneqq \tau_1$ and $r\coloneqq \tau_2-\tau_1$, then plugging \cref{eqprf:NFlem:5,eqprf:NFlem:6} into \cref{eqprf:NFlem:4} gives
    \begin{align}\begin{split}\label{eqprf:NFlem:7}
        \int \ftm(\tau) & \sup_{\boldsymbol x\in A} \int \rfy{\fsp}(\xi \,\vert\,\tau,\boldsymbol x) \diff \xi \diff \tau\\
        %&\le C \int_0^\infty \diff s \int_0^\infty \diff r (s+r)^{-b}(1+s+r)^C \e^{-\lambda_1 s}\e^{-\lambda_2 r} \sup_{\boldsymbol x\in A} G_1(\tau,\boldsymbol x) G_2(\tau,\boldsymbol x)\\
        &\le C \int_0^\infty \diff s\, \e^{-\lambda_1 s}s^{-1 / 2} \int_0^\infty \diff r\, (s+r)^{-b_2}(1+s+r)^C \e^{-\lambda_2 r} \sup_{\boldsymbol x\in A} G_2(\tau,\boldsymbol x).
    \end{split}\end{align}
    If the second merge involves at least two leaves, say $w\in \fpart^F_2\setminus \fpart^F_1$ with $u,v\in \ch(w)\cap \lf(F)$, then necessarily $\left\{ u,v \right\} \cap \left\{ u_0,v_0 \right\} = \emptyset $, so $\max_{u' \in \ch(w)}|\boldsymbol x_{u'}-\interp_{w}| \ge \frac{1}{2}|\boldsymbol x_u-\boldsymbol x_v| \ge \rfy{\varepsilon_1}(\boldsymbol x)$ for all $\boldsymbol x\in A$ and $\tau$, so we can bound \[
        \sup_{\boldsymbol x\in A} G_2(\tau,\boldsymbol x) \le C (1+\tau_2)^C \tau_2^{-b_1} \e^{-cD^2 / \tau_2}
    \] for some $b_1 \ge 0$, where $D = \inf_{\boldsymbol x\in A}\rfy{\varepsilon_1}(\boldsymbol x)$. With $b\coloneqq b_1+b_2$ we can further bound \cref{eqprf:NFlem:7} by
    \begin{align}\begin{split}\label{eqprf:NFlem:8}
        C \int_0^\infty & \diff s\, \e^{-\lambda_1 s} s^{-1 / 2} \int_0^\infty \diff r\, \underbrace{(s+r)^{-b} \e^{-cD^2 / (s+r)}}_{\le C D^{-2b} \, \forall s,r > 0} (1+s+r)^C \e^{-\lambda_2 r}\\[3pt]
                        &\le C D^{-2b} \int_0^\infty \diff s\, \e^{-\lambda_1 s}s^{-1 / 2} \int_0^\infty \diff r\, (1+s+r)^C \e^{-\lambda_2 r}\\
                        &\le C D^{-2b} \int_0^\infty \e^{-\lambda_1 s}s^{-1 / 2} (1+s)^C \diff s\\
                        &\le C D^{-2b}.
    \end{split}\end{align}
    If the second merge involves at most one additional leaf, then it must be a binary merge of $w_0$ and a leaf $u \not\in \left\{ u_0,v_0 \right\} $ with parent $w_1 = u \cup w_0$. Then $|\boldsymbol x_u - \interp_{w_1}| \ge \frac{1}{2} |\boldsymbol x_u - \interp_{w_0}| \ge c |\boldsymbol x_u - \boldsymbol x_{u_0}| \ge c \varepsilon_1(\boldsymbol x)$ for all $\boldsymbol x\in A$ and $\tau$, where we used \cref{lem:NFaux:2} in the second step. Then we can bound
    \begin{align*}
        G_2(\tau,\boldsymbol x) = \underbrace{r_{w_1}^{1 / 2}}_{\le \tau_2^{1 / 2}} \underbrace{\rfy{p}_{\tau_2 - \tau_1 / 2} (\interp_{w_0}-\interp_{w_1})}_{\le C \tau_2^{-1 / 2}}\underbrace{\rfy{p}_{\tau_2}(\boldsymbol x_u - \interp_{w_1})}_{\le C\tau_2^{-1 / 2}\e^{-cD^2 / \tau_2}}
        %&\le C \rfy{p}_{\tau_2}(\boldsymbol x_u-\interp_{w_1})
                                \le C \tau_2^{-1 / 2} \e^{-c D^2 / \tau_2},
    \end{align*}
    and finish the proof as in \cref{eqprf:NFlem:8}.

    %Continuity now follows: If $\boldsymbol x_n \to \boldsymbol x$ in $E^{\lf(F)}_\circ $, then $\int \rfy{\fsp}(\xi \,\vert\,\tau,\boldsymbol x_n)\diff \xi \to \int \rfy{\fsp}(\xi \,\vert\,\tau,\boldsymbol x) \diff \xi$ as $n\to \infty$ for every $\tau\in \dct(F)$ because the RHS of \cref{eq:NFaux} is continuous. Further $\varepsilon_d(\boldsymbol x_n) \to \varepsilon_d(\boldsymbol x) > 0$ and $\varepsilon_d(\boldsymbol x_n) > 0$ for all $n$, so $\inf_{n\in \N} \varepsilon_d(\boldsymbol x_n) > 0$ and thus \[
    %    \int \ftm(\tau) \sup_{n\in \N} \left( \int \rfy{\fsp}(\xi \,\vert\,\tau,\boldsymbol x_n) \diff \xi \right) \diff \tau < \infty,
    %\] which implies $N_F(\boldsymbol x_n) \to N_F(\boldsymbol x)$ by dominated convergence.
\end{proof}

The following lemma finishes the proof of \cref{lem:app:NF} together with \cref{eq:appNF:fsp2,lem:NFRd}.

\begin{lemma}
    For every $\fpart\in \mathcal{P}$, $\boldsymbol x\in E^\fpart_\circ $, and $\vec{\boldsymbol k}\colon \fpart \to \Z^d$, \[
        \rfy{\varepsilon_d}(\boldsymbol x +  \vec{\boldsymbol k}) \ge \varepsilon_d(\boldsymbol x).
    \]
\end{lemma}
\begin{proof}
    If $u,v\in \fpart$, then
    \begin{align*}
        | \boldsymbol x_u +  \vec{\boldsymbol k}_u - (\boldsymbol x_v +  \vec{\boldsymbol k}_v)| \ge \inf \{|\boldsymbol x_u - \boldsymbol x_v +  \vec{\ell}| \colon :\vec{\ell}\in \Z^d\} = \rho(\boldsymbol x_u,\boldsymbol x_v).
    \end{align*}
    With that, the claim follows directly from the definitions of $\varepsilon_d$ and $\rfy{\varepsilon_d}$.
\end{proof}

We finish by proving some technical lemmas used in the proof of \cref{lem:NFRd}.

\begin{lemma}\label{lem:NFaux:1}
    For $a,b,\lambda,y> 0$,
    \begin{align*}
        \int_0^\infty \e^{-\lambda s - y^2 / s} (1+s)^a s^{-b}\diff s \le C \left( 1 + y^{-2b} \right) ,
    \end{align*}
    where $C = C(a,b,\lambda) > 0$.
\end{lemma}
\begin{proof}
    First note that \[
        \sup_{s > 0}\left( s^{-b} \e^{-y^2 / s}\right) = y^{-2b} \sup_{s > 0}\left( s^{-b} \e^{-1 / s} \right) = C y^{-2b},
    \] so
    \begin{align*}
        \int_0^\infty \e^{-\lambda s - y^2 / s} (1+s)^a s^{-b} \diff s
        &\le 2^a \int_0^1 s^{-b}\e^{-y^2 / s} \diff s + \int_1^\infty \e^{-\lambda s} (1+s)^a \diff s \le C y^{-2b} + C.
    \end{align*}
\end{proof}

\begin{lemma}\label{lem:NFaux:2}
    If $x_1, \ldots ,x_n\in \R^d$ and $|x_1-x_2| \le |x_i-x_j|$ for all distinct $i,j\in [n]$, then \[
        |x_i - \frac{x_1+x_2}{2}| \ge \frac{3}{\sqrt{2} } |x_i - x_1|
    \] for all $i\in \left\{ 3, \ldots ,n \right\} $.
\end{lemma}
\begin{proof}
    Put $a \coloneqq |x_1-x_2|$ and $\overline{x} \coloneqq (x_1+x_2) / 2$, and let $i\in \left\{ 3, \ldots ,n \right\} $. Then $x_i \not\in B(x_1,a) \cup B(x_2,a)$, which means that $|x_i - \overline{x}|$ is minimised if $x_i \in B(x_1,a) \cap B(x_2,a)$, in which case $|x_i - x_1| = |x_i - x_2| = |x_1-x_2| = a$, so by Pythagoras $|x_i - \overline{x}| = \frac{3}{\sqrt{2} } a$.
\end{proof}

\section{Minor Technical Lemmas}
This section contains some minor technical lemmas used in the proofs of the main paper.

\begin{lemma}\label{lem:hmeasurable}
    Let $A$ be a measurable space and $B$ a Borel space, $m$ a probability measure on $B$ with no atoms, and $f \colon A \times B \to (0,\infty)$ measurable with $\int_B f(a,b) m(\diff b) = 1$ for every $a\in A$, so that $m_a(\diff b) \coloneqq f(a,b) m(\diff b)$ defines a probability measure on $B$ for every $a\in A$. Then there exists a measurable function $h\colon A \times B \to B$ such that \[
        \forall a\in A\colon h(a,\cdot )\# m = m_a.
    \] 
\end{lemma}
\begin{proof}
    Let $\varphi\colon B \to [0,1]$ be a Borel isomorphism, \[
        f'\colon A\times [0,1] \to (0,\infty);\, (a,x) \mapsto f(a,\varphi^{-1}(x)),
    \] and $m' \coloneqq \varphi \# m$. Then $\int_{[0,1]} f'(a,x) m'(\diff x) = \int_B f(a,b) m(\diff b) = 1$ for every $a\in A$, so $m_a'(\diff x) \coloneqq f'(a,x) m'(\diff x)$ is a probability measure on $[0,1]$, and \[
        \int_{[0,1]} g(x) m_a'(\diff x) = \int_{[0,1]} g(x) f(a,\varphi^{-1}(x)) m'(\diff x) = \int_B g(\varphi(b)) f(a,b) m(\diff x),
    \] so $m_a' = \varphi \# m_a$. Now if there exists a measurable $h'\colon A \times [0,1] \to [0,1]$ with $h'(a,\cdot ) \# m' = m_a'$ for every $a\in A$, then $h\colon A\times B \to B$ defined by $(a,b) \mapsto \varphi^{-1}(h'(a,\varphi(b)))$ satisfies 
    \begin{align*}
        \int_B g(b) (h(a,\cdot )\# m)(\diff b)
        %= \int_B g(h(a,b)) m(\diff b)
        = \int_B g(\varphi^{-1}(h'(a,\varphi(b)))) m(\diff b)
        &= \int_{[0,1]} g(\varphi^{-1}(h'(a,x))) m'(\diff x)\\
        %&= \int_{[0,1]} g(x) (h'(a,\cdot )\# m')(\diff x)\\
        &= \int_{[0,1]} g(\varphi^{-1}(x)) m_a'(\diff x)\\
        &= \int_B g(b) m_a(\diff b),
    \end{align*}
    so $h(a,\cdot )\# m = m_a$. This shows that it suffices to show the statement if $B = [0,1]$.

    Define $F\colon A \times [0,1] \to [0,1]$ by $F(a,y) = \int_{0}^y f(a,x) m(\diff x)$, which is measurable by Tonelli's theorem. Then $F^{-1}(a,u) \coloneqq \inf \{ y \in [0,1]\colon F(a,y) \ge u\}$ defines a function $A\times [0,1] \to [0,1]$ which is measurable because $F$ is c\`adl\`ag. Furthermore, if $\lambda$ denotes Lebesgue measure on $[0,1]$, then $F^{-1}(a,\cdot ) \# \lambda = m_a$. Let $G\colon [0,1]\to [0,1]$ be defined by $G(y) = \int_0^y m(\diff x)$, so that $G \# m = \lambda$ because $m$ has no atoms. Then $h(a,x) \coloneqq F^{-1}(a,G(x))$ is a measurable map $A\times [0,1]\to [0,1]$, and \[
        h(a,\cdot ) \# m = (F^{-1}(a,\cdot ) \circ G) \# m = F^{-1}(a,\cdot ) \# (G\# m) = F^{-1}(a,\cdot ) \# \lambda = m_a.
    \] 
\end{proof}

\begin{lemma}\label{lem:glem}
    Suppose $h\colon (\Q \cap [0,\infty)) \times A \times A \to \R$ is a function for a set $A$, and that there is a family of sets $(A_x\subset A)_{x\in A}$ such that $A_x \cap A_{x'} \neq \emptyset $ for all $x,x'\in A$, and for every $s,t,\in \Q \cap [0,\infty)$, $x,y\in A$, $z\in A_x$, \[
        g(t+s,x,y) = g(s,x,z) + g(t,z,y).
    \] Suppose further that $g(0,x,y) = b(x) - b(y)$ for some function $b\colon A \to \R$. Then there is $\la \in \R$ such that for all $t \in \Q \cap [0,\infty)$, $x\in A$ and $z\in A_x$, \[
    g(t,x,z) = \la t + b(x)-b(z).
    \]
\end{lemma}
\begin{proof}
    Fix $x\in A$. Then for $t,s\in \Q \cap [0,\infty)$ and $z\in A_x$ we have $g(t+s,x,z) = g(s,x,z) + g(t,z,z)$, and thus \[
        g(t,z,z) + g(s,z,z) = \Big[ g(t,x,z) - g(0,x,z)\Big] + \Big[ g(t+s,x,z) - g(t,x,z)\Big] = g(t+s,z,z),
    \] for any $t,s,\in \Q \cap [0,\infty)$ and $z\in A_x$. This implies that there exists $\la(z) \in \R$ such that $g(t,z,z) = \la(z) t$ for $t\in \Q \cap [0,\infty)$. Then for all $z\in A_x$ and $t \in \Q \cap [0,\infty)$,
    \begin{align}\label{eqprf:glem:1}
        g(t,x,z) = g(0,x,z) + g(t,z,z) = b(x)-b(z) + \la(z) t.
    \end{align}
    We show that $\la(z) = \la(z')$ for all $z,z'\in \bigcup_{x\in A} A_x$. By assumption and \cref{eqprf:glem:1}, for every $z\in A_x $ we have $g(t,x,x) = g(t,x,z) + g(0,z,x) = \la(z) t$, in particular $\la(z) = \la(z') \eqqcolon \la(x)$ for all $z,z'\in A_x$.

    Finally, if $x,x'\in A$ then there is $z\in A_x\cap A_{x'}$, so $\la(x) = \la(z) = \la(x')$, so in fact there is one $\la \in \R$ such that $\la(z) = \la$ for all $z\in \bigcup_{x\in A} A_x$.
\end{proof}

\bibliographystyle{plain}
\bibliography{references}

\end{document}